\documentclass{report} 
\usepackage[utf8]{inputenc}
\usepackage[T1]{fontenc} 
\usepackage{amsmath}
\usepackage{amssymb}
\usepackage{amsthm}
\usepackage{dsfont}
\usepackage[mathscr]{euscript}
\usepackage{stmaryrd}
\usepackage{hyperref}
\usepackage{empheq}
\usepackage{cleveref}
\usepackage{enumitem}
\usepackage{fullpage}
\usepackage{xcolor}
\usepackage{cancel}
\usepackage{tikzsymbols}
\usepackage{tikz}
\usepackage{pgfplots}
\usepackage{pgfplotstable}
\usepackage{multirow}
\pgfplotsset{compat=1.15}
\usepackage{esint}
\usepackage{caption}
\usepackage{authblk}
\usepackage{url}

\usepackage{lmodern}

\SetSymbolFont{stmry}{bold}{U}{stmry}{m}{n}

\theoremstyle{plain}
\newtheorem{Thm}{Theorem}[chapter]
\newtheorem{Cor}[Thm]{Corollary}
\newtheorem{Prop}[Thm]{Proposition}
\newtheorem{Lem}[Thm]{Lemma}
\theoremstyle{definition}
\newtheorem{Def}[Thm]{Definition}

\newtheorem{Rem}[Thm]{Remark}
\newtheorem{Ass}[Thm]{Assumption}
\newtheorem{Conv}[Thm]{Convention}

\theoremstyle{remark}
\newtheorem{Ex}[Thm]{Example}

\newcounter{stepa}
\newenvironment{stepa}[1]{\par \vspace{5pt} \refstepcounter{stepa}\noindent \underline{Step \thestepa}: #1. \par \vspace{2pt}}

\newcounter{stepb}
\newenvironment{stepb}[1]{\par \vspace{5pt} \refstepcounter{stepb}\noindent \underline{Step \thestepb}: #1. \par \vspace{2pt}}

\newcounter{stepc}
\newenvironment{stepc}[1]{\par \vspace{5pt} \refstepcounter{stepc}\noindent \underline{Step \thestepc}: #1. \par \vspace{2pt}}

\newcounter{stepd}
\newenvironment{stepd}[1]{\par \vspace{5pt} \refstepcounter{stepd}\noindent \underline{Step \thestepd}: #1. \par \vspace{2pt}}

\newcounter{stepe}
\newenvironment{stepe}[1]{\par \vspace{5pt} \refstepcounter{stepe}\noindent \underline{Step \thestepe}: #1. \par \vspace{2pt}}

\newcounter{stepf}
\newenvironment{stepf}[1]{\par \vspace{5pt} \refstepcounter{stepf}\noindent \underline{Step \thestepf}: #1. \par \vspace{2pt}}

\newcounter{stepg}
\newenvironment{stepg}[1]{\par \vspace{5pt} \refstepcounter{stepg}\noindent \underline{Step \thestepg}: #1. \par \vspace{2pt}}

\newcounter{steph}
\newenvironment{steph}[1]{\par \vspace{5pt} \refstepcounter{steph}\noindent \underline{Step \thesteph}: #1. \par \vspace{2pt}}

\newcounter{stepi}
\newenvironment{stepi}[1]{\par \vspace{5pt} \refstepcounter{stepi}\noindent \underline{Step \thestepi}: #1. \par \vspace{2pt}}

\newcommand{\R}{\mathbb{R}}

\newcommand{\F}{\mathcal{F}}
\newcommand{\G}{\mathcal{G}}
\newcommand{\T}{\mathbb{T}}
\newcommand{\E}{\mathbf{E}}
\newcommand{\Em}{\mathds{E}}

\newcommand{\XX}{\mathcal{X}}
\newcommand{\N}{\mathbb{N}}
\newcommand{\Prob}{\mathbb{P}}
\newcommand{\V}{\mathcal{V}}
\newcommand{\cg}{\langle}
\newcommand{\cd}{\rangle}
\newcommand{\pf}{{}_\#}
\newcommand{\1}{\mathds{1}}

\newcommand{\cadlag}{\mbox{\it c\`adl\`ag}}
\renewcommand{\P}{\mathcal{P}}
\newcommand{\M}{\mathcal{M}}
\newcommand{\BBM}{\mathsf{BBM}}

\renewcommand{\L}{\mathscr{L}}

\newcommand{\xx}{\mathbf{x}}

\newcommand{\dr}{\partial}
\newcommand{\Leb}{\mathrm{Leb}}
\newcommand{\RUOT}{\mathsf{RUOT}}
\newcommand{\ruot}{\mathsf{Ruot}}
\newcommand{\CE}{\mathcal{CE}}
\newcommand{\eps}{\varepsilon}

\newcommand{\Gc}{\mathcal{G}^{\text{c}}}
\newcommand{\Gst}{\mathcal{G}^{\text{st}}}
\newcommand{\Ec}{\mathcal{E}^{\text{c}}}
\newcommand{\Est}{\mathcal{E}^{\text{st}}}
\newcommand{\I}{\mathcal{I}}
\newcommand{\Hil}{\mathcal{H}}
\newcommand{\prox}{\mathrm{Prox}}
\newcommand{\proj}{\mathrm{Proj}}
\newcommand{\Sch}{\mathsf{Sch}}
\newcommand{\BrSch}{\mathsf{BrSch}}
\newcommand{\GBrSch}{\mathsf{GBrSch}}
\newcommand{\Init}{\mathsf{Init}}

\DeclareMathOperator{\Div}{div}

\DeclareMathOperator{\Supp}{Supp}
\DeclareMathOperator{\D}{d\!}

\DeclareMathOperator{\Id}{Id}
 
\DeclareMathOperator{\argmin}{argmin} 
\DeclareMathOperator*{\esssup}{ess\,sup}

\title{Regularized unbalanced optimal transport as entropy minimization with respect to branching Brownian motion}

\author[1]{Aymeric Baradat}
\author[2]{Hugo Lavenant}
\affil[1]{\small{Univ Lyon, CNRS, Université Claude Bernard Lyon 1, UMR5208, Institut Camille Jordan, F-69622 Villeurbanne, France}}
\affil[2]{\small{Department of Decision Sciences, Bocconi University, Via Röntgen 1, 20136 Milan, Italy}}

\date{\today}

\begin{document} 
	
	\maketitle
	
	\begin{abstract}
		We consider the problem of minimizing the entropy of a law with respect to the law of a reference branching Brownian motion under density constraints at an initial and final time. We call this problem the \emph{branching Schrödinger problem} by analogy with the Schrödinger problem, where the reference process is a Brownian motion. Whereas the Schrödinger problem is related to regularized (a.k.a.\ entropic) optimal transport, we investigate here the link of the branching Schrödinger problem with regularized \emph{unbalanced} optimal transport. 
		
		This link is shown at two levels. First, relying on duality arguments, the values of these two problems of calculus of variations are linked, in the sense that the value of the regularized unbalanced optimal transport (seen as a function of the initial and final measure) is the lower semi-continuous relaxation of the value of the branching Schrödinger problem. Second, we also explicit a correspondence between the competitors of these two problems, and to that end we provide a fine description of laws having a finite entropy with respect to a reference branching Brownian motion.
		
		We investigate the small noise limit, when the noise intensity of the branching Brownian motion goes to $0$: in this case we show, at the level of the optimal transport model, that there is convergence to partial optimal transport. We also provide formal arguments about why looking at the branching Brownian motion, and not at other measure-valued branching Markov processes, like superprocesses, yields the problem closest to optimal transport. Finally, we explain how this problem can be solved numerically: the dynamical formulation of regularized unbalanced optimal transport can be discretized and solved via convex optimization.
	\end{abstract}

	\tableofcontents
	\newpage

	\chapter{Introduction}

	It is now well understood that there is a strong link between models of regularized optimal transport and the large deviations of the space distribution of a large population of Brownian particles. By ``regularized optimal transport'', we mean quadratic optimal transport to which one adds an entropic penalization. By ``large deviations of the space distribution of a large population of Brownian particles'', we mean, correspondingly with Sanov's theorem, the minimization of the relative entropy of a law on the space of paths with respect to the Wiener measure, under prescribed initial and final temporal marginals. These two problems of calculus of variations are equivalent, in the sense that their (numerical) values are the same, and  we can build a relevant correspondence between competitors of both problems, that can only decrease the value of the objective function \cite{leonard2013survey,gentil2017analogy}.
	
	\emph{In this work, we will show that such a link holds between models of Regularized \emph{Unbalanced} Optimal Transport (RUOT) and entropy minimization with respect to the law of the \emph{branching} Brownian Motion (BBM)}. Unbalanced optimal transport models allow for mass not only to be transported, but also to be created or destroyed for some cost. They were introduced independently in \cite{kondratyev2016new,liero2018optimal,chizat2018unbalanced} and allow to match distributions of mass not sharing the same total mass. They are relevant from the point of view of the applications \cite{yang2019scalable,peyre2019computational,diMarino2020tumor}, and in particular we point out to its recent use in order to infer trajectories of cells by Schiebinger \emph{et al.}\ \cite{schiebinger2019optimal}. Regularized unbalanced optimal transport will be defined through the dynamical formulation, by adding both a growth rate as well as diffusion in the continuity equation. On the other hand, the branching Brownian motion models the evolution of a family of particles undergoing a Brownian motion, but also splitting or dying at random instants. We introduce the \emph{branching Schrödinger problem}: it is the minimization of the relative entropy of a law on the space of measure-valued paths with respect to the law of a reference BBM, under prescribed initial and final intensity measures. Given a branching mechanism for the BBM, our work identifies a cost for creating and destroying mass in regularized unbalanced optimal transport so that the two models become strongly linked.   
	
	This link is actually weaker than in the balanced case: this is because the constraints in the branching Schrödinger problem are not closed, which makes the problem ill-posed. We argue that the RUOT problem is the (well posed) lower semi-continuous (l.s.c.) relaxation of the branching Schrödinger problem. This will be seen at two levels. First, we show under basically no assumption on the reference BBM that the optimal value in RUOT is the l.s.c.\ envelope of the one in the branching Schrödinger problem, up to a term depending on the initial condition. It will rely on convex duality: in one sentence, the dual of these two problems happens to be the same. Second, adding some reasonable integrability assumptions on the BBM, we show that we can go further and build a relevant correspondence between the competitors of each problem. Given a law on the space of measure-valued paths with finite entropy with respect to the law of a BBM, one can define a predictable drift and a predictable branching mechanism generating the law of the process. These drifts and branching mechanisms can then be used to build a competitor in RUOT. On the other hand, from a smooth competitor of RUOT it is always possible to compute the law of a branching diffusion process with drift and branching mechanism chosen accordingly to the competitor in RUOT.    
	
	Having understood this equivalence, we then answer some questions which arise naturally from our result. First, we investigate a small noise limit: we show that when the diffusivity and the branching rate jointly go to $0$ at a controlled rate, the RUOT model converges to partial optimal transport. Second, even though we study the BBM here, it is reasonable to wonder what we would get by replacing it with more general measure-valued branching Markov processes, like superprocesses\footnote{Broadly, they correspond to branching Markov processes describing the behaviour of infinitely many independent particles that may split and die.}. We present formal computations which answer these questions and seem to indicate that other measure-valued branching Markov processes lead to problems sharing less structure with optimal transport, hence our focus on BBM. Lastly, we discuss the numerical aspects of the problem: a Sinkhorn-like approach seems to be not feasible, however one can still compute numerical solutions by discretizing the dynamical formulation of RUOT.   		
	
	\bigskip
	
	Our motivation was twofold. First, we give a probabilistic justification for models of unbalanced optimal transport, which has actually been an open question since their introduction. As far as we know, the recent paper~\cite{chen2021most} is the only other work treating this question, and it is done by looking at processes along which particles may be killed: from the unbalanced transport point of view, mass can be removed but never added. On the other hand in the present work both are possible. Second, the second author, in a work with Schiebinger and co-authors \cite{lavenant2021towards}, again in a biological context, not only used optimal transport similarly as in \cite{schiebinger2019optimal}, but was led to consider the BBM as a modelling tool. Finding a link between these two concepts became a natural question that we answer here.
	
	\bigskip
	
	The rest of this introduction presents our main result, with a presentation that we try to keep light. We present the formal argument based on duality about why the two problems are related. We then present a detailed exposition of the structure of the present work.

	\bigskip
	
	In the sequel we work on a temporal domain $[0,1]$, that is we set the initial and final time to be respectively $0$ and $1$. This can be done without loss of generality. On the other hand, we work for convenience on the $d$-dimensional torus $\T^d$. By doing that we lose generality but we avoid compactness problems (as $\T^d$ is compact) as well as spatial boundary issues. 
	
	\section{Optimal transport}
	\label{subsec:intro_OT}
	
	Optimal transport is now a well established theory, and we refer to \cite{villani2003topics,ambrosio2008gradient,santambrogio2015optimal,peyre2019computational} for textbooks about it. We focus here on the dynamical formulation (often called Benamou-Brenier formulation since this approach was initiated in~\cite{benamou2000computational}) because it is the only one in which our RUOT problem makes sense.

	The input of our regularized unbalanced optimal transport problem consist in $\rho_0, \rho_1$ a pair of nonnegative measures on $\T^d$. In its dynamical formulation, one then seeks a triple of space-time dependent fields $(\rho, v,r)$ where: $\rho$ is a space-time dependent density such that $\rho(0,\cdot) = \rho_0$ and $\rho(1,\cdot) = \rho_1$; $v$ is a space-time dependent velocity field taking its values in $\R^d$; and $r$ is a space-time dependent scalar field corresponding to a growth rate. These three variables are linked with the equation 
	\begin{equation*}
	\dr_t \rho + \Div(\rho v) = \frac{\nu}{2} \Delta \rho + r \rho,
	\end{equation*}
	where $\nu \in \R$ is the \textbf{diffusivity} (at this stage $\nu$ could be negative or equal to $0$). This corresponds to a Fokker-Planck equation with an additional source term $r \rho$: mass can be created ($r > 0$) or destroyed ($r < 0$). Among all triples $(\rho, v, r)$ such that $\rho(0,\cdot) = \rho_0$ and $\rho(1,\cdot) = \rho_1$, one seeks one which minimizes the energy
	\begin{equation*}
	\int_0^1 \hspace{-5pt} \int \left[ \frac{1}{2} |v(t,x)|^2  + \Psi(r(t,x)) \right] \rho(t,x) \D x \D t, 
	\end{equation*}
	where $\Psi : \R \to [0,+\infty]$ corresponds to a \textbf{growth penalization}. The velocity field is penalized via the kinetic energy (the integral of $1/2 \, |v|^2 \rho$) while $\Psi(r)$ measures the cost per unit of mass of having a growth rate $r$. We underline some particular cases, illustrated at Figure~\ref{fig:numerics} below:
	
	\begin{itemize}
		\item \textbf{(plain) Optimal Transport (OT)}, see \cite[Chapter 5]{santambrogio2015optimal} or \cite[Chapter 7]{ambrosio2008gradient}. It corresponds to $\nu = 0$ and $\Psi(r) = +\infty$, unless for $r=0$ where $\Psi(0) = 0$, so that $r$ is enforced to cancel. The problem is clearly symmetric in time, and one can see that $t \mapsto \int \rho(t, \cdot)$ is constant. In particular, for the problem to be well posed, $\rho_0$ and $\rho_1$ need to be balanced, that is, $\int \rho_0$ and $\int \rho_1$ need to be the same.
		
		\item \textbf{Regularized Optimal Transport (ROT)}, see \cite{gentil2017analogy} or \cite{gigli2020benamou} for a more general framework. It corresponds to $\nu \neq 0$ but still the same $\Psi$ as in the previous OT problem. Here again $\rho_0$ and $\rho_1$ have to be balanced for the problem to be well posed. Although this problem looks non-symmetric with respect to time, with the change of variable $(\rho, v) \to (\rho, w = v - \nu/2 \nabla \log \rho )$, there holds, with $H$ being the relative entropy (defined in \eqref{eq:def_H}), and $\Leb$ being the Lebesgue measure,
		\begin{align}
		\notag
		\min_{\rho, v} \bigg\{ &\int_0^1 \hspace{-5pt} \int \frac{1}{2} |v(t,x)|^2 \rho(t,x) \D x \D t \ : \ \dr_t \rho + \Div(\rho v) = \frac{\nu}{2} \Delta \rho \bigg\} +  \frac{\nu}{2} (H(\rho_0[\Leb]) - H(\rho_1|\Leb) )\\
		\label{eq:ROT_symmetric} &= \min_{\rho, w} \left\{ \int_0^1 \hspace{-5pt} \int \left[ \frac{1}{2} |w(t,x)|^2  + \frac{1}{2}\left|\frac{\nu}{2} \nabla \log \rho(t,x) \right|^2 \right]\rho(t,x) \D x \D t \ : \ \dr_t \rho + \Div(\rho w) = 0 \right\},
		\end{align}
		with the temporal constraints $\rho(0, \cdot) = \rho_0$ and $\rho(1, \cdot) = \rho_1$. The right hand side is symmetric with respect to the exchange $\rho_0 \leftrightarrow \rho_1$. 
		
		\item \textbf{Unbalanced Optimal Transport (UOT)}, see \cite{kondratyev2016new,liero2018optimal,chizat2018unbalanced}. It corresponds to the case where $\nu = 0$ but $r$ is no longer constrained to cancel. By changing $\Psi$, we modify the compromise between displacement and creation/destruction of mass chosen by the solutions of the corresponding models. A popular choice, that will not be ours, is to take $\Psi$ to be quadratic. As the presence of $r$ allows for the total mass $\int \rho(t,x) \D x$ to change over time, one no longer enforces $\rho_0$ and $\rho_1$ to have the same mass.
		
		\item \textbf{Regularized Unbalanced Optimal Transport (RUOT)}. It corresponds to allow both $\nu \neq 0$ and $r$ to be nonzero. As pointed out Subsection~\ref{subsec:no_sinkhorn}, this definition is \emph{not} symmetric in time as the change of variables $(\rho, v) \to (\rho, v - \nu/2 \nabla \log \rho )$ no longer works because of the growth part. We stick to this definition and not another one (e.g. starting from \eqref{eq:ROT_symmetric}) because it will appear to be the one linked to the entropic minimization problem with respect to BBM. The function $\Psi$ is not specified for the moment, but it will have to depend on the branching mechanism of the BBM for the equivalence to hold. We denote by $\ruot_{\nu, \Psi}$ the value of the problem, that is,
		\begin{equation}
		\label{eq:def_ruot}
		\begin{aligned}
		\ruot_{\nu,\Psi}(\rho_0,\rho_1)
		=  \min_{\rho, v, r}  \bigg\{ \int_0^1 \hspace{-5pt} \int &\left[ \frac{1}{2} |v(t,x)|^2  + \Psi(r(t,x)) \right] \rho(t,x) \D x \D t \ : \ \\ &\qquad\dr_t \rho + \Div(\rho v) = \frac{\nu}{2} \Delta \rho + r \rho, \ \rho(0, \cdot) = \rho_0, \ \rho(1, \cdot) = \rho_1  \bigg\}   
		\end{aligned}
		\end{equation}
		One can find a similar formulation in \cite{chen2021most} where the penalization $\Psi$ prevents $r$ from taking positive values: only mass removal is allowed.
	\end{itemize}
	
	\begin{figure}
		\begin{center}
			\begin{tabular}{cc}
				
				\multicolumn{2}{c}{
					\begin{tikzpicture}[scale = 0.4]
					\begin{axis}[
					ylabel={$t=0$},
					xmin=0, xmax=1,
					ymin = -0.5, ymax = 10]
					\addplot[color=blue] table [x=grid, y=rho0]{temporal_boundary.txt};
					\end{axis}
					\end{tikzpicture}} \\
				
				\multicolumn{2}{c}{Initial measure $\rho_0$} \\
				
				\begin{tikzpicture}[scale = 0.7] 
				\begin{axis}
				[
				xmin=0, xmax=1,
				ymin=0, ymax=1,
				zmin = 0, zmax = 10,
				view={60}{30},
				xlabel={$x$},
				ylabel={$t$},
				zlabel={density},
				]
				\addplot3[color=red,line width = 0.7pt] table [x=grid, y expr = 0.25, z=rho6]{store_OT.txt};
				\addplot3[color=red,line width = 0.7pt] table [x=grid, y expr = 0.5, z=rho12]{store_OT.txt};
				\addplot3[color=red,line width = 0.7pt] table [x=grid, y expr = 0.75, z=rho18]{store_OT.txt};
				\addplot3[color=blue,line width = 0.7pt,opacity=0.5] table [x=grid, y expr = 1, z=rho1]{temporal_boundary.txt};
				\addplot3[color=blue,line width = 0.7pt,opacity=0.5] table [x=grid, y expr = 0, z=rho0]{temporal_boundary.txt};
				\end{axis}
				\end{tikzpicture} & 
				
				\begin{tikzpicture}[scale = 0.7] 
				\begin{axis}
				[
				xmin=0, xmax=1,
				ymin=0, ymax=1,
				zmin = 0, zmax = 10,
				view={60}{30},
				xlabel={$x$},
				ylabel={$t$},
				zlabel={density},
				]
				\addplot3[color=red,line width = 0.7pt] table [x=grid, y expr = 0.25, z=rho6]{store_ROT.txt};
				\addplot3[color=red,line width = 0.7pt] table [x=grid, y expr = 0.5, z=rho12]{store_ROT.txt};
				\addplot3[color=red,line width = 0.7pt] table [x=grid, y expr = 0.75, z=rho18]{store_ROT.txt};
				\addplot3[color=blue,line width = 0.7pt,opacity=0.5] table [x=grid, y expr = 1, z=rho1]{temporal_boundary.txt};
				\addplot3[color=blue,line width = 0.7pt,opacity=0.5] table [x=grid, y expr = 0, z=rho0]{temporal_boundary.txt};
				
				\end{axis}
				\end{tikzpicture} \\
				
				(a) OT & (b) ROT \\
				
				\begin{tikzpicture}[scale = 0.7] 
				\begin{axis}
				[
				xmin=0, xmax=1,
				ymin=0, ymax=1,
				zmin = 0, zmax = 10,
				view={60}{30},
				xlabel={$x$},
				ylabel={$t$},
				zlabel={density},
				]
				\addplot3[color=red,line width = 0.7pt] table [x=grid, y expr = 0.25, z=rho6]{store_UOT.txt};
				\addplot3[color=red,line width = 0.7pt] table [x=grid, y expr = 0.5, z=rho12]{store_UOT.txt};
				\addplot3[color=red,line width = 0.7pt] table [x=grid, y expr = 0.75, z=rho18]{store_UOT.txt};
				\addplot3[color=blue,line width = 0.7pt,opacity=0.5] table [x=grid, y expr = 1, z=rho1]{temporal_boundary.txt};
				\addplot3[color=blue,line width = 0.7pt,opacity=0.5] table [x=grid, y expr = 0, z=rho0]{temporal_boundary.txt};
				
				\end{axis}
				\end{tikzpicture} & 
				
				\begin{tikzpicture}[scale = 0.7] 
				\begin{axis}
				[
				xmin=0, xmax=1,
				ymin=0, ymax=1,
				zmin = 0, zmax = 10,
				view={60}{30},
				xlabel={$x$},
				ylabel={$t$},
				zlabel={density},
				]
				\addplot3[color=red,line width = 0.7pt] table [x=grid, y expr = 0.25, z=rho6]{store_RUOT.txt};
				\addplot3[color=red,line width = 0.7pt] table [x=grid, y expr = 0.5, z=rho12]{store_RUOT.txt};
				\addplot3[color=red,line width = 0.7pt] table [x=grid, y expr = 0.75, z=rho18]{store_RUOT.txt};
				\addplot3[color=blue,line width = 0.7pt,opacity=0.5] table [x=grid, y expr = 1, z=rho1]{temporal_boundary.txt};
				\addplot3[color=blue,line width = 0.7pt,opacity=0.5] table [x=grid, y expr = 0, z=rho0]{temporal_boundary.txt};
				
				\end{axis}
				\end{tikzpicture} \\
				
				(c) UOT & (d) RUOT \\
				
				\multicolumn{2}{c}{
					\begin{tikzpicture}[scale = 0.4]
					\begin{axis}[
					ylabel={$t=1$},
					xmin=0, xmax=1,
					ymin = -0.5, ymax = 10]
					\addplot[color=blue] table [x=grid, y=rho1]{temporal_boundary.txt};
					\end{axis}
					\end{tikzpicture}} \\
				
				\multicolumn{2}{c}{Final measure $\rho_1$}
				
			\end{tabular}   
			\caption{Simulations of different models of optimal transport for the same initial and final densities $\rho_0, \rho_1$ (top and bottom line). (a) plain optimal transport, (b) regularized optimal transport, (c) unbalanced optimal transport and (d) regularized unbalanced optimal transport. \\
				Observe that in models (a) and (b), mass is only transported. As a result, a large portion of the mass of the big initial bump on the left hand side needs to be transported onto the big final bump in the right hand side. In models (c) and (d), this is not the case anymore: some mass from the big initial bump is destroyed, while mass is produced to form the big final bump. Also observe that in models (b) and (d), the densities at intermediate times look more regular than in the corresponding models (a) and (c).\\
				The simulations were performed with the same code which discretizes the dynamical formulation, see Section \ref{sec:numerics} for more details.}
			\label{fig:numerics}
		\end{center}
	\end{figure}
	
	\section{The Schrödinger problem and its link with regularized optimal transport}

	Before diving into our contributions, we briefly recall what the Schrödinger problem is and how it is linked to ROT (balanced) models. We refer to \cite{leonard2013survey,gentil2017analogy} and references therein for an exhaustive understanding.  
	
	We recall that if $\mathbf{X}$ is a polish space and $\mathsf{r} \in \P(\mathbf{X})$ is a probability measure on it, for all finite Borel measure $\mathsf{p}$ on $\mathbf{X}$, the relative entropy of $\mathsf p$ with respect to $\mathsf r$ is defined by:
	\begin{equation}
	\label{eq:def_H}
	H(\mathsf p| \mathsf r) := \left\{ \begin{aligned}
	&\E_{\mathsf p}\left[\log \frac{\D \mathsf p}{\D \mathsf r}\right], &&\mbox{if } \mathsf p \in \P(\mathbf{X}), \, \mathsf p\ll \mathsf r,\\
	&+\infty, && \mbox{else}.
	\end{aligned}
	\right.
	\end{equation} 	
	(Here and in the sequel, if $\mathsf p$ is a probability measure on a Polish space $\mathbf X$, $\E_{\mathsf p}$ stands for the expectation w.r.t.\ $\mathsf p$.)
	
	In the context of the Schrödinger problem, $\mathbf{X}$ is in fact $\Omega = C([0,1], \T^d)$ the space of continuous paths valued in the torus, and $R \in \P(\Omega)$ is the law of the Brownian motion of diffusivity $\nu$ and uniform initial condition. That is, under $R$, the canonical process $X = (X_t)_{t \in [0,1]}$ on $\Omega$ follows the law of Brownian motion of diffusivity $\nu$, while $X_0$ is uniformly distributed. Given $\rho_0, \rho_1 \in \P(\T^d)$, the Schrödinger problem reads as finding a law $P \in \P(\Omega)$ which minimizes $H(P|R)$ under the marginal constraints that under $P$, $X_0 \sim \rho_0$ and $X_1 \sim \rho_1$. We denote by $\Sch$ the value of the problem (scaled by a factor $\nu > 0$), that is
	\begin{equation*}
	\Sch_\nu(\rho_0, \rho_1) = \inf_P \left\{ \nu H(P|R) \ : \ \text{under } P, \; X_0 \sim \rho_0 \text{ and } X_1 \sim \rho_1  \right\}. 
	\end{equation*} 	
	This problem is well posed (in the sense that there exists a unique minimizer) if and only if $H(\rho_0|\Leb) < + \infty$ and $H(\rho_1|\Leb) < + \infty$.

	\begin{Rem}
		\label{rk:sanov_schrodinger}
		In view of Sanov's theorem~\cite{sanov1958probability}, this minimization of relative entropy has the following interpretation in terms of large deviations, that we state in an informal way and that were Schrödinger's original concern. Let us consider a large number of Brownian particles, each of them being chosen independently and at a random (uniform) initial position. Because of the law of large numbers, the initial and final distribution of these particles are close to being uniform, with very high probability. However, conditionning on the very rare event that it is not the case, and that instead these distributions are close to $\rho_0$ and $\rho_1$ respectively, then Sanov's theorem predicts that with very high probability, the behaviour of the system is the same as if all the particles had been chosen according to $P$, the solution of the entropic minimization problem.
	\end{Rem}
	
	One of the main result, for which various statements and proofs (in more general context than the state space being $\T^d$) can be found for instance in \cite{gentil2017analogy}, is the identity:
	\begin{equation}
	\label{eq:equivalence_schrodinger}
	\Sch_\nu(\rho_0,\rho_1) = \nu H(\rho_0|\Leb) + \min_{\rho, v} \bigg\{ \int_0^1 \hspace{-5pt} \int \frac{1}{2} |v(t,x)|^2 \rho(t,x) \D x \D t \ : \ \dr_t \rho + \Div(\rho v) = \frac{\nu}{2} \Delta \rho \bigg\} 
	\end{equation}	
	where the minimum in $\rho,v$ is taken among those with temporal boundary conditions $\rho(0,\cdot) = \rho_0$ and $\rho(1,\cdot) = \rho_1$. One can recognize in the right hand side, up to the term $\nu H(\rho_0|\Leb)$, the value of the ROT problem between $\rho_0$ and $\rho_1$. Moreover, there is a correspondence between the competitors: from any $P$ with $H(P|R)< + \infty$ and the right temporal marginals, the pair $(\rho,v)$ formed by its density and its current velocity (in the sense of Nelson~\cite{nelson1967dynamical}) is admissible in the ROT problem with a controlled energy; while if $(\rho,v)$ is admissible in ROT and $v$ is smooth, one should take $P$ the law of the Stochastic Differential Equation (SDE) $\D X_t = v(t,X_t) \D t + \sqrt{\nu} \D B_t$ and initial condition $\rho_0$ to build a competitor in the Schrödinger problem with controlled entropy. 
	
	Eventually, as can be seen formally in the models of optimal transport, in the limit $\nu \to 0$, the value of the Schrödinger problem converges to (plain) optimal transport. We refer to \cite{leonard2012schrodinger,baradat2020small} for statements and proofs of this limit in a probabilistic framework.

	\section{The branching Schrödinger problem} 
	
	Now we want to replace the Brownian motion by a branching Brownian motion. One difficulty is to understand what the probability space $\Omega$ (which was previously $C([0,1], \T^d)$) becomes: it will be replaced by measure-value paths, to take into account that in a BBM the number of particles is not constant.
	
	\paragraph{Branching Brownian motion}	
	
	We describe informally the BBM as follows, see \cite[Chapter~1]{etheridge2000introduction} or Section~\ref{sec:presentation_BBM} for a more detailed introduction. We consider a population of particles which evolve independently according to Brownian motion with a common \textbf{diffusivity} $\nu>0$. In addition, to each particle is attached an (independent) exponential clock with parameter being the \textbf{branching rate} $\lambda >0$. When the clock rings, the corresponding particle dies, and gives birth to a randomly chosen number of particles, which will in turn evolve in the same way. The number of offsprings is drawn accordingly to the (common) \textbf{law of offspring}, denoted by $\boldsymbol{p} = (p_k)_{k \in \N} \in \P(\N)$ (we identify probability measures on $\N$ with nonnegative sequences summing at one), where $p_k$ is the probability to give birth to $k$ descendants at a branching event. We refer to Figure~\ref{fig:example_BBM} for a schematic example of a realization of a branching Brownian motion.  	
	
	\begin{figure}
		\begin{center}
			\begin{tikzpicture}[scale = 1]
			\begin{axis}[
			xlabel={Time},
			ylabel={Space},
			xmin=0, xmax=1,
			ymin = 0.1, ymax = 0.8]
			\addplot[color=blue,line width = 0.7pt] table [x=t1, y=x1]{data_BBM.txt};
			\addplot[color=blue,line width = 0.7pt] table [x=t2, y=x2]{data_BBM.txt};
			\addplot[color=blue,line width = 0.7pt] table [x=t3, y=x3]{data_BBM.txt};
			\addplot[color=blue,line width = 0.7pt] table [x=t4, y=x4]{data_BBM.txt};
			\addplot[color=blue,line width = 0.7pt] table [x=t5, y=x5]{data_BBM.txt};
			\addplot[color=blue,line width = 0.7pt] table [x=t6, y=x6]{data_BBM.txt};
			\addplot[color=blue,line width = 0.7pt] table [x=t7, y=x7]{data_BBM.txt};
			\addplot[color=black,only marks,mark= star,mark size=3.9pt,] table[x=eDeatht,y= eDeathx]{data_BBM.txt};
			\addplot[only marks,mark=*,mark size=2.9pt,color=blue] table[x=eInitt,y= eInitx]{data_BBM.txt};
			\addplot[only marks,mark=square*,mark size=2.9pt,color=violet] table[x=eFinalt,y= eFinalx]{data_BBM.txt};
			\addplot[only marks,mark=triangle*,mark size=3.9pt,color=orange] table[x=eSplit2t,y= eSplit2x]{data_BBM.txt};
			\addplot[only marks,mark= diamond*,mark size=3.9pt,color=red] table[x=eSplit3t,y= eSplit3x]{data_BBM.txt};
			\end{axis}
			\end{tikzpicture}
		\end{center}
		\caption{Schematic illustration of the trajectory of a branching Brownian motion in one space dimension, starting ($\textcolor{blue}{\bullet}$) and ending ($\textcolor{violet}{\blacksquare}$) with two particles, with particles dying ($\star$), and branching events with $2$ ($\textcolor{orange}{\blacktriangle}$) and $3$ ($\textcolor{red}{\blacklozenge}$) offsprings.}
		\label{fig:example_BBM}
	\end{figure}

	We study the law of the process which at time $t$ gives the empirical measure associated with all the positions of all the particles alive at time $t$. It means that we describe a collection of particles $(x_1, \dots, x_p)$ on $\T^d$ by its (unnormalized) empirical distribution 
	\begin{equation*}
	\sum_{i=1}^p \delta_{x_i}
	\end{equation*}	
	which belongs to $\M_+(\T^d)$ the set of nonnegative Radon measures on $\T^d$. Thus, a realization of the BBM is a $\cadlag$ curve valued in $\M_+(\T^d)$ with jumps corresponding to branching events. In short the BBM corresponds to a stochastic process denoted by $(M_t)_{t \in [0,1]}$, where $M_t$ is the empirical measure describing the collection of particles alive at time $t$. The law of the BBM is an element of $\P(\cadlag([0,1]; \M_+(\T^d)))$.
	
	The last parameter needed to describe the law of the BBM is the \textbf{initial law} (or otherwise stated the law of $M_0$) denoted by $R_0 \in \P(\M_{\delta}(\T^d))$, where $\M_{\delta}(\T^d) \subset \M_+(\T^d)$ is the set of finite sum of Dirac masses on $\T^d$. Note that $R_0$ can be described as a \emph{point process}.

	\begin{Rem}
		\label{rem:q_instead_of_p}
		From the algebraic point of view, it will be less cumbersome to introduce the branching rates $q_k = \lambda p_k$ in order to build $\boldsymbol{q} = (q_k)_{k \in \N}$ as a positive measure on $\N$. The quantity $q_k$ corresponds to the (temporal) rate of branching events with $k$ offsprings. We also introduce the generating function of~$\boldsymbol{q}$ as
		\begin{equation*}
		\Phi_{\boldsymbol{q}} : z \mapsto \sum_{k \in \N} q_k z^k = \lambda \sum_{k \in \N} p_k z^k.
		\end{equation*} 
		Of course $\boldsymbol{q}$, that we call the \textbf{branching mechanism}, is entirely determined by $\Phi_{\boldsymbol{q}}$, while $\lambda$ and $\boldsymbol{p}$ can be reconstructed from $\boldsymbol{q}$ by $\lambda = \lambda_{\boldsymbol q} :=\sum_{k \in \N} q_k$ and $p_k = \lambda^{-1} p_k$ for $k \in \N$.
	\end{Rem}

	Under the mild assumption that the mean number of particle at each branching event is finite (that is, $\sum_k k q_k < + \infty$), these data are sufficient to give a full characterization of the BBM of diffusivity $\nu$, branching mechanism $\boldsymbol{q} = \lambda \boldsymbol{p} \in \M_+(\N)$ and initial law $R_0$: in the sequel we write
	\begin{equation*}
	R \sim \BBM(\nu, \boldsymbol{q}, R_0)
	\end{equation*}
	to indicate that $R$ is a such a law.

	\paragraph{Marginal constraints and entropy minimization}
	
	Similarly to the Schrödinger problem, we want to minimize $P \to H(P|R)$, being $R \sim \BBM(\nu, \boldsymbol{q}, R_0)$, with additional marginal constraints on $P$. Here $R$ and the law $P$ are elements of $\P(\cadlag([0,1]; \M_+(\T^d)))$. Under $P$, the random variables $M_0$ and $M_1$ are random measures, and the constraints will be about their intensity measures.
	
	Specifically, if $P \in \P( \cadlag([0,1]; \M_+(\T^d)))$, we define its expected marginal at time $t$ as $\E_P[ M_t ]$: this is the measure on $\M_+(\T^d)$ defined for any test function $\theta \in C(\T^d)$ by 
	\begin{equation*}
	\Big\cg \theta , \E_P[ M_t ]\Big\cd = \E_P \Big[ \cg \theta, M_t\cd \Big],
	\end{equation*}
	where here an in the rest of this work, $\cg \cdot , \cdot \cd$ stands for the duality bracket between continuous functions and measures of finite total variation on the torus (see also Definition~\ref{def:intensity_measure}).  
	
	\begin{Rem}
		Let us give an intuition for $\E_P[M_t]$. Under $P$, one can see $M_t$ as a random measure, and $\E_P[M_t]$ is simply its average value. Hence, for any Borel set $A \subset \T^d$, $\E_P[M_t](A) = \E_P[M_t(A)]$ is the average number of particles that one finds in $A$ at time $t$. If $M_t$ is interpreted as a random point process, $\E_P[M_t]$ is called its intensity or sometimes the expected measure. Let us provide three examples. If $\rho$ is a probability distribution on $\T^d$, consider the random measure $M$ built as follows: (i) draw $X$ according to $\rho$ and set $M = \delta_X$, or (ii) draw $X,Y$ according to $\rho$, for instance independently, and set $M = \frac{1}{2} \delta_X + \frac{1}{2} \delta_Y$, or (iii) $M$ is a Poisson point process, of intensity $\rho$. Then in the three cases $\E[M] = \rho$.
	\end{Rem}
	
	We can now state the branching Schrödinger problem. Let us fix $\nu$, $\boldsymbol{q}$ and $R_0$ as above. We consider $R \sim \BBM(\nu, \boldsymbol{q}, R_0)$ the ``reference'' measure, and we give ourselves $\rho_0, \rho_1$ two nonnegative measures on $\T^d$. The entropy minimization with respect to BBM is 
	\begin{equation}
	\label{eq:def_BrSch}
	\BrSch_{\nu, \boldsymbol q, R_0}(\rho_0,\rho_1) := \inf_{P} \left\{ \nu H(P|R) \ : \ \E_P[M_0] = \rho_0 \text{ and } \E_P[M_1] = \rho_1 \right\}.
	\end{equation} 
	Notice that as in the non-branching case, there is a factor $\nu$ in front of the entropy which does not change the minimizers but yields the right scaling when $\nu \to 0$. 
	
	\begin{Rem}
		Similarly to Remark \ref{rk:sanov_schrodinger}, this minimization of relative entropy has the following interpretation in terms of large deviations. Let us consider a large number of branching Brownian particles, each of them being chosen independently acording to the law $R$. Because of the law of large numbers, the initial and final distribution of these particles are close to $\E_R[M_0]$ and $\E_R[M_1]$ respectively, with very high probability. However, conditionning on the very rare event that it is not the case, and that instead these distributions are close to $\rho_0$ and $\rho_1$ respectively, then Sanov's theorem predicts that with very high probability, the behaviour of the system is the same as if all the particles had been chosen according to $P$, the solution of the entropic minimization problem.
	\end{Rem}
	
	\begin{Rem}
		\label{rem:ill-posed}
		The branching Schrödinger problem is ill-posed because the constraints are not closed. Indeed, for a given $\theta \in C(\T^d)$, the set of $P$ which satisfies
		\begin{equation*}
		\E_P \left[ \langle \theta , M_t \rangle \right] = \text{cst}
		\end{equation*} 
		is not closed with respect to weak convergence, $L^1$ convergence or even $I$-convergence following the terminology of \cite{csiszar2003information}. This is because $M_t(\T^d)$, the total number of particles alive at time $t$, is not bounded, so that actually the left hand side behaves like a moment of order one of $P$. In particular, it is not true in general that, for given $\rho_0, \rho_1$, an optimal $P$ exists, even if the minimization problem is not empty. Some counterexamples can be found in Subsection~\ref{subsec:counterexamples}. 
		
		On the other hand, the RUOT problem is well posed under mild assumptions. Our main result is that the (well posed) latter can be seen as the relaxation of the (ill-posed) former.
	\end{Rem}
	
	\section{A duality argument about the equivalence between the models}
	\label{sec:duality_informal}
	
	In this section, we introduce our main result about the equivalence between the models, and we  give a reason of why it should hold. Let us fix $\nu$, $\boldsymbol{q}$, $R_0$ as above (in particular with $\sum_k k q_k < + \infty$) and $R \sim \BBM(\nu,  \boldsymbol{q}, R_0)$. Let us exclude the pathological case where $R_0 = \delta_0$ where $R$-a.s.\ nothing happens (see Remark~\ref{rem:R0=delta0}). We will actually prove the following. 
	
	\bigskip

	{\it We \emph{define} a convex growth penalization $\Psi_{\nu, \boldsymbol q}$ by its Legendre transform}
	\begin{equation}
	\label{eq:def_Psi*}
	\Psi_{\nu, \boldsymbol q}^*(s) = \nu \Big( \Phi_{\boldsymbol{q}}\big(e^{s/\nu}\big)e^{-{s/\nu}} - \Phi_{\boldsymbol{q}}(1) \Big) = \nu \sum_{k=0}^{+\infty} q_k \left\{\exp\left( (k-1)\frac{s}{\nu}  \right) -1 \right\},
	\end{equation}
	{\it that is we define $\Psi_{\nu, \boldsymbol q}(r) = \sup_{s \in \R} rs - \Psi_{\nu, \boldsymbol q}^*(s)$. Let us also define $L_{R_0}^* : \M_+(\T^d) \to \R \cup \{ + \infty \}$ the Legendre transform of the log-Laplace transform of $R_0$:} 
	\begin{equation}
	\label{eq:def_L*}
	L_{R_0}^*(\rho_0) = \sup_{\sigma \in C(\T^d)} \langle \sigma, \rho_0 \rangle -  \log \E_{R_0} \left[ \exp( \langle \sigma, M \rangle ) \right], \qquad \rho_0 \in \M_+(\T^d).
	\end{equation}
	{\it Denoting by $\overline{\BrSch}_{\nu, \boldsymbol q, R_0}$ the l.s.c.\ envelope of the function $(\rho_0,\rho_1) \to \BrSch_{\nu, \boldsymbol q, R_0}(\rho_0,\rho_1)$ defined in~\eqref{eq:def_BrSch} for the topology of weak convergence on $\M_+(\T^d) \times \M_+(\T^d)$, there holds:} 	
	\begin{equation}
	\label{eq:equivalence_BrSchr}
	\overline{\BrSch}_{\nu, \boldsymbol q, R_0}(\rho_0,\rho_1) = \nu L_{R_0}^*(\rho_0) +  \ruot_{\nu,\Psi}(\rho_0,\rho_1) 
	\end{equation}	
	{\it where $\ruot_{\nu,\Psi}$ is defined in \eqref{eq:def_ruot}, provided $\Psi = \Psi_{\nu, \boldsymbol q}$.}

	\bigskip
	
	Of course the parallel should be done with \eqref{eq:equivalence_schrodinger}, with now the term $\nu L_{R_0}^*(\rho_0)$ playing the role of the ``initial term'' depending only on $\rho_0$. As emphazised in Remark~\ref{rem:ill-posed}, as the branching Schrödinger problem is ill-posed, the functional $ \BrSch$ is not necessarily lower semi continuous, while the right hand side of \eqref{eq:equivalence_BrSchr} is. In that sense, our result is one of the finest characterization one can hope. 
	
	\begin{Rem}
		\label{rk:psistar_cosh}
		In Appendix~\ref{app:plots}, we provide a set of plots and describe some relevant properties of $\Psi^*_{\nu, \boldsymbol q}$ and $\Psi_{\nu, \boldsymbol q}$ for various choices of $\nu$ and $\boldsymbol q$.
		
		Let us explicit $\Psi^*_{\nu, \boldsymbol q}$ in the simple case where $q_0 = q_2 = \lambda/2$ and $q_k = 0$ for $k \neq 0,2$. That is, branching events happen with rate $\lambda$ and at each branching event the particle divides in two or dies with equal probability. In that case $\Psi_{\nu, \boldsymbol q}^*(s) = \nu \lambda ( \cosh(s/\nu) - 1 )$. Then $\Psi_{\nu, \boldsymbol q}$ can be computed explicitly (but the expression is cumbersome and does not tell much): it yields a convex function which is minimal at $0$ and which grows at $\pm \infty$ as $\Psi_{\nu, \boldsymbol q}(r) \sim \nu |r| \ln |r|$. 
	\end{Rem}

	\bigskip	
	
	We will actually prove a finer result than \eqref{eq:equivalence_BrSchr}, in the sense that under reasonable assumptions on $R_0$ and $\boldsymbol q$, we will build a relevant correspondence between the competitors $P$ and the triples $(\rho,v,r)$. But for the moment, let us give a first argument for the equality of the values, that is, \eqref{eq:equivalence_BrSchr}. The reader who does not want to follow the details can directly jump to Section~\ref{sec:overview_article} to get an overview of the other results we prove in this article.
	
	The two problems are constrained convex optimization problems, and they are linked by their temporal boundary conditions. Therefore, we will rather look at the dual problems, and actually show that they are the same. Before writing the rigorous proof later in this article, let us present the formal computations leading to this result.
	
	\paragraph{Dual of the RUOT problem}
	A non-rigorous way to derive the dual of the RUOT problem (see e.g.\ \cite{benamou2017variational} for the general method) is to introduce a Lagrange multiplier $\phi : [0,1] \times \T^d \to \R$ associated to the differential constraint, and to write the problem as an $\inf \sup$. In other terms, consider the following functional
	\begin{equation*}
	(\rho, v, r) \mapsto  \sup_{\phi} \cg \phi(1), \rho_1\cd - \cg\phi(0), \rho_0\cd - \int_0^1 \hspace{-5pt} \int \rho \left( \dr_t \phi + v \cdot \nabla \phi + \frac{\nu}{2} \Delta \phi + r \phi \right) \D x \D t,
	\end{equation*}
	where to lighten the notations, here and later in this work, we denote by $\phi(0)$ and $\phi(1)$ the functions $\phi(0,\cdot)$ and $\phi(1, \cdot)$, so that
	\begin{equation*}
	\cg\phi(0), \rho_0\cd = \int \phi(0,x) \rho_0(x) \D x \qquad \mbox{and} \qquad   \cg \phi(1), \rho_1\cd = \int \phi(1,x) \rho_1(x) \D x .
	\end{equation*}
	
	This functional takes values in $\{ 0, + \infty \}$, and cancels exactly on those triples $(\rho,v, r)$ that satisfy both the PDE and the temporal constraints of the problem. This leads to the equality: 
	\begin{align*}
	\ruot_{\nu,\Psi}(\rho_0,\rho_1) = \min_{\rho,v,r} \sup_\phi \int_0^1 \hspace{-5pt} \int &\left[ \frac{1}{2} |v(t,x)|^2  + \Psi(r(t,x)) \right] \rho(t,x) \D x \D t  \\&+ \cg \phi(1), \rho_1\cd - \cg\phi(0), \rho_0\cd
	- \int_0^1 \hspace{-5pt} \int \rho \left( \dr_t \phi + v \cdot \nabla \phi + \frac{\nu}{2} \Delta \phi + r \phi \right) \D x \D t.   
	\end{align*}
	Swapping the $\inf$ and the $\sup$, which can be justified by convex analysis arguments, and then optimizing in $v$ (which yields $v = \nabla \phi$) and in $r$ (which yields $r = (\Psi^*)'(\phi)$), one ends up with
	\begin{equation*}
	\sup_{\phi} \min_\rho   \cg \phi(1), \rho_1\cd - \cg\phi(0), \rho_0\cd - \int_0^1 \hspace{-5pt} \int \rho \left( \dr_t \phi + \frac{1}{2} |\nabla \phi|^2 + \frac{\nu}{2} \Delta \phi + \Psi^*(\phi) \right) \D x \D t.
	\end{equation*} 
	As written here $\rho$ is unconstrained so that the term in front of $\rho$ should vanish. Thus, we get the dual problem of the RUOT problem with temporal boundary conditions $\rho_0$ and $\rho_1$:
	\begin{equation}
	\label{eq:RUOT_dual}
	\ruot_{\nu,\Psi}(\rho_0,\rho_1)  
	= \sup_{\phi} \bigg\{ \cg \phi(1), \rho_1\cd - \cg\phi(0), \rho_0\cd \ : \
	\dr_t \phi + \frac{1}{2} |\nabla \phi|^2 + \frac{\nu}{2} \Delta \phi + \Psi^*(\phi) = 0 \text{ on } [0,1] \times \T^d  \bigg\}.
	\end{equation}
	When the $\inf-\sup$ exchange is done rigorously there is an inequality sign in the constraint and not an equality (which makes the constraint set convex). However, taking the inequality to be an equality will always increase the value of the objective functional (this is because the Hamilton-Jacobi-Bellman equation, that is, the constraint in \eqref{eq:RUOT_dual}, has a comparison principle), so for this introduction we keep an equality sign and refer to the sequel for rigorous statements.

	\paragraph{Dual of the branching Schrödinger problem} 
	Let us now derive the dual problem of the minimization of the entropy with respect to BBM. Here we have two Lagrange multipliers $\sigma, \theta : \T^d \to \R$ for the constraints on the initial and final temporal marginals: the functional 
	\begin{equation*}
	P \mapsto \sup_{\sigma, \theta} \cg \sigma, \rho_0 \cd - \E_P [ \langle \sigma,M_0 \rangle ]+ \cg \theta, \rho_1\cd - \E_P [ \langle \theta,M_1 \rangle ]	    
	\end{equation*}
	also takes values in $\{ 0, + \infty \}$ and cancels exactly on the $P$ such that $\E_P[M_0] = \rho_0$ and $\E_P[M_1] = \rho_1$. Thus the problem reads
	\begin{align*}
	\BrSch_{\nu, \boldsymbol q, R_0}(\rho_0, \rho_1)&=\min_P \sup_{\sigma, \theta} \nu H(P|R) + \cg \sigma, \rho_0 \cd - \E_P [ \langle \sigma,M_0 \rangle ]+ \cg \theta, \rho_1\cd - \E_P [ \langle \theta,M_1 \rangle ] \\
	&= \min_P \sup_{\sigma,\theta, \Xi} \nu \E_P[\Xi] - \nu \log \E_R[\exp \Xi] + \cg \sigma, \rho_0\cd - \E_P [ \langle \sigma,M_0 \rangle ] + \cg \theta, \rho_1\cd - \E_P [ \langle \theta,M_1 \rangle ],
	\end{align*}  
	where we used that the Legendre transform of the relative entropy with respect to a given measure $R$ is $\Xi \to \log \E_R [ \exp \Xi ]$. Again by convex duality we can swap the $\min$ and the $\sup$, but this time as the problem is ill-posed we get the lower semi-convex envelope, leading to 
	\begin{equation*}
	\overline{\BrSch}_{\nu, \boldsymbol q, R_0}(\rho_0, \rho_1) = \sup_{\sigma,\theta, \Xi} \cg \sigma, \rho_0 \cd +\cg \theta, \rho_1 \cd- \nu \log \E_R[\exp \Xi] +  \min_P \E_P\big[ \nu \Xi  - \langle \sigma,M_0 \rangle - \langle \theta,M_1 \rangle \big],
	\end{equation*}
	Optimizing in $P$, we get $-\infty$, unless $\Xi = \frac{1}{\nu}(\langle \sigma,M_0 \rangle + \langle \theta,M_1 \rangle)$. The branching Schrödinger problem is hence reformulated:
	\begin{equation}
	\label{eq:dual_generic_R}
	\overline{\BrSch}_{\nu, \boldsymbol q, R_0}(\rho_0,\rho_1) = \sup_{\sigma,\theta} \cg \sigma, \rho_0 \cd +\cg \theta, \rho_1 \cd  - \nu \log \E_R \left[  \exp \left( \frac{1}{\nu} \big\{\langle \sigma,M_0 \rangle + \langle \theta,M_1 \rangle \big\}\right) \right].
	\end{equation}
	Until now $R$ could have been any process, but we will now dive in the specifics of the BBM by using a martingale characterization. The idea is that, for a smooth test function $\varphi \in C^2([0,1] \times \T^d)$ which satisfies the PDE
	\begin{equation}
	\label{eq:constraint_phi_BBM_unnormalized}
	\partial_t \varphi + \frac{\nu}{2} \left( \Delta\varphi + |\nabla \varphi|^2\right) + \left(\Phi_{\boldsymbol{q}}(e^\varphi) e^{-\varphi} - \Phi_{\boldsymbol{q}}(1) \right) = 0
	\end{equation}
	then the process whose value at time $t \in [0,1]$ is 
	\begin{equation}
	\label{eq:exponential_martingale_BBM}
	\exp \left( \langle \varphi(t), M_t \rangle \right)
	\end{equation} 
	(still using the notation $\varphi(t) := \varphi(t, \cdot)$) is a martingale under $R$, it can be read for instance in \cite[Theorem~1.4]{etheridge2000introduction}. Actually to write rigorous statements one usually assumes that $\varphi \leq 0$ in order for \eqref{eq:exponential_martingale_BBM} is bounded. One of the main technical difficulty of our proof will be to go beyond this case, that is to handle the case where $\varphi$ is not nonnegative. Also, we will prefer to make the change of variables $\phi := \nu \varphi$, in such a way that \eqref{eq:constraint_phi_BBM_unnormalized} translates in  
	\begin{equation}
	\label{eq:constraint_phi_BBM}
	\dr_t \phi + \frac{|\nabla \phi|^2}{2} + \frac{\nu}{2} \Delta \phi + \nu \left( \Phi_{\boldsymbol q}(e^{\phi/\nu}) e^{- \phi/\nu} - \lambda_{\boldsymbol q} \right) = 0,
	\end{equation}
	where $\lambda_{\boldsymbol q} := \sum_k q_k$, and
	\begin{equation*}
	\left( \exp \left( \frac{1}{\nu}\langle \phi(t), M_t \rangle \right) \right)_{t \in [0,1]}
	\end{equation*}
	is a martingale under $R$. Now assume that we solve the backward PDE \eqref{eq:constraint_phi_BBM} with terminal condition $\phi(1) = \theta$ up to time $0$ (this is another difficulty: the solution could blow up before time $t=0$, and there is some work to do to handle this case). Then going back to the problem and using the martingale property, we get
	\begin{align*}
	\sup_{\sigma,\theta} \cg \sigma, \rho_0 \cd +&\cg \theta, \rho_1 \cd - \nu \log \E_R \left[  \exp \left( \frac{1}{\nu} \big\{\langle \sigma,M_0 \rangle + \langle \theta,M_1 \rangle \big\}\right) \right] \\
	& = \sup_{\sigma,\phi} \cg \phi(1), \rho_1 \cd +\cg \sigma, \rho_0 \cd  - \nu \log \E_R \left[  \exp \left( \frac{1}{\nu} \langle \sigma ,M_0 \rangle\right) \E_R \left[ \left. \exp \left( \frac{1}{\nu} \langle \phi(1) ,M_1 \rangle\right)  \right| M_0 \right] \right] \\
	& = \sup_{\sigma,\phi} \cg \phi(1), \rho_1 \cd +\cg \sigma, \rho_0 \cd  - \nu \log \E_R \left[  \exp \left( \frac{1}{\nu} \langle \sigma + \phi(0),M_0 \rangle\right) \right].
	\end{align*}
	Changing the variables according to $\tilde{\sigma} = \nu^{-1} (\sigma + \phi(0))$ and optimizing in $\tilde{\sigma}$, we recognize the Legendre transform of the log Laplace transform of $R_0$, so in the end we end up with
	\begin{equation*}
	\overline{\BrSch}_{\nu, \boldsymbol q, R_0}(\rho_0,\rho_1) = \sup_{\phi} \cg \phi(1), \rho_1 \cd - \cg \phi(0), \rho_0 \cd + \nu L_{R_0}^*(\rho_0).
	\end{equation*}
	Up to the term $\nu L_{R_0}^*(\rho_0)$, we recognize exactly the objective functional of \eqref{eq:RUOT_dual}. The constraint on $\phi$, which is \eqref{eq:constraint_phi_BBM}, is the same as the constraint in~\eqref{eq:RUOT_dual} as soon as $\Psi^*(s) = \Psi^*_{\nu, \boldsymbol q}(s)$ as defined in~\eqref{eq:def_Psi*} which is the assumption in the statement of the theorem. 
	
	\begin{Rem}
		\label{rem:FKPP_log_exp}
		Notice that calling $u = e^{\phi/\nu} = e^\phi$, then $\phi$ solves \eqref{eq:constraint_phi_BBM} if and only if $u$ solves the FKPP equation 
		\begin{equation}
		\label{eq:FisherKPP}
		\dr_t u + \frac{\nu}{2} \Delta u  + \Phi_{\boldsymbol q}(u) - \lambda_{\boldsymbol q}u  = 0, 
		\end{equation}
		which is well known to be connected to BBM \cite{mckean1975application}. In our theoretical analysis, we will actually rely on this equivalence and switch between \eqref{eq:constraint_phi_BBM} and \eqref{eq:FisherKPP} depending on the context. Also, note that if $M_t = \sum_{i=1}^p \delta_{x_i} \in \M_+(\T^d)$ corresponds to $p$ particles located at $x_1, x_2, \ldots, x_p$ then 
		\begin{equation*}
		\exp \left( \frac{1}{\nu}\langle \phi(t), M_t \rangle \right) = \prod_{i=1}^p u(t, x_i).
		\end{equation*}
	\end{Rem}
	
	\begin{Rem}
		\label{rk:duality_Schrodinger_pb}
		As a safety check, let us explain what happens in the case of the Brownian motion. We restrict to $M_t = \delta_{X_t}$ where $(X_t)_{t \in [0,1]}$ follows a Brownian motion with diffusivity $\nu$. Thus here the dual \eqref{eq:dual_generic_R} reads
		\begin{equation*}
		\sup_{\sigma,\theta}\cg \sigma, \rho_0 \cd +\cg \theta, \rho_1 \cd  - \nu \log \E_R \left[  \exp \left( \frac{1}{\nu} \big\{ \sigma(X_0) + \theta(X_1) \big\}\right) \right],
		\end{equation*}
		and to follow the same strategy we only need to compute 
		\begin{equation*}
		\E_R \left[ \left. \exp \left( \frac{\theta(X_1)}{\nu} \right) \right| X_0 \right] = \left[\tau_\nu \ast \exp \left( \frac{\theta}{\nu} \right)  \right](X_0),
		\end{equation*}
		being $(\tau_s)_{s > 0}$ the heat kernel on the torus. This dual one gets is the same as in \cite[Section 2]{leonard2013survey}, which of course should be the case. However, when we go to the BBM, we cannot use this simple representation as convolution with a kernel, and we have to rely on \eqref{eq:constraint_phi_BBM}, or equivalently to the FKPP equation \eqref{eq:FisherKPP}. 
	\end{Rem}		
	
	We finish by referring the reader to Section~\ref{sec:general_superproc} about formal computations in cases where $R$ is no longer the law of the BBM, but of a more general measure-valued branching Markov processes.  	
	
	\section{Overview of the contributions of the article}
	\label{sec:overview_article}
	
	Now that we have presented the formal duality argument justifying the equivalence between the models, we will give an overview of the results of the present work. In particular some of them are concerned with a more precise description of this equivalence.
	
	\paragraph{Objects of interest and preliminary results}
	
	In this chapter, after presenting some notations that we will use throughout this work, we introduce both the RUOT model and the BBM, and we give some preliminary results about the relative entropy functional. At the level of optimal transport, we explain how the variational model presented in Section~\ref{subsec:intro_OT} can be made rigorous, why the problem is well posed and we establish the duality equality \eqref{eq:RUOT_dual}. As far as the BBM is concerned, we start by introducing it rigorously using a description where the particles have labels. Then, we explain how to build solutions to the equations~\eqref{eq:constraint_phi_BBM} and~\eqref{eq:FisherKPP} by computing moments of the BBM, and why processes such as~\eqref{eq:exponential_martingale_BBM} are martingales. Note that these are standard results in the case where the functions $\varphi$ in~\eqref{eq:constraint_phi_BBM} and~\eqref{eq:exponential_martingale_BBM} are nonpositive, but we will give some careful details to also treat the case where they take positive values. In this case, the notion of solution to~\eqref{eq:FisherKPP} needs to be weakened, and the process in~\eqref{eq:exponential_martingale_BBM} will only be a \emph{local} martingale. Finally, as far as the relative entropy is concerned, we compute its Legendre transform, show how conditioning makes it possible to decouple the ``initial'' part and the ``dynamical'' part of the entropy, and we introduce some additional assumptions on $R_0$ and $\boldsymbol q$ that will be necessary in Chapters~\ref{chap:characterization_finite_entropy} and~\ref{chap:equivalence_competitors}. These new assumptions require the existence of exponential moments for $R_0$ and $\boldsymbol q$, and we show that they guarantee that any $P$ with finite entropy w.r.t.~$R \sim\BBM(\nu, \boldsymbol q, R_0)$ admits an expected marginal $\E_P[M_t]$, for all $t \in [0,1]$.
	
	\paragraph{Equality of the values: a proof by duality}
	
	In this chapter we prove that indeed $(\rho_0, \rho_1) \to \nu L_{R_0}^*(\rho_0) + \ruot_{\nu, \Psi}(\rho_0,\rho_1)$ is the l.s.c.\ envelope of $\BrSch_{\nu, \boldsymbol q, R_0}(\rho_0,\rho_1)$, provided $\Psi = \Psi_{\nu, \boldsymbol q}$ as defined in~\eqref{eq:def_Psi*}. For this result, no assumption is needed on the parameters of the BBM, except for the existence of a first moment for $\boldsymbol q$ that is necessary to define the BBM. 
	
	The first step is actually the study of the static problem which corresponds to the initial distribution. That is, we show that $L_{R_0}^*$ is the l.s.c.\ envelope of the functional
	\begin{equation}
	\label{eq:intro_init}
	\rho_0 \to \Init_{R_0}(\rho_0) := \inf_{P_0} \{ H(P_0|R_0) \ : \ \E_{P_0}[M] = \rho_0 \},
	\end{equation}    
	where here the competitors are taken in $\P(\M_+(\T^d))$, that is they are laws of random measures. As explained in Remark~\ref{rem:grand_canonical}, this problem is nothing but an inverse problem in grand canonical classical statistical mechanics. For both this static case, and then the dynamical case, the computation of the l.s.c.\ envelope is done via two consecutive Legendre transform, which is standard in convex analysis. However, for the dynamical case the main difficulty will be to justify the equality:
	\begin{equation*}
	\E_R \left[ \left. \exp \left( \frac{1}{\nu}\langle \phi(t), M_t \rangle \right) \right| M_0 \right] = \exp \left( \frac{1}{\nu}\langle \phi(0), M_0 \rangle \right) 
	\end{equation*}
	when $\phi$ is a solution of \eqref{eq:constraint_phi_BBM}, but not necessarily nonpositive. Actually, we will see that we can guarantee it only for a restricted class of smooth solution of this equation, and it will be enough thanks to regularization arguments.
	
	To give a full picture about the link between the two problems, we end up this chapter by providing counterexamples showing that there is an actual discrepancy between both $\Init_{R_0}(\rho_0)$ and $\BrSch_{\nu, \boldsymbol q, R_0}$ and their respective l.s.c.\ envelopes, so that our results are somehow optimal.

	\paragraph{Laws with finite entropy w.r.t. a branching Brownian motion}
	
	When we assume the existence of exponential moments for $R_0$ and $\boldsymbol q$, we can go further and study an equivalence between competitors of both problems, that is, we can answer the question: how does one build a competitor in RUOT from one in the branching Schrödinger problem, and \emph{vice versa}?
	
	Note that if one has a vector field $\tilde{v} : [0,1] \times \T^d \to \R^d$ and a space-time dependent branching mechanism $\tilde{\boldsymbol{q}} : [0,1] \times \T^d \to \M_+(\N)$, one could intuitively build a branching process where particles undergo a SDE with drift $\tilde{v}$ and diffusion $\nu$, while they branch and give birth to $k$ offsprings at position $(t,x)$ with rate $\tilde{q}_k(t,x)$. We show that we can construct such a process $P$ under mild assumptions on $\tilde{v}$ and $\tilde{\boldsymbol{q}}$, but also we explicit the Radon-Nikodym density of $P$ with respect to $R$: 
	\begin{equation*}
	\frac{\D P}{\D R} := \frac{\D P_0}{\D R_0}(M_0) \exp\left( \frac{1}{\nu} I[\tilde v]_1 + \sum_{k \in \N} J^k\left[\log \frac{\tilde q_k}{q_k}\right]_1 - \int_0^1 \left\cg \frac{|\tilde v(t)|^2}{2\nu} + \lambda_{\tilde{\boldsymbol q}(t)} - \lambda_{\boldsymbol q}, M_t \right\cd \D t  \right),
	\end{equation*}
		where $\lambda_{\boldsymbol q}$ and $\lambda_{\tilde{\boldsymbol q}(t)}$ are the branching rates associated with $\boldsymbol q$ and $\tilde{\boldsymbol q}(t)$ respectively. This density features two stochastic processes adapted to the BBM. Indeed, given a (random) predictable vector field $\tilde v$, we define $(I[\tilde v]_t)_{t \in [0,1]}$ which will correspond to a classical stochastic integral with respect to Brownian motion, but summed over all particles. The corresponding term in the Radon-Nikodym density above is reminiscent of the Brownian case, see formula~\eqref{eq:RN_derivative_brownian_case} in the introduction of Chapter~\ref{chap:characterization_finite_entropy}. We will also define, for a (random) predictable scalar field $a$ and an index $k$, the quantity $(J^k[a]_t)_{t \in [0,1]}$, which corresponds to summing $a(t,x)$ every time that there is a branching event with $k$ offsprings at time $t$ and position $x$. This time, the corresponding term in the Radon-Nikodym density above is reminiscent of the pure-jump case, see~\cite[Formula~(12)]{leonard2012girsanov}. Further information on this analogy is given at Remark~\ref{rem:construction_modified_BBM}. In this chapter, we also establish an Itô formula, which reads
	\begin{equation*}
	\cg \varphi(t), M_t \cd = \cg \varphi(0), M_0\cd + I[\nabla \varphi]_t + \sum_{k=0}^{+\infty} J^k[(k-1)\varphi]_t +  \int_0^t \Big\cg \partial_t \varphi(s) + \frac{\nu}{2} \Delta \varphi(s), M_s \Big\cd \D s.
	\end{equation*} 
	
	In addition, we go also the other way around. Following the ideas of the proof of \cite{leonard2012girsanov}, we are able to characterize the laws $P \in \P(\cadlag([0,1], \M_+(\T^d)))$ having a finite entropy with respect to $R$. That is, if $H(P|R) < + \infty$, we can find a predictable drift $\tilde{v}$ and a new branching mechanism $\tilde{\boldsymbol{q}}$ such that $P$ corresponds to a BBM with this drift and this branching mechanism. Note that here the drift and the branching mechanism are not only space-time dependent, but also predictable, roughly meaning that their value at time $t$ can depend on the past trajectories of all the particles.

	\paragraph{Equivalence of the competitors}
	
	Equipped with the additional results of the previous chapter, we are able to state and prove the equivalence between the competitors. Once again, this chapter will rely on our stronger assumption that $R_0$ and $\boldsymbol q$ admit exponential moments.
	
	We start with the static case, that is the study of the problem defined in \eqref{eq:intro_init}. The idea of the proof is that we can approximate $R_0$ by conditioning it to allow only a bounded number of particles. We show that in that case, we can actually prove the existence of a minimizer for the problem defining $\Init_{R_0}(\rho_0)$, and that $\Init_{R_0}(\rho_0)$ actually coincides with its lower semi-continuous envelope $L_{R_0}^*(\rho_0)$. It is by dropping the approximation that discrepancy appears.
	
	We then move to the dynamical case. The first result is as follows: if $P$ is such that $H(P|R) < + \infty$, we can use the dirift $\tilde v$ and branching mechanism $\tilde{\boldsymbol{q}}$ to build $(\rho,v,r)$ a competitor in the RUOT problem with smaller energy. The vector field $v$ is simply an averaged version of the drift $\tilde v$, but in the case of $r$, we need to further reduce the complexity and build a growth rate out of a branching mechanism. This is achieved by choosing
	\begin{equation}
	\label{eq:link_tildeq_r}
	r = \sum_{k \in \N} (k-1) \tilde q_k.
	\end{equation} 
	
	On the other hand, if we have $(\rho,v,r)$ a competitor for the RUOT problem, we first smooth it out, and then build from it a $P$ such that the energy for the RUOT problem coincides, up to the initial term $\nu L_{R_0}^*(\rho_0)$, with $H(P|R)$. Again we use the results of the previous chapter by building a $P$ with prescribed drift and branching mechanism. For the drift we choose of course the (smoothed) $v$ from RUOT. On the other hand, for the branching mechanism we have to build it from the branching rate $r$. We actually choose $\tilde{\boldsymbol{q}}$ as the minimizer of 
	\begin{equation}
	\label{eq:def_h}
	h(\tilde{\boldsymbol{q}} | \boldsymbol{q}) := \sum_k \tilde q_k \log \frac{\tilde q_k}{q_k} + q_k - \tilde q_k
	\end{equation}
	among all $\tilde{\boldsymbol{q}} \in \M_+(\N)$ satisfying \eqref{eq:link_tildeq_r}. It actually yields a new expression for the growth penalization $\Psi_{\nu, \boldsymbol q}$, equivalent to \eqref{eq:def_Psi*}: for $r \in \R$
	\begin{equation*}
	\Psi_{\nu, \boldsymbol q}(r) = \inf_{\tilde{\boldsymbol{q}} \in \M_+(\N)} \left\{ \nu h(\tilde{\boldsymbol{q}} | \boldsymbol{q}) \ : \ \sum_{k \in \N} (k-1) \tilde q_k = r  \right\}.
	\end{equation*}
	The quantity in~\eqref{eq:def_h}, sometimes called Kullback-Liebler divergence, is a standard way to extend the relative entropy functional $H$ from~\eqref{eq:def_H} to nonnegative measures that are not of unite total mass.
	
	\paragraph{Additional results: asymptotics, generalizations, numerics}
	
	The last chapter concerns additional questions which are very natural once the equivalence is established.
	
	The first one is about the small noise limit. There are now two ``noise'' parameters one can tune: the diffusivity $\nu > 0$ and the branching rate $\lambda > 0$ (for the analysis of the limit we go back to a fixed $\boldsymbol{p} \in \P(\N)$ and $\boldsymbol{q} = \lambda \boldsymbol{p}$ with $\lambda \to 0$). We already mentioned that in non-branching case, the small noise limit has been investigated and, in the limit $\nu \to 0$, the problem converges to \emph{unregularized} optimal transport. Here the situation is different: in order to have a small noise limit where both transport and growth play a role, one has to choose a particular scaling between $\nu$ and $\lambda$. By a heuristic analysis detailed in the chapter, we find that the correct scaling to see both transport and growth at the limit is
	\begin{equation*}
	\lambda = \frac{1}{\nu} \exp \left(- \frac{\varpi}{\nu} \right)
	\end{equation*}
		for some constant $\varpi > 0$. This scaling could be deduced from known results, as explained at Remark~\ref{rem:regime_small_noise}. Moreover, the only kind of growth penalization that we see in the limit is an homogeneous one of type $\Psi(r) \sim |r|$, once again in accordance with known results (see the beginning of Section~\ref{sec:smallNoise} referring to~\cite{leonard2016lazy}), and hence yielding partial optimal transport \cite[Section 5.1]{chizat2018unbalanced}. Importantly, we did not manage to recover quadratic growth penalization.
	
	The second set of additional results that we present are formal computations for more general measure-valued branching Markov processes. That is, we see a more general measure-valued branching Markov processes $R$ as a law on $\cadlag([0,1]; \M_+(\T^d))$. We assume that there exists a differential operator $\mathcal{L} : [0,1] \times C^\infty(\T^d) \to C^\infty(\T^d)$ such that for all smooth function $\varphi: [0,1] \times \T^d \to \R$ which is a solution of $\dr_t \varphi + \mathcal{L}[t,\varphi] = 0$ then
	\begin{equation*}
	\left( \exp \left( \langle \varphi(t), M_t \rangle \right) \right)_{t \in [0,1]}
	\end{equation*} 
	is a martingale under $R$, where the previous case corresponds to a particular choice of $\mathcal{L}$ as in \eqref{eq:constraint_phi_BBM}. (Here, we stick to the formal level and we do not want to address the question of possible blow-ups.) We explain why we expect it to be linked to a variational model analogous to RUOT, and what the shape of the optimizer for the entropy minimization problem under marginal constraints should be. Interestingly, in the case of the Dawson-Watanabe superprocess, which corresponds to $\mathcal{L}[\varphi] = \frac{\nu}{2} \Delta \varphi + \frac{\gamma}{2} \varphi^2$ for some $\nu, \gamma>0$, the small noise limit yields the Fisher-Rao metric which is a pure ``vertical distance'', \emph{i.e.}\ for which transport plays no role. We do not prove the equivalence rigorously as it would likely require new technical ideas and we leave it for a future work.
	
	Last but not least, we discuss the numerical aspects of RUOT and the branching Schrödinger problem. We explain why the Sinkhorn approach, so efficient for the Schrödinger problem, does not work. There are two main reasons: first the problem is not symmetric in time (see Subsection~\ref{subsec:no_sinkhorn}, we cannot choose $R_0$ such that $(M_t)_{t \in [0,1]}$ and $(M_{1-t})_{t \in [0,1]}$ have the same law); and second, solutions to the PDE \eqref{eq:constraint_phi_BBM}, or equivalently to the FKPP PDE \eqref{eq:FisherKPP}, do not exist in closed form. This is in contrast with the Schrödinger problem where this PDE is simplified into the heat equation, for which solutions can be easily computed. However, it is still possible to discretize the dynamical formulation of RUOT and to solve the resulting problem by proximal splitting. Such a method to solve the optimal transport problem (without regularization or unbalanced term) was started by Benamou and Brenier \cite{benamou2000computational} and has been used successively in a wide variety of context as explained in the chapter. We explain why, in this setting, RUOT is (almost) no more complicated to solve than OT, we implement it and we illustrate the results, see Figure~\ref{fig:numerics}. 
	
	\section*{Acknowledgements}
	\addcontentsline{toc}{section}{Acknowledgements}
	
	We wish to express our gratitude to Michel Pain for being always open to reply our beginner questions concerning the branching Brownian motion or other probabilistic topics. Also, we would like to thank Andreas Kyprianou for kindly discussing with us the state of the art about absolutely continuous changes of laws w.r.t.\ a reference branching Brownian motion.

	\chapter{Objects of interest and preliminary results}
	
	As explained in the introduction, this chapter is devoted to the presentation of the models that will appear throughout this work: the Regularized Unbalanced Optimal transport (RUOT) in Section~\ref{sec:presentation_RUOT}, and the branching Brownian motion (BBM) in Section~\ref{sec:presentation_BBM}. In Section~\ref{sec:entropy}, we introduce rigorously the branching Schrödinger problem, we provide some useful properties of the relative entropy functional that will be used in the next chapters, and we introduce the set of assumptions under which we will be able to perform our analysis in Chapter~\ref{chap:characterization_finite_entropy} and~\ref{chap:equivalence_competitors}. As a preliminary section we first introduce notations which we will use throughout the present work.
	
	\section{Notations}
	\label{sec:notations}
	
	\paragraph{Some standard notations}
	
	We will denote the minimum of two real numbers $a,b$ by $a \wedge b$. If $a$ is a real number, we will call $a_+$ and $a_-$ its positive and negative part respectively. 
	
	If $A \subseteq \mathbf{X}$ is the subset of a set, we denote by $\1_A$ its $0/1$ indicator function and $\iota_X$ its $0/+\infty$ indicator function. That is, for $x \in \mathbf{X}$,
	\begin{equation*}
	\1_A(x) := \begin{cases}
	1 & \text{if } x \in A, \\
	0 & \text{otherwise},
	\end{cases}
	\hspace{1cm}
	\iota_A(x) := \begin{cases}
	0 & \text{if } x \in A, \\
	+ \infty & \text{otherwise}.
	\end{cases}
	\end{equation*}

	\paragraph{Functional spaces}
	
	If $\mathbf{X}$ is a polish space endowed with its Borel $\sigma$-algebra, we denote by respectively $\P(\mathbf{X})$, $\M(\mathbf{X})$ and $\M_+(\mathbf{X})$ the space of Borel probability measures, Borel finite measures and Borel nonnegative finite measures on $\mathbf{X}$. The integral of a continuous integrable function $\phi : \mathbf{X} \to \R$ against a measure $\alpha$ in $\P(\mathbf{X})$, $\M(\mathbf{X})$ or $\M_+(\mathbf{X})$ is denoted by $\langle \phi, \alpha \rangle$. Observe that this notation does not refer explicitly to the the space $\mathbf X$ on which we integrate: this space is always the domain on which the measure and test function are defined. Each space is endowed with the topology of weak convergence, that is, the coarsest one making all the linear forms $\alpha \mapsto \langle \phi, \alpha \rangle$ continuous when $\phi \in C(\mathbf{X})$ is a continuous and bounded function. With such topology, $\P(\mathbf{X})$, $\M(\mathbf{X})$ and $\M_+(\mathbf{X})$ are polish spaces. Note that it is equivalent to \emph{narrow} convergence. When $\mathbf{X}$ is compact, it coincides also with the weak-$\star$ convergence as $\M(\mathbf{X})$ is the topological dual of $C(\mathbf{X})$.  
	
	If $\alpha \in \M(\mathbf{X})^d$ is a vectorial measure, we also denote by $\langle \phi, \alpha \rangle$ the integral of the vector valued $\phi : \mathbf{X} \to \R^d$ function against the measure $\alpha$.  
	
	The spaces $\T^d$ and $[0,1] \times \T^d$ are endowed with the Lebesgue measure. We write respectively $\D x$ and $\D t \otimes \D x$ when we want to integrate with respect to this measure, and we also use $\Leb \in \P(\T^d)$ when we view the Lebesgue measure on $\T^d$ as an element of $\M(\T^d)$.  
	
	We introduce properly the Kullback-Liebler divergence that we already saw in~\eqref{eq:def_h} in a particular case. We write $\mathsf p \ll \mathsf r$ if $\mathsf p$ is a measure absolutely continuous with respect to the nonnegative measure $\mathsf r$, and in this case $\frac{\D \mathsf p}{\D \mathsf r}$ denotes the Radon-Nikodym density of $\mathsf p$ with respect to $\mathsf r$.
	\begin{Def}
		\label{def:KL}
		If $\mathbf{X}$ is a polish space and $\mathsf{r} \in \M_+(\mathbf{X})$ is a nonnegative finite measure on it, for all measure $\mathsf{p}\in \M(\mathbf{X})$, the Kullback-Liebler divergence of $\mathsf p$ with respect to $\mathsf r$ is defined by:
		\begin{equation*}
		\label{eq:def_KL}
		h(\mathsf p| \mathsf r) := \left\{ \begin{aligned}
		&\int \left\{ \frac{\D \mathsf p}{\D \mathsf r} \log \frac{\D \mathsf p}{\D \mathsf r} +  1 - \frac{\D \mathsf p}{\D \mathsf r} \right\} \D \mathsf r, &&\mbox{if } \mathsf p \in \M_+(\mathbf{X}), \, \mathsf p\ll \mathsf r,\\
		&+\infty, && \mbox{else}.
		\end{aligned}
		\right.
		\end{equation*} 	
	\end{Def} 
	An argument of convexity yields $h(\mathsf p| \mathsf r) \geq 0$ with equality if and only if $\mathsf p = \mathsf r$.
	
	We will make a brief use of $L^p$ spaces, for $p \in [1, + \infty]$. We denote by $L^p(\T^d)$ and $L^p([0,1] \times \T^d)$ the set of Borel measurable functions whose $p$-power is integrable with respect to respectively $\D x$ and $\D t \otimes \D x$, and where two functions coinciding a.e.\ are identified. The case $p = + \infty$ corresponds as usual to functions which are essentially bounded. 
	
	If $\mathbf{X}$ is either $\T^d$ or $[0,1] \times \T^d$, we denote by $C^k(\mathbf{X})$ the set of functions whose derivatives up to order $k$ exist and are continuous. When $k = 0$, we simply denote by $C(\mathbf{X})$ the set of continuous functions. 
	
	\paragraph{Legendre transform}
	
	If $\V$ is a separable Banach space and $\V'$ its topological dual, we also denote the duality pairing between $v \in \V$ and $w \in \V'$ as $\langle v,w \rangle$. It is consistent with previous notations when $\mathbf{X} = \T^d$ or $[0,1] \times \T^d$ and $\V = C(\mathbf{X})$, as in this case $\V' = \M(\mathbf{X})$. If $F : \V \to (- \infty, + \infty]$, we write $F^* : \V' \to (-\infty, + \infty]$ for its Legendre transform defined by
	\begin{equation*}
	F^*(w) := \sup_{v \in \V} \langle v, w \rangle - F(v), \qquad w \in \V',
	\end{equation*}
	while if $G : \V' \to (- \infty, + \infty]$ we rather define its Legendre transform only on $\V$ as 
	\begin{equation*}
	G^*(v) := \sup_{w \in \V'} \langle v, w \rangle - G(w), \qquad v \in \V.
	\end{equation*}
	
	We recall that a convex function is said proper if it is not identically $+ \infty$, and if it does not take the value $- \infty$. The Legendre transform of a any proper convex function is always a proper and l.s.c.\ convex function.

	\paragraph{Convolution and heat kernel}
	
	We will denote the space convolution between a measure $\alpha \in \M(\T^d)$ and a kernel $K : \T^d \to \R$ by $\alpha*K$. It is the measure which reads
	\begin{equation*}
	K \ast \alpha = \left( \int_{\T^d} K(x - y) \D \alpha(y) \right) \D x. 
	\end{equation*}
	We will sometimes identity $K \ast \alpha$ with its density with respect to Lebesgue measure. Actually, if $f : \T^d \to \R$ is a function, we also denote its convolution with $K$ as $K \ast f$, which is the function 
	\begin{equation*}
	K \ast f : x \mapsto   \int_{\T^d} K(x - y) f(y) \D y.
	\end{equation*}
	If $f : [0,1] \times \T^d \to \R$ is a function depending on time and space, we will use the shortcut $f(t)$ to denote the function $x \in \T^d \mapsto f(t,x) \in \R$, and $K \ast f$ will be understood only as a \emph{space} convolution, that is, $K \ast f$ is the function such that $(K \ast f)(t) = K \ast f(t) $.
	
	For $s > 0$, we write $\tau_s : \T^d \to (0, + \infty)$ for the heat kernel on the torus defined as
	\begin{equation}
	\label{eq:def_heat_kernel}
	\tau_s(x) \propto \sum_{k \in \mathbb{Z}^d} \exp \left( \frac{|x-k|^2}{2  s} \right), 	    
	\end{equation}
	with the normalizing constant being chosen such that $\int_{\T^d} \tau_s = 1$. For any $s > 0$, it belongs to $C^\infty(\T^d)$ and is strictly positive. As a joint function of $s$ and $x \in \T^d$ it belongs to $C^\infty((0,+\infty) \times \T^d)$. When $s = 0$ we set by convention $\tau_0 := \delta_0$, the Dirac mass at $0$.

	If $\alpha \in \M(\T^d)$ is a finite measure we can define the curve $s \mapsto \tau_s \ast \alpha$ which is a curve of measures, that is a function from $(0,+\infty)$ into $\M(\T^d)$ together with the convention $\tau_0 \ast \alpha := \alpha$. For each $s > 0$ the measure $\tau_s \ast \alpha$ has a density with respect to the Lebesgue measure. Identifying a measure with its density, as a joint function of $(s,x) \in (0,+\infty) \times \T^d$ the following equation holds
	\begin{equation*}
	\partial_s [\tau_s \ast \alpha] = \frac{1}{2} \Delta [\tau_s \ast \alpha]
	\end{equation*}
	in a strong sense, together with the temporal boundary conditions
	\begin{equation*}
	\lim_{s \to 0} [\tau_s \ast \alpha] = \alpha	    
	\end{equation*}
	for the topology of weak convergence.

	\section{Models of regularized unbalanced optimal transport}
	\label{sec:presentation_RUOT}
	In this section, we give a rigorous formulation of the dynamical model of Regularized Unbalanced Optimal Transport (RUOT) and show the existence of competitors (Subsection~\ref{subsec:def_RUOT}), we establish the dual formulation (Subsection~\ref{subsec:duality_RUOT}) as well as technical results that will be useful in the sequel, and that have to do with coercivity of the energy functional and regularization (Subsection~\ref{subsec:technical_lemmas}). We will make rather weak assumptions on the parameters of the model, namely, the diffusivity coefficient $\nu$ and the growth penalization~$\Psi$. This is because we believe that the study of RUOT can be interesting in itself \emph{i.e.}\ even when $\Psi$ is not of type $\Psi_{\nu, \boldsymbol q}$ as defined through formula~\eqref{eq:def_Psi*}. However, each time we will introduce a new assumption for $\Psi$, this assumption will hold for $\Psi_{\nu, \boldsymbol q}$ provided the diffusivity coefficient $\nu$ is positive, and provided the branching mechanism $\boldsymbol q$ satisfies some properties that we will explicit. Importantly, we prove duality with very few assumptions. In particular, it will always hold for $\Psi = \Psi_{\nu, \boldsymbol q}$, whatever $\nu>0$ and $\boldsymbol q$.
	
	But for now, we give ourselves any real number $\nu$ (even possibly $0$ or negative), and a proper l.s.c.\ convex function $\Psi : \R \to [0, + \infty]$ such that $\Psi^*(0) = 0$ (or equivalently, such that $\inf \Psi = 0$). Further assumptions will come later, in Proposition~\ref{prop:ruot_existence_competitor}, Subsection~\ref{subsec:duality_RUOT} and Subsection~\ref{subsec:technical_lemmas}.
	
	\subsection{Rigorous definitions}
	\label{subsec:def_RUOT}
	Let us start with the precise definitions of our problems. As recalled in the introduction, in RUOT, the unknowns are the density $\rho$, the velocity $v$ and the growth rate $r$. For the latter two, to make the problem convex, we rather work with the momentum $m = v \rho$ and the net source of mass $\zeta = r \rho $. 
	
	Actually, when $\Psi$ is not superlinear at $\pm \infty$, the net source of mass $\zeta$ is not necessarily absolutely continuous with respect to $\rho$. As this phenomenon can occur for $ \Psi = \Psi_{\nu, \boldsymbol q}$ at $+\infty$ when $q_k$ does not decrease sufficiently fast with $k$, see Section~\ref{sec:no_exp_moment} and~\ref{sec:exp_moment} from Appendix~\ref{app:plots}, we have develop the theory of the RUOT problem in this setting where $\Psi$ is not assumed to be superlinear. 
	
	To get good compactness properties, we consider our unknowns as integrated in time. More precisely, $\rho, m$ and $\zeta$ will be measures on $[0,1] \times \T^d$ where we think of $\D \rho(t,x)$ as $\rho(t, \mathrm{d} x) \D t$, and similarly for $m$ and~$\zeta$ (even though this representation is not rigorous in the case of $\zeta$). 
	
	\begin{Def}
		\label{def:CE}
		A triple $(\rho, m, \zeta)$ where $\rho \in \M([0,1] \times \T^d)$, $m \in \M([0,1] \times \T^d)^d$ and $\zeta \in \M([0,1] \times \T^d)$ is said to satisfy the equation
		\begin{equation}
		\label{eq:continuity_linear}
		\partial_t \rho + \Div m = \frac{\nu}{2} \Delta \rho + \zeta
		\end{equation}
		in a weak sense if for all smooth test function $\phi : [0,1] \times \T^d \to \R$ which vanish on $\{ 0,1 \} \times \T^d$ there holds 
		\begin{equation}
		\label{eq:continuity_weak_form_without_boundary}
		\left\cg \partial_t \phi + \frac{\nu}{2} \Delta \phi, \rho \right\cd + \langle \nabla \phi, m \rangle + \langle \phi, \zeta \rangle = 0.
		\end{equation}
		
		In addition, if $\rho_0, \rho_1 \in \M(\T^d)$, we say that $(\rho, m, \zeta)$ satisfies the equation \eqref{eq:continuity_linear} with boundary conditions $\rho_0, \rho_1$ in a weak sense if for all smooth test function $\phi : [0,1] \times \T^d \to \R$ there holds
		\begin{equation}
		\label{eq:continuity_weak_form}
		\left\cg \partial_t \phi + \frac{\nu}{2} \Delta \phi, \rho \right\cd + \langle \nabla \phi, m \rangle + \langle \phi, \zeta \rangle = \cg \phi(1), \rho_1 \cd - \cg \phi(0), \rho_0\cd
		\end{equation}
	\end{Def}
	
	Clearly, the set of triples $(\rho, m, \zeta)$ which satisfy \eqref{eq:continuity_linear} with or without given boundary conditions is closed for the topology of weak convergence. Being a solution imposes some mild regularity on $\rho$.

	\begin{Lem}
		\label{lem:desintegration_rho_ruot}
		Let $(\rho,m,\zeta)$ be a weak solution of~\eqref{eq:continuity_linear}. Then there exists a measurable map $t \mapsto \rho_t \in \M_+(\T^d)$ such that $\rho$ is of the form $\D t \otimes \rho_t$. Moreover, the map $t \mapsto \rho_t(\T^d)$ is of bounded variations. In particular, up to redefining $\rho_t$ on a Lebesgue-negligible set of $t$, there holds $\sup_{t \in [0,1]} \rho_t(\T^d) = \esssup_{t \in [0,1]} \rho_t(\T^d) < + \infty$. In what follows, we will always assume that this is the case.
	\end{Lem}
	
	\begin{proof}
		This is the same as \cite[Lemma 1.1.2]{chizat2017unbalanced}, as the presence of diffusion does not change anything (the equation~\eqref{eq:continuity_weak_form_without_boundary} is tested with a function depending only on time). Note that the temporal derivative of $\rho_t(\T^d)$, in the sense of measures, coincides with $\pi \pf \zeta$, being $\pi : (t,x) \in [0,1] \times \T^d \to t \in [0,1]$ the projection on the temporal axis. 
	\end{proof}
	
	An easy generalization of~\cite[Remark~1.3]{ambrosio2014continuity} yields that if $m$ and $\zeta$ are absolutely continuous with respect to $\rho$, then the map $t \mapsto \rho_t$ is continuous for the topology of weak convergence. However, as already mentioned, given the type of growth penalization $\Psi$ that we will consider, it is not always guaranteed that $\zeta$ is absolutely continuous with respect to $\rho$. Even if $t \mapsto \rho_t$ is not continuous, it is still possible to give a meaning to the temporal boundary conditions.
	
	\begin{Lem}
		\label{lem:existence_boundary_CE}
		Let $(\rho,m,\zeta)$ be a solution of~\eqref{eq:continuity_linear}, in the sense of~\eqref{eq:continuity_weak_form_without_boundary}. Assume in addition that $m \ll \rho$. Then there exists a unique pair $\rho_0, \rho_1 \in \M(\T^d)$ such that $(\rho,m,\zeta)$ is a solution of~\eqref{eq:continuity_linear} with boundary conditions $\rho_0, \rho_1$, in the sense of~\eqref{eq:continuity_weak_form}. Moreover, $(\rho_0, \rho_1)$ depends in a continuous way of $(\rho,m,\zeta)$ for the topology of weak convergence.  
	\end{Lem}
	
	\begin{Rem}
		\begin{itemize}
		\item Even when $\rho$ is nonnegative, we cannot guarantee that $\rho_0$ (resp.\ $\rho_1$) is a nonnegative measure because $\zeta$ could give a positive (resp.\ negative) mass to $\{ 0 \} \times \T^d$ (resp. $\{ 1 \} \times \T^d$).
		\item In terms of the curve $t \mapsto \rho_t$ from Lemma~\ref{lem:desintegration_rho_ruot}, these boundary conditions can be understood as follows. Take $(\rho,m,\zeta)$ as in Lemma~\ref{lem:existence_boundary_CE} and call $\tilde m,\tilde \zeta$ the extensions of $m,\zeta$ on $\R \times \T^d$ which cancel outside of $[0,1]\times \T^d$. Then there is a unique $\tilde \rho = \D t \otimes \tilde \rho_t \in \M(\R \times \T^d)$ such that $(\tilde \rho, \tilde m, \tilde \zeta)$ solves~\eqref{eq:continuity_linear} on $\R \times \T^d$, and such that $\tilde \rho$ coincides with $\rho$ on $[0,1]\times \T^d$. Moreover $t \mapsto \tilde \rho_t \in \M(\T^d)$ is of bounded variations, and hence has left and right limits at each time. The boundary conditions from Lemma~\ref{lem:existence_boundary_CE} correspond to the left limit at time $0$ and to the right limit at time $1$ of this map.
		\end{itemize}
	\end{Rem}
	
	\begin{proof}
		Take $(\rho,m,\zeta)$ a solution of~\eqref{eq:continuity_linear} in a weak sense. As a function of $\phi$, the l.h.s.\ of \eqref{eq:continuity_weak_form} defines a distribution on $[0,1] \times \T^d$. Moreover, thanks to \eqref{eq:continuity_weak_form_without_boundary}, we see that this distribution is supported on $\{ 0,1 \} \times \T^d$. To prove that it can be written as the r.h.s.\ of \eqref{eq:continuity_weak_form}, it is enough to prove that it is a distribution of order~$0$, that is, the l.h.s.\ can be controlled by $\| \phi \|_\infty$ only. 
		
		To that end, let us fix $\phi \in C^2([0,1] \times \T^d)$ and write $\rho = \D t \otimes \rho_t$ as per Lemma~\ref{lem:desintegration_rho_ruot}. We look at the function $f : [0,1] \to \R$ defined (\emph{a.e.}) by $f(t) = \cg \phi(t), \rho_t \cd$, $t \in [0,1]$. Using \eqref{eq:continuity_weak_form_without_boundary} with a test function of type $(t,x)\mapsto\chi(t) \phi(t,x)$ for some $\chi : [0,1] \times \R$ canceling at $0$ and $1$, we can see that the function $f$ is of bounded variation, its distributional derivative $\dot{f}$ being the measure defined for all $\chi$ by:
		\begin{equation*}
		\cg \chi, \dot{f} \cd = \left\cg \chi \left( \partial_t \phi + \frac{\nu}{2} \Delta \phi \right), \rho \right\cd + \langle \chi \nabla \phi, m \rangle + \langle \chi \phi, \zeta \rangle.
		\end{equation*}  
		As $m \ll \rho$ and $\rho = \D t \otimes \rho_t$, the only singular part of $\dot{f}$ with respect to $\D t$ is given by the last term containing~$\zeta$. In particular, the mass of the singular part of $\dot{f}$ is controlled by $\| \phi \|_\infty |\zeta|([0,1] \times \T^d)$. Then, note that the l.h.s.\ of \eqref{eq:continuity_weak_form} is nothing but $\dot f([0,1])$. As $f$ is of bounded variation, this quantity is controlled by $|f(0+)| + |f(1-)| + |\dot{f}^s|(\{ 0,1 \} \times \T^d)$, being $\dot{f}^s$ the singular part of the measure $\dot{f}$. As mentioned above, this singular part is controlled by $\| \phi \|_\infty |\zeta|([0,1] \times \T^d)$. Moreover thanks to Lemma~\ref{lem:desintegration_rho_ruot} and the definition of $f$, it is easy to see that $|f(0+)| + |f(1-)|$ is controlled by $2 \sup_t \rho_t(\T^d) \, \| \phi \|_\infty$. Thus we conclude that the l.h.s.\ of~\eqref{eq:continuity_weak_form} is a distribution of order $0$ in $\phi$, hence the existence of $\rho_0, \rho_1$. 
		
		Continuity of $\rho_0, \rho_1$ with respect to $(\rho,m,\zeta)$ is direct thanks to~\eqref{eq:continuity_weak_form}.  
	\end{proof}

	Next we need to define the energy. Given our change of variables $(\rho,m,\zeta) = (\rho, \rho v, \rho r)$, the density of energy would be
	\begin{equation*}
	\frac{|v|^2}{2} \rho + \Psi \left( r \right) \rho = \frac{|m|^2}{2 \rho} + \Psi \left( \frac{\zeta}{\rho} \right) \rho.
	\end{equation*}
	To justify that this term is well defined, we follow the usual definition by convex duality. There is a subtlety here: when $\Psi$ grows linearly at $\pm \infty$, the expression $\Psi ( \zeta/\rho ) \rho$ still makes sense when $\zeta$ is not absolutely continuous w.r.t.~$\rho$. Actually, it is necessary to allow this possibility in order to make the energy l.s.c. Namely, we define the energy of a triple $(\rho, m, \zeta)$ as follows.

	\begin{Def}
		\label{def:energy_UOT}
		If $\rho \in \M([0,1] \times \T^d)$, $m \in \M([0,1] \times \T^d)^d$ and $\zeta \in \M([0,1] \times \T^d)$, then we define the energy $E_\Psi(\rho, m, \zeta) \in [0, + \infty]$ of the triple $(\rho, m, \zeta)$ as
		\begin{equation*}
		E_\Psi(\rho,m,\zeta) = \sup_{(a,b,c) \in \mathcal{K}} \cg a, \rho \cd +  \cg b, m \cd +  \cg c, \zeta \cd  
		\end{equation*}
		where $\mathcal{K}$ is made of triples $(a,b,c)$ of continuous functions $a,c : [0,1] \times \T^d \to \R$ and $b : [0,1] \times \T^d \to \R^d$ such that, for all $(t,x) \in  [0,1] \times \T^d$,
		\begin{equation*}
		a(t,x) + \frac{|b(t,x)|^2}{2} + \Psi^*(c(t,x)) \leq 0. 
		\end{equation*}
	\end{Def}
	
	Some useful and classical properties of this functional are collected below. To that end, we introduce the following notations:
	\begin{equation}
	\label{eq:def_horizon_Psi_star}
	L^+_{\Psi} =\lim_{r \to + \infty} \frac{\Psi(r)}{r}, \qquad L^-_{\Psi} = \lim_{r \to - \infty} \frac{\Psi(r)}{|r|}.
	\end{equation} 
	
	It is easy to see that whenever $\Psi$ is convex, proper and nonnegative, these quantities are well defined in $[0, + \infty]$. Alternatively, $L^+_\Psi$ (resp.~$L^-_\Psi$) can be defined as the supremum (resp.\ infimum) of the domain where $\Psi^*$ is finite. 	
	
	\begin{Prop}
		\label{prop:energy_UOT}
		The functional $E_\Psi$ is convex and l.s.c.\ for the topology of weak convergence on the product space $\M([0,1] \times \T^d)\times \M([0,1] \times \T^d)^d \times \M([0,1] \times \T^d)$. Moreover, if $E_\Psi(\rho,m,\zeta) < + \infty$ then $\rho$ belongs to the set $\M_+([0,1] \times \T^d)$ of nonnegative measures and $m\ll \rho$. Furthermore, calling $v := \frac{\D m}{\D \rho}$, writing $\zeta = r \rho + \zeta^s$ the Radon-Nikodym decomposition of $\zeta$ with respect to $\rho$, and $\zeta^s = \zeta^s_+ - \zeta^s_-$ the Jordan decomposition of the singular part $\zeta^s$, there holds
		\begin{equation}
		\label{eq:energy_UOT}
		E_\Psi(\rho,m,\zeta) =\int_0^1 \hspace{-5pt} \int \left\{ \frac{1}{2} |v(t,x)|^2 + \Psi(r(t,x))\right\} \D \rho(t,x) + L^+_{\Psi} \zeta^s_+([0,1] \times \T^d) + L^-_{\Psi} \zeta^s_-([0,1] \times \T^d) .
		\end{equation}
	\end{Prop}
	
	Note that if $\Psi$ is superlinear at $\pm \infty$, which is equivalent to $L^\pm_{\Psi} = + \infty$, then $E_\Psi(\rho,m,\zeta) < + \infty$ implies that $\zeta \ll \rho$ and the last two terms in~\eqref{eq:energy_UOT} vanish.  		
	
	\begin{proof}
		This is classical in unbalanced optimal transport: for instance we use \cite[Lemma 2.9]{chizat2018unbalanced} with $X = \T^d$, $n = 1+d+1$ and 
		\begin{equation*}
		f(x,(\rho,m,\zeta)) = f((\rho,m,\zeta)) = \frac{|m|^2}{2 \rho} + \Psi \left( \frac{\zeta}{\rho} \right) \rho. 
		\end{equation*}
		Then, for the expression~\eqref{eq:energy_UOT} we essentially use \cite[Proposition 7.7]{santambrogio2015optimal}.
	\end{proof}
	
	We are now in position to define the RUOT problem. 
	
	\begin{Def}
		\label{def:RUOT}
		Let $\rho_0, \rho_1 \in \M_+(\T^d)$ and $\nu, \Psi$ be as above. The regularized unbalanced optimal transport problem $\RUOT_{\nu,\Psi}(\rho_0,\rho_1)$ is the minimization of $E_\Psi(\rho,m,\zeta)$ among all triples $(\rho,m,\zeta) \in \M([0,1] \times \T^d)\times \M([0,1] \times \T^d)^d\times \M([0,1] \times \T^d)$ satisfying \eqref{eq:continuity_linear} with boundary conditions $\rho_0, \rho_1$ in a weak sense. We call $\ruot_{\nu,\Psi}(\rho_0,\rho_1)$ its optimal value, that is 
		\begin{equation}
		\label{eq:def_ruot_rigorous}
		\ruot_{\nu,\Psi}(\rho_0,\rho_1) = \inf \left\{ E_\Psi(\rho,m,\zeta) \ : \ (\rho,m,\zeta) \text{ satisfies \eqref{eq:continuity_linear} with boundary conditions } \rho_0, \rho_1 \right\}.
		\end{equation}
		
		We call a competitor for the RUOT problem $\RUOT_{\nu,\Psi}(\rho_0,\rho_1)$ any triple $(\rho,m,\zeta)$ satisfying~\eqref{eq:continuity_linear} with boundary conditions $\rho_0, \rho_1$ in a weak sense, and such that $E_\Psi(\rho,m,\zeta) < + \infty$.
		
		When the initial and final times $0,1$ are replaced by $t_0,t_1 \in \R$, we adapt in a straightforward way the definitions of $\RUOT_{\nu, \Psi}$ and $\ruot_{\nu, \Psi}$ to provide respectively $\RUOT_{\nu, \Psi,t_0,t_1}$ and $\ruot_{\nu, \Psi,t_0,t_1}$.
	\end{Def}
	
	If $\nu = 0$, then we recover the unbalanced optimal transport in its dynamical formulation. If $\Psi(r)$ is $+ \infty$ for all $r$ but $r = 0$, or equivalently if $\Psi^* = 0$ everywhere (that is we enforce $\zeta$ to be $0$), then we get the regularized optimal transport problem. 
	
	\bigskip
	
	Facing this kind of optimization problems, there are three results that are very natural to look for: existence of a competitor, existence of a solution and duality. For the second one, we could use the direct method of calculus of variations. Indeed, once obtained the existence of competitors, this method boils down to proving coercivity, which we actually do, as the latter is a consequence of Proposition~\ref{prop:control_mass_from_energy} below. Instead, we will get solutions directly from the duality result Theorem~\ref{thm:existence_ruot}, as they are given for free by the Fenchel-Rockafellar theorem. 
	
	\bigskip
	
	But for the moment, let us start by the existence of competitors. For this result, we need very mild upper bounds on $\Psi$. This is to be compared with our assumption for duality and coercivity, which consist in a below bound. The proof relies on the corresponding results for the regularized optimal transport, see for instance~\cite{gentil2017analogy}.
	
	\begin{Prop}
		\label{prop:ruot_existence_competitor}
		Let $\rho_0, \rho_1 \in \M_+(\T^d)$. Assume that $h(\rho_0|\Leb) < + \infty$ if $\nu < 0$, and that $h(\rho_1|\Leb) < + \infty$ if $\nu > 0$ (recall that $h$ was defined in Definition~\ref{def:KL}). In addition:
		\begin{itemize}
			\item If $\rho_0(\T^d) \leq \rho_1(\T^d)$ assume that $\Psi$ is finite on $[0,+\infty)$, and, in addition, if $\rho_0=0$ that $\Psi$ grows at most polynomially at $+ \infty$.
			\item If $\rho_0(\T^d) \geq \rho_1(\T^d)$ assume that $\Psi$ is finite on $(- \infty,0]$, and, in addition, if $\rho_1 = 0$ that $\Psi$ grows at most polynomially at $- \infty$.
		\end{itemize}
		
		Then there exists a competitor with finite energy for $\RUOT_{\nu, \Psi}(\rho_0, \rho_1)$. 
	\end{Prop}

	\begin{Rem}
		\label{rk:ruot_existence_competitor_q}
		If $\boldsymbol{q} = (q_k)_{k \in \N}$ is a positive measure on $\N$ and $\nu > 0$ while $\Psi_{\nu, \boldsymbol{q}}$ is defined with the help of its Legendre transform given in formula~\eqref{eq:def_Psi*} then the assumptions on $\Psi$ in Proposition~\ref{prop:ruot_existence_competitor} can be translated into assumptions on $\boldsymbol{q}$. Indeed, notice that $\Psi^*_{\nu, \boldsymbol{q}}$ grows exponentially fast at $+ \infty$ (resp.\ $- \infty$) if $\sum_{k \geq 2} q_k > 0$ (resp.\ $q_0 > 0$). This translates in a finiteness and a polynomial growth of $\Psi_{\nu, \boldsymbol{q}}$ at respectively $+ \infty$ and $- \infty$. Excluding the case $q_k = 0$ for all $k$ but $k=1$ (which corresponds to balanced transport) we fall into three different cases: take $\rho_0, \rho_1 \in \M_+(\T^d)$ which can be the zero measure and assume $h(\rho_1|\Leb) < + \infty$. 
		\begin{itemize}
			\item If $q_0 > 0$ and $\sum_{k \geq 2} q_k > 0$, then there exists a competitor in $\RUOT_{\nu, \Psi_{\nu, \boldsymbol{q}}}(\rho_0, \rho_1)$.
			\item If $q_0 = 0$ but $\sum_{k \geq 2} q_k > 0$, then there exists a competitor in $\RUOT_{\nu, \Psi_{\nu, \boldsymbol{q}}}(\rho_0, \rho_1)$ if $\rho_0(\T^d) \leq \rho_1(\T^d)$.
			\item If $\sum_{k \geq 2} q_k = 0$ but $q_0>0$, then there exists a competitor in $\RUOT_{\nu, \Psi_{\nu, \boldsymbol{q}}}(\rho_0, \rho_1)$ if $\rho_0(\T^d) \geq \rho_1(\T^d)$.
		\end{itemize}
		It is easily shown that in these items, the sufficient conditions are also necessary. We refer to Appendix~\ref{app:plots}, and in particular to Section~\ref{sec:psi_infinite} for illustrations of these different cases.
	\end{Rem}
	
	\begin{proof}
		The case $\nu < 0$ can be reduced to $\nu > 0$ by doing a change of variables $t \leftrightarrow -t$. Moreover, the case $\nu = 0$ is covered in \cite[Proposition 1.1.14]{chizat2017unbalanced}. Thus we restrict our attention to the case $\nu > 0$. 
		
		Then, for two nonnegative measures $\alpha, \beta$ on the torus, let us write $\alpha \to \beta$ if for any $t_0 < t_1$, there exists a competitor in the problem $\RUOT_{\nu, \Psi,t_0,t_1}(\alpha, \beta)$. By concatenation, the relation $\alpha \to \beta$ is transitive on $\M_+(\T^d)$. The proof relies on showing that in the following diagram 
		\begin{equation*}
		\rho_0  \to \rho_0(\T^d) \Leb  \to \rho_1(\T^d) \Leb \to \rho_1,
		\end{equation*}
		each arrow is justified, so that we can conclude that $\rho_0 \to \rho_1$ using transitivity. So let us detail why each of the arrow drawn above holds given the assumptions we made.
		
		First, we claim that if $\alpha, \beta \in \M_+(\T^d)$ have the same total mass, and if $h(\alpha|\Leb),h(\beta|\Leb) < + \infty$, then $\alpha \to \beta$. Indeed, one can just take $\zeta = 0$ and rely on the known results for the regularized optimal transport problem, that is the Schrödinger problem \cite[Theorem 5.1]{gentil2017analogy}. Actually, by taking $m$ and $\zeta$ to be $0$ for a short time at the beginning (thus $\rho$ follows the forward heat equation with diffusivity $\nu$) we can dispense from the assumption $h(\alpha|\Leb) < + \infty$. It justifies the two external arrows.
		
		If $\Psi$ is finite everywhere on $[0, + \infty)$, it is easy to show that $\alpha \to \beta$ provided $\alpha = c_\alpha\Leb$, $\beta = c_\beta \Leb$ with $0 < c_\alpha \leq c_\beta$ by taking $r = \frac{\D \zeta}{\D \rho}$ constant in space and time. A similar statement holds if $0 < c_\beta \leq c_\alpha$ and $\Psi$ is finite everywhere on $(-\infty,0]$.
		
		On the other hand, assume in addition that $\Psi$ has a polynomial growth at $+ \infty$, that is, $\Psi(r) \leq C r^p$ for some $p < + \infty$ and $r$ large enough. Then, if $\alpha = c_\alpha \Leb$ is a multiple of the Lebesgue measure there holds $0 \to \alpha$: to build a competitor with finite energy it is enough to take $\rho(t,\cdot) = c(t) \Leb$ with 
		\begin{align*}
		c(t) = c_\alpha \left( \frac{t - t_0}{t_1-t_0} \right)^{p}
		\end{align*}  
		as this yields a curve with finite energy given the polynomial growth of $\Psi$. The same reasoning also yields $\alpha \to 0$ if $\Psi$ has a polynomial growth at $- \infty$.
	\end{proof}

	\subsection{Duality for the RUOT problem}
	\label{subsec:duality_RUOT}
	In this subsection, we prove a duality result for the RUOT problem defined in Definition~\ref{def:RUOT}, under the following assumption on $\Psi$, that can be seen as a below bound.
	
	\begin{Ass}
		\label{ass:ruot_weak}	
		We assume that $\Psi^*$ is finite on $[-s_0,0]$ for some $s_0 > 0$ (and therefore continuous on $(-s_0,0)$). This is equivalent to a growth of $\Psi$ that is at least linear at $-\infty$, \emph{i.e.}\ to the existence of $C>0$ such that for all $r \in \R$, $\Psi(r) \geq C^{-1}(x+C)_-$
	\end{Ass}
	
	\begin{Rem}
		\label{rem:Psi_nu_q_grows_at_m_infty}
		If $\boldsymbol{q} = (q_k)_{k \in \N}$ is a positive measure on $\N$ and $\nu > 0$ while $\Psi_{\nu, \boldsymbol{q}}$ is defined with the help of its Legendre transform given in formula \eqref{eq:def_Psi*}, one can easily check that it satisfies Assumption~\ref{ass:ruot_weak}, the worst case being the one presented in Section~\ref{sec:no_exp_moment} from Appendix~\ref{app:plots}.   
	\end{Rem}
	
	In the following theorem, we state the dual problem and the absence of duality gap, which has for corollary the existence of an optimal solution to the primal problem. 	
	
	\begin{Thm}
		\label{thm:existence_ruot}
		Let us suppose that $\Psi$ satisfies Assumption~\ref{ass:ruot_weak}. Let $\rho_0, \rho_1 \in \M_+(\T^d)$ be two nonnegative measures. Then there holds 
		\begin{align}
		\label{eq:duality_RUOT}
		\ruot_{\nu, \Psi}(\rho_0,\rho_1) 
		= \sup_{\phi} \left\{\cg \phi(1), \rho_1 \cd - \cg \phi(0), \rho_0 \cd  \ : \ \partial_t \phi + \frac{1}{2} |\nabla \phi|^2 + \frac{\nu}{2} \Delta \phi + \Psi^*(\phi) \leq 0 \text{ on  } [0,1] \times \T^d \right\},
		\end{align}
		where the supremum is taken among all $C^2$ functions $\phi : [0,1] \times \T^d \to \R$. In addition, if there exists a competitor for $\RUOT_{\nu, \Psi}(\rho_0,\rho_1)$, then there exists an optimal solution of $\RUOT_{\nu, \Psi}(\rho_0,\rho_1)$, that is, the infimum in \eqref{eq:def_ruot_rigorous} is attained.  
	\end{Thm}

	\begin{proof}
		We use the Fenchel-Rockafellar theorem (see for instance \cite[Theorem 1.12]{brezis2011functional}). Let $\V$ be the Banach space $C([0,1] \times \T^d; \R \times \R^d \times \R)$ endowed with the topology of uniform convergence. By the Riesz representation theorem its topological dual is identified to 
		\begin{equation*}
		\V' = \M([0,1] \times \T^d)\times \M([0,1] \times \T^d)^d \times \M([0,1] \times \T^d).
		\end{equation*}
		We define two functionals on $\V$: $F : \V \to [0,+ \infty]$ is defined by 
		\begin{equation*}
		F(a,b,c) = \left\{\begin{aligned}
		&- \cg \phi(1), \rho_1 \cd + \cg \phi(0), \rho_0 \cd &&  
		\text{if }(a,b,c) = (\partial_t \phi + \textstyle{ \frac{\nu}{2} } \Delta \phi, \nabla \phi, \phi)\text{ for some }\phi \in C^2([0,1]\times \T^d),
		\\[5pt]
		&+ \infty && \text{otherwise},
		\end{aligned}
		\right.
		\end{equation*} 
		while $G : \V \to [0,+ \infty]$ is defined by 
		\begin{equation*}
		G(a,b,c) = \begin{cases}
		0 & \text{if } \forall (t,x) \in [0,1] \times \T^d, \ \displaystyle{a(t,x) + \frac{|b(t,x)|^2}{2} + \Psi^*(c(t,x)) \leq 0}, \\
		+ \infty & \text{otherwise}.
		\end{cases}
		\end{equation*}
		Both functions are convex on $\V$. We need to find one point of $\V$ where $F$ is finite while $G$ is continuous. 
		
		Call $C := |\Psi^*(s_0)| / s_0 < + \infty$, where $s_0$ is given by Assumption~\ref{ass:ruot_weak}. As $\Psi^*(0) = 0$, we have by convexity for all $s \in [-s_0,0]$, $\Psi(s) \leq C |s|$. Consider 
		\begin{equation*}
		y(t) := - s_0\frac{\exp(Ct)-1}{\exp(C)-1}, \qquad t \in [0,1].
		\end{equation*} 
		Notice that therefore, for all $t \in [0,1]$, $y(t) \in [-s_0, 0]$, so that our bound on $\Psi^*$ applies. Finally, define $\phi(t,x) := y(t)$ (that is, we take a competitor of the dual problem that is homogeneous in space) and $(a,b,c):=(\partial_t \phi + \frac{\nu}{2} \Delta \phi, \nabla \phi, \phi)=(\dot y, 0, y))$. The fact that $F(a,b,c)$ is finite is clear from the definitions. The fact that $G$ is continuous at $(a,b,c)$ follows from the fact that for all $t \in [0,1]$ and $x\in \T^d$:
		\begin{equation*}
		a(t,x) + \frac{|b(t,x)|^2}{2} + \Psi^*(c(t,x)) = \dot y(t) + \Psi^*(y(t)) \leq \dot y(t) + C|y(t)| = -\frac{C s_0}{\exp(C) - 1}  < 0.
		\end{equation*}
		
		From these observations, and using the Fenchel-Rockafellar theorem, we can write 
		\begin{equation*}
		\min_{(\rho,m,\zeta) \in \V'} (F^*(-(\rho,m,\zeta)) + G^*(\rho,m, \zeta) ) = - \inf_{\V} (F + G)  .
		\end{equation*}
		In the right hand side we recognize the right hand side of \eqref{eq:duality_RUOT}: $F$ and $G$ were built for that. On the other hand, $G^* = E_\Psi$ by definition while $F^*(-(\rho,m,\zeta))$ is $0$ if and only if $(\rho,m,\zeta)$ is a weak solution of \eqref{eq:continuity_linear} with boundary conditions $\rho_0, \rho_1$, and $+\infty$ otherwise. Thus the left hand side is precisely $\ruot_{\nu, \Psi}(\rho_0,\rho_1)$, as defined by formula~\eqref{eq:def_ruot_rigorous}. 
		
		Eventually, we recall that the Fenchel-Rockafellar Theorem guarantees existence of a solution for the problem posed over $\V'$, at least if the value of the problems is not $+ \infty$.
	\end{proof}
	
	Observe that the following consequence of duality gives a first hint why it can be the l.s.c.\ relaxation of the branching Schrödinger problem.
	
	\begin{Prop}
		Given $\nu, \Psi$ satisfying Assumption~\ref{ass:ruot_weak}, the function $(\rho_0, \rho_1) \mapsto \ruot_{\nu, \Psi}(\rho_0, \rho_1)$ is l.s.c.\ for the topology of weak convergence on $\M_+(\T^d)^2$.    
	\end{Prop}
	
	\begin{proof}
		Equation \eqref{eq:duality_RUOT} shows that $\ruot_{\nu, \Psi}$ can be expressed as a supremum of continuous functions for the topology of weak convergence, hence it is l.s.c.
	\end{proof}

	\subsection{Some useful technical lemmas}
	\label{subsec:technical_lemmas}
	
	This subsection is devoted to prove estimates implying coercivity that will be useful later on, and to describe a regularization procedure for the competitors of the RUOT problem. The reader can safely ignore this section at a first reading and come back to it when needed.

	Let us start with our coercivity estimates, under the same Assumption~\ref{ass:ruot_weak} as for the duality.
	
	\begin{Prop}
		\label{prop:control_mass_from_energy}
		Suppose that $\Psi$ satisfies Assumption~\ref{ass:ruot_weak}, and let $C>0$ be the corresponding constant. Let $\rho_0, \rho_1 \in \M_+(\T^d)$ and
		let $(\rho,m,\zeta)$ be a competitor for $\RUOT_{\nu, \Psi}(\rho_0, \rho_1)$. Let $(\rho_t)_{t \in [0,1]}$ be the disintegration of $\rho$ with respect to time, given by Lemma~\ref{lem:desintegration_rho_ruot}. Then there exists a constant $\tilde C$ only depending on $C$ such that the following estimates hold.
		\begin{itemize}
			\item Concerning the density, we have the pointwise estimate
			\begin{equation}
			\label{eq:uniform_bound_density}
			\sup_{t \in [0,1]} \rho_t(\T^d) \leq \tilde C \big( \rho_1(\T^d)  + E_{\Psi}(\rho,m,\zeta)\big).   
			\end{equation}
			\item Concerning the momentum, for all $0 \leq t_0 \leq t_1 \leq 1$, we have the local estimate
			\begin{equation}
			\label{eq:bound_velocity_momentum}
			|m|([t_0,t_1] \times \T^d) \leq \sqrt{t_1-t_0} \times \sqrt{\tilde C \big(\rho_1(\T^d) + E_{\Psi}(\rho,m,\zeta)\big)E_\Psi(\rho,m,\zeta)}.
			\end{equation}
			\item Concerning the net mass balance, we have the global estimate
			\begin{equation}
			\label{eq:bound_growth_momentum_new}|\zeta|([0,1] \times \T^d) \leq \tilde C \big( \rho_1(\T^d)  + E_{\Psi}(\rho,m,\zeta)\big).
			\end{equation}
		\end{itemize}				
	\end{Prop}
	\begin{Rem}
		\label{rem:coercivity}
		Note that having only a linear growth of $\Psi$ does not prevent $\zeta$ from concentrating in time. Thus in~\eqref{eq:bound_growth_momentum_new} we cannot have a control of $|\zeta|([t_0,t_1] \times \T^d)$ which goes to $0$ as $t_1 - t_0$ goes to $0$.
	\end{Rem} 
	
	\begin{proof}
		The proof is easily adapted from the one of \cite[Proposition 1.1.21]{chizat2017unbalanced}. We assume $E_\Psi(\rho,m,\zeta) < + \infty$, otherwise all our estimates trivially hold.
		
		As $(\rho,m,\zeta)$ solves~\eqref{eq:continuity_linear}, recall that calling $\pi: (t,x) \in [0,1]\times \T^d \mapsto t$ and $\pi \pf \zeta$ the temporal marginal of $\zeta$, the map $t \in [0,1] \mapsto \rho_t (\T^d)$ is of bounded variation, of derivative $\pi \pf \zeta$. In particular, integrating between $t$ and $1$, we get for almost every $t$
		\begin{equation}
		\label{eq:zz_aux_coercivity_1}
		\rho_t(\T^d) = \rho_1(\T^d) - \zeta([t,1] \times \T^d).
		\end{equation} 		
		
		We want to control the r.h.s.\ with the energy. To that end, let us write $\zeta = r \rho + \zeta^s_+ - \zeta^s_-$, following the notations of Proposition~\ref{prop:energy_UOT}. Note that by Assumption~\ref{ass:ruot_weak}, the quantity $L^-_{\Psi}$, defined in~\eqref{eq:def_horizon_Psi_star}, is bounded from below by $C^{-1}$. Also, $L^+_\Psi$ is nonnegative. Thus, given~\eqref{eq:energy_UOT}, for any $t \in [0,1]$,
		\begin{align}
		\notag
		E_\Psi(\rho,m,\zeta) & \geq \int_0^t\hspace{-5pt}\int \Psi(r(s,x)) \D \rho(s,x) + L^+_{\Psi} \zeta^s_+([0,t] \times \T^d) + L^-_{\Psi} \zeta^s_-([0,t] \times \T^d) \\
		\notag
		& \geq \int_0^1\hspace{-5pt}\int C^{-1}(r(s,x) + C)_- \D \rho(s,x)  + C^{-1} \zeta^s_-([0,t] \times \T^d) \\
		\notag
		& \geq -  \int_0^t \rho_s(\T^d) \D s + C^{-1} \int_0^1\hspace{-5pt}\int r(s,x)_-\D \rho(s,x) + C^{-1}\zeta^s_-([0,s] \times \T^d) \\
		\label{eq:zz_aux_coercivity_2}
		& = -  \int_0^t \rho_s(\T^d) \D s + C^{-1} \zeta_-([0,t] \times \T^d),
		\end{align}
		where we used the inequality $(r + C)_- \geq r_- - C$ to get the third line, and where $\zeta = \zeta_+ - \zeta_-$ is the Jordan decomposition of $\zeta$. Plugging that back in the previous estimate~\eqref{eq:zz_aux_coercivity_1} and using $- \zeta \leq \zeta_-$, we get for almost every $t$
		\begin{equation*}
		\rho_t(\T^d) \leq \rho_1(\T^d) + C E_\Psi(\rho,m,\zeta) + C \int_t^1 \rho_s(\T^d) \D s.
		\end{equation*} 
		The result follows by Gronwall's lemma and the fact that we chose $t \mapsto \rho_t$ is such a way that $\sup_t \rho_t(\T^d) = \esssup_t \rho_t(\T^d)$.  
		
		Next, we prove the estimate on $m$. We have by the Cauchy-Schwarz inequality
		\begin{equation*}
		|m|([t_0,t_1] \times \T^d) \leq \int_{t_0}^{t_1} \hspace{-5pt} \int_{\T^d} |v(t,x)| \D \rho_t(x) \D t \leq\sqrt{\int_{t_0}^{t_1} \hspace{-5pt} \int_{\T^d} \D \rho_t(x) \D t} \sqrt{\int_{t_0}^{t_1} \hspace{-5pt} \int_{\T^d} |v(t,x)|^2 \D \rho_t(x) \D t}, 
		\end{equation*}
		which provides the result once used our bound on $\rho$.
		
		Eventually, for the bound on $\zeta$, we have:
		\begin{align*}
		|\zeta|([0,1]\times \T^d) = \zeta([0,1] \times \T^d) + 2 \zeta_-([0,1]\times \T^d) = \rho_1(\T^d) - \rho_0(\T^d) + 2 \zeta_-([0,1]\times \T^d).
		\end{align*}
		Therefore, the result follows from~\eqref{eq:zz_aux_coercivity_2} and our estimate on $\rho$, up to considering a bigger $\tilde C$.
	\end{proof}
	
	Our second result in this section provides a way to regularize competitors of the RUOT problem which will be used in several contexts, namely, for building a competitor for the branching Schrödinger problem out of a competitor for the RUOT problem (Proposition~\ref{prop:from_RUOT_to_BrSch}), and for the small noise limit (Theorem~\ref{Thm:small_noise_limit}). The spirit of this lemma is that every competitor can be regularized in time and space without increasing too much the value of the objective functional. This kind of result is rather classical in the theory of unbalanced optimal transport, see for instance the Step 3 in the proof of \cite[Theorem~1.3.3]{chizat2017unbalanced}. However, for Proposition~\ref{prop:from_RUOT_to_BrSch}, we will need a quite specific shape of the regularizing sequence, hence our statement is a bit lengthier and the proof is a bit more tedious than what can be found in the literature.
	
	For this lemma, we will need a few assumptions on $\Psi$, including an upper bound, that we gather here.    
	
	\begin{Ass}
		\label{ass:regularization}
		We assume that: 
		\begin{itemize}
			\item $\Psi(0) < +\infty$ (or equivalently, $\inf \Psi^* > - \infty$).
			\item There exists a number $\bar r \in \R$ such that $\Psi(\bar r) = 0$.
			\item There exists a number $C>1$ such that for all $r \in \R$, we have $\Psi(2r) \leq C(1+\Psi(r))$. 
		\end{itemize}
	\end{Ass} 
	
	\begin{Rem} 
		For all $\nu>0$ and all branching mechanism $\boldsymbol q$, $\Psi_{\nu, \boldsymbol q}$ satisfies Assumption~\ref{ass:regularization}, with $\Psi_{\nu, \boldsymbol q}(0)\leq \lambda_{\boldsymbol q}$, and $\Psi_{\nu, \boldsymbol q} (\bar r)=0$ with $\bar r := \sum_k (k-1)q_k$. The third point is a bit more involved, and is left to the reader as an exercise. Observe that it includes our two possible very different cases: the case when $\Psi \equiv + \infty$ on $\R_+^*$ (resp.\ $\R_-^*$), and the case when $\Psi$ grows at most polynomially in $+\infty$ (resp.\ $-\infty$). We refer to Appendix~\ref{app:plots} to see how these two cases appear.
	\end{Rem}

	\begin{Lem}
		\label{lem:mollification}
		Let us give ourselves a diffusivity $\nu \geq 0$ and a proper l.s.c.\ convex function $\Psi$. Let us assume that $\Psi$ satisfies Assumption~\ref{ass:regularization}, and let us consider $\bar r$ as given in this assumption. Finally, let us take $\bar \nu > 0$ and a positive measure $\bar \rho_0 \in \M_+(\T^d) \backslash\{0\}$. 
		
		Let $(\rho, m, \zeta)$ be a solution of~\eqref{eq:continuity_linear} with $E_{\Psi}(\rho, m, \zeta) < + \infty$ and with nonnegative temporal marginals $\rho_0, \rho_1\in \M_+(\T^d)$. There exists a family $(\rho^\eps, m^\eps, \zeta^\eps)_{\eps>0}$ of solutions of~\eqref{eq:continuity_linear} with nonnegative temporal marginals $\rho_0^\eps, \rho_1^\eps \in \M_+(\T^d)$, satisfying the following properties.
		\begin{enumerate}
			\item \label{item:convergence} The family $(\rho^\eps, m^\eps, \zeta^\eps)_{\eps>0}$ weakly converges towards $(\rho,m,\zeta)$ as $\eps \to 0$. 

			\item \label{item:rho_smooth} For all $\eps>0$, the density $\rho^\eps$ is smooth and below bounded by a positive number uniformly on $[0,1] \times \T^d$.
			\item \label{item:fields} For all $\eps>0$, there is a vector field $v^\eps$ and a growth rate $r^\eps$ that are smooth on $[0,1] \times \T^d$ s.t.\ $m^\eps = v^\eps \rho^\eps$ and $\zeta^\eps = r^\eps \rho^\eps$. In addition, the measures $m^\eps$ and $\zeta^\varepsilon$ are concentrated on $[0, 1-\varepsilon] \times \T^d$, that is, $v^\eps(t,x)$ and $r^\varepsilon(t,x)$ vanish whenever $t \geq 1 - \varepsilon$, and $\Psi(r^\eps)$ is bounded.
			\item \label{item:energy_estimate} 
			There holds 
			\begin{equation*}
			\limsup_{\eps \to 0} E_{\Psi}(\rho^\eps, m^\eps, \zeta^\eps) \leq E_{\Psi} (\rho, m, \zeta).
			\end{equation*}
			\item \label{item:initial_condition} For all $\eps>0$, the initial density satisfies $\rho^\eps_0 = \tau_{\bar \nu \eps} \ast ((1-\eps)\rho_0 + \eps \bar \rho_0 )$.
		\end{enumerate}
	\end{Lem}
	\begin{Rem}
		\begin{itemize}
			\item In the proof of Proposition~\ref{prop:from_RUOT_to_BrSch}, we will crucially need the fifth point of this lemma, as well as the property on the support of $\zeta^\eps$.
			\item Notice that we allow $\nu$ to cancel in this lemma. This is important since for proving the small noise limit stated at Theorem~\ref{Thm:small_noise_limit}, we will need to regularize competitors of the limiting problem. 
			\item In Proposition~\ref{prop:from_RUOT_to_BrSch}, we will use $\bar{\nu} = \nu > 0$ for the fifth point. On the other hand, for the small noise limit, we can take any $\bar \nu > 0$ as the precise expression for $\rho^\eps_0$ does not matter, but we cannot take $\bar \nu = \nu$ as $\nu$ could cancel. 
		\end{itemize}
	\end{Rem}
	
	\begin{proof}[Proof of Lemma~\ref{lem:mollification}]
		There will be several levels of regularization that we will perform in separated steps, all of them being parametrized by the same coefficient $\eps>0$. Then, we will prove the properties of the mollified fields obtained at Step~\ref{step:mollification_properties}. So we fix $\eps>0$, sufficiently small for everything to be well defined, and we let it go to $0$ at Step~\ref{step:mollification_properties}.
		\begin{stepi}{Adding a positive amount of mass}
			First of all, call $\bar\rho \in \M_+([0,1]\times \T^d)$ the measure defined by $\bar \rho = \D t \otimes \bar \rho_t$, where for all $t \in [0,1]$, 
			\begin{equation*}
			\bar \rho_t := e^{\bar r t}  \tau_{\nu t} \ast \bar \rho_0 .
			\end{equation*}
			We can check that $(\bar \rho,0,\bar r \bar \rho)$ is a solution of the continuity equation~\eqref{eq:continuity_linear} with
			\begin{equation*}
			E_{\Psi}(\bar \rho, 0, \bar r \bar \rho) = \int_0^1\hspace{-5pt} \int \Psi(\bar r) \D \bar \rho(t,x) = 0.
			\end{equation*}
			
			Let $(\rho, m, \zeta)$ be as in the statement of the lemma. Let us call
			\begin{equation*}
			(\rho'^\eps,  m'^\eps, \zeta'^\eps) := (1-\eps) (\rho, m ,\zeta) + \eps (\bar \rho, 0, \bar r \bar \rho).
			\end{equation*}
			First, it is a solution of the continuity equation~\eqref{eq:continuity_linear}, as the latter is linear. Second, by convexity,  
			\begin{equation*}
			E_\Psi(\rho'^\eps,  m'^\eps, \zeta'^\eps) \leq (1-\eps) E_\Psi(\rho, m, \zeta) + \eps E_\Psi(\bar \rho, 0, \bar r \bar \rho) = (1-\eps) E_{\Psi} (\rho, m, \zeta) <+\infty.
			\end{equation*}
			Finally, using the notations of Lemma~\ref{lem:desintegration_rho_ruot}, we see that for all $t \in [0,1]$, 
			\begin{equation*}
			\eps e^{\bar r t} \bar \rho_0(\T^d) \leq \rho'^\eps_t(\T^d) \leq \eps e^{\bar r t} \bar \rho_0(\T^d) + (1-\varepsilon) \sup_{s \in [0,1]} \rho_s(\T^d).
			\end{equation*} 
			Therefore, on the one hand, $\rho'^\eps_t(\T^d)$ is below bounded by a positive constant $c_\eps>0$, depending on $\eps$ but not on $t$, and on the other hand, it is bounded from above by a constant $C>0$ uniformly in $t$ and $\eps$, that is,
			\begin{equation}
			\label{eq:zz_aux_control_mass_mollifier}
			0<c_\eps \leq \inf_{t \in [0,1]} \rho'^\eps_t(\T^d ) \leq \sup_{t \in [0,1]} \rho'^\eps_t(\T^d) \leq C< + \infty.
			\end{equation}		
		\end{stepi}
		\begin{stepi}{Regularization in space}
			For the regularization in space, we use the heat kernel (we could use any approximation of unity, but then we would not have the explicit rate of convergence given by the Li-Yau inequality, see below). Namely, we define the measures $\tilde \rho^\eps, \tilde m^\eps$ and $\tilde \zeta^\eps$ by duality as follows: $a,b,c$ three smooth test functions,
			\begin{equation}
			\label{eq:convolution_test_function}
			\cg a, \tilde \rho^\eps \cd := \cg \tau_{\bar \nu \eps} \ast a, \rho'^\eps \cd, \qquad
			\cg b, \tilde m^\eps \cd := \cg \tau_{\bar \nu \eps} \ast b,  m'^\eps \cd, \qquad
			\cg b, \tilde \zeta^\eps \cd := \cg \tau_{\bar \nu \eps} \ast c,  \zeta'^\eps \cd,
			\end{equation}	
			where the convolution with the test function is taken only in space. Note that using the notations of Lemma~\ref{lem:desintegration_rho_ruot}, $\tilde \rho^\eps_t = \tau_{\bar \nu \eps} \ast \rho'^\eps_t$. However, we cannot write anything similar for $\zeta$, as we cannot decompose \emph{a priori} $\zeta'^\eps$ as $\D t \otimes \zeta'^\eps_t$. The triple $(\tilde \rho^\eps, \tilde m^\eps, \tilde \zeta^\eps)$ still satisfies the continuity equation~\eqref{eq:continuity_linear}, now with boundary conditions $(\tilde \rho^\eps_0, \tilde \rho^\eps_1)$.
			
			At this stage, its is clear that $\tilde \rho^\eps$ is below bounded by a positive number, uniformly in time and space (this is because, $\tau_{\bar \nu \eps}$ is positive everywhere and $\rho'^\eps_t(\T^d)$ is below bounded, as in~\eqref{eq:zz_aux_control_mass_mollifier}).
			Finally, by joint convexity of the maps $(\rho,m) \mapsto |m|^2/\rho$ and $(\rho, \zeta) \mapsto \Psi(\zeta/\rho)\rho$, we have for all $\eps>0$
			\begin{equation}
			\label{eq:energy_estimate_intermediate_triple}
			E_\Psi(\tilde \rho^\eps, \tilde m^\eps, \tilde\zeta^\eps) \leq E_\Psi( \rho'^\eps,  m'^\eps, \zeta'^\eps) \leq  (1-\eps) E_{\Psi} (\rho, m, \zeta) <+\infty.
			\end{equation}
			(See also~\cite[Lemma~1.1.23]{chizat2017unbalanced}.) 
		\end{stepi}
		\begin{stepi}{Extension to times before $0$ and after $1$}
			\label{step:extension}
			Now, our mollified competitors are smooth with respect to space, and we need to regularize them in time. Because our set of times has boundaries, it is a bit more involved, and we need first to extend $(\tilde \rho^\eps, \tilde m^\eps, \tilde \zeta^\eps)$ on $\R^d \times \T^d$. We do not do it in the same way before $t=0$ and after $t=1$, for reasons that will be clear from the rest of the proof. However in both case, we set $\zeta = 0$. In the past, we take $\rho$ constant in time, with a momentum $m$ which compensates the effects of the diffusion. This is in that way that our mollified densities will satisfy Point~\ref{item:initial_condition}. In the future, we take $m = 0$, so that $\rho$ follows the heat flow. This is in order to satisfy our condition on the support of $m^\eps$ and $\zeta^\eps$ stated in Point~\ref{item:fields} after squeezing in time at Step~\ref{step:squeezing}, without increasing the value of the energy.

			For all $t \leq 0$, we set
			\begin{equation*}
			\tilde \rho^\eps_t := \tilde \rho^\eps_0 = \tau_{\bar \nu \eps} \ast ((1-\eps)\rho_0 + \eps \bar \rho_0), \qquad \tilde m^\eps_t := -\frac{\nu}{2} \nabla \tilde \rho^\eps_0 = -\frac{\nu}{2} (\nabla \tau_{\bar \nu \eps}) \ast ((1-\eps)\rho_0 + \eps \bar \rho_0), \qquad \tilde \zeta^\eps_t := 0.
			\end{equation*}
			On the other hand, for all $t \geq 1$, we set
			\begin{equation*}
			\tilde \rho^\eps_t := \tau_{\nu(1-t)} \ast \tilde \rho^\eps_1, \qquad \tilde m^\eps_t := 0, \qquad \tilde \zeta^\eps_t := 0.
			\end{equation*}	
			Then, outside of $[0,1] \times \T^d$, we set $\tilde \rho^\eps |_{(\R \backslash [0,1]) \times T^d} := \D t |_{(\R \backslash [0,1])\times \T^d} \otimes \tilde \rho^\eps_t$, and similarly for the momentum and net source of mass. Now, $(\tilde \rho^\eps, \tilde m^\eps, \tilde\zeta^\eps)$ is defined on $\R \times \T^d$, it is a solution of~\eqref{eq:continuity_linear}, and $\tilde \rho^\eps$ has a density below bounded by a positive number uniformly in time and space.  
		\end{stepi}
		\begin{stepi}{Regularization in time}
			Let us choose $(K_\eps)_{\eps>0}$ a smooth approximation of unity in $\R$. Importantly, we assume that for all $\eps>0$, $K_\eps$ has its support included in $[0,\eps^2] \subset \R_+$. We set for all  $t \in \R$
			\begin{gather*}
			\hat \rho^{\eps}_t := \int_{t-\eps}^t K_{\eps}(t-s) \tilde \rho^\eps_s \D s, \qquad
			\hat m^{\eps}_t := \int_{t-\eps}^t K_{\eps}(t-s) \tilde m^\eps_s \D s, \qquad
			\hat \zeta^{\eps}_t:= \int_{t-\eps}^t K_{\eps}(t-s) \tilde \zeta^\eps_s \D s.
			\end{gather*}
			Note that the last integral should be understood as the convolution (in time) between the measure $\tilde \zeta^\varepsilon$ and the kernel $K_{\eps}$, as done in formula~\eqref{eq:convolution_test_function} in the case of space convolution. The result is a measure on $\R \times \T^d$.	
			The triple $(\hat \rho^\eps, \hat m^\eps, \hat \zeta^\eps)$ is clearly a solution of the continuity equation~\eqref{eq:continuity_linear}. 
			
			Moreover, now our three fields $\hat \rho^{\eps}$, $\hat m^\eps$ and $\hat \zeta^\eps$ have densities with respect to $\D t \otimes \D x$, which are smooth, and $\hat \rho^\eps$ is below bounded by a positive constant uniformly in time and space. Identifying a measure with its density, we can now define $\hat v^\eps = \hat m^\eps/ \hat \rho^\eps$ and $\hat r^\eps =  \hat \zeta^\eps / \hat \rho^\eps$. They are smooth functions on $\R \times \T^d$. 
			
			Eventually, $\hat \rho^\eps_t$ coincides with $\tilde \rho^\eps_0$ as soon as $t \leq 0$ while $\hat m^\eps$ and $\hat \zeta^\eps$ are supported in $[0, 1+\eps^2] \times \T^d$.   
		\end{stepi}
		\begin{stepi}{Squeezing in time}
			\label{step:squeezing}
			The last step in the construction of the final triple is to squeeze in time. Let us define $s^\eps$ by
			\begin{equation*}
			s^\eps(t) = (1+2\eps)t, \qquad t \in [0,1].
			\end{equation*}
			It is an affine function satisfying $s_\eps(0) = 0$ and $s_\eps(1-\eps) \geq 1 + \eps^2$, as soon as $\eps$ is sufficiently small. Finally, for all $t \in \R$, we define
			\begin{equation*}
			\rho^\eps_t := \hat \rho^\eps_{s^\eps(t)}, \qquad m^{\eps}_t := (1+2\eps) \hat m^\eps_{s^\eps(t)} - \eps \nu \nabla \rho^\eps_{t},  \qquad \zeta^\eps_t :=(1+2\eps)\hat \zeta^\eps_{s^\eps(t)}. 
			\end{equation*} 
			This triple satisfies equation~\eqref{eq:continuity_linear} (note the term involving $\nabla \rho^\eps_{t}$ in the momentum, added to compensate the effect of the squeezing on the diffusivity). Of course, we have for all $t \in \R$
			\begin{equation}
			\label{eq:mollified_v_r}
			v^\eps_t := \frac{m^\eps_t}{\rho^\eps_t} = (1+2\eps) \hat v^\eps_{s^\eps(t)} - \eps \nu \nabla \log \rho^\eps_{t}, \qquad r^\eps_t := \frac{\zeta^\eps_t}{\rho^\eps_t}= (1+2\eps) \hat r^\eps_{s^\eps(t)}.
			\end{equation}
			By our choice of $s^\eps$, $m^\eps$ and $\zeta^\eps$ are supported in $[0, 1-\eps] \times \T^d$ for small values of $\eps$, and $\rho_0^\eps = \tilde \rho^\eps_0$. Our final mollified triple is the restriction of $(\rho^\eps, m^\eps, \zeta^\eps)$ to $[0,1] \times \T^d$.
		\end{stepi}
		\begin{stepi}{Properties of the mollified triples}
			\label{step:mollification_properties}
			Let us check the different properties given in the statement of the lemma. First, the points~\ref{item:rho_smooth},~\ref{item:fields} and~\ref{item:initial_condition} are obvious by construction. The only subtlety is that $\Psi(r^\eps)$ is bounded. To prove that, notice that because of the third point of Assumption~\ref{ass:regularization} and convexity, there is only two possibilities: either $\Psi$ is finite everywhere on $\R^*_+$ and hence locally bounded on $\R^*_+$, or $\Psi$ is infinite everywhere on $\R_+^*$. In the first case, $\Psi(r^\eps) \1_{r^\eps>0}$ remains bounded because $r^\eps$ is bounded, and in the second case, because $E_\Psi(\rho,m,\zeta) < +\infty$ and $\Psi(\bar r)=0$, both $\zeta$ and $\bar r$ are nonpositive, and therefore $\zeta^\eps$ and $r^\eps$ are nonpositive as well by construction. Thus, in this case as well, $\Psi(r^\eps) \1_{r^\eps>0}\equiv 0$ is bounded. The same goes for the negative values of $r^\eps$, and as $\Psi(0)$ is finite by assumption, we can conclude.
			
			It remains to prove Point~\ref{item:convergence} and Point~\ref{item:energy_estimate}. Let us start with Point~\ref{item:convergence}.
			
			By properties of approximations of unity, it is clear that as a pair of measures on $\R \times \T^d$, $(\rho^\eps,\zeta^\eps)$ converges towards $(\rho,\zeta)$, where the latter is extended on $t \leq 0$ by $\D t \otimes (\rho_0, 0)$ and on $t \geq 1$ by $\D t \otimes (\tau_{\nu(1-t)} \ast \rho_1, 0)$, in duality with $C_c(\R \times \T^d)$. In the case of $\rho$, as $\pi \pf \rho$ is absolutely continuous with respect to the Lebesgue measure  (where $\pi: (t,x) \in \R \times \T^d \mapsto t$), it is enough to conclude that when restricted to $[0,1] \times \T^d$, $\rho^\eps$ weakly converges to $\rho$. In the case of $\zeta$, this is also true, but this time because for all $\eps$, the support of $\zeta^\eps$ is included in $[0,1]\times \T^d$.
			
			Let us now treat the case of the momentum. Having a close look at its construction, we can see that it can be decomposed as
			\begin{equation}
			\label{eq:decomposition_m}
			m^\eps_t = \check m^\eps_t - f_0^\eps(t) \times \frac{\nu}{2}\nabla \rho^\eps_0- \eps\nu \nabla \rho^\eps_t, \qquad t \in [0,1],
			\end{equation}
			where:
			\begin{itemize}
				\item $\check m^\eps$ is a rescaled regularization of $m$, so that $\check m^\eps \rightharpoonup m$ (here as well one has to use either the fact that $\pi \pf m \ll \Leb$ as $E_\Psi(\rho,m,\zeta)< + \infty$, or the fact that $\check m^\eps$ is supported in $[0,1]\times \T^d$).
				\item $f^\eps_0$ is a smooth function with values in $[0,1 + 2\eps]$ which cancels for $t \geq \eps^2$ (as $s^\eps(\eps^2) \geq \eps^2$). This term comes from the extension in the past and the need to counterbalance the diffusivity when keeping a constant density.
				\item The last term comes directly from Step~\ref{step:squeezing}.
			\end{itemize} 
			Our task is to show that the two last terms in the r.h.s.\ of~\eqref{eq:decomposition_m} converge to zero. In both cases, we will rely on the fact that the densities involved can be written as the convolution of a nonnegative measure with the heat kernel $\tau_{\bar \nu \eps}$, so that the Li-Yau inequality \cite[Theorem 1.1]{li1986parabolic} applies and gives an explicit rate of convergence. For instance, for the last term, this inequality writes
			\begin{equation}
			\label{eq:LiYau}
			\int_{\T^d} |\nabla \log \rho^\eps_{t}|^2 \D \rho^\eps_t \leq \frac{d \rho^\eps_t(\T^d)}{2 \bar \nu \varepsilon}, \qquad t \in [0,1].  
			\end{equation}   
			Therefore, by the Cauchy-Scwharz inequality,
			\begin{equation*}
			\int_{\T^d} \left| \nabla \rho^\eps_{t}(x) \right| \, \D x =   \int_{\T^d} \left| \nabla \log  \rho^\eps_{t}(x) \right| \, \D \rho^\eps_t(x)
			\leq \sqrt{\int_{\T^d} |\nabla \log \rho^\eps_{t}|^2 \D \rho^\eps_t } \sqrt{\rho^\eps_t(\T^d)} \leq \sqrt{\frac{d}{2\bar \nu \eps}}\rho^\eps_t(\T^d).
			\end{equation*}
			As $\rho^\eps_t(\T^d)$ is clearly bounded uniformly in $t$ and $\eps$ (it comes from the bound~\eqref{eq:zz_aux_control_mass_mollifier}), we conclude that the l.h.s.\ is of order $1/\sqrt \eps$, uniformly in $t$, and hence tends to $0$, uniformly in $t$ when multiplied by the perfactor $\nu \eps$ from~\eqref{eq:decomposition_m}. This is enough to show that the last term in~\eqref{eq:decomposition_m} goes to $0$ in $L^1([0,1]\times \T^d)$, and hence for the weak convergence of measures. 
			
			The boundary terms are treated in the same way: $\nabla \tilde \rho_0^\eps$ is of size $1/\sqrt \eps$ in $L^1(\T^d)$ because of the space regularization and the Li-Yau inequality, and the temporal support of $f^\eps_0$ is of size $\eps^2$, so that the product of the two is of size $\eps\sqrt\eps$ in $L^1([0,1]\times \T^d)$.

			Finally, we will prove Point~\ref{item:energy_estimate}. That is, we want to estimate from above
			\begin{equation*}
			E_\Psi(\rho^\eps, m^\eps, \zeta^\eps) = \int_0^1 \hspace{-5pt} \int \left\{ \frac{1}{2} |v^\eps(t,x)|^2 \rho^\eps(t,x) + \Psi(r^\eps(t,x)) \rho^\eps(t,x) \right\} \D x \D t.
			\end{equation*}
			(This formula holds as a consequence of Proposition~\ref{prop:energy_UOT}, since $\zeta^\eps$ has no singular parts thanks to the regularization.) The first step is to go back to the ``hat'' variable. For that, recall the formula for $v^\eps$ given in~\eqref{eq:mollified_v_r}.
			Thanks to the Li-Yau inequality~\eqref{eq:LiYau},
			\begin{equation*}
			\left\|\nabla \log \rho^\eps_t \right\|_{L^2([0,1] \times \T^d, \rho^\eps)} = \sqrt{\int_0^1 \hspace{-5pt} \int |\nabla \log \rho^\eps_{t}|^2 \D \rho^\eps_t} \leq   \sqrt{\frac{d}{2\bar \nu \eps}\sup_{t \in [0,1]} \rho^\eps_t(\T^d)},
			\end{equation*}
			which goes to $0$ as $\eps \to 0$, when multiplied by the prefector $\nu \eps$ from~\eqref{eq:mollified_v_r} (for the bound on the mass of $\rho^\eps$ we use again \eqref{eq:zz_aux_control_mass_mollifier}). It enables to ``replace'' $v^\eps_t$ by $(1+2\eps) \hat v^\eps_{s^\eps(t)}$, that is: 
			\begin{multline*}
			\limsup_{\eps \to 0} \int_0^1 \hspace{-5pt} \int \left\{ \frac{1}{2} |v^\eps(t,x)|^2 \rho^\eps(t,x) + \Psi(r^\eps(t,x)) \rho^\eps(t,x) \right\} \D x \D t \\
			= \limsup_{\eps \to 0} \int_0^1 \hspace{-5pt} \int \left\{ \frac{\left(1 + 2\eps\right)^2}{2} |\hat v^\eps(s^\eps(t),x)|^2 \rho^\eps(t,x) + \Psi(r^\eps(t,x)) \rho^\eps(t,x) \right\} \D x \D t.
			\end{multline*}
			Then, using the definition of $\rho^\eps$ and $r^\eps$ in terms of $\hat \rho^\eps$ and $\hat r^\eps$, and doing the change of variables $t \leftrightarrow s^\eps(t)$, we find 
			\begin{multline*}
			\limsup_{\eps \to 0} \int_0^1 \hspace{-5pt} \int \left\{ \frac{1}{2} |v^\eps(t,x)|^2 \rho^\eps(t,x) + \Psi(r^\eps(t,x)) \rho^\eps(t,x) \right\} \D x \D t \\
			= \limsup_{\eps \to 0} \int_{0}^{1+ 2\eps} \hspace{-5pt} \int \left\{ \frac{(1+ 2\eps)^2}{2} |\hat v^\eps(t,x)|^2 \hat \rho^\eps(t,x) + \frac{1}{1+2\eps} \Psi \big( (1+2\eps) \hat r^\eps(t,x) \big) \hat \rho^\eps(t,x) \right\} \D x \D t.
			\end{multline*}
			Now, using the third point of Assumption~\ref{ass:regularization} on $\Psi$, we claim that the effect temporal scaling will vanish when $\eps \to 0$, that is
			\begin{multline}
			\label{eq:regularization_intermediate_limsup}
			\limsup_{\eps \to 0} \int_0^1 \hspace{-5pt} \int \left\{ \frac{1}{2} |v^\eps(t,x)|^2 \rho^\eps(t,x) + \Psi(r^\eps(t,x)) \rho^\eps(t,x) \right\} \D x \D t \\
			\leq \limsup_{\eps \to 0} \int_{0}^{1+2\eps} \hspace{-5pt} \int \left\{ \frac{1}{2} |\hat v^\eps(t,x)|^2 \hat \rho^\eps(t,x) +  \Psi \left( \hat r^\eps(t,x) \right) \hat \rho^\eps(t,x) \right\} \D x \D t.
			\end{multline}
			The term which is not obvious to pass to the limit is the one involving $\Psi$. But for this, we just have to observe that using the convex inequality (for $\eps$ sufficiently small)
			\begin{equation*}
			\Psi\left( (1+2\eps) r \right) \leq (1-2\eps) \Psi(r) + 2 \eps \Psi(2 r), \qquad r \in \R,
			\end{equation*}
			we find
			\begin{align*}
			\limsup_{\eps \to 0} \int_{0}^{1 + 2\eps} \hspace{-5pt} \int &\left\{  \frac{1}{1+2\eps} \Psi \big( (1+2\eps) \hat r^\eps(t,x) \big) - \Psi \left( \hat r^\eps(t,x) \right)  \right\} \hat \rho^\eps(t,x) \D x \D t\\
			&\leq \limsup_{\eps \to 0} \int_{0}^{1 + 2\eps} \hspace{-5pt} \int \left\{  \frac{-4\eps}{1+2\eps} \Psi \left( \hat r^\eps(t,x) \right) + \frac{2\eps}{1+2\eps} \Psi \left( 2 \hat r^\eps(t,x) \right)  \right\} \hat \rho^\eps(t,x) \D x \D t \\
			&\leq \limsup_{\eps \to 0} \frac{2\eps}{1+2\eps} \int_{0}^{1 + 2\eps} \hspace{-5pt} \int \Psi \left( 2 \hat r^\eps(t,x) \right) \hat \rho^\eps(t,x) \D x \D t.
			\end{align*}
			But this last quantity goes to zero as soon as the r.h.s.\ of~\eqref{eq:regularization_intermediate_limsup} is finite thanks to the third point in Assumption~\ref{ass:regularization}, and the uniform bound on $\rho^\eps_t(\T^d)$ deduced from~\eqref{eq:zz_aux_control_mass_mollifier}. Inequality~\eqref{eq:regularization_intermediate_limsup} follows.
			
			Now, we want to unwrap the effect of the smoothing in time. For this, we call $e^\eps$ the locally finite nonnegative measure defined for all $B \in \mathcal B(\R)$ by
			\begin{equation*}
			e^\eps(B) :=  \int_B \int_{\T^d} \left\{ \frac{1}{2} |\tilde v^\eps(t,x)|^2  + \Psi(\tilde r^\eps(t,x))  \right\} \D \tilde \rho^\eps(t,x) + L^+_{\Psi} \left(\tilde \zeta^\eps\right)^s_+(B \times \T^d) + L^-_{\Psi} \left(\tilde \zeta^\eps\right)^s_-(B \times \T^d) ,
			\end{equation*}
			with the notations of Proposition~\ref{prop:energy_UOT}. By the same convexity argument of the function $(\rho,m,\zeta) \to |m/\rho|^2 \rho + \Psi(\zeta/\rho) \rho$ as before, we find that for all $t \in \R$, we have 
			\begin{equation*}
			\int \left\{ \frac{1}{2} |\hat v^\eps(t,x)|^2 +  \Psi \left( \hat r^\eps(t,x) \right)  \right\} \hat \rho^\eps(t,x)\D x \leq e^\eps \ast K_\eps(t).
			\end{equation*}
			We integrate this inequality between times $0$ and $1 + 2\eps$, and we find
			\begin{equation*}
			\int_{0}^{1+2\eps}\hspace{-5pt}\int \left\{ \frac{1}{2} |\hat v^\eps(t,x)|^2 +  \Psi \left( \hat r^\eps(t,x) \right) \right\} \hat \rho^\eps(t,x) \D x \D t \leq \int_{\R} F^\eps(t) \D e^\eps(t),
			\end{equation*}
			where $F^\eps(t) := \int_0^{1+2\eps} K^\eps(s-t) \D s$, $t \in \R$ is a smooth function with values in $[0,1]$, which coincides with $1$ on $[0,1]$ and cancels outside of $[-\eps^2, 1+2\eps]$. Now, we use the fact that $e([0,1]) = E_\Psi(\tilde \rho^\eps, \tilde m^\eps, \tilde \zeta^\eps)$, that $\tilde \zeta^\eps$ is supported in $[0, 1] \times \T^d$, and that $\tilde v^\eps$ cancels for $t \geq 1$. We find  
			\begin{multline*}
			\int_{0}^{1+2\eps}\hspace{-5pt}\int \left\{ \frac{1}{2} |\hat v^\eps(t,x)|^2 +  \Psi \left( \hat r^\eps(t,x) \right) \right\} \hat \rho^\eps(t,x) \D x \D t \\
			\leq E_{\Psi}(\tilde \rho^\eps, \tilde m^\eps, \tilde \zeta^\eps) + \int_{-\eps^2}^0  \int \frac{1}{2} | \tilde{v}^\eps(t,x)|^2 \D \tilde \rho^\eps(t,x) + (2\eps + \eps^2) \Psi(0)  \sup_{t \in \R} \tilde \rho^\eps(\T^d). 
			\end{multline*}
			The last term in the r.h.s.\ clearly goes to $0$ as $\eps \to 0$, thanks to~\eqref{eq:zz_aux_control_mass_mollifier}. The $\limsup$ of the first term is controlled by $E(\rho,m,\zeta)$ thanks to~\eqref{eq:energy_estimate_intermediate_triple}. So in view of~\eqref{eq:regularization_intermediate_limsup}, the last thing that we have to prove is
			\begin{equation*}
			\int_{-\eps^2}^0  \int \frac{1}{2} | \tilde{v}^\eps(t,x)|^2 \D \tilde \rho^\eps(t,x) \\
			= \frac{\eps^2\nu^2}{8}\int_{\T^d} |\nabla \log \tilde \rho^\eps_0|^2 \D \tilde \rho^\eps_0\underset{\eps \to 0}{\longrightarrow}0.
			\end{equation*}
			where the equality comes from the definitions given in Step~\ref{step:extension}. We use again the Li-Yau inequality~\eqref{eq:LiYau} which yields that $\int| \nabla \log \tilde \rho^\eps_0|^2\D  \tilde \rho^\eps_0$ is of order $1/\eps$, and hence tends to zero when multiplied by the prefactor $\eps^2\nu^2 /8$, and the result follows. 
		\end{stepi}
		
	\end{proof}
	
	\section{The branching Brownian motion}
	\label{sec:presentation_BBM}
	In this section, we introduce the branching Brownian motion (BBM) described heuristically in the introduction. The definition of this object is given in Subsection~\ref{subsec:def_BBM}. Then in Subsection~\ref{subsec:prop_BBM}, we present its fundamental properties: it is Markov and branching, and we compute its generator. In Subsection~\ref{subsec:evolution_PDE_and_martingale_problem}, we show that these properties let us derive the \emph{evolution equation} of the BBM, and show that the process is characterized by a \emph{martingale problem}. Most of these definitions and proprieties are already well known by specialists and we rather present here the versions of them that are adapted to our context for the sake of completeness. 
	
	\subsection{Definition of the branching Brownian motion}
	\label{subsec:def_BBM}
	
	In the introduction, we explained formally what is the BBM corresponding to a diffusivity parameter $\nu>0$, a branching rate $\lambda>0$, a law of offspring $\boldsymbol{p} \in \P(\N)$ and an initial law $R_0 \in \P(\M_\delta(\T^d))$: starting from an initial configuration of particles drawn according to $R_0$, the particles follow Brownian trajectories of diffusivity $\nu$, and die at rate $\lambda$ while giving birth to $k$ particles with probability~$p_k$.
	
	Let us give a meaning to this idea, and start with the measurable space that we are going to work with.
	
	\paragraph{Canonical space for the BBM}
	We recall that the BBM will be seen as a $\cadlag$ process with values in the set $\M_+(\T^d)$ of all nonnegative finite Borel measures on $\T^d$. In the following, we denote by $M: \M_+(\T^d) \to \M_+(\T^d)$ the identity map. Observe that a given law $R_0 \in \P(\M_+(\T^d))$ is characterized by the values of its \emph{Laplace transform}:
	\begin{equation*}
	\E_{R_0} \Big[ \exp \Big( \cg \varphi, M \cd \Big) \Big], \qquad \varphi:\T^d \to \R \ \mbox{measurable}.
	\end{equation*}
	To characterize $R_0$, it is enough to stick to $\varphi$ nonpositive, but we will sometimes need to plug $\varphi$ taking positive values. The Laplace transform is still defined in the latter case but may take the value $+ \infty$.
	
	More precisely, the BBM will be a process with values in $\M_\delta(\T^d)$, the subset of $\M_+(\T^d)$ consisting of all the sums of a finite number of unitary Diracs. If $\mu \in \M_\delta(\T^d)$, we denote by $\Supp(\mu)$ the multiset defined as the support of $\mu$ counted with multiplicity, that is such that
	\begin{equation*}
	\mu = \sum_{x \in \Supp \mu} \delta_x.
	\end{equation*} 
	With this notation, depending on the circumstances, we will sometimes choose to write the Laplace transform in terms of $u = \exp(\varphi)$, through the correspondence
	\begin{equation*}
	\E_{R_0} \Big[ \exp \Big( \cg \varphi, M \cd \Big) \Big] = \E_{R_0}\Bigg[ \prod_{X \in \Supp(M)} u(X) \Bigg].
	\end{equation*}
	
	The BBM will be a law on the following space:
	
	\begin{Def}[Canonical space and related objects]
		\label{def:canonical_space}
		\begin{itemize}
			\item We call
			\begin{equation*}
			\Omega := \cadlag([0,1]; \M_+(\T^d)).
			\end{equation*}
			This is a Polish space when endowed with the Skorokhod topology (see~\cite[Section~12]{billingsley1999convergence}).
			
			\item We call $\mathcal{F}$ the corresponding Borel $\sigma$-algebra.
			
			\item We denote by $(M_t)_{t \in [0,1]}$ the canonical process on the space, \emph{i.e.}\ for all $t \in [0,1]$, we call $M_t$ the evaluation map on $\cadlag([0,1]; \M_+(\T^d))$, at time $t$. We also call $M_{t-} := \lim_{s \nearrow t} M_s$ the value of the left limit of the curve at time $t$ if $t\in(0,1]$, extended by $M_{0-} := M_0$.
			\item We define $(\mathcal{F}_t)_{t \in [0,1]}$ as the filtration generated by this process, that is for all $t \in [0,1]$, $\F_t := \sigma(M_s, \, s \in [0,t])$. Notice that in this framework, $\mathcal{F}$ coincides with $\F_1$ (see~\cite[Theorem~12.5]{billingsley1999convergence}).
		\end{itemize}
	\end{Def}
	
	\paragraph{Parameters of the BBM}
	
	Let us make three observations concerning the parameters of the BBM:
	\begin{itemize}
		\item As already explained in Remark~\ref{rem:q_instead_of_p}, we will always chose to work with $\boldsymbol{q} := \lambda \boldsymbol{p} \in \M_+(\N)\backslash\{0\}$ instead of $\lambda>0$ and $\boldsymbol{p} \in \mathcal{P}(\N)$ separately. From $\boldsymbol{q}$, it is possible to recover $\lambda$ and $\boldsymbol{p}$ thanks to the formula:
		\begin{equation}
		\label{eq:def_lambda_p}
		\lambda_{\boldsymbol q} := \boldsymbol q(\N), \qquad \boldsymbol{p}_{\boldsymbol q} := \frac{\boldsymbol q}{\lambda_{\boldsymbol q}}.
		\end{equation}
		By standard considerations on exponential times, $q_k := \boldsymbol{q}(\{k\})$ is the rate at which a particle is replaced by $k$ particles.
		
		Actually, one can observe that we could even chose $\boldsymbol q = 0$, which corresponds to non-branching Brownian particles. In this case, we set $\lambda_{\boldsymbol q} := 0$ and, for instance, $\boldsymbol p_{\boldsymbol{q}} = \delta_0$. From now on, this possibility is allowed. However, to avoid particular cases in our presentation, we will most of the time make the unnecessary assumption that $\boldsymbol q \neq 0$.
		\item If $q_1 \neq 0$, then there can be some branching events when a particle is replaced by one particle. But if we do not care about the labeling of the particles, this is exactly the same as if no branching event had occurred at all. Hence, without loss of generality, we will always assume that $p_1= 0$.
		\item Under the assumption that $\boldsymbol{q}$ has a finite first order moment:
		\begin{equation}
		\label{eq:q_finite_mean}
		\sum_k k q_k < + \infty,
		\end{equation} 
		it is not hard to see that for all $t \geq 0$, the number of branching events occurring before time $t$ is \emph{a.s.}\ finite, and therefore the evolution of the BBM remains well defined. Actually, this property would still be satisfied under a weaker assumption given in~\cite{schuh1982sums}. We decided to stick to the easy case where the average number of particles remains finite along the process (at least when it is finite at the initial time), which is precisely when~\eqref{eq:q_finite_mean} holds, see Corollary~\eqref{cor:evolution_density} below.
	\end{itemize}
	We end up with the following definition of what we call an \emph{branching mechanism}.
	\begin{Def}
		\label{def:branching_mechanism}
		We say that a measure $\boldsymbol q \in \M_+(\N)$ is a branching mechanism provided $q_1 = 0$ and the bound~\eqref{eq:q_finite_mean} holds.
	\end{Def}

	Now, we aim to define the BBM corresponding to a given diffusivity $\nu>0$, a given branching mechanism $\boldsymbol q$ and a given initial law $R_0 \in \P(\M_\delta(\T^d))$. To do so, the standard method, explained for instance in~\cite{pain2019these} or~\cite[Chapter~4]{li2010measure}, consists in first drawing a tree using Neveu's formalism~\cite{neveu1986arbres}, then drawing a branching time for each node, and finally plugging a Brownian trajectory on each branch. 
	
	We reproduce almost exactly what is done in~\cite{pain2019these}, adapting it very slightly in order to allow more than one particle at the initial time (so that we define the set of forests of a finite number of trees, and not the set of trees). 
	
	\paragraph{Set of forests with a finite number of trees}
	We first define the set that we will use to label the particles.
	\begin{Def}
		We call $\mathcal{U} := \bigcup_{k \in \N^*} (\N^*)^k$, the set of finite sequences of natural numbers.
	\end{Def}
	
	Here is the interpretation of this definition:  $n = n(1) \dots n(k) \in \mathcal U$ represents the $n(k)$-th descendant of the $n(k-1)$-th descendant of the ... of the $n(2)$-th descendant of the $n(1)$-th initial particle, see Figure~\ref{fig:convention}. 
	
	This set comes with a natural partial ordering defined as such. 
	\begin{Def}
		\label{def:partial_ordering}
		If $n = n(1)\dots n(k) \in \mathcal{U}$, and $l \in \{ 1, \dots, k\}$, we call $n_l := n(1)\dots n(l)$, and for all $n \in \mathcal{U}$, we write $n' \leq n$ (resp. $n' < n$) if there is $l \in \{1, \dots, k\}$ (resp. $l \in \{1, \dots, k-1 \}$) such that $n'=n_l$, that is, if $n'$ is an ancestor of $n$. Observe that for all $n \in \mathcal{U}$, $n_1$ stands for the original ancestor of particle $n$.
	\end{Def}
	
	Now, we define forests as follows.
	\begin{Def}
		\label{def:forest}
		A forest with a finite number of tress $\mathcal T$ is a subset of $\mathcal U$ that satisfies:
		\begin{itemize}
			\item The set of  $j \in \N^*$ such that $j \in \mathcal T$ is either empty or of the form $\{ 1, \ldots, k \}$, for some $k \in \N^*$.
			\item For all $n \in \mathcal T$, the set of $j \in \N^*$ such that $nj \in \mathcal T$ is either empty or of the form $\{ 1, \ldots, k \}$, for some $k \in \N^*$.
			\item For all $n \in \mathcal T$ and $n' \leq n$, we have $n' \in \mathcal T$.
		\end{itemize} 
	\end{Def}
	
	\paragraph{BBM with deterministic initial configuration}
	In this paragraph, we choose a family of initial positions $\xx = (x_1, \dots, x_p) \in \cup_{k \in \N} (\T^d)^k$, and we define the BBM of diffusivity $\nu>0$ and branching mechanism $\boldsymbol q\neq 0$, starting from the deterministic initial configuration $\mu = \delta_{x_1} + \dots + \delta_{x_p}$. It corresponds to the case where $R_0 = \delta_\mu$.
	
	Let us give ourselves an abstract probability space $(\Omega^\xx, \F^\xx, \mathbb{P}^\xx)$. We write $\Em_\xx$ for the expectation w.r.t.~$\Prob^\xx$. To lighten the formulas, we will denote with the same notations the random variables that we define on the different probability spaces, labeled by $\xx$. We suppose that $\Omega^\xx$ is endowed with a sequence $(L_n,e_n,Y_n)_{n \in \mathcal U}$ of independent random variables where for all $n \in \mathcal U$:
	\begin{itemize}
		\item $L_n$ has its values in $\N$, follows the law $\boldsymbol{p}_{\boldsymbol q}$ from formula~\eqref{eq:def_lambda_p}, and represents the number of descendants of particle $n$;
		\item $e_n$ has its values in $\R_+$, follows an exponential law of parameter $\lambda_{\boldsymbol q}$ from formula~\eqref{eq:def_lambda_p}, and represents the lifetime of particle $n$;
		\item $Y_n$ has its values in the set of continuous curves $C(\R_+; \T^d)$, follows the law of the Brownian motion of diffusivity $\nu$ starting from $0$, and represents the trajectory of particle $n$.
	\end{itemize}
	
	From $(L_n)_{n \in \mathcal U}$, we can define a corresponding random forest as follows.
	\begin{Def}
		We define $\mathcal{T}$ recursively by $\mathcal{T}_0 := \{1, \dots, p \}$ and for all $k \geq 0$:
		\begin{equation*}
		\mathcal{T}_{k+1} := \Big\{ nj \, : \, n \in \mathcal{T}_k, \,  j \in \{ 1, \ldots, L_n \} \Big\},
		\end{equation*}
		and finally $\mathcal{T} := \bigcup_{k \in \N} \mathcal{T}_k$.
	\end{Def} 
	It is easy to check that the random object $\mathcal T$ has its values in the set of forests with a finite number of trees, as defined in Definition~\ref{def:forest}.
	
	Next, we define the moments of birth and death of each particle and the set of particles alive at time $t \geq 0$ using the random variables $(e_n)_{n \in \mathcal U}$. We also define the position of these particles, using the random variables $(Y_n)_{n \in \mathcal U}$.
	\begin{Def}
		\begin{itemize}
			\item For a particle $n$ belonging to the forest $\mathcal{T}$, we call 
			\begin{equation*}
			b_n := \sum_{q < n} e_q \qquad \mbox{and} \qquad d_n := b_n + e_n,
			\end{equation*}
			the moment of birth and the moment of death of $n$ respectively. For $n \notin \mathcal T$, we set $b_n = d_n := + \infty$.
			
			\item The set of particles alive at time $t \geq 0$ is the set
			\begin{equation*}
			\mathcal{N}^0(t) := \Big\{ n \in \mathcal{U} \, : \, b_n \leq t < d_n \Big\}.
			\end{equation*}
			
			\item For a given $t \geq 0$ and $n \in \mathcal{N}^0(t)$, the position of the particle $n$ at time $t$ is given by
			\begin{equation*}
			X_n(t) := x_{n_1} + \sum_{p < n} Y_p(e_p) + Y_n(t-b_n).
			\end{equation*}
		\end{itemize}
	\end{Def}
	
	\begin{Conv}
		\label{conv:trajectory}
		As it is practical in several situations, we also define $X_n(t)$ for $t\geq 0$ outside of $[b_n,d_n)$ as follows:
		\begin{itemize}
			\item If $t \geq d_n$, we set $X_n(t) := X_n(d_n)$.
			\item If $n_1>p$, the initial number of particles, that is if $n$ has no ancestor in $\mathcal T$, we set $X_n(t) = 0$.
			\item If $n_1 \leq p$ and $t < b_n$, then $n$ has an ancestor in the set $\cup_{s \leq t} \mathcal N^0(s)$ of all particles that has been alive before time $t$. We call $X_n(t)= X_{n'}(t)$, where $n'$ is the biggest of these ancestors, for the partial ordering defined in Definition~\ref{def:partial_ordering}.
		\end{itemize}
		In plain English, before its birth a particle is located at the position of its biggest ancestor alive at time $t$, and after its death we locate the particle where it died. The subtlety is for those particle which will never be born, and which have no ancestor alive at time $t$. For those, we choose the dying position of its last ancestor that has been alive, see Figure~\ref{fig:convention}. 
	\end{Conv}
	\begin{figure}
		\begin{center}
			\begin{tikzpicture}[scale = 1]
			\begin{axis}[
			xlabel={Time},
			ylabel={Space},
			xmin=0, xmax=1,
			ymin = 0.1, ymax = 0.8]
			\addplot[color=blue,line width = 0.7pt] table [x=t1, y=x1]{data_BBM.txt};
			\addplot[color=blue,line width = 0.7pt] table [x=t2, y=x2]{data_BBM.txt};
			\addplot[color=blue,line width = 0.7pt] table [x=t3, y=x3]{data_BBM.txt};
			\addplot[color=blue,line width = 0.7pt] table [x=t4, y=x4]{data_BBM.txt};
			\addplot[color=blue,line width = 0.7pt] table [x=t5, y=x5]{data_BBM.txt};
			\addplot[color=blue,line width = 0.7pt] table [x=t6, y=x6]{data_BBM.txt};
			\addplot[color=blue,line width = 0.7pt] table [x=t7, y=x7]{data_BBM.txt};
			\end{axis}
			\draw[thick, dashed] (5.9,5.7) -- (5.9,0);
			\draw[thick] (5.9,.05) -- (5.9,-.05) node[below=-1pt]{$t$};
			\draw (.8,3.6) -- (.8,4) -- (1.4,4) -- (1.4,3.6) -- cycle;
			\draw (1.11,3.8) node{$2$}; 
			\draw (.7,.3) -- (.7,.7) -- (1.3,.7) -- (1.3,.3) -- cycle;
			\draw (1.01,.5) node{$1$};
			\draw (2.4,2.1) -- (2.4,2.5) -- (3,2.5) -- (3,2.1) -- cycle;
			\draw (2.71,2.3) node{$22$};
			\draw (2.6,4) -- (2.6,4.4) -- (3.2,4.4) -- (3.2,4) -- cycle;
			\draw (2.91,4.2) node{$21$};
			\draw (6.1,2.5) -- (6.1,2.9) -- (6.7,2.9) -- (6.7,2.5) -- cycle;
			\draw (6.41,2.7) node{$212$};
			\fill[white] (6.5,4.2) -- (6.5,4.6) -- (7.1,4.6) -- (7.1,4.2) -- cycle;
			\draw (6.5,4.2) -- (6.5,4.6) -- (7.1,4.6) -- (7.1,4.2) -- cycle;
			\draw (6.81,4.4) node{$213$};
			\draw (5.1,4.3) -- (5.1,4.7) -- (5.7,4.7) -- (5.7,4.3) -- cycle;
			\draw (5.41,4.5) node{$211$};
			\draw[color=red] plot[mark=*,mark size=3.9pt] (5.9,4.5);
			\draw[color=orange] plot[mark=triangle*,mark size=3.9pt] (2.2,2.9);
			\end{tikzpicture}
		\end{center}
		\caption{We took back Figure~\ref{fig:example_BBM}, and specified the labels of the particles. According to our conventions, at time $t$, $X_{211}(t) = \textcolor{red}{\bullet}$ (the particle is alive), $X_2(t) = \textcolor{orange}{\blacktriangle}$ (the particle is already dead), $X_{2115}(t) = \textcolor{red}{\bullet}$ (the particle is not born yet, or more precisely, we do not know yet at time $t$ if it will be born one day or not), and $X_{232}(t) = \textcolor{orange}{\blacktriangle}$ (we already know at time $t$ that the particle will never be born).}
		\label{fig:convention}
	\end{figure}
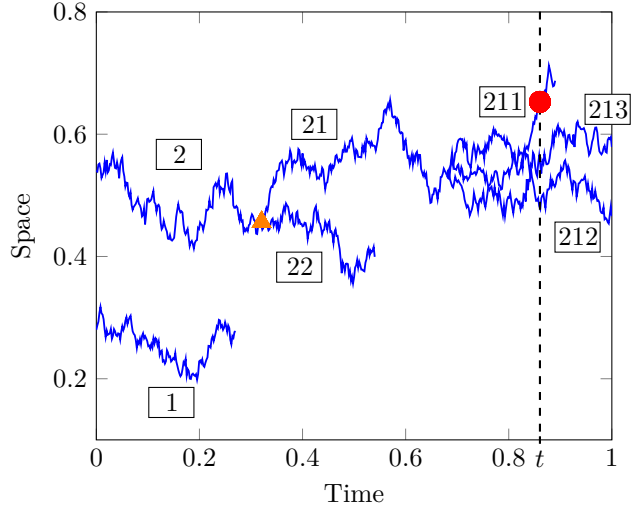

	The following existence proposition, which follows from~\cite[Proposition~4.1]{li2010measure}, is essentially a consequence of the bound~\eqref{eq:q_finite_mean}.
	\begin{Prop}
		\label{prop:existence_BBM}
		Let us define for $t \geq 0$ the empirical measure associated with the positions of the particles alive at time $t$: 
		\begin{equation*}
		M^0_t := \sum_{n \in \mathcal{N}^0(t)} \delta_{X_n(t)}.
		\end{equation*}
		The measured valued curve $M^0 = (M^0_t)_{t \in [0,1]}$ is almost surely $\cadlag$ for the topology of weak convergence. Otherwise stated, almost surely, $M^0 \in \Omega$, the canonical space defined in Definition~\ref{def:canonical_space}.
		
		Also, its law only depends on $\mu = \delta_{x_1} + \dots + \delta_{x_p}$, and not on the labeling of $x_1, \dots, x_p$.
	\end{Prop}
	
	Hence, we can define the BBM $R^\mu$ starting from the deterministic initial measure $\mu$ as the law of this random variable $M^0$.
	
	\begin{Def}
		\label{def:BBM_deterministic_R0}
		Let $\mu \in \M_\delta(\T^d)$. We define the BBM of parameters $\nu$ and $\boldsymbol q$, and starting from the deterministic initial configuration $\mu$, as the law $R^\mu \in \P(\Omega)$ of $M^0$ defined in Proposition~\ref{prop:existence_BBM}. 
		
		In order to stick with standard notations in the theory of Markov processes, in this very case where the initial configuration is deterministic, we will use the notation $\Em_\mu$ instead of $\E_{R^\mu}$ for the expectation w.r.t.~$R^\mu$. Hence, in this case, the dependence of these expectations w.r.t.~$\nu$ and $\boldsymbol q$ is implicit.
	\end{Def}
	
	\paragraph{BBM with general initial law}
	Now, the diffusivity $\nu>0$ and the branching mechanism $\boldsymbol q\neq 0$ are still fixed, and we consider a general initial law $R_0 \in \P(\M_\delta(\T^d))$. 
	\begin{Def}
		\label{def:BBM_general_R0}
		The BBM $R \sim \BBM(\nu, \boldsymbol{q}, R_0)$ is
		\begin{equation*}
		R := \int R^{\mu} \D R_0(\mu) = \E_{R_0}\left[R^M\right].
		\end{equation*}
		That is, with our notations, for every random variable $Z$ on $\Omega$ for which the expectation exists, we have
		\begin{equation}
		\label{eq:computation_expectation}
		\E_R[Z] = \E_{R_0}\Big[ \Em_M[Z] \Big].
		\end{equation}
	\end{Def}
	
	\subsection{First properties of the branching Brownian motion}
	\label{subsec:prop_BBM}
	Here, we give without proofs some classical properties of the BBM. Here again, $\nu$ and $\boldsymbol q\neq 0$ are fixed for the whole subsection.
	\paragraph{Correspondence between $\boldsymbol{\Omega}^\xx$ and $\boldsymbol{\Omega}$} Some of the quantities that we used to define the BBM on the space $\Omega^\xx$ can be seen to be measurable w.r.t.~the $\sigma$-algebra generated by the process $M^0$, and hence have corresponding equivalents in the canonical space $\Omega$. This correspondence will be very useful in the proofs of the properties given in the next subsection.
	\begin{Def}
		\label{def:N_S}
		Let $\omega \in \Omega$ be a $\cadlag$ measure-valued curve and $t \in [0,1]$. We call
		\begin{gather*}
		\mathcal N(\omega,t) := \Supp\omega(t) \subset \T^d, \quad \mbox{the support of }\omega(t)\mbox{ counted with multiplicity},\\
		S(\omega,t) := \inf\{ s > t \ : \ \omega(s) \neq \omega(s-) \}, \quad \mbox{the instant of the first jump of }\omega\mbox{ after time }t,
		\end{gather*}
		where $\inf \emptyset := 1$.
	\end{Def}
	Correspondingly with the use in probability theory, when we work on the canonical space $\Omega$, we will omit the dependence of these objects w.r.t.~$\omega$. The following proposition is a direct consequences of the definitions given in the previous subsection, and we omit the proof.
	
	\begin{Prop}
		\label{prop:correspondence_N_S}
		Let $\xx \in \cup_k (\T^d)^k$. For all $t \in [0,1]$, we have $\Prob^\xx$-\emph{a.s.}
		\begin{gather*}
		\mathcal N(M^0,t) = \Big\{ X_n(t)\ : \ n \in \mathcal N^0(t) \Big\},\\
		S(M^0,t) = \inf\Big\{ d_n \ : \ n \in \mathcal N^0(t) \Big\}.
		\end{gather*}
		where the first equality is understood in the sense of multisets, that is, allowing multiple instances.
	\end{Prop}
	
	\paragraph{Strong Markov property} The BBM satisfies the strong Markov property. We formulate this Markov property in a non-completely standard way but which is easily implied by the classical versions stated for instance~\cite{ethier1986markov}.
	\begin{Prop}
		\label{prop:strong_markov}
		Let $R \sim \BBM(\nu, \boldsymbol q, R_0)$ be a BBM, $T$ be a stopping time on $\Omega$ with respect to $(\F_t)_{t \in [0,1]}$, and $t \in [0,1]$. Assume that $T \leq t$, $R$-\emph{a.s.} Let $F: \M_+(\T^d) \to \R_+ \cup\{ + \infty \}$ be measurable, and call
		\begin{equation*}
		\Gamma(s,\mu) := \Em_\mu[F(M_s)], \qquad s \in [0,1], \, \mu \in \M_\delta(\T^d).
		\end{equation*}
		Then, $R$-\emph{a.s.}
		\begin{equation*}
		\E_R\Big[ F(M_t) \Big| \F_{T} \Big] =\Gamma(t-T, M_T),
		\end{equation*} 
		where $\F_{T}$ is the $\sigma$-algebra of the past before time $T$. If $R_0 = \delta_\mu$ for some $\mu \in \M_\delta(\T^d)$, this rewrites $R^\mu$-\emph{a.s.}
		\begin{equation*}
		\Em_\mu\Big[ F(M_t) \Big| \F_{T} \Big] =\Gamma(t-T, M_T).
		\end{equation*}
	\end{Prop}
	\paragraph{Branching property}
	The BBM also satisfies the branching property, meaning that all the particles alive at a certain time have independent evolutions. Because of the Markov property, it suffices to state the branching property at the initial time. We write this property as follows.
	
	\begin{Prop}
		\label{prop:branching_property}
		Let $\mu \in\M_\delta(\T^d)$, $u:\T^d \to \R_+\cup\{+\infty\}$ be a measurable function and $t \in [0,1]$.
		We have
		\begin{equation}
		\label{eq:branching_property_u}
		\Em_\mu\Bigg[ \prod_{X \in \mathcal N(t)} u (X) \Bigg] = \prod_{x \in \Supp \mu} \Em_{\delta_{x}}\Bigg[\prod_{X \in \mathcal N(t)} u ( X) \Bigg],
		\end{equation}
		with convention $0 \times + \infty = 0$.
		
		Equivalently, for all $\varphi:\T^d \to \R \cup \{\pm \infty \}$, calling for $t \in [0,1]$ and $x \in \T^d$
		\begin{equation*}
		\bar \varphi(t,x) := \log \Em_{\delta_x} \Big[ \exp\big( \cg  \varphi, M_t \big) \Big]
		\end{equation*}
		with the convention $(-\infty)+ (+ \infty) = - \infty$, there holds
		\begin{equation}
		\label{eq:branching_property_varphi}
		\Em_\mu\left[ \exp\Big( \cg \varphi, M_t \cd \Big) \right] = \prod_{x \in \Supp \mu} \Em_{\delta_{x}}\Big[\exp\big(\cg \varphi, M_t \cd\big)\Big] = \exp\left(\cg \bar \varphi(t) ,\mu \cd\right).
		\end{equation}
	\end{Prop}
	
	Observe that if $\xx =(x_1, \dots, x_p)$ and $\mu := \delta_{x_1} + \dots + \delta_{x_p}$, defining for all $i = 1, \dots, p$ and $t \in [0,1]$ the measure $M^i_t := \sum_{n\geq i, \, n \in \mathcal N^0(t)} \delta_{X_n(t)}$ (which is the empirical measure associated with the positions of the descendants of particle $i$ alive at time $t$), the l.h.s.\ of formula~\eqref{eq:branching_property_varphi} rewrites, with the use of Proposition~\ref{prop:correspondence_N_S},
	\begin{equation*}
	\Em_\mu\left[ \exp\Big( \cg  \varphi, M_t\cd \Big) \right] = \Em_\xx\Bigg[ \prod_{i=1}^p \exp\Big( \cg  \varphi, M^i_t\cd \Big) \Bigg],
	\end{equation*}
	whereas the r.h.s.\ of the first inequality rewrites
	\begin{equation*}
	\prod_{x \in \Supp \mu} \Em_{\delta_{x}}\Big[\exp\big(\cg \varphi, M_t \cd\big)\Big] =  \prod_{i=1}^p \Em_{\delta_{x_i}}\Big[\exp\big(\cg \varphi, M_t \cd\big)\Big]
	\end{equation*}
	Hence, the proposition states that under $\Prob^\xx$ the measure-valued processes $M^1, \dots, M^p$ are independent, of respective law $R^{\delta_{x_1}}, \dots, R^{\delta_{x_p}}$.
	
	\paragraph{Stability by translation}
	Next, the BBM is invariant by translation, we omit the proof as it is very simple and left and as an exercise to the reader. 
	\begin{Prop}
		\label{prop:translation_invariance}
		Let $x \in \T^d$, $t \in [0,1]$, and $\varphi:\T^d \to \R$ be a measurable function. There holds
		\begin{equation*}
		\Em_{\delta_x} \Big[ \exp\Big( \cg \varphi, M_t \cd \Big) \Big] = \Em_{\delta_0} \Big[ \exp\Big( \cg \varphi(x + \cdot ), M_t \cd \Big) \Big].
		\end{equation*}
	\end{Prop}
	\paragraph{Associated operators}
	Recall that 
	\begin{equation}
	\Phi_{\boldsymbol q}: z \in \R_+ \mapsto \sum_{k \neq 1} q_k z^k \in [0,+\infty]
	\end{equation}
	is the generating function of the branching mechanism $\boldsymbol q$. This is a power series with nonnegative coefficients. It satisfies $\Phi_{\boldsymbol q}(1) = \lambda_{\boldsymbol q}$, and by assumption~\eqref{eq:q_finite_mean}, $\Phi_{\boldsymbol q}'(1) < +\infty$. Also recall the definition~\eqref{eq:def_Psi*} of $\Psi^*_{\nu, \boldsymbol q}$.
	This is a convex function with values in $\R \cup \{+ \infty\}$ satisfying $\Psi^*_{\nu,\boldsymbol q}(0) = 0$ and $(\Psi^*_{\nu,\boldsymbol q})'(0) < + \infty$.
	
	The last fundamental property is a computation that involves the two following operators:
	\begin{align}
	\label{eq:operator_K} \mathcal K_{\nu,\boldsymbol q}[u] :=  \frac{\nu}{2} \Delta u + \Phi_{\boldsymbol q}(u) - \lambda_{\boldsymbol q}u , \qquad u:\T^d \to \R_+, \mbox{ smooth},\\
	\label{eq:operator_L} \L_{\nu, \boldsymbol q}[\phi] := \frac{1}{2} |\nabla \phi|^2 + \frac{\nu}{2} \Delta \phi + \Psi^*_{\nu,\boldsymbol q}(\phi), \qquad \phi: \T^d \to \R, \mbox{ smooth}.
	\end{align}
	The second one is called the (rescaled) generator of the BBM. This operators are linked through the following correspondence formula:
	\begin{equation*}
	\mathcal L_{\nu, \boldsymbol q}[\phi] = \nu \frac{\mathcal K_{\nu, \boldsymbol q}[u]}{u}, \qquad u = \exp\left(\frac{\phi}{\nu}\right).
	\end{equation*}
	The proposition writes as follows.
	\begin{Prop}
		\label{prop:generator}
		Let $u: \T^d \to [0,1]$ be smooth. For all $x \in \T^d$, there holds
		\begin{equation}
		\label{eq:derivative_BBM_u}
		\frac{\D\ }{\D t+}\Em_{\delta_x}\Bigg[ \prod_{X \in \mathcal N(t)} u(X) \Bigg]\Bigg|_{t = 0} = \mathcal K_{\nu, \boldsymbol q}[u](x).
		\end{equation}
		
		Correspondingly, for all nonnegative smooth $\phi: \T^d \to \R_-$, there holds
		\begin{equation}
		\label{eq:derivative_BBM_phi}
		\frac{\D\ }{\D t+} \Em_{\delta_x}\left[ \exp\left(\frac{\cg \phi, M_t \cd}{\nu} \right) \right]\bigg|_{t=0} =\exp\left( \frac{\phi(x)}{\nu} \right) \mathcal L_{\nu, \boldsymbol q} [\phi](x).
		\end{equation}
	\end{Prop}
	Let us sketch the proof to give an intuition for this result.
	\begin{proof}[Elements of proof]
		Because of the translation stability stated in Proposition~\ref{prop:translation_invariance}, it suffices to check the formula for $x=0$. The second formula is a consequence of the first one through the change of variable $u = \exp(\phi/\nu)$, so we just check the first one.
		
		Given a small time $t$, as we are starting with one particle only, the probability that there has been exactly one branching event with $k$ descendants before time $t$ is $q_k t + o(t)$, the probability that there has been no jump at all before time $t$ is $1 - \lambda_{\boldsymbol q} t + o(t)$, and the probability that there has been at least to jumps before time $t$ is $o(t)$. Therefore, calling $\Em_0 := \Em_\xx$ for $\xx = 0 \in \T^d$, we have with the help of Proposition~\ref{prop:correspondence_N_S}	
		\begin{align*}
		\Em_{\delta_0}\Bigg[&\prod_{X \in \mathcal N(t)} u(X) \Bigg] = \Em_{0}\Bigg[\prod_{X \in \mathcal N(M^0,t)} u(X) \Bigg] = \Em_{0}\Bigg[\prod_{n \in \mathcal N^0(t)} u\big(X_n(t)\big) \Bigg] \\
		&=(1 - \lambda_{\boldsymbol q} t)\Em_{0}\Big[ u\big(  X_1(t)\big)\Big| d_1>t \Big] + t \sum_k q_k \Em_{0}\Bigg[\prod_{i=1}^k u\big( X_{1i}(t)\big)\Big|  d_1 \leq t < d_{1i}, \, i = 1, \dots, k  \Bigg] + o(t).
		\end{align*}
		Now, $X_1$ is a Brownian motion of diffusivity $\nu$ and which is independent of $d_1$. Hence, using the fact that the generator of the Brownian motion of diffusivity $\nu$ is $\frac{\nu}{2}\Delta$, we get
		\begin{align*}
		\Em_{0}\Big[ u\big(  X_1(t)\big)\Big| d_1>t \Big] = \Em_{0}\Big[ u\big(  X_1(t)\big)\Big] = u(0) + t \frac{\nu}{2} \Delta u(0) + o(t).
		\end{align*}
		On the other hand, it's not hard to convince oneself that the displacement of $X_{1i}(t)$, $i=1, \dots, k$, before time $t$ does not give any contribution of order $0$ in the expectation of the formula above. That leads to
		\begin{align*}
		\Em_{0}\Bigg[\prod_{i=1}^k u\big( X_{1i}(t)\big)\Big|  d_1 \leq t < d_{1i}, \, i = 1, \dots, k  \Bigg] &= \Em_{0}\Bigg[\prod_{i=1}^ku(0)\Big|  d_1 \leq t < d_{1i}, \, i = 1, \dots, k  \Bigg] + o(1)\\
		&=u(0)^k + o(1).
		\end{align*}
		The result follows easily.
	\end{proof}

	\subsection{Consequence: evolution equation and martingale problem}
	\label{subsec:evolution_PDE_and_martingale_problem}
	The properties of the last subsection let us derive an evolution equation associated with the BBM, namely, the corresponding Fisher-Kolmogorov-Petrovsky-Piskunov (FKPP) equation. They can also be used to prove that the BBM solves a martingale problem that characterizes it. The purpose of this subsection is to introduce these two notions. The link between the BBM and the FKPP equation dates back to the seminal work by McKean~\cite{mckean1975application}. The idea to characterize branching processes as solutions of martingale problems is also very old, and we refer to~\cite{ethier1986markov} for an extensive study of the subject.
	
	For both of these notions, we give the standard versions with sketches of proofs, using the properties of the previous subsection for functions $u$ that are bounded above by one, or correspondingly, for nonpositive functions $\varphi$ or $\phi$, so that all the expectations of products appearing in the formulas are trivially finite. But we will also need more involved versions in the next chapters, where the bounds on $u$, $\varphi$ of $\phi$ are relaxed. In these cases, the notion of solution of the FKPP equation needs to be weakened, and the martingales of the martingale problem will become local martingales. In these more general cases, we provide complete proofs.
	
	In the whole subsection, the diffusivity $\nu>0$ and the branching mechanism $\boldsymbol q\neq 0$ are fixed once for all.
	
	\paragraph{The evolution equation of the BBM}
	
	The main idea of this section is that the function of $t$ and $x$ defined as the expectation appearing in formula~\eqref{eq:derivative_BBM_u} from Proposition~\ref{prop:generator} solves the FKPP equation
	\begin{equation}
	\label{eq:FKPP}
	\partial_t u = \frac{\nu}{2} \Delta u + \Phi_{\boldsymbol q}(u) - \lambda_{\boldsymbol q}u.
	\end{equation}
	
	The most standard version of this result is the following, where the initial condition $u_0$ is supposed to have its values in $[0,1]$.
	\begin{Prop}
		\label{prop:FKPP_below1}
		Let $u_0:\T^d \to [0,1]$ be a smooth function.
		For all $t \in[0,1] $ and $x \in \T^d$, we call
		\begin{equation*}
		u(t,x) := \Em_{\delta_x}\left[ \prod_{X \in \mathcal{N}(t)} u_0(X) \right] \in [0, 1].
		\end{equation*}
		The function $u$ is a classical solution of the FKPP equation~\eqref{eq:FKPP}.
	\end{Prop}
	A complete proof can be found in~\cite[Section~4.1]{li2010measure}. Here, we just provide the main ideas.
	\begin{proof}[Elements of proof]
		Let us compute $u(t+s, x)$ for $t\in [0,1)$, $0\leq s \leq 1-t$ and $x \in \T^d$, by conditioning the definition of $u$ w.r.t.\ $\F_s$. We find, using the Markov property from Proposition~\ref{prop:strong_markov} with $T:= s$ and $F(\mu) := \prod_{x \in \Supp \mu} u_0(x)$, and the branching property~\eqref{eq:branching_property_u}:
		\begin{equation*}
		u(t+s, x) = \Em_{\delta_x}\left[\Em_{\delta_x} \left[ \prod_{X \in \mathcal{N}(t+s)} u_0(X) \Bigg| \F_s\right]\right] = \Em_{\delta_x}\left[\Em_{M_s} \left[ \prod_{X \in \mathcal{N}(t)} u_0(X)\right] \right] = \Em_{\delta_x}\left[ \prod_{X \in \mathcal{N}(s)} u(t,X) \right].
		\end{equation*}
		Admitting that $u$ remains smooth and below $1$, we can use formula~\eqref{eq:derivative_BBM_u} to get
		\begin{equation*}
		\frac{\D \, }{\D s+} u(t+s,x)|_{s = 0} = \mathcal K_{\nu, \boldsymbol q}[u(t)](x),
		\end{equation*}
		and because of the definition~\eqref{eq:operator_K} of $\mathcal K_{\nu, \boldsymbol q}$, the result follows. 
	\end{proof}
	
	A consequence of this result is the following evolution equation for the density of any branching Brownian motion $R\sim \BBM(\nu, \boldsymbol q, R_0)$ that satisfies $\E_{R_0}[M] < + \infty$.
	\begin{Cor}
		\label{cor:evolution_density}
		Let $R \sim \BBM(\nu, \boldsymbol q, R_0)$ be a branching Brownian motion satisfying $\E_{R_0}[M(\T^d)]< + \infty$. Then its intensity $\rho_t := \E_{R}[M_t]$, $t \in [0,1]$ solves the following PDE in the weak sense:
		\begin{equation*}
		\partial_t \rho_t = \frac{\nu}{2}\Delta \rho_t + \bar r \rho_t,
		\end{equation*}
		where $\bar r := \sum_{k \neq 1}(k-1) q_k$. Otherwise stated, for all $t \in [0,1]$, $\rho_t = e^{\bar r t} \tau_{\nu t} * \rho_0$, being $(\tau_s)_{s\geq 0}$ the heat flow on the torus.
	\end{Cor}
	\begin{Rem}
		Notice that the PDE in this corollary is reminiscent of the continuity equation~\eqref{eq:continuity_linear} where $m = 0$ and $\zeta = \bar r \rho$. This is actually a first instance where branching processes provide competitors for RUOT problems.
	\end{Rem}
	Once again and as this result is very classical, we only provide elements of the proof. A reader interested in writing a complete proof could for instance see this result as a straightforward consequence of Theorems~\ref{thm:definition_stochastic_integral}, \ref{thm:properties_jump_processes} and~\ref{thm:branching_Ito} below.
	\begin{proof}[Elements of proof]
		Take $\varphi \in C^\infty(\T^d)$ a smooth test function. The goal is to prove that for all $t \in [0,1]$,
		\begin{equation}
		\label{eq:evolution_density_weak_form}
		\frac{\D}{\D t} \cg \varphi, \rho_t\cd = \frac{\nu}{2} \cg \Delta \varphi, \rho_t \cd + \bar r \cg \varphi, \rho_t\cd.
		\end{equation}
		Decomposing $\varphi$ as $\varphi_+ - \varphi_-$, we can assume that $\varphi \geq 0$. Moreover, it can be deduced from the branching property stated in Proposition~\ref{prop:branching_property} that
		\begin{equation*}
		\rho_t = \E_{R_0} \Big[ \Big\cg \Em_{\delta \cdot}[M_t], M \Big\cd \Big] = \E_{R_0}\left[ \sum_{X \in \Supp M} \Em_{\delta_X}[M_t] \right].
		\end{equation*}
		Therefore, by linearity of~\eqref{eq:evolution_density_weak_form} w.r.t.\ $\rho_t$, up to integrating w.r.t.\ $\E_{R_0}[M]$, which is a finite measure by assumption, we just have to prove the result for $\rho_t = \Em_{\delta_x}[M_t]$, $x \in \T^d$.
		
		If $\eps >0$ is sufficiently small, $u_0 := 1 - \eps \varphi$ can be used in Proposition~\ref{prop:FKPP_below1}. Moreover, we can show that with the same notations as there, for all $t \in [0,1]$ and $x \in \T^d$,
		\begin{gather*}
		u(t,x) = 1 - \eps \Em_{\delta_x}\Big[\cg \varphi, M_t \cd\Big] + \underset{\delta \to 0}{o}(\delta^2),\\
		\Delta u(t,x) = - \eps \Em_{\delta_x}\Big[\cg \Delta \varphi, M_t \cd\Big] + \underset{\delta \to 0}{o}(\delta^2),\\
		\Phi_{\boldsymbol q}(u) - \lambda_{\boldsymbol q}u = \lambda_{\boldsymbol q} - \eps \bar r \Em_{\delta_x}\Big[\cg \varphi, M_t \cd\Big] + \underset{\delta \to 0}{o}(\delta^2),
		\end{gather*}
		where the second line uses the translation invariance of the BBM stated in Proposition~\ref{prop:translation_invariance}, and the last one exploits the fact that the left derivative of $\Phi_{\boldsymbol q}$ at $u=1$ is $\bar r + \lambda_{\boldsymbol q}$. The results follows from identifying the terms of order one in $\eps$ in~\eqref{eq:FKPP}.
	\end{proof}
	
	In the case where $u_0$ is allowed to take values above $1$, the result of Proposition~\ref{prop:FKPP_below1} still holds, up to weakening the notion of solution of the FKPP equation. Later in Chapter~\ref{chap:duality}, we will refer to this notion of solutions as $D$-weak solutions as they are linked to classical solutions thanks to Duhamel's formula \cite[Section 2.3.1]{evans1998}. Recall that $(\tau_s)_{s \geq 0}$ is the heat kernel on the torus.  
	\begin{Prop}
		\label{prop:evolution_eq_u_from_BBM}
		Let $u_0:\T^d \to \R_+$ be measurable and nonnegative.
		We call for all $t \in[0,1] $ and $x \in \T^d$
		\begin{equation}
		\label{eq:def_u_t}
		u(t,x) := \Em_{\delta_x}\left[ \prod_{X \in \mathcal{N}(t)} u_0(X) \right] \in [0, +\infty].
		\end{equation}
		For all $t \in [0,1]$, there holds
		\begin{equation*}
		u(t) =  e^{-\lambda_{\boldsymbol q} t} \tau_{\nu t} \ast u_0 + \int_0^t e^{-\lambda_{\boldsymbol q}(t-s)} \tau_{\nu(t-s)}\ast  \Phi_{\boldsymbol q}(u(s))  \D s,
		\end{equation*}
		where both side of this equality might be infinite.
	\end{Prop}
	\begin{proof}
		Let us apply the strong Markov property of Proposition~\ref{prop:strong_markov} with $R := R^{\delta_x}$ for some $x \in \T^d$, $t \in [0,1]$, $T := S(0) \wedge t$ as defined in Definition~\ref{def:N_S}, and with $F(\mu) := \prod_{x \in \Supp \mu} u_0(x)$. As in Proposition~\ref{prop:strong_markov}, we call
		\begin{equation*}
		\Gamma(s,\mu) := \Em_\mu[F(M_s)] = \Em_{\mu}\left[ \prod_{X \in \mathcal N(s)} u_0(X) \right], \qquad s \in [0,1], \, \mu \in \M_\delta(\T^d).
		\end{equation*}
		We find
		\begin{equation}
		\label{eq:application_strong_markov}
		\Em_{\delta_x}\Bigg[ \prod_{X \in \mathcal N(t)} u_0(X) \Bigg| \F_{T} \Bigg] = \Gamma(t-T, M_T).
		\end{equation}
		The goal is to compute the expectation of this formula. Denoting by $\Prob^x := \Prob^\xx$ with $\xx = x \in \T^d$, and using Proposition~\ref{prop:correspondence_N_S}, we have
		\begin{align*}
		\Em_{\delta_x}\Big[ \Gamma(t-T, M_T) \Big] &= \Em_x\Big[ \Gamma\Big(t-S(M^0, 0)\wedge t), M^0_{S(M^0, 0)\wedge t}\Big) \Big] \\
		&= \Em_x\Big[ \Gamma\Big( t - d_1\wedge t, M_{d_1\wedge t}^0 \Big) \Big],
		\end{align*}
		where the second line is obtained noticing that under $\Prob^x$, $S(M^0,0) = d_1$, \emph{a.s.} Hence, it is an exponential time of parameter~$\lambda_{\boldsymbol q}$, and with probability $e^{-\lambda_{\boldsymbol q}t}$, we have $d_1 > t$. In that case, $d_1 \wedge t = t$, $M^0_{d_1 \wedge t} = \delta_{X_1(t)}$ and the r.h.s.\ in~\eqref{eq:application_strong_markov} rewrites
		\begin{equation*}
		\Prob^x\mbox{-\emph{a.s.} on }\{ d_1 > t \}, \qquad \Gamma\Big( t - d_1\wedge t, M_{d_1\wedge t}^0 \Big)= \Gamma(0, \delta_{X_1(t)}) = u_0\big( X_1(t) \big).
		\end{equation*}
		
		Else, if $d_1 \leq t$, we have $M^0_{d_1 \wedge t} = M^0_{d_1}= L_1 \delta_{X_1(d_1)}$. In this case, the r.h.s.\ in~\eqref{eq:application_strong_markov} rewrites
		\begin{align*}
		\Prob^x\mbox{-\emph{a.s.} on }\{ d_1 \leq t \}, \qquad \Gamma\Big( t - d_1\wedge t, M_{d_1\wedge t}^0 \Big) &= \Gamma\Big( t - d_1, L_1 \delta_{X_1(d_1)} \Big) \\
		&=  \Gamma\Big( t - d_1, \delta_{X_1(d_1)} \Big)^{L_1}\\
		&= u\Big(t-d_1, X_1(d_1)\Big)^{L_1},
		\end{align*}
		where we used the branching property~\eqref{eq:branching_property_u} to get the second line, and the definition~\eqref{eq:def_u_t} of $u$ to get the third one.
		
		As a consequence, taking the expectation of~\eqref{eq:application_strong_markov} provides
		\begin{equation*}
		u(t,x) = e^{- \lambda_{\boldsymbol q}t} \Em_{x}\Big[ u_0\big(X_1(t)\big) \Big| d_1>t \Big] + \Big( 1 - e^{-\lambda_{\boldsymbol q}t} \Big)\Em_{x}\Bigg[ \Big(u\big((t-d_1), X_1(d_1)\big)\Big)^{L_1} \Bigg| d_1 \leq t \Bigg].
		\end{equation*}
		The random variables $d_1 = e_1$ and $X_1(t) = x + Y_1(t)$ are independent under $\Prob^x$, and the law of $Y_1(t)$ is $\tau_{\nu t}$. Consequently, the first expectation in the r.h.s.\ is
		\begin{equation*}
		\Em_{x}\Big[ u_0\big(X_1(t)\big) \Big| d_1>t \Big] = \Em_{x}\Big[ u_0\big(X_1(t)\big)\Big] =  \tau_{\nu t}\ast u_0 (x).
		\end{equation*}
		
		For the second term, the computation is slightly more involved, and can be written as:
		\begin{align*}
		\Em_{x}\Bigg[ \Big(u\big((t-d_1), X_1(d_1)\big)\Big)^{L_1} \Bigg| d_1 \leq t \Bigg] &= \int_0^t  \Em_{x}\Bigg[ \Big(u\big((t-s), X_1(s)\big)\Big)^{L_1} \Bigg| d_1 = s \Bigg] \Prob^{x}(d_1 \in \D s)\\
		&= \int_0^t  \Em_{x}\Bigg[ \Big(u\big((t-s), X_1(s)\big)\Big)^{L_1} \Bigg] \lambda_{\boldsymbol q}e^{-\lambda_{\boldsymbol q}s}\D s\\
		&= \int_0^t e^{-\lambda_{\boldsymbol q}s} \left[ \tau_{\nu s} \ast \Phi_{\boldsymbol q}\big(u(t_s)\big) \right](x) \D s,
		\end{align*}
		where the second lines uses the law of $d_1=e_1$ and the independence of $d_1$ and $(L_1, X_1)$, and the third line is a direct computation using the independence of $L_1$ and $X_1$, and their respective laws. The result follows from gathering everything and changing the variable according to $s \mapsfrom t-s$
	\end{proof}

	\paragraph{The martingale problem of the branching Brownian motion}
	
	For a given $R_0 \in \P(\M_\delta(\T^d))$, the martingale problem associated with the initial condition $R_0$ and the operator $\mathcal L_{\nu, \boldsymbol q}$ defined in~\eqref{eq:operator_L} is defined as follows.
	\begin{Def}
		\label{def:martingale_characterization}
		We say that $R \in \P(\Omega)$ solves the martingale problem associated with $R_0$ and $\L_{\nu, \boldsymbol q}$ provided 
		\begin{enumerate}
			\item The law of $M_0$ under $R$ is $R_0$.
			\item For all nonpositive smooth function $\psi: [0,1]\times \T^d \to \R_-$, the process 
			\begin{equation}
			\label{eq:exp_mart_BBM}
			\exp\left( \frac{\cg \psi(t), M_t\cd}{\nu} -\frac{\cg \psi(0), M_0\cd }{\nu} - \frac{1}{\nu}\int_0^t \big\cg \partial_t \psi(s) + \mathcal L_{\nu, \boldsymbol q}[\psi(s)], M_s \big\cd \D s \right), \qquad t \in [0,1],
			\end{equation}
			is a $(\F_t)_{t \in [0,1]}$-martingale.
		\end{enumerate}
	\end{Def}

	This property completely characterizes $R\sim\BBM(\nu, \boldsymbol q, R_0)$ as stated in the proposition below. Later, in Chapter~\ref{chap:characterization_finite_entropy}, we will use such formulations to define BBMs where particle undergo a drift in addition to there Brownian trajectories, and where the branching mechanism depends on $t$ and $x$, see Definition~\ref{def:modified_BBM}.
	\begin{Prop}
		\label{prop:martingale_problem}
		The law $R \sim \BBM(\nu, \boldsymbol q, R_0)$ is the unique solution of the martingale problem associated with $R_0$ and $\L_{\nu, \boldsymbol q}$.
	\end{Prop}
	We do not provide a complete proof, that the interested reader could find in~\cite[Section~9.4]{ethier1986markov}. However, let us give the main ideas. Note in particular that in the proof below, we do not rely on the tree structure of the BBM, we only use the Markov property, the branching property and formula~\eqref{eq:derivative_BBM_phi}.  
	\begin{proof}[Elements of proof]
		Let us first justify why for all nonpositive and smooth $\psi$, formula~\eqref{eq:exp_mart_BBM} defines a martingale. By a classical result in stochastic calculus (this is a consequence of \cite[Lemma~2.1]{stroock1971diffusion}) it suffices to prove that
		\begin{equation*}
		\exp\left( \frac{\cg \psi(t), M_{t} \cd}{\nu} \right) - \int_0^{t} \exp\left( \frac{\cg \psi(s), M_{s} \cd}{\nu} \right) \cg \partial_s \psi(s) + \L_{\nu, \boldsymbol q}[\psi(s)], M_s \cd \D s, \qquad t \in [0,1],
		\end{equation*}
		is a martingale.
		
		To do so, the Markov property from Proposition~\ref{prop:strong_markov} shows that we only need to justify that for all $t \in [0,1]$ and all $\mu \in \M_\delta(\T^d)$,
		\begin{equation*}
		\Em_\mu\left[ \exp\left( \frac{\cg \psi(t), M_{t} \cd}{\nu} \right) - \int_{0}^{t}\! \exp\left( \frac{\cg \psi(s), M_{s} \cd}{\nu} \right) \cg \partial_s \psi(s) + \L_{\nu, \boldsymbol q}[\psi(s)], M_s \cd \D s  \right] =\exp\left( \frac{\cg \psi(0), \mu \cd}{\nu} \right).
		\end{equation*}
		As this is true for $t = 0$, it suffices to show that the right derivative w.r.t.\ $t$ of the l.h.s.\ cancels. We compute separately the first and the second term. For the second, we have
		\begin{align*}
		\frac{\D \, }{\D t+}	\Em_\mu\Bigg[ \int_{0}^{t} \exp\left( \frac{\cg \psi(s), M_{s} \cd}{\nu} \right) &\cg \partial_s \psi(s) + \L_{\nu, \boldsymbol q}[\psi(s)], M_s \cd \D s  \Bigg]\\
		& = \Em_\mu\left[ \exp\left( \frac{\cg \psi(t), M_{t} \cd}{\nu} \right)\cg \partial_t \psi(t)+ \L_{\nu, \boldsymbol q}[\psi(t)], M_{t} \cd \right].
		\end{align*}
		For the first one, understanding $\D/\D s$ taken as a right derivative at $s=0$:
		\begin{align*}
		\frac{\D \, }{\D t+} \Em_\mu\Bigg[ \exp\left( \frac{\cg \psi(t), M_{t} \cd}{\nu} \right)  \Bigg] &= \Em_\mu\left[ \frac{\D \, }{\D s} \Em_\mu\left[\exp\left( \frac{\cg \psi(t+s), M_{t + s} \cd}{\nu} \right) \bigg| \F_{t} \right]\right]\\
		&=  \Em_\mu\left[ \frac{\D \, }{\D s} \Em_{M_{t}}\left[\exp\left( \frac{\cg \psi(t + s), M_{s} \cd}{\nu} \right) \right]  \right]\\
		&= \Em_\mu\left[ \frac{\D \, }{\D s} \prod_{X \in \mathcal N(t)} \Em_{\delta_X}\left[ \exp\left( \frac{\cg \psi(t + s), M_{s} \cd}{\nu} \right) \right] \right]\\
		&= \Em_\mu\left[ \exp\left( \frac{\cg \psi(t), M_{t} \cd}{\nu} \right) \cg\partial_t \psi(t) +  \L_{\nu, \boldsymbol q}[\psi(t)], M_{t} \cd  \right].
		\end{align*} 
		where we used the Markov property from Proposition~\ref{prop:strong_markov} at the second line, the branching property~\eqref{eq:branching_property_varphi} at the third line, and formula~\eqref{eq:derivative_BBM_phi} at the last line. The martingale property follows.
		
		Concerning the uniqueness, the method classically used is the so-called method of duality (see~\cite[Section~1.6]{etheridge2000introduction}), and develops as follows. To prove that $R$ is unique, it suffices to show that for all $0 \leq t_0 \leq t_1 \leq 1$ and all smooth nonpositive function $\psi_{t_1} = \psi_{t_1}(x)$ on $\T^d$, the value of the random variable
		\begin{equation*}
		\E_R\left[ \exp\left( \frac{\cg \psi_{t_1}, M_{t_1}\cd}{\nu} \right) \bigg| \F_{t_0} \right]
		\end{equation*}
		is prescribed by the martingale problem. Indeed, in that case, the conditional law of the measure-valued random variable $M_{t_1}$ knowing $\mathcal{F}_{t_0}$ is prescribed, and as a consequence, as $R_0$ is given, so is the finite-dimensional distributions of the canonical process, which suffices to characterize $R$. 
		
		Now, let us consider $\phi:\R_+ \times \T^d \to \R$ the unique classical solution of
		\begin{equation*}
		\left\{ \begin{gathered}
		\partial_t \phi = \L_{\nu, \boldsymbol q}[\phi], \\
		\phi(0) = \psi_{t_1}.
		\end{gathered} \right.
		\end{equation*}
		As $\phi(0) \leq 0$, it is not hard to see that $\phi$ exists, is unique, and remains nonpositive and smooth everywhere. This is mainly because this PDE, similarly to the ODE $\partial_t \phi = \Psi^*_{\nu, \boldsymbol q}(\phi)$, does not blow-up when restricted to the stable set of nonpositive functions.
		
		Next, for $t \leq t_1$, call $\psi(t) := \phi(t_1 - t)$, so that for all $t \leq t_1$, we have $\partial_t \psi + \L_{\nu, \boldsymbol q}[\psi] = 0$. Using the martingale property, we get
		\begin{equation*}
		\E_R\left[ \exp\left( \frac{\cg \psi(t_1), M_{t_1}\cd}{\nu} - \frac{\cg \psi(0) ,M_0 \cd}{\nu} \right) \bigg| \F_{t_0} \right] = \exp\left( \frac{\cg \psi(t_0), M_{t_0}\cd}{\nu} - \frac{\cg \psi(0) ,M_0 \cd}{\nu} \right).
		\end{equation*}
		This identity rewrites
		\begin{equation*}
		\E_R\left[ \exp\left( \frac{\cg \psi_{t_1}, M_{t_1}\cd}{\nu} \right) \bigg| \F_{t_0} \right] = \exp\left( \frac{\cg \phi(t_1-t_0), M_{t_0}\cd}{\nu} \right),
		\end{equation*}
		which lets us conclude.
	\end{proof}
	
	In the case where $\psi$ is not nonpositive, it is still possible to prove that the process defined in~\eqref{eq:exp_mart_BBM} is a local martingale, provided $\Psi^*_{\nu, \boldsymbol q}(\psi)$ remains bounded.
	\begin{Prop}
		\label{prop:loc_exp_mart_BBM}
		Let $R_0 \in \P(\M_{\delta}(\T^d))$ be an initial law and $R \sim \BBM(\nu, \boldsymbol q, R_0)$ be a BBM. Let $\psi: \R_+ \times \T^d \to \R$ be smooth, and assume that $\Psi^*_{\nu, \boldsymbol q}(\psi)$ remains bounded on $[0,1]\times \T^d$. For all $x \in \T^d$, the process
		\begin{equation}
		\label{eq:loc_exp_mart_BBM}
		\exp\left( \frac{\cg \psi(t), M_t\cd}{\nu} - \frac{\cg \psi(0), M_0\cd}{\nu}  - \frac{1}{\nu}\int_0^t \big\cg \partial_t \psi(s) + \mathcal L_{\nu, \boldsymbol q}[\psi(s)], M_s \big\cd \D s \right), \qquad t \in [0,1],
		\end{equation}
		is a $(\F_t)_{t \geq 0}$-local martingale under $R$.
		As it is nonnegative, it is a supermartingale as well.
	\end{Prop}
	It would be possible to write a tedious proof of this result only using the properties of the branching Brownian motion derived until now. However, later in Chapter~\ref{chap:characterization_finite_entropy}, we will introduce new tools that will make this proposition rather simple. For this reason, we prefer to postpone this proof to Subsection~\ref{subsec:extended_Ito}. The reader interested in understanding this proposition as quickly as possible only needs to read Sections~\ref{sec:discontinuous_stochastic_calculus} and~\ref{sec:new_processes}.  
	
	\section{The entropy functional and the branching Schrödinger problem}
	\label{sec:entropy}
	In this section, we state rigorously the branching Schrödinger problem, and then we gather some information about the relative entropy functional~\eqref{eq:def_H} that will be useful later on for our analysis. We start with a crucial convex inequality providing the Legendre transform of this functional, and we present its behavior under conditioning of the measures involved. Finally, we introduce a new set of assumptions for the BBM that allows to define the density of \emph{any competitor} of the branching Schrödinger problem. These assumptions will be extensively used in Chapters~\ref{chap:characterization_finite_entropy} and~\ref{chap:equivalence_competitors}, but nowhere else. Therefore, the reader who is only interested in the results of Chapters~\ref{chap:duality} and~\ref{chap:additional_results} can skip Subsection~\ref{subsec:exponential_bounds} below.
	
	\subsection{The branching Schrödinger problem}
	
	Having at our disposal the entropy functional $H$ defined by formula~\eqref{eq:def_H} and the law of the BBM defined as a probability measure on $\cadlag([0,1],\M_+(\T^d))$ in Section~\ref{sec:presentation_BBM}, we only need to define the expectation measure of the marginals in order to state the branching Schrödinger problem.
	
	\begin{Def}
		\label{def:intensity_measure}
		Let $P \in \P(\cadlag([0,1],\M_+(\T^d)))$ and $t \in [0,1]$. Then $\E_P[M_t]$ is defined as the function defined on the Borel sets of $\T^d$ by $A \mapsto \E_P[M_t(A)] \in [0, +\infty]$. We say that the expectation measure of $M_t$ is finite if $\E_P[M_t] \in \M_+(\T^d)$.  
	\end{Def}
	
	As discussed at the beginning of Section 9.5 of \cite{daley2008introduction}, as $\E_P[M_t]$ is always countably additive, for the expectation measure to be finite it is enough to assume $\E_P[M_t(\T^d)] < + \infty$. We present below in Section~\ref{subsec:exponential_bounds} a set of assumptions on $R_0$ and $\boldsymbol q$ which guarantee the finiteness of the expectation of $M_t$ under $R$ or under a measure $P$ having a finite entropy with respect to $R$. However, these assumptions will \emph{not} be needed in Chapter~\ref{chap:duality} about the equivalence of the values. In any case, to define the branching Schrödinger problem, we simply (implicitly) restrict to $P$ for which the expectation measure of $M_0,M_1$ is finite. 
	
	\begin{Def}
		\label{def:BrSch}
		Let $R \sim \BBM(\nu, \boldsymbol{q}, R_0)$, where $\nu > 0$, $\boldsymbol{q}$ is a branching mechanism and $R_0 \in \P(\M_\delta(\T^d))$. For $\rho_0, \rho_1 \in \M_+(\T^d)$ we define 
		\begin{equation*}
		\BrSch_{\nu, \boldsymbol q, R_0}(\rho_0,\rho_1) := \inf_{P} \left\{ \nu H(P|R) \ : \ \E_P[M_0] = \rho_0 \text{ and } \E_P[M_1] = \rho_1 \right\}    
		\end{equation*}
		where the minimum is taken among all probability distributions $P$ on $\cadlag([0,1],\M_+(\T^d))$. 
	\end{Def}
	
	As it will be discussed later in Subsection~\ref{subsec:counterexamples}, it is not always easy to know when the problem is not empty, and when one can guarantee the existence of a competitor. 
	
	We will also need a simpler problem, where one is only concerned with the initial intensity. Again for this one non-emptiness of the problem or existence of a minimizer can be a delicate question.
	
	\begin{Def}
		\label{def:init}
		Let $R_0 \in \P(\M_\delta(\T^d))$. For $\rho_0 \in \M_+(\T^d)$ we define 
		\begin{equation}
		\label{eq:def_init}
		\Init_{R_0}(\rho) := \inf_{P_0} \left\{  H(P_0|R_0) \ : \ \E_{P_0}[M] = \rho_0 \right\}    
		\end{equation}
		where the minimum is taken among all $P_0$ probability distributions on $\M_+(\T^d)$. 
	\end{Def}
		In Subsection~\ref{subsec:example_static}, we will give examples where solutions of this problem exist and can be computed explicitly. But for now, let us show that actually, this problem has already been studied in the literature. This remark can easily be skipped at a first reading.
	
	\begin{Rem}
		\label{rem:grand_canonical}
		We became aware after finishing a first version of this work that the problem that we just defined is a reformulation of the so-called inverse problem in classical statistical mechanics, in the grand canonical ensemble, as studied in~\cite{chayes1984inverse,chayes1984validity}. This problem, and in particular its zero temperature limit, is currently being reinterpreted in the formalism of optimal transport by several authors~\cite{dimarino2021grand}. As our formulation is different, let us explain further the analogy.
		
		In~\cite{chayes1984validity}, a reference random system of a random number of particles (at fixed time and temperature 1) is described by a probability space $(\Lambda, \mu)$ where the particles live, and a family of interaction potentials $w_N : \Lambda^N \to \R \cup \{ + \infty \}$, $N \in \N^*$. The idea is that given a family of functions $f_N: \Lambda^N \to \R$, $N \in \N^*$, which are interpreted as a quantity depending on the configuration of particles, the mean value of this quantity is
		\begin{equation}
		\label{eq:grand_canonical_probability}
		\frac{\displaystyle 1 + \sum_{N \geq 1} \frac{1}{N!} \int f_N e^{-w_N}\D\mu^{\otimes N}}{\displaystyle 1 + \sum_{N \geq 1} \frac{1}{N!} \int e^{-w_N}\D \mu^{\otimes N}}.
		\end{equation}
		Note the prefactor $1/N!$ reflecting the fact that the particles are assumed to be indistinguishable. It means that a configuration of particles is seen as an element $(x_1, \dots, x_N)$ of the disjoint union $\cup_{k \in \N} \Lambda^k$, and the formula above describes a certain probability measure on this disjoint union: the probability of seeing an (ordered) configuration $(x_1, \dots, x_N)$ is, proportional to $(1/N!) e^{-w_N(x_1, \dots, x_N)}\D \mu(x_1) \dots \D \mu(x_N)$.
		
		In our case, we describe the same system by a law $R_0 \in \P(\M_\delta(\T^d))$ (in particular, we assume $\Lambda$ to be the $d$-dimensional torus, but it does not play a role in the present discussion). A collection of particles is now simply seen as an element of $\M_{\delta}(\T^d)$. This is very practical because now, the indistinguishability of particles is already induced by the nature of $\M_\delta(\T^d)$. Of course, up to applying the push-forward operation to the probability measure defined by~\eqref{eq:grand_canonical_probability} by the map
		\begin{equation*}
		\Theta: (x_1, \dots, x_N) \in \cup_{k \in \N} (\T^d)^k \mapsto \delta_{x_1} + \dots + \delta_{x_N} \in \M_\delta(\T^d),
		\end{equation*}
		we get a probability $R_0 \in \P(\M_{\delta}(\T^d))$, so that actually our formulation is more general. More precisely, as the map $\Theta$ has a measurable right inverse, it is easy to convince oneself that the case of~\cite{chayes1984validity} is the case where $R_0(M=0)>0$ and where there exists a single probability measure $\mu \in \P(\T^d)$ such that the law $R_0$ conditioned to have exactly $N \in \N$ particles $R_0( \cdot | M(\T^d) = N)$ is absolutely continuous with respect to $\Theta \pf \mu^{\otimes N}$.
		
		In~\cite{chayes1984inverse} and~\cite{chayes1984validity}, the authors ask the following question: given the interactions $(w_N)_{N \in \N^*}$ encoded in the reference law $R_0$, and an observed density $\rho_0 \in \M_+(\T^d)$, can one find an external potential $U$ (in the following of this work, we will rather use the notation $\varphi := - U$) such that the grand canonical ensemble (still at temperature $1$) where the interactions are the same as in the reference one, but with the additional external potential $U$, has intensity~$\rho_0$? This is indeed an inverse problem, as we seek for $U$, knowing $\rho$. This condition on $U$ writes
		\begin{equation*}
		\frac{\E_{R_0}\Big[ M \exp\big(-\cg U, M \cd\big) \Big]}{\E_{R_0}\Big[\exp\big(-\cg U, M \cd\big)\Big]} = \rho_0.
		\end{equation*}	
		As explained in~\cite[Proposition~2.1]{chayes1984validity}, if this problem has a solution, then 
		\begin{equation*}
		P_0 := \frac{\exp\big(-\cg U, M \cd\big)}{\E_{R_0}\Big[\exp\big(-\cg U, M \cd\big)\Big]}\cdot R_0
		\end{equation*}
		is a solution to the problem stated at Definition~\ref{def:init}, the latter being stated in their framework in~\cite[Formula~(2.14)]{chayes1984validity}. These formulas are the formal optimality conditions for our initial problem and they will play a role in the next chapters, see for instance~\eqref{eq:def_P_0}. The potential $U$ is nothing but the dual variable (a.k.a. Lagrange multiplier) for the convex optimization problem that we are looking at. 
		
		Unfortunately, we will always consider more general situations than what is studied in~\cite{chayes1984inverse,chayes1984validity}. For instance in our case, dual existence is not guaranteed. Therefore we are not able to use directly their results in what follows.
	\end{Rem}

	\subsection{Some properties of the relative entropy}
	Recall that the entropy functional $H$ was defined by formula~\eqref{eq:def_H}. Let us give some of its properties that will be useful in the sequel.
	
	\paragraph{Duality inequalities} 
	
	Let us call $l$ the nonnegative function defined by 
	\begin{equation}
	\label{eq:def_l}
	l: z \in \R_+ \longmapsto z \log z + 1 - z.
	\end{equation}
	$l$ is convex and s.c.i.\ when extended by $+\infty$ on $\R_-^*$. A computation of its Legendre transform provides:
	\begin{Prop}
		For all $y,z \in \R$,
		\begin{equation}
		\label{eq:convex_inequality_l}
		yz \leq l(z) + \exp( y) - 1,
		\end{equation}
		with equality if and only if $z=\exp( y)$.
	\end{Prop}
	We can apply this inequality to prove the following duality inequality for the relative entropy, that will be crucial throughout this work.
	
	\begin{Prop}
		\label{prop:legendre_transf_entropy}
		Let $\mathbf{X}$ be a Polish space, and $\mathsf{p},\mathsf{r} \in \P(\mathbf{X})$. For all real valued random variable $Y$ on $\mathbf{X}$, we have
		\begin{equation}
		\label{eq:convex_ineq_entropy}
		\E_{\mathsf p}[Y] \leq H(\mathsf p|\mathsf r) + \log \E_{\mathsf r}[\exp( Y)],
		\end{equation}
		where the l.h.s.\ is automatically well defined in $[-\infty, + \infty)$ whenever the r.h.s.\ is finite. Moreover, when the r.h.s.\ is finite, equality holds if and only if
		\begin{equation}
		\label{eq:equality_condition_duality_entropy}
		\frac{\D \mathsf p}{\D \mathsf r} = \frac{\exp( Y)}{ \E_{\mathsf r}[\exp( Y)]}.
		\end{equation}
	\end{Prop}
	
	\begin{proof}
		Starting from~\eqref{eq:convex_inequality_l} with $y+\alpha$ in place of $y$,
		\begin{equation*}
		\forall y,z \in \R,  \qquad yz \leq l(z) + \exp( y + \alpha) - 1 - \alpha z,
		\end{equation*}
		with equality if and only if $z = \exp(y + \alpha)$. Assume that $\mathsf p \ll \mathsf r$ (else, there is nothing to prove). Call $Z$ the Radon-Nikodym derivative of $\mathsf p$ w.r.t.~$\mathsf r$, in particular $\E_{\mathsf r}[Z] = 1$. Integrating the previous inequality w.r.t.~$\mathsf r$, we get for all $\alpha \in \R$
		\begin{equation*}
		\E_{\mathsf p}[Y] = \E_{\mathsf r}[YZ] \leq H(\mathsf p|\mathsf r) + \E_{\mathsf r}[\exp( Y+ \alpha)] - 1 - \alpha,
		\end{equation*} 
		with equality if and only if $\mathsf r$-\emph{a.s,} $Z = \exp(Y + \alpha)$. Optimizing w.r.t.\ $\alpha$ leads to~\eqref{eq:convex_ineq_entropy} and~\eqref{eq:equality_condition_duality_entropy}.
	\end{proof}
	
	\begin{Rem}
		Proposition~\ref{prop:legendre_transf_entropy} easily implies that the Legendre transform of the function $H(\cdot | \mathsf r)$, that we extend by $+ \infty$ on $\M(\T^d) \setminus \P(\T^d)$, is $\phi \in C(\T^d) \mapsto \log \E_{\mathsf r}[\exp(\phi(X))]$, where $X$ is the identity map on $\mathbf{X}$.    
	\end{Rem}

	\paragraph{Conditioning to the initial configuration}
	
	The behavior of the relative entropy towards disintegration is very well known, and follows from the additivity properties of the logarithm, see~\cite[Theorem~2.4]{leonard2014some}. Here, we just provide the case where we condition to $M_0$, the initial configuration, because this is the only formula that we will use.
	\begin{Prop}
		Let $P,R \in \P(\Omega)$, being $\Omega$ the canonical space of Subsection~\ref{subsec:def_BBM}. Let us call $R_0$ the law of $M_0$ under $R$, and $R^\mu := R(\,\cdot\,|M_0 = \mu)$, which is defined for $R_0$-almost all $\mu \in \M_+(\T^d)$. Define $P_0$ and $P^\mu$, $\mu \in \M_+(\T^d)$ in the same way starting from $P$. There holds
		\begin{equation}
		\label{eq:disintegration_entropy}
		H(P|R) = H(P_0|R_0) + \E_{P_0} \Big[ H\big(P^{M} \big| R^{M}\big) \Big].
		\end{equation}
	\end{Prop}
	\begin{Rem}
		This formula is a first hint why the value of the objective functional in the Schrödinger problem and in the branching Schrödinger problem can be written as a sum of an initial part, and a dynamical part, as seen in formulas~\eqref{eq:equivalence_schrodinger} and~\eqref{eq:equivalence_BrSchr}.
		
		Observe that when $R\sim\BBM(\nu, \boldsymbol q, R_0)$, there is no conflict of notations: because of~\eqref{eq:computation_expectation}, we have indeed $R(\, \cdot \, | M_0 = \mu) = R^\mu$, as defined in Definition~\ref{def:BBM_deterministic_R0}.
	\end{Rem}
	\begin{proof}
		If $P\cancel{\ll}R$, then the l.h.s.\ in~\eqref{eq:disintegration_entropy} is infinite. In this case, either $P_0\cancel{\ll}R_0$, or the set of $\mu \in \M_+(\T^d)$ for which $P^\mu\cancel{\ll} R^\mu$ is not $P_0$-negligible. In both case, the r.h.s.\ in~\eqref{eq:disintegration_entropy} is infinite as well. 
		
		Then, if $P \ll R$, the disintegration theorem ensures that $R$-\emph{a.s.}
		\begin{equation*}
		\frac{\D P}{\D R} = \frac{\D P_0}{\D R_0}(M_0) \times \frac{\D P^{M_0}}{\D R^{M_0}}.
		\end{equation*}
		In particular, using the additivity property of the $\log$ in the first line, we find
		\begin{align*}
		H(P|R) &= \E_P\left[ \log \frac{\D P_0}{\D R_0}(M_0) \right] + \E_P\left[ \log \frac{\D P^{M_0}}{\D R^{M_0}} \right] \\
		&=\E_{P_0}\left[ \log \frac{\D P_0}{\D R_0} \right] + \E_P\left[\E_{P^{M_0}}\left[ \log \frac{\D P^{M_0}}{\D R^{M_0}}\right] \right] \\
		&= H(P_0 | R_0) + \E_P\left[ H(P^{M_0}| R^{M_0}) \right],
		\end{align*}
		and the result follows.
	\end{proof}
	\subsection{A stronger set of assumptions for the branching Brownian motion}
	\label{subsec:exponential_bounds}
	Once again, this subsection is only useful in Chapters~\ref{chap:characterization_finite_entropy} and~\ref{chap:equivalence_competitors}. The reader who is not interested in the results of these chapters can jump directly to Chapter~\ref{chap:duality}.
	
	We present here the assumptions on the reference BBM that we will make in order to be able to prove the equivalence of competitors in Chapter~\ref{chap:equivalence_competitors}. They will also be useful in Chapter~\ref{chap:characterization_finite_entropy} to characterize the laws of finite entropy w.r.t.\ a reference BBM. They consist in the existence of exponential moments for $R_0$ and $\boldsymbol q$.
	
	\begin{Ass}
		\label{ass:exponential_bounds}
		Let $R \sim \BBM(\nu, \boldsymbol q,R_0)$ be a BBM, where $\nu > 0$, $\boldsymbol{q}$ is a branching mechanism and $R_0 \in \P(\M_\delta(\T^d))$. We say that:
		\begin{itemize}
			\item $\boldsymbol{q}$ has an exponential bound if there exists $\theta_{\boldsymbol{q}}>0$ such that
			\begin{equation}
			\label{eq:exponential_moment_q}
			\sum_{k \neq 1} q_k \exp(k \theta_{\boldsymbol{q}}) < + \infty.
			\end{equation}
			\item $R_0$ has an exponential bound if there exists $\theta_0>0$ such that:
			\begin{equation}
			\label{eq:exponential_moment_R_0}
			\E_{R_0}\Big[\exp\big(\theta_0 M(\T^d)\big)\Big] < +\infty.
			\end{equation}
		\end{itemize}
	\end{Ass}
	\begin{Rem}
		To have an idea of what condition~\eqref{eq:exponential_moment_q} implies on the growth penalization $\Psi_{\nu, \boldsymbol q}$ defined through~\eqref{eq:def_Psi*}, observe that in the set of examples provided in Appendix~\ref{app:plots}, the only case where condition~\eqref{eq:exponential_moment_q} does not hold is the one of Section~\ref{sec:no_exp_moment}. The most pathological choices of $\boldsymbol q$ satisfying condition~\eqref{eq:exponential_moment_q} are presented in Section~\ref{sec:exp_moment}. In particular, we show there that some choices of $\boldsymbol q$ lead to a linear growth of $\Psi_{\nu, \boldsymbol q}$ at $+\infty$, and not more.
	\end{Rem}
	
	These assumptions are reasonable. Indeed, in a lot of applications, $\boldsymbol q$ has a bounded support, which is much stronger than~\eqref{eq:exponential_moment_q}. Moreover, if~\eqref{eq:exponential_moment_q} holds, then the existence of an exponential moment under $R_0$ stated in formula~\eqref{eq:exponential_moment_R_0} is preserved through time, that is, the law $R_t$ of $M_t$ under $R$ also admits an exponential moment, as stated in the proposition below. 
	
	A reason justifying that these assumptions are useful is that when $R_0$ and $\boldsymbol q$ satisfy these bounds, then any $P$ with finite entropy w.r.t.~$R\sim\BBM(\nu,\boldsymbol q, R_0)$ admits an expected marginal $\E_P[M_t] \in \M_+(\T^d)$ at all time $t \in [0,1]$.
	
	\begin{Prop}
		\label{prop:exponential_bounds}
		Let $R \sim \BBM(\nu, \boldsymbol{q},R_0)$ be a BBM, and assume both bounds from Assumption~\ref{ass:exponential_bounds}. Then, there exists $\bar \kappa>0$ such that
		\begin{equation*}
		\E_R\Big[\exp\Big(\bar \kappa \sup_{t \in [0,1]} M_t(\T^d)\Big)\Big] < +\infty.
		\end{equation*}
		
		In particular, for all $P \in \P(\Omega)$ such that $H(P|R) < + \infty$, the map $t \in [0,1] \mapsto \E_P[M_t] \in \M_+(\T^d)$ is well defined and weakly $\cadlag$. Moreover, the following bound holds, where $C>0$ only depends on $R_0$ and 
		$\boldsymbol q$:
		\begin{equation}
		\label{eq:unif_bound_density}
		\sup_{t \in [0,1]} \E_P[M_t(\T^d)] \leq C\big( 1 + H(P|R) \big).
		\end{equation}
	\end{Prop}
	\begin{proof}
		Consider the following ODE involving $\Psi^*_{\nu, \boldsymbol q}$ defined in~\eqref{eq:def_Psi*}:
		\begin{equation}
		\label{eq:ODE_exponential_lemma}
		\left\{
		\begin{gathered}
		\dot \psi(t) + \Psi^*_{\nu, \boldsymbol q}(\psi(t)) = 0,\\
		\psi(0) = \psi_0.
		\end{gathered}
		\right.
		\end{equation}
		The function $\psi(t) \equiv 0$ is the unique global solution starting from $\psi_0 = 0$.
		
		Under assumption~\eqref{eq:exponential_moment_q}, $\psi \in \R \mapsto\Psi^*_{\nu, \boldsymbol q}(\psi)$ is smooth in the neighborhood of $\psi = 0$. By standard results about ODEs, there exists $\bar \psi>0$ such that for all $|\psi_0| \leq \bar \psi$, the unique solution of~\eqref{eq:ODE_exponential_lemma} blows-up at a time $T^*>1$.  
		
		Take $\theta_0$ as given by assumption~\eqref{eq:exponential_moment_R_0}. Let us take $0< \psi^1_0 \leq  \min (\nu \theta_0,\bar \psi)$, and $\psi^1:[0,1] \to \R$ the unique solution of~\eqref{eq:ODE_exponential_lemma} starting from $\psi^1_0$ up to time $t=1$. Exploiting the martingale property~\eqref{eq:loc_exp_mart_BBM} for the test function $(t,x) \mapsto \psi^1(t)$, we see that
		\begin{equation*}
		\exp\left(\frac{\psi^1(t) M_t(\T^d)}{\nu}\right), \qquad t \in [0,1]
		\end{equation*}	
		is a nonnegative local martingale, and hence a supermartingale as well. In particular, for all $t \in [0,1]$,
		\begin{equation*}
		\E_R\left[\exp\left(\frac{\psi^1(t) M_t(\T^d)}{\nu}\right)\right] \leq \E_R\left[\exp\left( \frac{\psi^1_0 M_0(\T^d)}{\nu} \right)\right] < + \infty.
		\end{equation*} 
		
		Now, let us take $0<\psi^2_0<\psi^1_0$ so small that the unique solution $\psi^2$ of~\eqref{eq:ODE_exponential_lemma} starting from $\psi^2_0$ satisfies $\psi^2(t) \leq \psi^1(t)/2$, for all $t \in [0,1]$. Doing so, the local martingale
		\begin{equation*}
		\exp\left(\frac{\psi^2(t) M_t(\T^d)}{\nu}\right), \qquad t \in [0,1]
		\end{equation*}	
		is bounded in $L^2$, and therefore is a true martingale. Applying to it Doob's maximal inequality in $L^2$, we get	
		\begin{equation*}
		\E_R\left[\sup_{t \in [0,1]} \exp\left( \frac{2\psi^2(t) M_t(\T^d)}{\nu} \right) \right] \leq 4 \E_R\left[ \exp\left( \frac{2 \psi^2(1) M_1(\T^d)}{\nu} \right) \right] \leq 2 \E_R\left[ \exp\left( \frac{ \psi^1(1) M_1(\T^d)}{\nu} \right) \right] < +\infty.
		\end{equation*}
		We conclude the first part of the claim by taking $\bar \kappa := \inf_{t \in [0,1]} 2\psi^2(t)/\nu$, observing that~\eqref{eq:ODE_exponential_lemma} preserves positivity.
		
		Then, let us take $P \in \P(\Omega)$ such that $H(P|R) < + \infty$, and let us apply formula~\eqref{eq:convex_ineq_entropy} with $Y = \bar \kappa  \sup_{t \in [0,1]} M_t(\T^d) $. We get
		\begin{equation*}
		\sup_{t \in [0,1]} \E_P[M_t(\T^d)] \leq \E_P\left[  \sup_{t \in [0,1]} M_t(\T^d) \right] \leq \frac{\displaystyle H(P|R) + \E_R\left[  \exp\left( \bar \kappa \sup_{t \in [0,1]} M_t(\T^d)\right)\right]}{\bar \kappa} < + \infty.
		\end{equation*}
		Hence, for all $t\in[0,1]$, $\rho(t) := \E_P[M_t]$ is well defined, bounded in $\M_+(\T^d)$, uniformly in $t$, and~\eqref{eq:unif_bound_density} holds (the fact that $C$ does not depend on $\nu$ relies on an analysis of the scaling of~\eqref{eq:ODE_exponential_lemma}).
		
		Then, for all $\varphi$, $t \mapsto \cg \varphi, M_t \cd$ is $\cadlag$, $P$-\emph{a.s.}\ (as $P$ is a probability measure on the set of $\cadlag$ measure-valued curves $\cadlag([0,1]; \M_+(\T^d))$, endowed with the topology of weak convergence of measures). Hence, the announced property follows from the dominated convergence theorem and the estimate $\E_P[\sup_t M_t(\T^d)] < + \infty$.
	\end{proof}
	\begin{Rem}
		\label{rem:rho_continuous}
		It is not hard to prove that for all $t \in [0,1]$,
		\begin{equation*}
		R\Big( \Big\{ s \mapsto M_s(\T^d) \mbox{ undergoes a jump at time }s = t \Big\} \Big) = 0.
		\end{equation*}
		Hence, this is also true under $P$, for all $P \ll R$. As a consequence, slightly adapting the proof, we can show that $t  \mapsto \E_P[M_t] $ is not only $\cadlag$ but also continuous for the weak topology of measures.
	\end{Rem}

	\chapter{Equivalence of the values: a proof by duality}
	\label{chap:duality}
	
	The goal of this chapter is to state and prove one of our main result: the value of the Regularized Unbalanced Transport (RUOT) problem is, up to a term depending only on the initial condition, the lower semi-continuous (l.s.c.) envelope of the value of the branching Schrödinger problem. We start by recalling the definition of the l.s.c.\ envelope and then state our result. The basic idea of the proof is that the computation of the l.s.c.\ envelope of a convex function boils down to computing two Legendre transforms in a row. We apply this method for the initial ``static'' problem, which corresponds to the term depending only on the initial condition $\rho_0$ and the initial law $R_0$. Next, before moving to the proof of the main result, we first prove additional results on the FKPP equation and weak solutions of it, in order to plug nonnegative functions in the martingale characterization of the BBM. After this, we are in position to prove our main result. Eventually we close this chapter with examples for which the initial ``static'' problem can be computed explicitly, and with counterexamples illustrating the ill-posedness of the branching Schrödinger problem.
	
	\section{Statement of the results}
	
	We start by recalling the definition of the l.s.c.\ envelope. Let $\V$ be a separable Banach space, and let us call $\V'$ its topological dual. For $v \in \V$ and $w \in \V'$ we denote the duality pairing by $\langle v, w \rangle$. We endow $\V'$ with the topology $\sigma(\V', \V)$ which is the coarsest making the linear forms $w \mapsto \langle v, w \rangle$ continuous. The space $\V'$ endowed with $\sigma(\V',\V)$ is a locally convex space. In addition, as $\V$ is separable, $\sigma(\V',\V)$ is metrizable. We will use the results of this section with $\V$ being be the set $C(\T^d)$ while $\V' = \M(\T^d)$ will be the set of finite Borel measures, or $\V = C(\T^d)^2$ and $\V' = \M(\T^d)^2$. Then, $\sigma(\V', \V)$ is the topology of weak convergence.	
	
	For simplicity, we will consider functions $F$ defined on $\V'$ and which take nonnegative values, including the value $+ \infty$. The nonnegativity assumption could be relaxed: it would be enough to be above a linear continuous function. 
	
	\begin{Def}
		Let $F : \V' \to [0, + \infty ]$. Its $\sigma(\V', \V)$-l.s.c.\ envelope is defined by
		\begin{equation*}
		\overline{F} := \sup \left\{ G \ : \ G \leq F \text{ and } G \text{ is } \sigma(\V',\V) \text{ lower semi-continuous} \right\},
		\end{equation*}
		where the supremum is taken on functions $G : \V' \to \R \cup \{ + \infty \}$.
	\end{Def}	
	
	The lower semi-continuous envelope $\overline{F}$ is  l.s.c.\ as the supremum of l.s.c.\ functions is l.s.c. One can actually represent it as (see \cite[Corollary 2.1]{ekeland1999convex})
	\begin{equation*}
	\overline{F}(w) = \inf \left\{ \liminf_{n \to + \infty}  F(w_n) \ : \ \lim_{n \to + \infty} w_n = w  \right\},
	\end{equation*} 	
	where the infimum is taken among all sequences $(w_n)_{n \in \N}$ which converge to $w$ for the $\sigma(\V',\V)$ topology. 
	
	\bigskip
	
	We are now in position to state our main results. Recall that the functionals $\BrSch_{\nu, \boldsymbol q, R_0}$ and $\Init_{R_0}$ are defined in Definition \ref{def:BrSch} and \ref{def:init} respectively.

	We start with the initial static problem, which is when we only look at the initial marginal constraint in the branching Schrödinger problem. For this one, we can actually remove the assumption that $R_0$ is supported on $\M_\delta(\T^d)$.
	
	\begin{Thm}
		\label{thm:initial_lsc_envelope}
		Let $R_0 \in \P(\M_+(\T^d))$ be a probability distribution on $\M_+(\T^d)$. Let us define the Legendre transform of its log-Laplace transform as in~\eqref{eq:def_L*}. The functional $L^*_{R_0}$ is the ls.c.\ envelope of $\Init_{R_0}$ for the topology of weak convergence of measures.
	\end{Thm}
	
	The second result is the following, making the link with the RUOT problem.
	
	\begin{Thm}
		\label{thm:equality_values_dyn}
		Let $R \sim \BBM(\nu, \boldsymbol{q}, R_0)$, where $\nu > 0$, $\boldsymbol{q}$ is a branching mechanism and $R_0 \in \P(\M_\delta(\T^d))$. Assume that $R_0 \neq \delta_0$. Let us define $\Psi = \Psi_{\nu, \boldsymbol q}$ where the latter is defined through its Legendre transform by formula \eqref{eq:def_Psi*}. Then the function $(\rho_0, \rho_1) \mapsto \nu L^*_{R_0}(\rho_0) + \ruot_{\nu, \Psi}(\rho_0,\rho_1)$ is the l.s.c.\ envelope of $\BrSch_{\nu, \boldsymbol q, R_0}$ for the topology of weak convergence of measures. 
	\end{Thm}
	\begin{Rem}
		\label{rem:R0=delta0}
		Here it is crucial to exclude the case $R_0 = \delta_0$. Indeed, in this case, one can see that $\BrSch_{\nu, \boldsymbol q, R_0}(\rho_0, \rho_1)$ is $+\infty$ unless $\rho_0 = \rho_1 = 0$, while $L^*_{R_0}(0) + \ruot_{\nu, \Psi}(0,\Leb)$ is finite provided $\sum_{k \geq 2} q_k >0$ (use Proposition~\ref{prop:ruot_existence_competitor} together with Remark~\ref{rk:ruot_existence_competitor_q}). Therefore, the result cannot hold. A very similar argument is developped further at the very end of Subsection~\ref{subsec:counterexamples}.	
	\end{Rem}
	
	For convex functions, such as $\Init_{R_0}$ or $\BrSch_{\nu, \boldsymbol q, R_0}$, lower semi-continuous envelopes can be computed with the help of Legendre transforms. Specifically, for Theorem \ref{thm:initial_lsc_envelope} and Theorem \ref{thm:equality_values_dyn} we will rely on the following result, which we state in an abstract context. As above $\V$ is a separable Banach space and $\V'$ is its topological dual. 
	
	\begin{Thm}
		\label{thm:legendre_transform_lsc_envelope}
		Let $F : \V' \to [0 + \infty]$ be a convex function which is not identically equal to $+ \infty$. Then its $\sigma(\V',\V)$-l.s.c.\ envelope $\overline{F}$ is convex and is given by 
		\begin{equation*}
		\overline{F} = (F^*)^*,
		\end{equation*}
		the Legendre transform of the Legendre transform of $F$, as defined in Section~\ref{sec:notations}.
	\end{Thm}
	
	This result is very classical. Here we emphasize that the framework is a bit unusual: the functional $F$ is defined on the dual $\V'$, and its Legendre transform $F^*$ is defined on $\V$ and not $\V''$ (which differs from $\V$ if it is not reflexive). It is because we are looking at the $\sigma(\V',\V)$-envelope and not the $\sigma(\V',\V'')$ one.
	
	\begin{proof}
		Using for instance \cite[Proposition 3.3]{ekeland1999convex}, as $F$ is convex we know that $\overline{F}$ coincides with the $\Gamma$-regularization of $F$, that is the supremum of affine $\sigma(\V',\V)$-continuous functions which are below $F$. 
		
		By \cite[Proposition 3.14]{brezis2011functional}, a function $f : \V' \to \R$ which is $\sigma(\V',\V)$-continuous and affine can be written $f(w) = \langle v, w \rangle + c$ for some $v \in \V$ and $c \in \R$. For a given $v \in \V$ the largest affine function below $F$ with slope $v$ is $f : w \mapsto \langle v, w \rangle - F^*(w)$ by definition of $F^*$. 
		
		Putting these two results together is enough to conclude.
	\end{proof}
	
	We first prove Theorem \ref{thm:initial_lsc_envelope}: as one can see, it will be an almost immediate consequence of Theorem~\ref{thm:legendre_transform_lsc_envelope} and an easy approximation argument. The proof of Theorem~\ref{thm:equality_values_dyn} will be more involved and we will need some preliminary results obtained in Section~\ref{sec:linkFKPP_BBM}.

	\section{Proof of Theorem~\ref{thm:initial_lsc_envelope}, the static case}
	\label{sec:static_legendre_transform}
	
	In this section, we provide a direct proof of Theorem~\ref{thm:initial_lsc_envelope}. Because of Theorem~\ref{thm:legendre_transform_lsc_envelope} and as $\Init_{R_0}$ is convex, this is a direct consequence of the following proposition.
	
	\begin{Prop}[Legendre transform of $\Init$]
		\label{prop:legendre_transform_init}
		Let $R_0 \in \P(\M_+(\T^d))$. We have for all $\varphi \in C(\T^d)$
		\begin{equation*}
		\sup_{\rho \in \M_+(\T^d)} \cg \varphi, \rho \cd - \Init_{R_0}(\rho) = \log \E_{R_0}[\exp(\cg\varphi, M \cd)].
		\end{equation*}
	\end{Prop}
	\begin{proof}
		Let $\varphi \in C(\T^d)$. By definition of $\Init_{R_0}$, the statement is clearly equivalent to
		\begin{equation*}
		\sup_{\substack{P_0 \in \P(\M_+(\T^d))\\ \E_{P_0}[M(\T^d)] < + \infty}} \E_{P_0}[\cg \varphi, M \cd] - H(P_0|R_0)= \log \E_{R_0}[\exp(\cg\varphi, M \cd)].
		\end{equation*}
		
		Let us first prove the inequality ``$\leq$''. For a given $P_0 \in \P(\M_+(\T^d))$, by~\eqref{eq:convex_ineq_entropy}:
		\begin{equation}
		\label{eq:aux_proof_init_Legendre}
		\E_{P_0}[\cg \varphi, M \cd ] \leq H(P_0|R_0) + \log \E_{R_0}[\exp(\cg \varphi ,M\cd)].
		\end{equation}
		We conclude by putting the entropic term in the l.h.s.\ and by taking the supremum w.r.t.\ $P_0$.
		
		Let us prove ``$\geq$''. The main idea is to set
		\begin{equation}
		\label{eq:def_P_0}
		P_0 := \frac{\exp(\cg \varphi, M\cd) }{\E_{R_0}[\exp(\cg \varphi, M\cd)]} \cdot R_0
		\end{equation}
		to find an equality in~\eqref{eq:aux_proof_init_Legendre}. The two subtleties consist in dealing with the case where $\E_{R_0}[\exp(\cg\varphi,M\cd)]=+\infty$ (and hence $P_0$ given by~\eqref{eq:def_P_0} is not well defined), and with the case where $\E_{R_0}[\exp(\cg\varphi,M\cd)]<+\infty$, but $P_0$ given by~\eqref{eq:def_P_0} satisfies $\E_{P_0}[M(\T^d)] = + \infty$ (and hence $P_0$ is not a competitor for $\Init_{R_0}(\rho)$, with any~$\rho$). To do this, we set for $n \in \N$ big enough for the denominator to be nonzero:
		\begin{equation*}
		P^n_0 := \frac{\exp(\cg \varphi, M\cd) \1_{M(\T^d) \leq n} }{\E_{R_0}[\exp(\cg \varphi, M\cd)\1_{M(\T^d) \leq n}]} \cdot R_0.
		\end{equation*}
		Then we have automatically $\E_{P^n_0}[M(\T^d)] \leq n$ and a direct computation of the entropy leads to
		\begin{equation*}
		\E_{P_0^n}[\cg \varphi , M \cd] - H(P_0^n|R_0) =  \log \E_{R_0}[\exp(\cg \varphi, M \cd)\1_{M(\T^d) \leq n}],
		\end{equation*}
		so that the result follows from the monotone convergence theorem.
	\end{proof}
	
	\section{On the link between solutions of FKPP and the branching Brownian motion}
	\label{sec:linkFKPP_BBM}

	In this section, we consider the FKPP equation and its link with the martingale characterization of the BBM, as a preliminary for the proof of Theorem~\ref{thm:equality_values_dyn}. Our main results are:
	\begin{itemize}
		\item Theorem~\ref{thm:martingale_smooth_case}, which extends the martingale characterization of Definition~\ref{def:martingale_characterization} to the case where the function $\psi$ is no longer nonpositive but under the assumption that one has strong \emph{interior} solution of FKPP (see below);
		\item Proposition~\ref{prop:FKPP_approx} which states that one can always approximate weak solutions of FKPP by strong interior ones.
	\end{itemize}   
	
	Let us consider the (forward) FKPP equation:
	\begin{equation}
	\label{eq:FKPP_forward}
	\dr_t u = \frac{\nu}{2} \Delta u +  \Phi(u) - \lambda u,
	\end{equation}
	where $\nu > 0$ and $\lambda \geq 0$. Later, we will restrict to $\Phi = \Phi_{\boldsymbol q}$ and $\lambda = \lambda_{\boldsymbol q}$, being $\boldsymbol q$ a branching mechanism as defined in Definition~\ref{def:branching_mechanism}, so that this equation is nothing but~\eqref{eq:FKPP}. But for the moment, we only assume that $\Phi : [0, + \infty] \to [0,+\infty]$ is a nondecreasing convex function, and that $\Phi$ is finite at least on a right neighborhood of $0$. Note that these properties are satisfied when $\Phi = \Phi_{\boldsymbol q}$ is the generating function of a branching mechanism $\boldsymbol q \in \M_+(\N)$.
	
	Let us call $r_\text{max} \in (0,+\infty]$ the supremum of those $r$ for which $\Phi(r) < + \infty$. By standard properties of convex functions, $\Phi$ is Lipschitz on $[0,r]$ for any $r < r_\text{max}$. 
	
	We recall that we defined the heat kernel $(\tau_s)_{s > 0}$ and the convolution operator $*$ in Section~\ref{sec:notations}. We also recall that if $u$ is defined on $[0,1] \times \T^d$, we use $u(t)$ as a shortcut for $u(t,\cdot)$.

	\begin{Def}
		Let $u : [0,1] \times \T^d \to [0,+\infty]$ be a nonnegative measurable function. We call $u$:
		\begin{itemize}
			\item A \emph{strong} solution of \eqref{eq:FKPP_forward} if $u \in C^2([0,1] \times \T^d)$ and satisfies \eqref{eq:FKPP_forward} pointwise. 
			\item A \emph{strong interior} solution of \eqref{eq:FKPP_forward} if it is a strong solution and, in addition, $\min_{[0,1] \times \T^d} u > 0$ and $\max_{[0,1] \times \T^d} u < r_\text{max}$.
			\item A $D$-weak solution of \eqref{eq:FKPP_forward} if for all $t \in [0,1]$, 
			\begin{equation}
			\label{eq:FKPP_weak}
			u(t) =  e^{-\lambda t} \tau_{\nu t} \ast u(0) +  \int_0^t  e^{-\lambda (t-s)} \tau_{\nu (t-s)} \ast \Phi(u(s)) \, \D s .
			\end{equation}
			\item A $D$-weak supersolution of 
			\eqref{eq:FKPP_forward} if for all $t \in [0,1]$,
			\begin{equation}
			\label{eq:FKPP_weak_sub}
			u(t) \geq  e^{-\lambda t} \tau_{\nu t} \ast u(0) +  \int_0^t  e^{-\lambda (t-s)} \tau_{\nu (t-s)} \ast \Phi(u(s)) \, \D s.
			\end{equation}
		\end{itemize}
	\end{Def}
	
	Note that for $D$-weak solutions and supersolutions the right hand side is well defined as a function taking nonnegative (possibly infinite) values.
	
	Obviously a strong interior solution is a strong solution. In addition a strong solution is a $D$-weak solution. To see this, we just have to use Duhamel's formula (see for instance \cite[Section 2.3.1]{evans1998}) by splitting $\frac{\nu}{2} \Delta u +  ( \Phi(u) - \lambda u )$ in the r.h.s.\ of equation~\eqref{eq:FKPP_forward} as a linear part $\frac{\nu}{2} \Delta u - \lambda u$ and a nonlinear part~$\Phi(u)$. Eventually, note that this notion is precisely the kind of solutions that we got in Proposition~\ref{prop:evolution_eq_u_from_BBM}.

	\begin{Rem}
		\label{rem:non_unique_FKPP}
		Importantly, strong solutions are not necessarily unique, even starting from the same initial condition. This can happen if $r_\text{max} < + \infty$ with $\Phi(r_{\text{max}}) = \lambda r_\text{max}$ and $\Phi'(r_{\text{max}}) = + \infty$ in such a way that $\Phi$ is not locally Lipschitz in a left neighborhood of $r_\text{max}$.
		
		Specifically, let us fix $\lambda = 1$ and look at the function $\Phi$ defined by $\Phi(r) = 2 - \sqrt{2-r}$ if $r \leq 2$ and $\Phi(r) = + \infty$ if $r > 2$. This function is convex, l.s.c., and here $r_\text{max} = 2$. Moreover, $\Phi(1) = 1$, and one could check that it is the generating function associated with a branching mechanism. As $\Phi(2) = 2$ the (space-time) constant function equal to $2$ is a solution of~\eqref{eq:FKPP_forward}. However it is not the only one as any solution of the ODE $\dot{r}(t) = (\Phi(r(t)) -  r(t))$ is a solution of \eqref{eq:FKPP_forward} (taking a function which is constant in space), and this ODE has multiple solutions starting from $r=2$. The reason is that $\Phi$ is not globally Lipschitz on $[0,2]$.
		
		This is why we introduce the notion of strong \emph{interior} solutions, which stay far away from the singularity corresponding to $r_\text{max}$.   
	\end{Rem}
	
	As illustrated by the remark, uniqueness can be a tricky issue, but can actually be guaranteed for strong interior solutions. More precisely we will rely only on this weak-strong comparison principle.
	
	\begin{Prop}
		\label{prop:comparision_FKPP}
		Let $u$ be a strong interior solution, and $v$ be a $D$-weak supersolution of the same equation~\eqref{eq:FKPP_forward}. Assume furthermore that $u(0,\cdot) \leq v(0,\cdot)$. Then $u \leq v$ everywhere on $[0,1] \times \T^d$. 
	\end{Prop}
	
	\begin{proof}
		Using that $u$ is a $D$-weak solution of \eqref{eq:FKPP_forward}, we can subtract \eqref{eq:FKPP_weak_sub} from \eqref{eq:FKPP_weak} and say that for all $t \in [0,1]$ there holds
		\begin{align}
		\notag u(t) - v(t) & \leq  e^{-\lambda t} \tau_{\nu t} \ast [ u(0) - v(0)]  + \int_0^t  e^{-\lambda (t-s)} \tau_{\nu (t-s)} \ast [\Phi(u(s)) - \Phi(v(s))] \, \D s \\ 
		\label{eq:zz_aux_proof_gronwall} &\leq  \int_0^t   e^{-\lambda (t-s)} \tau_{\nu (t-s)} \ast [\Phi(u(s)) - \Phi(v(s))] \, \D s,
		\end{align}
		where both side may be $- \infty$. 
		We then want to look at the supremum of $(u - v)_+$ in time, where $a_+$ denotes the positive part of $a \in \R$. To avoid measurability issues we rather look at the function $w : [0,1] \to \R_+$ defined by $w(t) = \vee_{x \in \T^d} (u(t,x)-v(t,x))_+$, where $\vee$ denotes the lattice supremum, see \cite[Lemma 2.6]{hajlasz2002approximation}. Specifically the function $w$ is measurable and for all $x \in \T^d$, there holds $w(t) \geq (u(t,x)-v(t,x))_+$ for \emph{a.e.}\ $t \in [0,1]$. By Fubini's theorem, we can write that this inequality holds for \emph{a.e.}\ $(t,x) \in [0,1]$. We deduce that for \emph{a.e.}\ $(s,x) \in [0,1] \times \T^d$,  
		\begin{equation*}
		\Phi(u(s,x)) - \Phi(v(s,x)) \leq [\Phi(u(s,x)) - \Phi(v(s,x))]_+ \leq C (u(s,x) - v(s,x))_+ \leq C w(s),
		\end{equation*} 
		where $C$ is the Lipschitz constant of $\Phi$ on the image of $u$. As $u$ is a strong interior solution the constant $C$ is finite. We conclude that for \emph{a.e.}\ $t \in [0,1]$ there holds
		\begin{equation*}
		w(t) \leq C \int_0^t e^{-\lambda (t-s)} w(s) \, \D s.
		\end{equation*}
		Gronwall's Lemma (see for instance \cite[Appendix 5]{ethier1986markov} for the case where $w$ measurable) yields that $w$ is nonpositive at least \emph{a.e.}\ That is, for \emph{a.e.}\ $(t,x) \in [0,1] \times \T^d$, $u(t,x) \leq v(t,x)$. Equation~\eqref{eq:zz_aux_proof_gronwall} enables us to go from an equality holding \emph{a.e.}\ to an equality that holds everywhere.
	\end{proof}
	
	The next result states that $D$-weak supersolutions can be approximated from below by strong interior ones. We first start with an easy lemma.
	
	\begin{Lem}
		\label{lemma_D_weak_sup}
		Let $u$ be a D-weak supersolution of \eqref{eq:FKPP_forward}. Let $\eta,\eps > 0$. We define the function $u^{\eta,\varepsilon}$ on $[0,1] \times \T^d$ as
		\begin{equation}
		\label{eq:def_reg_u}
		u^{\eta,\varepsilon}(t) = \tau_\eta\ast u(t) - e^{-\lambda t} \varepsilon, \qquad \forall t \in [0,1],
		\end{equation}
		and we assume that it stays positive on $[0,1] \times \T^d$. Then it is still a D-weak supersolution of \eqref{eq:FKPP_forward}. 
	\end{Lem}
	
	\begin{proof}
		We apply the heat kernel $\tau_\eta$ on each side of \eqref{eq:FKPP_weak_sub}. Note that $\tau_\eta$ commutes with $\tau_{\nu t}$, with the multiplication by $e^{-\lambda t}$ and with the integral sign. Thus we get
		\begin{equation*}
		\tau_\eta \ast u(t) = e^{-\lambda t} \tau_{\nu t} \ast u(0) +  \int_0^t  e^{-\lambda (t-s)} \tau_{\nu (t-s)} \ast  \Big(\tau_\eta \ast \Phi(u(s))\Big) \, \D s, \qquad \forall t \in [0,1]
		\end{equation*}
		With Jensen's inequality we see that $\tau_\eta \ast \Phi(u(s)) \geq \Phi( \tau_\eta \ast u(s) )$. Next, as  $\tau_\eta \ast u(t)  = u^{\eta,\varepsilon} + e^{-\lambda t} \varepsilon$, we see that we can write
		\begin{equation*}
		u^{\eta,\varepsilon}(t) + e^{-\lambda t} \varepsilon \geq e^{-\lambda t} \tau_{\nu t}\ast  u^{\eta,\varepsilon}(0)  + e^{-\lambda t} \varepsilon +  \int_0^t  e^{-\lambda (t-s)} \tau_{\nu (t-s)}\ast \Phi(u^{\eta,\varepsilon}(s) + e^{\lambda t} \varepsilon) \, \D s, \qquad \forall t \geq 0.
		\end{equation*}
		As $\Phi$ is nondecreasing, $ \Phi(u^{\eta,\varepsilon}(s) + e^{\lambda t} \varepsilon) \geq  \Phi(u^{\eta,\varepsilon}(s))$, and we reach the conclusion by removing $e^{-\lambda t} \eps$ on each side.
	\end{proof}
	
	We now move to the main result of this subsection dealing with approximations of $D$-weak supersolutions.

	\begin{Prop}
		\label{prop:FKPP_approx}
		Let $u$ be a D-weak supersolution of \eqref{eq:FKPP_forward}. Assume that $u(0)$ is continuous and strictly positive, while $u(1)$ is not identically $+ \infty$. Then there exists a sequence $(u_n)_{n \in \N}$ of strong interior solutions of \eqref{eq:FKPP_forward} defined on $[0,1] \times \T^d$ such that $(u_n(0,x))_{n \in \N}$ converges to $u(0,x)$ increasingly for all $x \in \T^d$. 
	\end{Prop}
	
	We emphasise the low regularity assumption that we impose on $u$: finiteness only at one point at the finite time~$t=1$, whereas each function in the sequence $(u_n)_{n \in \N}$ is defined and smooth everywhere on $[0,1] \times \T^d$. Note that by the comparison principle given in Proposition~\ref{prop:comparision_FKPP}, we know that $u_n \leq u$ on $[0,1] \times \T^d$ for all $n \in \N$. 
	
	\begin{proof}
		We first evacuate the case where $\Phi$ is constant: in this case, \eqref{eq:FKPP_forward} is linear, any strong solution is a strong interior solution and it is enough to take $u_n$ (independent on $\N$) to be the unique solution of \eqref{eq:FKPP_forward} starting from $u(0)$. Thus we restrict to the case where $\Phi$ is not constant. Being convex and nondecreasing we deduce that it grows at least linearly at $+ \infty$.

		\begin{stepa}{A priori bounds on $u$}
			First, using \eqref{eq:FKPP_weak_sub} and neglecting the integral term on the r.h.s., we see that $u$ is in fact uniformly strictly positive not only on $\{0 \} \times \T^d$, but even on $[0,1] \times \T^d$. 
			
			Second, as there exists a point $x \in \T^d$ such that $u(1,x) < + \infty$, using again \eqref{eq:FKPP_weak_sub}, we can see that, at least for a.e.\ $s \in [0,1]$, the quantity $  \tau_{\nu(1-s)}\ast \Phi(u(s))$ is finite at one point. This can happen, for $s < 1$, only if $\Phi(u(s)) \in L^1(\T^d)$. As $\Phi$ grows at least linearly, and recalling the definition of $r_\text{max}$ as the upper bound of the domain of $\Phi$, we deduce that for a.e.\ $s \in [0,1]$:
			\begin{itemize}
				\item $u(s)$ and $\Phi(u(s))$ are in $\in L^1(\T^d)$,
				\item $u(s,x) \leq r_\text{max}$ for a.e.\ $x \in \T^d$.
			\end{itemize}
		\end{stepa}
		
		\begin{stepa}{Definition of a sequence of supersolutions}
			We take $(\eta_n)_{n \in \N}, (\varepsilon_n)_{n \in \N}$ two sequences which go to $0$, and $v_n := u^{\eta_n,\varepsilon_n}$ as defined in formula~\eqref{eq:def_reg_u}.  
			
			As $u$ is uniformly strictly positive, $v_n$ is also strictly positive provided $\varepsilon_n$ is small enough. Thus, $v_n$ is a $D$-weak supersolution of \eqref{eq:FKPP_forward} as well, thanks to Lemma~\ref{lemma_D_weak_sup}.
			
			In addition, the bounds proved in the previous step, together with $\eta_n > 0$ and $\varepsilon_n > 0$, show that for a.e.~$t \in [0,1]$, the function $v_n(t)$ is bounded everywhere on $\T^d$ by $r_\text{max} - e^{-\lambda} \varepsilon_n$.
		\end{stepa}
		
		\begin{stepa}{Definition of the approximate solutions}
			We define $u_n$ as a maximal strong solution of \eqref{eq:FKPP_forward} starting from $v_n(0)$, we want to show that its existence time is at least $1$. More precisely, for any $\delta > 0$ let us call $T_n^\delta$ the infimum of those times $t$ for which $u_n(t,x) \geq r_\text{max} - \delta$ at at least one point $x \in \T^d$. We will show that $T^\delta_n \geq 1$ as soon as $\delta$ is small enough. 
			
			Let us fix $n \in \N$ and choose $\delta := e^{-\lambda} \varepsilon_n$. Assume by contradiction that $T_n^\delta < 1$. For any $t \in (T^\delta_n, T^{\delta/2}_n)$, we know that $u_n$ is a strong interior solution of \eqref{eq:FKPP_forward} on $[0,t] \times \T^d$. In particular, using the comparision principle of Proposition~\ref{prop:comparision_FKPP}, we can say that $u_n(t) \leq v_n(t)$. But this directly contradicts the bound $v_n(t) \leq r_\text{max} - e^{-\lambda} \varepsilon_n$ which holds for a.e. $t$. We conclude that $u_n \leq r_\text{max} - e^{-\lambda} \varepsilon_n$ on $[0,1] \times \T^d$: not only its existence time is more than $1$, but it is actually a strong interior solution. Note that the strict positivity of $u_n$ comes from $u_n(0,\cdot) > 0$ and, for instance, the weak formulation \eqref{eq:FKPP_weak}. 
			
			The conclusion follows by taking $\eta_n$ and $\varepsilon_n$ tending both to $0$, with $\eta_n$ going to $0$ sufficiently fast so that the sequence $(u_n(0))_{n \in \N} = (v_n(0))_{n \in \N}$ is increasing while converging to $u(0)$.
		\end{stepa}
	\end{proof}
	
	We now make the link between strong interior solutions of FKPP and the martingale problem for the BBM. Specifically, we prove the following.
	
	\begin{Thm}
		\label{thm:martingale_smooth_case}
		Let $R \sim \BBM(\nu, \boldsymbol{q}, R_0)$, where $\nu > 0$, $\boldsymbol{q}$ is a branching mechanism and $R_0 \in \P(\M_\delta(\T^d))$. Let $u \in C^2([0,1] \times \T^d)$ be a strong interior solution of \eqref{eq:FKPP_forward} with $(\Phi, \lambda) = (\Phi_{\boldsymbol q},\lambda_{\boldsymbol q})$. We define $\varphi = \log u$. Then for all $t \in [0,1]$, there holds
		\begin{equation*}
		\exp \left( \langle \varphi(t), M_0 \rangle \right)  = \E_R \left[ \left. \exp \left(  \langle \varphi(0), M_t \rangle \right) \right| \F_0 \right].
		\end{equation*}
	\end{Thm}
	
	Note that as the right hand side depends only on $\varphi(0)$ (or equivalently on $u(0)$), such a result could not hold with general strong solutions if~\eqref{eq:FKPP_forward} had different strong solutions starting from the same initial condition. Following Remark~\ref{rem:non_unique_FKPP}, the assumption that $u$ is a strong \emph{interior} solution is reasonable.
	
	\begin{proof}
		A direct consequence of Definition~\ref{def:BBM_general_R0} is that $R$-\emph{a.s.}, $R(\, \cdot \, | \F_0) = R(\, \cdot \, | M_0) = R^{M_0}$. Hence, the claim is a consequence of the following one, stated for all deterministic initial configuration $\mu \in \M_\delta(\T^d)$:
		\begin{equation}
		\label{eq:martingale_smooth_case}
		\exp \left(  \langle \varphi(t), \mu \rangle \right)  = \E_{R^\mu} \left[  \exp \left( \langle \varphi(0), M_t \rangle \right) \right] = \Em_{\mu} \left[ \exp \left( \langle \varphi(0), M_t \rangle \right) \right], \qquad t \in [0,1].
		\end{equation}
		Notice that this identity~\eqref{eq:martingale_smooth_case} is trivial when $\mu = 0$. Otherwise, we decompose $\mu$ as a sum of Dirac masses and we simply apply the branching property, Proposition~\ref{prop:branching_property}: it shows that it is enough to prove \eqref{eq:martingale_smooth_case} when $\mu = \delta_x$ to conclude.
		
		\bigskip
		
		Thus from now on we assume $\mu = \delta_x$. In this case, we need to show
		\begin{equation*}
		u(t,x) = \exp \left(\varphi(t,x) \right) = \Em_{\delta_x} \left[ \exp \left(  \langle \varphi(0), M_t \rangle \right) \right].
		\end{equation*}
		We give a name to the right-hand side, that is, we define
		\begin{equation*}
		v(t,x) = \Em_{\delta_x} \left[ \exp \left( \langle \varphi(0), M_t \rangle \right) \right].
		\end{equation*} 
		We know that $v(0) = u(0)$ by definition of $\varphi$, and that $v$ is a $D$-weak solution of \eqref{eq:FKPP_forward} thanks to Proposition~\ref{prop:evolution_eq_u_from_BBM}. We want to show that $u = v$ everywhere on $[0,1] \times \T^d$. 
		
		By monotonicity for $D$-weak solutions, specifically Proposition~\ref{prop:comparision_FKPP}, we deduce that $u \leq v$. 
		
		On the other hand, to go the other way around let us define $\tilde{\varphi}(s,x) = \log u(t-s,x)$ in such a way that $\tilde{\phi} = \nu \tilde{\varphi}$ satisfies $\partial_t \tilde{\phi}(s) + \mathcal L_{\nu, \boldsymbol q}[\tilde{\phi}(s)] = 0$ in a strong sense. By Proposition~\ref{prop:loc_exp_mart_BBM}, we know that the process whose value at time $s \in [0,t]$ is 
		\begin{equation*}
		\exp \left( \frac{1}{\nu} \langle \tilde{\phi}(s), M_s \rangle \right) = \exp \left( \langle \tilde{\varphi}(s), M_s \rangle \right)
		\end{equation*}
		is a nonnegative local martingale and a supermartingale. As a consequence 
		\begin{equation*}
		\Em_{\delta_x} \left[\exp \left(  \langle \tilde{\varphi}(t), M_t \rangle \right) \right] \leq 
		\Em_{\delta_x} \left[ \exp \left(  \langle \tilde{\varphi}(0), M_0 \rangle \right) \right]
		\end{equation*}
		But by definition of $\tilde{\varphi}$, the left hand side is $v(t,x)$ while the right hand side is $u(t,x)$.
	\end{proof}
	The following remark is an outcome of the previous proof.
	\begin{Rem}
		Thanks to the Markov property of the BBM, it is not difficult to see that, with the same notations and assumptions as in Theorem~\ref{thm:martingale_smooth_case}, if we define $\tilde \phi(t) = \nu \varphi(1-t)$ as in the proof above, then $\tilde \phi$ satisfies $\partial_t \tilde \phi(s) + \mathcal L_{\nu, \boldsymbol q}[\tilde \phi(t)]$ and the process whose value at time $t \in [0,1]$ is 
		\begin{equation*}
		\exp \left( \frac{1}{\nu} \cg \tilde \phi(t), M_t \cd \right)    
		\end{equation*}
		is a true martingale. This is in that sense that Theorem~\ref{thm:martingale_smooth_case} extends Proposition~\ref{prop:martingale_problem}, as crucially $\tilde \phi$ can take positive values.
	\end{Rem}
	
	\section{Proof of Theorem~\ref{thm:equality_values_dyn}, the dynamical case}
	
	We are now ready to give a proof to the main theorem of this Chapter, that is, Theorem~\ref{thm:equality_values_dyn}.
	
	\begin{proof}[Proof of Theorem \ref{thm:equality_values_dyn}]
		
		The proof follows the one of the static case: to compute a l.s.c.\ envelope we compute two successive Legendre transforms and apply Theorem~\ref{thm:legendre_transform_lsc_envelope}.
		
		\begin{stepb}{Legendre transform of $\BrSch$}
			This first step is almost identical to what was done in Section~\ref{sec:static_legendre_transform}. Indeed, we claim that for all continuous functions $\sigma,\theta \in C(\T^d)$, 
			\begin{align*}
			\BrSch_{\nu, \boldsymbol q, R_0}^*(\sigma,\theta) &:= \sup_{\rho_0, \rho_1 \in\M_+(\T^d)} \langle \sigma, \rho_0 \rangle + \langle \theta, \rho_1 \rangle - \BrSch_{\nu, \boldsymbol q, R_0}(\rho_0,\rho_1) \\
			&= \nu \log \E_R \left[ \exp \left( \frac{1}{\nu} \langle \sigma, M_0 \rangle + \frac{1}{\nu} \langle \theta, M_1 \rangle \right) \right].
			\end{align*} 
			
			By definition, we have
			\begin{align*}
			\sup_{\rho_0, \rho_1 \in\M_+(\T^d)} &\langle \sigma, \rho_0 \rangle + \langle \theta, \rho_1 \rangle - \BrSch_{\nu, \boldsymbol q, R_0}(\rho_0,\rho_1) \\
			&=\sup_{\substack{P \in \P(\cadlag([0,1], \M_+(\T^d)))\\ \E_P[M_0] < +\infty \mbox{ and }\E_P[M_1] < + \infty}} \E_{P} [\langle \sigma, M_0 \rangle] + \E_{P} [\langle \theta, M_1 \rangle] - \nu H(P|R) \\
			&= \nu \left( \sup_{\substack{P \in \P(\cadlag([0,1], \M_+(\T^d)))\\ \E_P[M_0] < +\infty \mbox{ and }\E_P[M_1] < + \infty}} \E_{P} \left[ \frac{1}{\nu}  \langle \sigma, M_0 \rangle \right] + \E_{P} \left[ \frac{1}{\nu} \langle \theta, M_1 \rangle \right] - H(P|R) \right).
			\end{align*}
			Therefore, our claim is equivalent to
			\begin{equation*}
			\sup_{\substack{P \in \P(\cadlag([0,1], \M_+(\T^d)))\\ \E_P[M_0] < +\infty \mbox{ and }\E_P[M_1] < + \infty}} \hspace{-16pt} \E_{P} \left[ \frac{1}{\nu}  \langle \sigma, M_0 \rangle \right] + \E_{P} \left[ \frac{1}{\nu} \langle \theta, M_1 \rangle \right] - H(P|R) = \log \E_R \left[ \exp \left( \frac{1}{\nu} \langle \sigma, M_0 \rangle + \frac{1}{\nu} \langle \theta, M_1 \rangle \right) \right].
			\end{equation*}
			As for the static case the inequality ``$\leq$'' comes from \eqref{eq:convex_ineq_entropy}.
			
			On the other hand, to find a maximizing sequence, we consider for $n \in \N$
			\begin{equation*}
			P^n := \frac{\exp( \nu^{-1} \langle \sigma, M_0\rangle + \nu^{-1} \langle \theta, M_1\rangle) \1_{M_0(\T^d) + M_1(\T^d) \leq n}}{\E_R[ \exp( \nu^{-1} \langle \sigma, M_0\rangle + \nu^{-1} \langle \theta, M_1\rangle) \1_{M_0(\T^d) + M_1(\T^d) \leq n}]} \cdot R .
			\end{equation*} 
			Then the conclusion comes from exactly the same path as Proposition~\ref{prop:legendre_transform_init} and we do not detail it.
		\end{stepb}
		
		Applying then Theorem~\ref{thm:legendre_transform_lsc_envelope}, we end up with the following expression for the l.s.c.\ envelope:
		\begin{equation*}
		\overline{\BrSch}_{\nu, \boldsymbol q, R_0}(\rho_0,\rho_1) = \sup_{\sigma,\theta \in C(\T^d)} \langle \sigma, \rho_0 \rangle + \langle \theta, \rho_1 \rangle - \nu \log \E_R \left[ \exp \left( \frac{1}{\nu} \langle \sigma, M_0 \rangle + \frac{1}{\nu} \langle \theta, M_1 \rangle \right) \right].  
		\end{equation*}
		Let us now work on this expression.
		
		\begin{stepb}{Getting $\theta$ back to $t=0$}
			Now comes the most technical step of the proof. It amounts to drop $\theta$ and replace it with $\nu \log u(0,\cdot) = \nu \log u(0)$, being $u \in C^2([0,1] \times \T^d)$ a strong interior solution of \eqref{eq:FKPP_forward}. Specifically, we claim that
			\begin{multline}
			\label{eq:aux_to_prove_dual}
			\sup_{\sigma,\theta \in C(\T^d)} \langle \sigma, \rho_0 \rangle + \langle \theta, \rho_1 \rangle - \nu \log \E_R \left[ \exp \left( \frac{1}{\nu} \langle \sigma, M_0 \rangle + \frac{1}{\nu} \langle \theta, M_1 \rangle \right) \right] \\
			= \sup_{\substack{\sigma \in C(\T^d),\\ u \in C^2([0,1] \times \T^d)}} \langle \sigma, \rho_0 \rangle + \langle  \nu \log u(0), \rho_1 \rangle - \nu \log \E_R \left[ \exp \left( \frac{1}{\nu} \langle \sigma + \nu \log u(1), M_0 \rangle \right) \right]
			\end{multline}
			where the supremum is taken over all strong interior solutions $u$ of equation~\eqref{eq:FKPP_forward}. Importantly, the last term of the right hand side depends only on $R_0$.
			
			The ``$\geq$'' inequality in \eqref{eq:aux_to_prove_dual} is easy: we start from the left hand side and restrict to $\theta = \nu \log u(0)$ (where~$u$ is a strong interior solution), then Theorem~\ref{thm:martingale_smooth_case} justifies the last equality in what follows:  
			\begin{align}
			\notag \E_R \left[ \exp \left( \frac{1}{\nu} \langle \sigma, M_0 \rangle +  \langle \log u(0), M_1 \rangle \right) \right] &= \E_R \left[ \exp \left( \frac{1}{\nu} \langle \sigma, M_0 \rangle \right) \E_R \left[ \left. \exp \left( \langle \log u(0), M_1 \rangle \right) \right| \F_0 \right] \right] \\
			\label{eq:aux_theta_back_0} &= \E_R \left[ \exp \left( \frac{1}{\nu} \langle \sigma + \nu \log u(1), M_0 \rangle \right) \right].
			\end{align}
			
			On the other hand, let us take $\theta \in C(\T^d)$ such that the left hand side is not $- \infty$ in~\eqref{eq:aux_to_prove_dual}, that is such that $\E_R [ \exp ( \nu^{-1} \langle \sigma, M_0 \rangle + \nu^{-1} \langle \theta, M_1 \rangle ) ] < + \infty$. Thanks to Proposition~\ref{prop:evolution_eq_u_from_BBM} we know that if we define 
			\begin{equation*}
			v(t,x) := \Em_{\delta_x} \left[ \exp \left( \frac{1}{\nu} \langle \theta, M_t \rangle \right) \right],
			\end{equation*}
			it yields a $D$-weak solution of \eqref{eq:FKPP_forward}. On the other hand, by Definition~\ref{def:BBM_general_R0} and the branching property, Proposition~\ref{prop:branching_property}, there holds
			\begin{equation*}
			\E_R \left[ \exp \left( \frac{1}{\nu} \langle \sigma, M_0 \rangle + \frac{1}{\nu} \langle \theta, M_1 \rangle \right) \right] = \E_{R_0} \left[ \prod_{x \in \Supp \mu} \exp \left( \frac{\sigma(x)}{\nu}  \right) v(1,x)  \right].    
			\end{equation*}
			As the left hand side is finite, so is the right hand side. For the latter to be finite we need that for $R_0$-a.e.~$\mu$, there holds $v(1,x) < + \infty$ for all $x \in \Supp \mu$. As we have excluded the case $\mu = 0$ $R_0$-\emph{a.s.}\ by assumption, this shows that that $v(1)$ is not identically $+ \infty$. In addition, $v(0) = \exp(\nu^{-1} \theta)$ is strictly positive and continuous. Thanks to Proposition~\ref{prop:FKPP_approx}, we can find a sequence $(u_n)_{n \in \N}$ of strong interior solutions of \eqref{eq:FKPP_forward} such that $u_n(0)$ converges increasingly to $\exp(\nu^{-1} \theta)$ (that is equivalent to $\nu \log u_n(0)$ converging increasingly to $\theta$). In particular, combining the comparison principle (Proposition~\ref{prop:comparision_FKPP}) and Theorem~\ref{thm:martingale_smooth_case}, similarly to \eqref{eq:aux_theta_back_0},
			\begin{align*}
			\E_R \left[ \exp \left( \frac{1}{\nu} \langle \sigma, M_0 \rangle + \frac{1}{\nu} \langle \theta, M_1 \rangle \right) \right] &\geq   \E_R \left[ \exp \left( \frac{1}{\nu} \langle \sigma, M_0 \rangle +  \langle \log u_n(0), M_1 \rangle \right) \right] \\
			&=   \E_R \left[ \exp \left( \frac{1}{\nu} \langle \sigma, M_0 \rangle \right) \E_R \left[ \left. \exp \left( \langle \log u_n(0), M_1 \rangle \right) \right| \F_0 \right] \right] \\
			&= \E_R \left[ \exp \left( \frac{1}{\nu} \langle \sigma + \nu \log u_n(1), M_0 \rangle \right) \right].
			\end{align*}
			Plugging that back, we see that for every $n \in \N$,
			\begin{align*}
			\langle \sigma,& \rho_0 \rangle  + \langle \theta, \rho_1 \rangle - \nu \log \E_R \left[ \exp \left( \frac{1}{\nu} \langle  \sigma, M_0 \rangle + \frac{1}{\nu} \langle \theta, M_1 \rangle \right) \right] \\
			& \leq \langle \sigma, \rho_0 \rangle + \langle \theta, \rho_1 \rangle - \nu \log \E_R \left[ \exp \left( \frac{1}{\nu} \langle \sigma + \nu \log u_n(1), M_0 \rangle \right) \right] \\
			& \leq \langle \theta - \nu \log u_n(0), \rho_1 \rangle +  \sup_{\substack{\sigma \in C(\T^d),\\ u \in C^2([0,1] \times \T^d)}} \langle \sigma, \rho_0 \rangle + \langle \nu \log u(0), \rho_1 \rangle - \nu \log \E_R \left[ \exp \left( \frac{1}{\nu} \langle \sigma + \nu \log u(1), M_0  \right) \right] .
			\end{align*}
			Sending $n \to + \infty$, the first term goes to $0$ by monotone convergence. Taking then the supremum in $\theta$ yields the inequality ``$\leq$'' in \eqref{eq:aux_to_prove_dual}.
		\end{stepb}
		
		\begin{stepb}{Conclusion}
			From the previous step and by doing the change of variables
			\begin{equation}
			\label{eq:zz_aux_change_variables_u_phi}
			\begin{cases}
			\phi(t,x) & = \nu \log u(1-t,x), \\
			u(t,x) & = \displaystyle{\exp\left\{ \frac{1}{\nu} \phi(1-t,x) \right\}},
			\end{cases}    
			\end{equation}
			we can rewrite the convex envelope as 
			\begin{equation*}
			\overline{\BrSch}_{\nu, \boldsymbol q, R_0}(\rho_0,\rho_1) = \sup_{\sigma \in C(\T^d), \phi \in C^2([0,1] \times \T^d)} \langle \sigma, \rho_0 \rangle + \langle \phi(1), \rho_1 \rangle - \nu \log \E_R \left[ \exp \left( \frac{1}{\nu} \langle \sigma + \phi(0), M_0 \rangle  \right) \right],  
			\end{equation*}
			where the supremum is taken over all $\phi$ that are smooth solutions of the backward equation
			\begin{equation}
			\label{eq:HJ_backward}
			\partial_t \phi + \L_{\nu, \boldsymbol q}[\phi] = 0,     
			\end{equation}
			being $\L_{\nu, \boldsymbol q}$ defined in \eqref{eq:operator_L}, with the additional constraint that they come from a strong interior solution of \eqref{eq:FKPP_forward}. Doing the change of variables $\tilde{\sigma} = \nu^{-1} (\sigma + \phi(0))$, we can write
			\begin{equation*}
			\overline{\BrSch}_{\nu, \boldsymbol q, R_0}(\rho_0,\rho_1) = \sup_{\tilde{\sigma} \in C(\T^d), \phi \in C^2([0,1] \times \T^d)}  \langle \phi(1), \rho_1 \rangle - \langle \phi(0), \rho_0 \rangle + \nu \langle  \tilde{\sigma}, \rho_0 \rangle - \nu \log \E_R [ \exp ( \langle \tilde \sigma, M_0 \rangle  ) ].  
			\end{equation*}
			The supremum is now decoupled in $\phi$ and $\tilde{\sigma}$. The one in $\tilde{\sigma}$ yields $\nu L_{R_0}^*(\rho_0)$ by definition. Next we want to take the supremum in $\phi$ and use the duality result of Theorem~\ref{thm:existence_ruot}. Given the choice of $\Psi^*_{\nu, \boldsymbol q}$ made in~\eqref{eq:def_Psi*}, any solution $\phi$ of equation~\eqref{eq:HJ_backward} also satisfies the inequality constraint from the definition~\eqref{eq:duality_RUOT} of $\ruot_{\nu, \Psi}$. Consequently, there holds 
			\begin{equation*}
			\overline{\BrSch}_{\nu, \boldsymbol q, R_0}(\rho_0,\rho_1) \leq \nu L^*_{R_0}(\rho_0) + \ruot_{\nu,\Psi}(\rho_0,\rho_1).    
			\end{equation*}
			
			On the other hand, take $\phi$ which satisfies the inequality \eqref{eq:duality_RUOT} in the dual formulation of $\ruot_{\nu,\Psi}$. Doing again the change of variables \eqref{eq:zz_aux_change_variables_u_phi}, we get a supersolution of \eqref{eq:FKPP_forward}. Regularizing it with Proposition~\ref{prop:FKPP_approx}, going back to the $\phi$ domain thanks to~\eqref{eq:zz_aux_change_variables_u_phi} we see that we have a sequence $\phi_n$ of smooth solution of \eqref{eq:HJ_backward}, coming from an interior solution of \eqref{eq:FKPP_forward}, such that $\phi_n(1)$ converges increasingly to $\phi(1)$ and $\phi_n(0) \leq \phi(0)$. Thus, we have
			\begin{equation*}
			\langle \phi(1), \rho_1 \rangle - \langle \phi(0), \rho_0 \rangle \leq \langle \phi(1) - \phi_n(1), \rho_1 \rangle + \langle \phi_n(1), \rho_1 \rangle - \langle \phi_n(0), \rho_0 \rangle,    
			\end{equation*}
			and taking the limit $n \to + \infty$, by monotone convergence, the first term of the right hand side vanishes. It allows us to conclude that 
			\begin{equation*}
			\langle \phi(1), \rho_1 \rangle - \langle \phi(0), \rho_0 \rangle \leq   \overline{\BrSch}_{\nu, \boldsymbol q, R_0}(\rho_0,\rho_1) - \nu L^*_{R_0}(\rho_0).   
			\end{equation*}
			Eventually, taking the supremum in $\phi$ on the left hand side yields $\ruot_{\nu,\Psi}(\rho_0,\rho_1)$ thanks to Theorem~\ref{thm:existence_ruot}.
		\end{stepb}
		
	\end{proof}

	\section{Examples and counterexamples for the branching Schrödinger problem}
	In this section, we provide a few explicit computations in order to illustrate our result. In Subsection~\ref{subsec:example_static}, we give some examples where $L^*_{R_0}$ and $\Init_{R_0}$ coincide and can be computed explicitely. Then, in Subsection~\ref{subsec:counterexamples}, we show that both for the static and the dynamical problem, our Theorem~\ref{thm:initial_lsc_envelope} and~\ref{thm:equality_values_dyn} are somehow optimal, in the sense that in general, neither $\Init_{R_0}$ nor $\BrSch_{\nu, \boldsymbol q, R_0}$ is l.s.c. 
	
	\subsection{Cases when the static problem can be computed}
	\label{subsec:example_static}
	In the three following examples, the value of $\Init_{R_0}$ is easly computed, and seen to be l.s.c. In particular, Theorem~\ref{thm:initial_lsc_envelope} guarantees that in these cases, $L^*_{R_0} = \Init_{R_0}$. Later in Theorem~\ref{thm:dual_initial_entropic_minimization}, we will see that there is a general assumption on $R_0$, satisfied in our three examples, ensuring that $\Init_{R_0} = L^*_{R_0}$. However, in the next subsection, we will see examples where this assumption is not satisfied, and where $\Init_{R_0}$ is not l.s.c.
	
	\begin{itemize}
		\item Assume $R_0 = \delta_\mu$, for some $\mu \in \M_+(\T^d)$. Then for all $\rho_0 \in \M_+(\T^d)$,
		\begin{equation*}
		\Init_{R_0}(\rho_0) = L^*_{R_0} (\rho_0) = 
		\left\{ \begin{aligned}
		&0 &&\mbox{if }\rho_0 = \mu,\\
		&+\infty && \mbox{else},
		\end{aligned}\right.
		\end{equation*}
		and the optimal $P_0$ in the r.h.s.\ of~\eqref{eq:def_init} is $\delta_{\rho_0} = \delta_\mu$.
		
		\item Assume $R_0$ only charge unitary Dirac masses, that is, there is $\bar \rho \in \P(\T^d)$ such that $R_0$ is the law of $\delta_X$, where $X$ is a random variable with values in $\T^d$, of law $\bar \rho$.
		Then for all $\rho_0 \in \M_+(\T^d)$,
		\begin{equation*}
		\Init_{R_0}(\rho_0) = L^*_{R_0} (\rho_0) = H(\rho_0 | \bar \rho),
		\end{equation*}
		where $H$ is defined by~\eqref{eq:def_H}, and the optimal $P_0$ in the r.h.s.\ of~\eqref{eq:def_init} is the law of $\delta_Y$, where $Y$ is a random variable with values in $\T^d$, of law $\rho_0$.
		
		\item Assume $R_0$ is the law of a Poisson point process of intensity $\bar \rho \in \M_+(\T^d)$ (see~\cite{kingman1992poisson}).
		Then for all $\rho_0 \in \M_+(\T^d)$,
		\begin{equation*}
		\Init_{R_0}(\rho_0) = L^*_{R_0} (\rho_0) = h(\rho_0|\bar \rho),
		\end{equation*}
		where $h$ is defined by~\eqref{eq:def_KL}, and the optimal $P_0$ in the r.h.s.\ of~\eqref{eq:def_init} is the law of the Poisson point process on $\T^d$, of intensity $\rho_0$. This can be proven using Cambell's formula \cite[Section 3.2]{kingman1992poisson}, but we leave the details to the reader.
	\end{itemize}
	
	\subsection{Examples of ill-posedness in the branching Schrödinger problem}
	\label{subsec:counterexamples}
	
	As mentioned in the introduction, see Remark~\ref{rem:ill-posed}, the branching Schrödinger problem is ill-posed. In this section we show counterexamples to illustrate the ill-posedness. Actually, we will explain that the initial problem from Definition~\ref{def:init} is already ill-posed. Specifically: the infimum in~\eqref{eq:def_init} is not always achieved, and $\Init_{R_0}$ is not l.s.c.\ for the topology of weak convergence. This is enough to justify why the branching Schrödinger problem is ill-posed. However we will also give counterexamples to well-posedness that has nothing to do with the initial problem: the functional $\BrSch_{\nu, \boldsymbol q, R_0}$ is not l.s.c.\ in general at couples of type $(0,\rho)$, whenever $\rho \neq 0$, and in the case when $\boldsymbol q$ has no exponential moment, it is free to create particles at large rate in RUOT, but not in the branching Schrödinger problem.
	
	\paragraph{The infimum for $\boldsymbol{\Init_{R_0}}$ is not achieved}	
	As we will see, there can be a lack of existence for the problem associated with $\Init_{R_0}$ when the generating function corresponding to the number of particles at the initial time has a domain that is bounded above, with a finite slope at its supremum, as represented in Figure~\ref{fig:c-ex_existence}. In this case, calling $\boldsymbol \pi$ the law of the number of particles at the initial time under the reference measure $R_0$, $\Phi_{\boldsymbol \pi}$ its generating function and $z_\text{max}$ the supremum of the domain of $\Phi_{\boldsymbol \pi}$, we will see that there are existence issues for $\Init_{R_0}(\rho)$ whenever $\rho(\T^d)$, the target mean number of particles, is larger than the slope $(\log \Phi_{\boldsymbol \pi})'(z_\text{max})$.
	
	\begin{figure}
		\centering
		\begin{tikzpicture}
		\draw[->] (-0.3,0) -- (4.5,0);
		\draw[->] (0,-1) -- (0,2.3)node[left]{$\log \Phi_{\boldsymbol \pi}$};
		\draw (2, -2pt) -- (2, 2pt) node[below=5pt]{$1$};
		\draw (4, -2pt) -- (4, 2pt) node[below=5pt]{$z_\text{max}$};
		\draw (-2pt, 1.5) -- (2pt, 1.5) node[left]{$\log \Phi_{\boldsymbol \pi}(z_\text{max})$};
		\draw (-2pt, -.5) -- (2pt, -.5) node[left]{$\log \Phi_{\boldsymbol \pi}(0)$};
		\draw[very thick, blue][domain=0:4] plot(\x, {\x*\x/8 - 0.5});
		\fill[blue] (4,1.5) circle (3pt);
		\end{tikzpicture}
		\caption{\label{fig:c-ex_existence} In this figure, the convex function $\log \Phi_{\boldsymbol \pi}$ has a domain that is bounded above, and has a finite slope at its extreme point.}
	\end{figure}
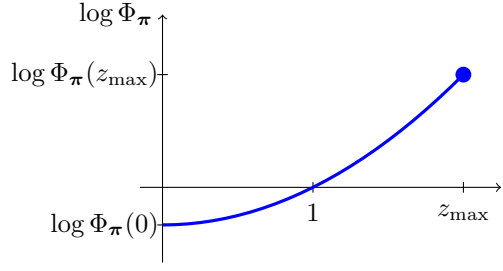
	
	Let $\boldsymbol \pi\in \P(\N)$ be a probability measure whose generating function $\Phi_{\boldsymbol \pi}$ satisfies
	\begin{equation*}
	\Phi_{\boldsymbol \pi}'(e)= \Phi_{\boldsymbol \pi}(e), \qquad \forall z>e, \quad \Phi_{\pi}(z) = + \infty,
	\end{equation*}
	that is, we choose $z_\text{max} = e$ and $(\log \Phi_{\boldsymbol \pi})'(z_\text{max})=1$. For $n \in \N$, we call $\delta^n$ the map defined by
	\begin{equation*}
	\delta^n: (x_1,\dots, x_n) \in (\T^d)^n \mapsto \delta_{x_1} + \dots + \delta_{x_n} \in \M_+(\T^d),
	\end{equation*}
	and we set
	\begin{equation*}
	R_0 := \sum_n \pi_n  \delta^n \pf (\Leb)^{\otimes n}.
	\end{equation*}
	That is, to draw from $M_0$ under $R_0$, we first draw $n \in \N$ according to the law $\boldsymbol \pi$, and then $M_0$ corresponds to $n$ particles i.i.d.\ and uniformly distributed. Finally, we take $\bar \rho >1$, and we show that the infimum for $\Init_{R_0}(\bar \rho \Leb)$ is not achieved.
	
	First, the convex inequality~\eqref{eq:convex_ineq_entropy} with $Y := M(\T^d)$ shows that for all competitor $P_0$ of $\Init_{R_0}(\bar \rho_0 \Leb)$,
	\begin{equation}
	\label{eq:convex_inequality_cex1}
	\bar \rho \leq H(P_0|R_0) + \log \E_{R_0} [\exp (M(\T^d))] =  H(P_0|R_0) + \log \Phi_{\boldsymbol \pi}(e),
	\end{equation}
	with equality if and only if 
	\begin{equation*}
	P_0 := \frac{\sum_n \pi_n \times \exp(n)  \delta^n \pf (\Leb)^{\otimes n}}{\sum_n \pi_n \times \exp(n)}.
	\end{equation*}
	A quick computation shows that for this specific $P_0$, there holds $\E_{P_0}[M] = (\log \Phi_{\boldsymbol \pi})'(e)\cdot\Leb  = \Leb$. Hence, it is not compatible with the constraint $\E_{P_0}[M] = \bar\rho \Leb$ and this $P_0$ cannot be a competitor of $\Init_{R_0}(\bar \rho_0 \Leb)$. Consequently, for all competitor $P_0$ of $\Init_{R_0}(\bar \rho_0 \Leb)$, the inequality in~\eqref{eq:convex_inequality_cex1} is strict, that is, 
	\begin{equation*}
	H(P_0|R_0) > \bar\rho - \log \Phi_{\boldsymbol \pi}(e).
	\end{equation*}
	
	\noindent It remains to show that $\Init_{R_0}(\bar\rho \Leb) \leq \bar\rho - \log \Phi_{\boldsymbol \pi}(e)$. To do so, take $(\bar \varphi_N)_{N \in \N}$ the sequence of numbers defined implicitly for all $N \in \N$ large enough by
	\begin{equation*}
	\frac{\sum_{n\leq N} n \pi_n \times (e^{\bar\varphi_N})^n  }{\sum_{n \leq N} \pi_n \times (e^{\bar\varphi_N})^n} = \bar \rho.
	\end{equation*}
	Observe that for all $N$, $\bar \varphi_N >1$, and that $\varphi_N$ decreases to $1$ as $N$ tends to $+\infty$. 
	Finally, call $(P_0^N)_{N\in\N}$ the sequence of laws defined for all $N \in \N$ by
	\begin{equation*}
	P^N_0 := \frac{\sum_{n\leq N} \pi_n \times (e^{\bar\varphi_N})^n  \delta^n \pf (\Leb)^{\otimes n}}{\sum_{n\leq N} \pi_n \times (e^{\bar\varphi_N})^n}.
	\end{equation*}
	For all $N$, $P_0^N$ satisfies the constraint, and hence the following direct computation lets us conclude:
	\begin{align*}
	\Init_{R_0}(\bar\rho \Leb)&\leq H(P_0^N | R_0) = \bar \varphi_N \bar \rho - \log \sum_{n\leq N} \pi_n \times (e^{\bar\varphi_N})^n \\
	&\leq \bar \varphi_N \bar \rho - \log \sum_{n\leq N} \pi_n \times e^n \underset{N\to +\infty}{\longrightarrow} \bar\rho - \log \Phi_{\boldsymbol \pi}(e).
	\end{align*}
	
	\paragraph{$\boldsymbol{\Init_{R_0}}$ is not l.s.c.}
	
	To prove that $\Init_{R_0}$ is not l.s.c.\ w.r.t.\ the topology of weak convergence, we will find an initial law $R_0\in \P(\M_{\delta}(\T^d))$, and a measure $\rho \in \M_+(\T^d)$ such that the infimum in~\eqref{eq:def_init} is achieved at a law $P_0$, but
	\begin{equation*}
	H(P_0|R_0) > L^*_{R_0}(\rho).
	\end{equation*}
	The result follows as therefore, in virtue of Theorem~\ref{thm:initial_lsc_envelope}, $\Init_{R_0}$ does not coincide with its l.s.c.\ envelope. The main idea is that when $R_0$ is too far from being invariant by translations, there can be a gap between $L^*_{R_0}$ and the bigger functional defined in the same way, but allowing discontinuous test functions in~\eqref{eq:def_L*}.
	
	Let us consider $A$ a dense set of $\T^d$ with $\Leb(A) = 1/2$. We call $B := \T^d \backslash A$. We also give ourselves $\boldsymbol{\alpha},\boldsymbol{\beta} \in \P(\N)$ and denote by $\Phi_{\boldsymbol \alpha}$ and $\Phi_{\boldsymbol \beta}$ their generating functions.
	
	We define
	\begin{equation*}
	R_0 := \int_{\T^d} \sum_k \Big\{ \1_A(x) \alpha_k \delta_{k \delta_x} + \1_B(x) \beta_k \delta_{k\delta_x} \Big\} \D x.
	\end{equation*}
	That is, to draw a measure $\mu$ from $R_0$, one first draws $x \in \T^d$ uniformly, then draws $k$ according to $\boldsymbol{\alpha}$ (resp.\ $\boldsymbol{\beta}$) if $x \in A$ (resp.\ $B$), and eventually set $\mu = k \delta_{x}$. Notice that
	\begin{equation*}
	\E_{R_0}[M] = \int_{\T^d} \sum_k \Big\{ \1_A(x) k\alpha_k \delta_x + \1_B(x) k\beta_k \delta_x \Big\} \D x = \Phi_{\boldsymbol \alpha}'(1)\Leb\llcorner_A + \Phi_{\boldsymbol \beta}'(1)\Leb \llcorner_B,
	\end{equation*}
	and that for all $\varphi \in C(\T^d)$,
	\begin{equation}
	\label{eq:laplace_transform_cex_init_not_lsc}
	\E_{R_0}[\exp(\cg \varphi, M\cd)] = \int \left\{ \1_A(x) \Phi_{\boldsymbol \alpha}\left(e^{\varphi(x)}\right) + \1_B(x) \Phi_{\boldsymbol \beta}\left(e^{\varphi(x)}\right) \right\}\D x.
	\end{equation}
	
	Now we consider competitors $P_0$ of the form
	\begin{equation*}
	P_0^{a,b} := \frac{\exp\Big( a M(A) + b M(B) \Big)}{\E_{R_0}\left[ \exp\Big( a M(A) + b M(B) \Big) \right]} \cdot R_0, \qquad a,b \in \R,
	\end{equation*}
	provided the denominator is well defined. It is easy to see that $P_0^{a,b}$ is an optimizer for $\Init_{R_0}$ corresponding to its own intensity: it comes from~\eqref{eq:convex_ineq_entropy} applied with $Y := a M(A) + b M(B)$. Indeed, the expectation $\E_{P_0}[Y]$ depends only on the intensity measure of $P_0$, thus is constant when restricted to competitors for $\Init_{R_0}$ corresponding to the intensity measure of $P_0^{a,b}$. However, the subtlety is that $Y = \langle \varphi, M \rangle$ with $\varphi:= a \1_A + b \1_b$ which is discontinuous.   
	
	Provided $\Phi_{\boldsymbol \alpha}'(e^a)$ and $\Phi_{\boldsymbol \beta}'(e^b)$ are finite, a tedious but straightforward computation leads to
	\begin{equation*}
	H(P_0^{a,b} | R_0) =  \frac{ae^a \Phi_{\boldsymbol \alpha}'(e^a) + be^b\Phi_{\boldsymbol \beta}'(e^b)}{\Phi_{\boldsymbol \alpha}(e^a) + \Phi_{\boldsymbol \beta}(e^b)} - \log \frac{\Phi_{\boldsymbol \alpha}(e^a) + \Phi_{\boldsymbol \beta}(e^b)}{2}.
	\end{equation*}
	Moreover, calling $\rho := \E_{P_0^{a,b}}[M]$ and taking $\varphi \in C(\T^d)$,  
	\begin{align*}
	\cg \varphi, \rho \cd - \log \E_{R_0}[\exp(\cg \varphi, M\cd)] &= \frac{\left(\fint_A\varphi\right) e^a\Phi_{\boldsymbol \alpha}'(e^a)  + \left(\fint_B\varphi\right)e^b\Phi_{\boldsymbol \beta}'(e^b)}{\Phi_{\boldsymbol \alpha}(e^a) + \Phi_{\boldsymbol \beta}(e^b)}
	\\ &\hspace{50pt} - \log  \int \left\{ \1_A(x) \Phi_{\boldsymbol \alpha}\left(e^{\varphi(x)}\right) + \1_B(x) \Phi_{\boldsymbol \beta}\left(e^{\varphi(x)}\right) \right\}\D x\\
	&\leq \frac{\left(\fint_A\varphi\right) e^a\Phi_\alpha'(e^a)  + \left(\fint_B\varphi\right)e^b\Phi_{\boldsymbol \beta}'(e^b)}{\Phi_{\boldsymbol \alpha}(e^a) + \Phi_{\boldsymbol \beta}(e^b)} - \log  \frac{\Phi_{\boldsymbol \alpha}(e^{\fint_A\varphi}) + \Phi_{\boldsymbol \beta}(e^{\fint_B\varphi})}{2}\\
	&=: \Lambda_{a,b}\left( \fint_A\varphi,\fint_B \varphi \right),
	\end{align*}
	where $\Lambda_{a,b}: \R^2 \to \R\cup\{-\infty\}$ is a strictly concave function, whose maximum is $H(P_0^{a,b} | R_0)$ at point $(a,b)$. 
	No, we set $\boldsymbol \alpha$, $\boldsymbol \beta$, $a$ and $b$ in such a way that the closure of
	\begin{equation*}
	\left\{ \left( \fint_A\varphi,\fint_B \varphi \right) \ : \ \varphi \in C(\T^d) \mbox{ and } \E_{R_0}[\exp(\cg \varphi, M \cd)] < + \infty\right\}
	\end{equation*}
	does not contain $(a,b)$. The result follows easily. To do that, we can take for instance $a = \log 2$, $b=\log 3$ and $\boldsymbol \alpha, \boldsymbol \beta$ such that:
	\begin{gather*}
	\Phi'_{\boldsymbol \alpha}(1) = \Phi'_{\boldsymbol \beta}(1) = 1,\\
	\Phi_{\boldsymbol \alpha}'(2) = 3, \quad \Phi_{\boldsymbol \beta}'(3) = 2,\\
	\forall z>5/2, \quad \Phi_{\boldsymbol \alpha}(z) = + \infty. 	
	\end{gather*}
	In this setting, if $\varphi \in C(\T^d)$ is such that $\E_{R_0}[\exp(\cg \varphi, M \cd)] < + \infty$, in view of~\eqref{eq:laplace_transform_cex_init_not_lsc}, $e^\varphi \leq 5/2$, almost everywhere on $A$. But $\varphi$ is continuous and $A$ is dense, so in particular, $e^\varphi \leq 5/2$ everywhere on $\T^d$. Consequently, $\fint_B\varphi \leq 5/2 < b$ and the result follows.
	
	Note that in this setting $\rho$ is a multiple of $\Leb$, so that the lack of lower semi-continuity has nothing to do with the regularity of $\rho$, but only on the one of $R_0$.
	
	\begin{Rem}[From $\Init_{R_0}$ to $\BrSch_{\nu, \boldsymbol q, R_0}$]
		Let us justify why the two previous counterexamples are enough to prove that the branching Schrödinger problem is ill-posed in general. 
		
		To do this, fix $\nu,\boldsymbol q, R_0$, and call $\bar r := \sum_k (k-1) q_k$. Then, we have for all $\rho \in \M_+(\T^d)$:
		\begin{equation*}
		\BrSch_{\nu, \boldsymbol q, R_0}(\rho, e^{\bar r} \tau_{\nu} \ast\rho)  = \Init_{R_0}(\rho),
		\end{equation*}
		where $(\tau_s)_{s\geq 0}$ is the heat flow on the torus, and where because of~\eqref{eq:disintegration_entropy}, if $P_0$ is a competitor for the problem in the r.h.s., $P := \E_{P_0}[R^M] \sim \BBM(\nu, \boldsymbol q, P_0)$ is a competitor for the problem in the l.h.s.\ with $H(P_0 | R_0) = H(P|R)$ (see Corollary~\ref{cor:evolution_density}). The claim follows easily.
	\end{Rem}
	
	We close this section with counterexamples showing that there are couples $(\rho_0,\rho_1) \in \M_+(\T^d)^2$ at which $\BrSch_{\nu, \boldsymbol q, R_0}$ is not l.s.c.\ for the topology of weak convergence (even though $\Init_{R_0}$ is l.s.c.\ at $\rho_0$), and couples for which the infimum in the branching Schrödinger problem is not attained (even though the infimum in $\Init_{R_0}(\rho_0)$ is). 
	
	\paragraph{The branching Schrödinger problem starting from zero particles} 
	
	In plain English, one main difference between the the branching Schrödinger problem and the RUOT problem is the following: with a BBM, if you start from nothing, you cannot create particles; whereas with RUOT you can create mass from nothing. We leverage this observation to give another counterexample to lower semi-continuity. 
	
	Specifically let us consider $\nu>0$, a branching mechanism $\boldsymbol q$ such that $\sum_{k \geq 2} q_k >0$, and $R_0$ such that $R_0(\{0\}) \in (0,1)$, that is, with positive probability, there is no particle under $R_0$, and with positive probability, there are particles under $R_0$. (Actually, the important part of the assumption for the reasonning, is $R_0(\{0\}) \neq 0$, the other part, namely, $R_0(\{0\})<1$ is only here for our counterexample to be in the scope of Theorem~\ref{thm:equality_values_dyn}.) Call $R \sim \BBM(\nu, \boldsymbol q, R_0)$. Finally, chose $\rho \in \M_+(\T^d)$ such that $\rho \neq 0$ and $h(\rho|\Leb) < + \infty$.
	
	Let us prove that calling $\Psi := \Psi_{\nu,\boldsymbol q}$, we have
	\begin{equation}
	\label{eq:gap_0_particle}
	\BrSch_{\nu, \boldsymbol q, R_0}(0,\rho) > L^*_{R_0}(0) + \ruot_{\nu, \Psi}(0, \rho).
	\end{equation}
	A direct consequence is that $\BrSch_{\nu, \boldsymbol q, R_0}$ is not l.s.c.\ at $(0,\rho)$ as it does not coincide with its l.s.c.\ envelope, in virtue of Theorem~\ref{thm:equality_values_dyn}. Let us compute one by one the quantities in~\eqref{eq:gap_0_particle}.
	
	First, $\BrSch_{\nu, \boldsymbol q, R_0}(0,\rho) = + \infty$. Indeed, if $P$ is a competitor for this problem, with finite entropy w.r.t. $R$, because of the constraint $\E_P[M_0] = 0$, we must have $P_0 = \delta_0$. But as $P \ll R$, it implies $P \ll R^0 = R(\, \cdot\, | M_0 = 0)$. As $R^0$ is a Dirac mass on the curve stationing at $0 \in \M_+(\T^d)$, so is $P$, which is incompatible with $\E_P[M_1] = \rho \neq 0$. Hence, there is no competitor and the result follows. 
	
	Hence, it suffices to show that the r.h.s.\ in~\eqref{eq:gap_0_particle} is finite. Concerning $L^*_{R_0}(0)$, we have thanks to Theorem~\ref{thm:initial_lsc_envelope}
	\begin{equation*}
	L^*_{R_0}(0) \leq \Init_{R_0}(0) \leq H(\delta_0 | R_0) = - \log R(\{0\}). 
	\end{equation*}
	Concerning the finiteness of $\ruot_{\nu, \Psi}(0, \rho)$, this is a direct consequence of Proposition~\ref{prop:ruot_existence_competitor} and Remark~\ref{rk:ruot_existence_competitor_q}. 
	
	\paragraph{Free creation of particles} 
	Another crucial difference between RUOT and the branching Schrödinger problem is that in the latter, it always costs something to create more particles than the reference process, whereas in RUOT, it can be free to create particles at an arbitrary large rate. This difference will provide another example where the infimum in the branching Schrödinger problem is not achieved.
	
	To see this, take any $\nu>0$, and $\boldsymbol q$ a branching mechanism which admits no exponential moment, that is, such that $\Psi^*_{\nu, \boldsymbol q}(s) = + \infty$ whenever $s >0$. Doing so, we are exactly in the case of Section~\ref{sec:no_exp_moment}, and $\Psi_{\nu,\boldsymbol q}(r) = 0$ for all $r \geq \bar r := \sum_{k} (k-1)q_k$. Consider any $R_0 \neq \delta_0$ such that $\E_{R_0}[M(\T^d)] < +\infty$, and call $R := \BBM(\nu, \boldsymbol q, R_0)$ and $\rho_0 := \E_{R_0}[M] \in \M_+(\T^d)$.
	
	In this context, calling $\Psi := \Psi_{\nu, \boldsymbol q}$, we have for all $r > \bar r$ that $\ruot_{\nu, \Psi}(\rho_0, e^r \tau_\nu \ast \rho_0)= 0$ (the corresponding solution being $(\rho,m,\zeta) := \D t \otimes (e^{rt}\tau_{\nu t } \ast \rho_0, 0, r e^{rt}\tau_{\nu t } \ast \rho_0)$), and $L^*_{R_0}(\rho_0) = \Init_{R_0}(\rho_0) = 0$. 
	
	On the other hand, for all $P \in \P(\cadlag([0,1]; \M_+(\T^d)))$, either $P=R$, and then $\E_P[M_1] \neq e^{r} \tau_\nu \ast \rho_0$, or $H(P|R) >0$. Therefore, there is no competitor for the branching Schrödinger problem between $\rho_0$ and $e^r \tau_\nu \ast \rho_0$ that satisfies $H(P|R) = L^*_{R_0}(\rho_0) + \ruot_{\nu, \Psi}(\rho_0,e^r \tau_\nu \ast \rho_0) = 0$. 
	
	So given Theorem~\ref{thm:equality_values_dyn}, there are two possibilities. Either $\BrSch_{\nu, \boldsymbol q, R_0}(\rho_0,e^r \tau_\nu \ast \rho_0) >0$, and then $\BrSch_{\nu, \boldsymbol q, R_0}$ it is not l.s.c.\ at $(\rho_0,e^r \tau_\nu \ast \rho_0)$, or $\BrSch_{\nu, \boldsymbol q, R_0}(\rho_0,e^r \tau_\nu \ast \rho_0) =0$ but the infimum for the branching Schrödinger problem between $\rho_0$ and $e^r \tau_\nu \ast \rho_0$ is not achieved. With the tools developed in the next chapters, and in particular with Lemma~\ref{lem:Psi_compromise}, we could see that we are actually in the second case.

	\chapter{Laws of finite entropy w.r.t.\ a branching Brownian motion}
	\label{chap:characterization_finite_entropy}

	In this section we show that up to technicalities, the set of laws with finite entropy with respect to the BBM $R \sim \BBM(\nu, \boldsymbol{q}, R_0)$ is the set of BBMs with modified initial law, for which the particles undergo a Brownian trajectory with an additional drift, and for which the branching mechanism is changed. We also build such laws in great generality by providing explicitly their Radon-Nikodym derivative w.r.t.\ $R$, and we furthermore give a formula for their relative entropy w.r.t.\ $R$.

	We follow the same strategy as the one used in~\cite{leonard2012girsanov} for showing that the set of laws with finite entropy w.r.t.\ a Brownian motion somehow coincides with the set of Brownian motions with modified initial law, and with drift. For short, dropping all the difficulties, for all law $P$ with finite entropy w.r.t. the law of a Brownian motion $R$ of diffusivity $\nu$, there exists a drift $v=(v_t)_{t \in [0,1]}$ such that $P$ is a Brownian motion of same diffusivity $\nu$ and with drift $v$. Moreover, calling $(X_t)_{t \in [0,1]}$ the canonical process:
	\begin{equation}
	\label{eq:RN_derivative_brownian_case}
	\frac{\D P}{\D R} = \frac{\D P_0}{\D R_0}(X_0) \cdot \exp\left( \frac{1}{\nu}\int_0^1 v_t \cdot \D X_t - \frac{1}{2\nu} \int_0^1 |v_t|^2 \D t \right).
	\end{equation}
	The main part of this strategy relies on the construction and on the study of families of \emph{exponential martingales}, whose value at time $1$ will play the role of the Radon-Nikodym derivatives of our modified laws $P$ w.r.t. to the reference law $R$, as in~\eqref{eq:RN_derivative_brownian_case}. 
	
	As visible in this formula, in the non-branching case, the stochastic integral with respect to the Brownian motion, which is a continuous martingale, is the only type of processes that is necessary to describe all the Radon-Nikodym derivatives of laws with finite entropy with respect to the Brownian motion. 
	A way to understand this fact is a famous result asserting that all the local martingales with respect to the filtration of a Brownian motion are of this type, see \cite[Section~5.4]{legall2016brownian}. Therefore, only continuous stochastic calculus is necessary to understand the relationship between the Schrödinger problem and the regularized optimal transport problem. 
	
	On the other hand, in the branching case, we will need some notions of \emph{discontinuous} stochastic calculus. A first hopefully convincing reason for this is that the simplest class of processes studied in the branching case are the ones of type $(\cg \varphi, M_t \cd)_{t \in [0,1]}$, where $\varphi$ is a smooth function on the torus, and they are discontinuous. A deeper reason is the study of pure branching processes (\emph{i.e.} the case where we are only interested in the number of particles, and not in their spatial positions). These processes are instances of continuous time Markov chains, and for the latter, there are well known formulas for the Radon-Nikodym derivatives between processes with different jump rates (see \cite[Section 2 of Appendix 1]{kipnis1998scaling}). Once again, these are values at time one of exponential martingales, but this time involving discontinuous processes.
	
	\bigskip
	
	The notions that we will use are classical in the theory of stochastic calculus, and we will consider~\cite{jacod2013limit} as a reference book. But for the sake of completeness, we give a full presentation adapted to our framework (without proofs) in Section~\ref{sec:discontinuous_stochastic_calculus}. We start by explaining what are \emph{predictable} random objects. Not only this notion is central in discontinuous stochastic calculus, but also, the drift and the modified branching mechanism of the modified BBMs that appear in our study will be of this type: both will depend on $t$ and $x$, but also on the whole past of the process, up to time $t$. Then, we introduce the type of processes that we will study all along this chapter, that we call semi-martingales with simple jumps. Finally, we state a version of the Itô formula taking the jumps into account. We also provide a way to define exponential martingales out of a semi-martingale with simple jumps. 
	
	Then, in Section~\ref{sec:new_processes}, we introduce the two types of processes that are necessary to describe all the Radon-Nikodym derivatives of laws with finite entropy with respect to the BBM (postponing the rigorous construction to Appendix~\ref{app:stochastic_calculus}). The first one is a continuous martingale which is a generalization of the standard stochastic integral. It will be helpful later on to add drifts to the trajectories of individual particles. The second one is a pure-jump process whose jump instants are precisely the times of the branching events. It will be used to change the branching mechanism of the reference law. Afterwards, we show how these processes can be used to decompose the processes of type $(\cg \varphi, M_t\cd)_{t \in [0,1]}$ as a sum of a continuous martingale, a pure jump process and a continuous finite variation process. This is the purpose of the main result of this section, namely Theorem~\ref{thm:branching_Ito}, which can be seen as an other generalization of the Itô formula where instead of studying processes that are smooth functions of real or vector-valued continuous semi-martingales, we study processes of type $ (\cg \varphi, M_t \cd)_{t \in [0,1]}$, under the law of the BBM. 
	
	In Section~\ref{sec:construction_modified_BBM}, we build some modified BBMs. More precisely, we show that if we give ourselves an initial law $P_0\ll R_0$, a predictable vector field $\tilde v$ and a predictable measure-valued field $\tilde{\boldsymbol q}$, and under nothing but boundedness assumptions (no regularity is needed), there exists a BBM with drift $\tilde v$, and branching mechanism $\tilde{\boldsymbol q}$. Moreover, we give formulas for the Radon-Nikodym derivative and for the relative entropy of its law with respect to $R$.
	
	Finally, in Section~\ref{sec:girsanov}, we follow the strategy of a proof by Léonard~\cite{leonard2012girsanov} based on the Riesz representation theorem to show that for each law $P$ with finite entropy with respect to $R$, there are predictable $\tilde v$ and $\tilde{\boldsymbol q}$ such that $P$ is a BBM with drift $\tilde v$ and branching mechanism $\tilde{\boldsymbol q}$.
	
	In the whole section, the diffusivity $\nu>0$ is fixed once for all, and we do not systematically refer to it in the definitions. For instance ``Brownian motion'' will stand for ``Brownian motion of diffusivity $\nu$''.

	\section{Preliminaries on discontinuous stochastic calculus}
	\label{sec:discontinuous_stochastic_calculus} 
	
	Let us present a few notions of discontinuous stochastic calculus. We do not pretend to introduce this widely studied topic in full generality, and prefer to keep the exposition as simple as possible. As we will only introduce the notions that let us study the rather simple processes that we will meet later on, our formulation will differ significantly with reference books such as~\cite{jacod2013limit}, but all the results will be strictly simpler than what is written there. On the one hand, our goal is to provide to the reader who is not familiar with stochastic calculus all the definitions, results (without proofs) and heuristics that are necessary to understand the rest of this chapter. On the other hand, we also want to give to the reader who is more acquainted with stochastic calculus all the correspondences necessary to switch from our formulation to more classical (but more involved) ones.
	
	In this section and in this section only, $(\Omega, \F, (\F_t)_{t \in [0,1]}, \Prob)$ is a general filtered probability space. We first define what are predictable random objects

	\subsection{The notion of predictable random fields}
	We recall that a process $X = (X_t)_{t \in [0,1]}$ defined on the filtered space $(\Omega, \F, (\F_t)_{t \in [0,1]}, \Prob)$ is said to be adapted provided for all $t \in [0,1]$, $X_t$ is $\F_t$-measurable. In the definitions to come, such a process is seen as a function on the product space $\Omega \times [0,1]$.
	
	In the theory of discontinuous stochastic calculus, the notion of predictable processes play a significant role. It is defined as follows (see~\cite[Definition~I.2.1]{jacod2013limit}).
	\begin{Def}
		\label{def:predictable_process}
		\begin{itemize}
			\item\emph{Predictable $\sigma$-algebra}. The predictable $\sigma$-algebra is the $\sigma$-algebra $\G$ on $\Omega \times [0,1]$ generated by all left-continuous adapted processes on $(\Omega, \F, (\F_t)_{t \in [0,1]}, \Prob)$.
			\item \emph{Predictable processes}. Let $X = (X_t)_{t \in [0,1]}$ be a process with values in some measurable space $(\mathcal{Y}, \F_{\mathcal{Y}})$. We say that $X$ is a predictable process if it is measurable w.r.t. $\G$.  
		\end{itemize}
	\end{Def}
	
	\begin{Rem}
		Heuristically, with the classical interpretation that for some $t \in [0,1]$, $\F_t$ gathers all the events that occur before or at time $t$, a process $X$ is predictable provided its trajectory up to time $t$ can be deduced from all the events occurring strictly before time $t$. A predictable process is always adapted, but the converse does not hold.  
	\end{Rem}
	
	In our case, the parameters of the BBMs that we will consider will be processes with values in a set of functions on the torus (for instance, one realization of the drift at time $t$ will be a vector field). So we have to add a dependency to space in these definitions. Once again, we follow~\cite[Paragraph~II.1.4]{jacod2013limit}.
	
	\begin{Def}[Predictable fields]
		\label{def:predictable_fields}
		Let us call $\XX := \Omega \times [0,1]\times \T^d$, and let us endow this set with the $\sigma$-algebra $\tilde \G := \G \otimes \mathcal B(\T^d)$. A function $a$ on $\XX$ valued in some measurable set $(\mathcal Y, \mathcal F_{\mathcal Y})$ is called a predictable field whenever it is $\tilde \G$-measurable.
	\end{Def}
	\begin{Rem}
		As before, we will extensively use some abusive notations. If $a$ is a predictable field: $a(\omega,t,x)$ will denote its value at $(\omega, t, x) \in \XX$; $a(t,x)$ will denote, given $(t,x) \in [0,1]\times \T^d$, the random variable $ a(\cdot,t,x)$; and if $M = M(\omega)$ is a measure-valued random variable and $t \in [0,1]$, $\cg a(t),M\cd$ will denote the random variable consisting for a fixed $\omega$ in integrating the function $a(\omega, t,\cdot)$ w.r.t. $M(\omega)$.
	\end{Rem} 
	
	The following proposition will be useful in the sequel. It is an easy consequence of~\cite[Theorem~I.2.2]{jacod2013limit}, so that we omit the proof.
	\begin{Prop}
		\label{prop:predictable_sigma_field}
		If $\F$ is generated by a countable family of sets, then so is $\tilde \G$ defined in Definition~\ref{def:predictable_fields}. 
	\end{Prop}
	
	\subsection{Semi-martingales with simple jumps}
	
	The class of processes $X = (X_t)_{t \in [0,1]}$ that we want to study are real-valued processes of the form:
	\begin{equation}
	\label{eq:semi-martingale_with_jumps}
	X_t = X_0 + Z_t + J_t + V_t, \quad t \in [0,1],
	\end{equation}
	where $X_0$ is $\F_0$-measurable, $Z=(Z_t)_{t \in [0,1]}$ is a continuous local martingale starting from~$0$, $J = (J_t)_{t \in [0,1]}$ is a pure-jump process starting from~$0$, and $V = (V_t)_{t \in [0,1]}$ is a continuous finite variation process starting from $0$.
	
	We will use precisely the same definitions of continuous local martingales and of continuous finite variation processes as in~\cite[Chapter~4]{legall2016brownian}, and hence we refer to it for a precise introduction to these notions. In this work, if $X,Y$ are two continuous martingales, we will denote by $[X] = ([X]_t)_{t \in [0,1]}$ the quadratic variation of $X$, and by $[X,Y] = ([X,Y]_t)_{t \in [0,1]}$ the bracket of $X$ and $Y$ (sometimes called the quadratic covariation of $X$ and $Y$).
	
	However, the pure-jump processes that we will deal with are simpler than the ones appearing for instance in the general theory of Levy processes~\cite{bertoin1996levy}. We will call them simple pure-jump processes, and give a separate definition. 
	
	In this definition, the set $\M_+(\R)$ of nonnegative finite Borel measure on $\R$ will play a role. As recalled in Section~\ref{sec:notations}, it is seen as a measurable set endowed with the Borel $\sigma$-algebra corresponding to the topology of weak convergence.
	
	\begin{Def}[Simple pure-jump processes]
		\label{def:simple_pure-jump_processes}
		Let $J = (J_t)_{t \in [0,1]}$ be an adapted process. We call it a simple pure-jump process provided:
		\begin{itemize}
			\item It is \emph{a.s.}\ \emph{càdlàg}, piecewise constant, and with a finite number of jumps. 
			\item There is a predictable process $\eta = (\eta_t)_{t\in[0,1]}$ with values in $\M_+(\R)$ such that:
			\begin{itemize}
				\item Almost surely,
				\begin{equation}
				\label{eq:integrability_jump_measure}
				\int_0^1 \eta_t(\R) \D t < +\infty.
				\end{equation}
				\item For all function $f \in C_b(\R)$, the real-valued process
				\begin{equation}
				\label{eq:martingale_property_def_simple_jump_processes}
				f(J_t) - \int_0^t \hspace{-5pt} \int_\R \Big\{f(J_s + y) - f(J_s) \Big\}\D \eta_s(y) \D s, \quad t \in [0,1],
				\end{equation}
				is a local martingale.
			\end{itemize} 
		\end{itemize}
		In this case, we say that $\eta$ is the \emph{jump measure} of $J$.
	\end{Def}
	\begin{Rem}
		The heuristic meaning of the jump measure of a simple pure-jump process is as follows. For a fixed $s \in [0,1]$ and $A \in \mathcal{B}(\R)$, the random quantity $\eta_s(A) \D s$ represents the probability, knowing the whole past of the trajectory, that the process undergoes a jump whose size is in $A$ between the times $s$ and $s + \D s$.
	\end{Rem}
	It is easy to see that the jump measure of a simple pure-jump process is unique: this is because a continuous martingale with finite variation is constant. 
	
	Usually, when pure-jump processes are studied, they are not assumed to have a finite number of jumps, as the general theory of Lévy processes suggests~\cite{bertoin1996levy}. In that case, the definition of the jump measure includes a cutoff procedure, leading to slightly more complicated formulas. Here, we do not need to introduce any truncation.
	
	We are now ready to give the definition of the class of processes that we will study in the whole section, that we call semi-martingales with simple jumps.
	\begin{Def}
		A process $X$ is called a semi-martingale with simple jumps provided it has a decomposition of the form~\eqref{eq:semi-martingale_with_jumps}, where:
		\begin{itemize}
			\item $X_0$ is $\F_0$-measurable.
			\item $Z$ is a continuous local martingale with $Z_0=0$.
			\item $J$ is a simple pure jump process with $J_0=0$.
			\item $V$ is a continuous finite variation process with $V_0=0$.
		\end{itemize}
		If it exists, this decomposition is easily seen to be unique. In this case, the quadratic variation of $X$, denoted by $[X] = ([X_t])_{t \in [0,1]}$ is defined as the quadratic variation of $Z$, and the jump measure of $X$ is defined as the jump measure of $J$.
	\end{Def}
	
	\begin{Rem}[Correspondence with the canonical representation of semi-martingales]
		\label{rem:correspondence_semi-martingales}
		The interested reader can compare our decomposition~\eqref{eq:semi-martingale_with_jumps} with the one given in~\cite[Formula~II.2.35]{jacod2013limit}, whereas the reader who is not familiar with stochastic calculus can easily skip this remark. The main message is that our semi-martingales are quasi-left-continuous (more on this below) and only undergo a finite number of jumps. These two properties allow to simplify the general formula~\cite[Formula~II.2.35]{jacod2013limit}, and to write it as~\eqref{eq:semi-martingale_with_jumps}. Let us develop this idea.
		\begin{itemize}
			\item Our $Z$ is precisely their $X^c$, the continuous martingale part of $X$, defined in~\cite[Proposition~I.4.27]{jacod2013limit}.
			\item Their $\nu$ translates in our case to $\D t \otimes \eta$, as can be deduced from~\eqref{eq:martingale_property_def_simple_jump_processes} and~\cite[Theorem~II.1.8]{jacod2013limit}. They call it the compensator of the random measure
			\begin{equation*}
			\mu^X := \sum_{s\in[0,1]} \1_{\Delta X_s \neq 0} \delta_{(s,\Delta X_s)},
			\end{equation*}
			where $\Delta X_s := X_s - X_{s-}$, see \cite[Definition~II.2.6]{jacod2013limit}. In virtue of~\cite[Proposition~II.1.17]{jacod2013limit}, the fact that in our case, its temporal marginal can be chosen to be the Lebesgue measure reflects the fact that we will only consider processes whose jump instants are \emph{totally inaccessible} in the sense of~\cite[Definition~1.2.20]{jacod2013limit}, so that $X$ is \emph{quasi-left-continuous} (see \cite[Definition~I.2.25, Proposition~I.2.26]{jacod2013limit}). For the same reason, in our case, their $B$ is continuous.
			\item In our case, their integrability property~\cite[Formula~II.2.13]{jacod2013limit} becomes the much stronger one~\eqref{eq:integrability_jump_measure}. Namely, in general and contrarily with our case, we cannot assume $\int_0^t \eta_s([-\eps, \eps]) \D s$ to be finite, for any $t,\eps>0$. This is because there exist semi-martingales that undergo an infinite number of small jumps during finite time intervals, whereas the semi-martingales that we will study only allow a finite number of jumps. In that case, the processes that they call $h * \mu^X$, $h*\nu$ and $x * \mu^X$ are all well defined, so that in their formula, we can simplify $h * \mu^X$. Then, their $x * \mu^X$ is our $J$, and their $B - h * \nu$ (which can be seen to be independent of $h$) is our $V$. It is of course continuous, and so automatically predictable.
			\item Finally, notice that in our real-valued case, their $C$ is simply a real-valued process. In view of~\cite[Definition~II.2.6]{jacod2013limit}, it is precisely what we call $[X]$.
			
		\end{itemize}
	\end{Rem}
	
	The main result of this subsection is that if $X$ is a semi-martingale with simple jumps, and if $F:\R \to \R$ is smooth, then there exists a continuous finite variation process, that can be expressed in terms of the continuous finite variation component, the bracket and the jump measure of $X$, such that $F(X)$ is this process plus a local martingale. Applying this result to $F = \exp$, we discover that the knowledge of the continuous finite variation component of $X$, of its quadratic variation and of its jump measure is enough to build exponential martingales.
	\begin{Thm}[Martingales associated with a semi-martingale with simple jumps]
		\label{thm:Ito_with_jumps}
		Let $X=(X_t)_{t \in [0,1]}$ be a semi-martingale with simple jumps. Let us call $V = (V_t)_{t \in [0,1]}$ its continuous finite variation component, $[X]= ([X]_t)_{t \in [0,1]}$ its quadratic variation and $\eta = (\eta_t)_{t \in [0,1]}$ its jump measure.
		
		\begin{enumerate}
			\item Let $F:\R \to \R$ be smooth and assume that almost surely, for all $K>0$, 
			\begin{equation}
			\label{eq:assumption_F(X)}
			\int_0^1 \sup_{|x| \leq K} \left\{ \int_{\R}|F(x + y) - F(x)|\D \eta_t(y) \right\} \D t < + \infty.
			\end{equation}
			Then the process whose value at time $t \in [0,1]$ is
			\begin{equation*}
			F(X_t) - F(X_0) - \int_0^t F'(X_s)\D V_s - \frac{1}{2} \int_0^t F''(X_s) \mathrm{d} [X]_s - \int_0^t \hspace{-5pt} \int_{\R}\Big\{ F(X_s + y) - F(X_s)\Big\} \D \eta_s(y) \D s
			\end{equation*}
			is well defined, and it is a local martingale.
			\item Assume that almost surely,
			\begin{equation*}
			\int_0^1 \hspace{-5pt} \int_{\R}\Big\{ \exp(y) - 1 \Big\} \D \eta_s(y) \D s < + \infty.
			\end{equation*}
			(Observe that this quantity is always defined in $(-\infty, +\infty]$ because of~\eqref{eq:integrability_jump_measure}.)
			Then, the process 
			\begin{equation}
			\label{eq:exponential_martingale}
			\exp\left( X_t - X_0 - V_t- \frac{1}{2} [X]_t - \int_0^t\hspace{-5pt}\int_{\R}\Big\{ \exp(y) - 1 \Big\}\D \eta_s(y) \D s \right), \qquad t \in [0,1],
			\end{equation}
			is a nonnegative local martingale. In particular, it is a supermartingale.
		\end{enumerate}
	\end{Thm}
	\begin{proof}
		Once again, the reader who is not familiar with stochastic calculus can skip this proof. The first point of this theorem is a direct consequence of~\cite[Theorem~II.2.42]{jacod2013limit}, Remark~\ref{rem:correspondence_semi-martingales}, \cite[Theorem~II.1.8]{jacod2013limit}, and a localization procedure. Notice that in all the integrands, we take the value at time $s$, and not at the left limit $s-$. As all the processes involved are $\D s$-\emph{a.e.}\ continuous, this does not change the value of the integrals. The second point is a consequence of the first one and of~\cite[Lemma~2.1]{stroock1971diffusion}.
	\end{proof}
	
	We will also use the following elementary lemma, that can \emph{e.g.}\ be deduced from~\cite[Theorem~I.4.57]{jacod2013limit}.
	\begin{Lem}
		\label{lem:mg_times_finite_var}
		Le $X = (X_t)_{t \in [0,1]}$ be a semi-martingale with simple jumps that is in addition a positive local martingale, and let $V = (V_t)_{t \in [0,1]}$ be a continuous process with finite variation. Then the process
		\begin{equation*}
		X_t V_t, \quad t \in [0,1],
		\end{equation*}
		is a local martingale if and only if $V$ is \emph{a.s.}\ constant.
	\end{Lem}

	\section{Two types of processes defined upon the branching Brownian motion}
	\label{sec:new_processes}

	As already said at the beginning of this chapter, in order to find generalizations of formula~\eqref{eq:RN_derivative_brownian_case} in the case where $R \sim\BBM(\nu, \boldsymbol{q}, R_0)$, we need to introduce two types of processes. The first one will be a continuous martingale that is a direct generalization of the stochastic integral with respect to the Brownian motion: it will consist in adding all the stochastic integral on all the branches of the BBM. The second one will be a simple pure-jump process whose jump instants will be the branching times, and whose jump sizes will depend on the time of the branching event, the position of the branching event, the number of descendants of the branching events, and the past of the whole process. For both of these processes, we provide here a definition and the properties that we will use in the rest of the chapter, and we postpone the details of the construction to Appendix~\ref{app:stochastic_calculus}.
	
	Here, we go back to $(\Omega, \F, (\F_t)_{t \in [0,1]})$ being the canonical space for the BBM, as introduced in Subsection~\ref{subsec:def_BBM}. In particular, $\Omega = \cadlag([0,1]; \M_+(\T^d))$. In all the definitions to come, $R\sim\BBM(\nu, \boldsymbol{q}, R_0)$ is given once for all, and we do not refer to it explicitly in the notations. 
	
	Before starting, let us give a definition that we will use constantly in the remaining sections of this chapter.
	
	\begin{Def}
		\label{def:Pi_Q}
		Let $Q \in \P(\Omega)$. We denote by $\Pi_Q$ the measure $Q \otimes \D t \otimes M_t$ on the measurable set $(\XX, \tilde \G)$ defined in Definition~\ref{def:predictable_fields}.
	\end{Def}
	
	\subsection{A generalized stochastic integral}
	\label{subsec:def_I}
	
	The stochastic integral with respect to the Brownian motion associates to every square integrable predictable process a square integrable martingale. Our generalized stochastic integral will associate to each predictable \emph{vector field} in $L^2(\XX, \tilde \G, \Pi_R)$ a martingale that is square integrable with respect to $R$. If $v$ is such a vector field, for almost all time and $R$-\emph{a.e.}\ on $\Omega$, $v(t,x)$ is well defined for $M_t$-almost all $x$. Otherwise stated, such a predictable vector fields provides in a predictable way a vector to each particle alive at time $t$.
	
	Therefore, the predictable vector field $v$ stores the same information as if we had one predictable vector-valued process \emph{per branch} of the BBM. Our generalized stochastic integral will consist in adding over all branches of the BBM the stochastic integrals of this process w.r.t.\ the trajectory of the corresponding branch. This is done in the following theorem defining our stochastic integral, and proved in Appendix~\ref{app:stochastic_calculus}. In this theorem, we  use the notations that we introduced in Subsection~\ref{subsec:def_BBM}. We recall that for for two real numbers $a,b\in\R$, $a \wedge b$ stands for the minimum between $a$ and $b$.
	
	\begin{Thm}
		\label{thm:definition_stochastic_integral}
		Let $v$ be a vector field in $L^2(\XX,\tilde \G, \Pi_{R})$. There exists a continuous adapted process $I[v]=(I[v]_t)_{t \in [0,1]}$ on $(\Omega, \F, (\F_t)_{t \in [0,1]})$, well defined $R$-\emph{a.s.}, such that for $R_0$-almost all $\mu$, and for all $\xx= (x_1, \dots, x_p) \in \cup_{k \in \N}(\T^d)^k$ such that $ \delta_{x_1}+\dots + \delta_{x_p} = \mu$, we have $\Prob^\xx$-\emph{a.s.} for all $t \in [0,1]$,
		\begin{equation}
		\label{eq:def_stochastic_integral}
		I[v]_t(M^0) = \sum_{n \in \mathcal{U}} \int_0^t \1_{b_n < \tau \leq d_n} v(M^0,s, X_n(s)) \D X_n(s), \qquad t \in [0,1].
		\end{equation}
		Moreover, $I[v]$ satisfies the following properties.
		\begin{itemize}
			\item $I[v]$ is a square integrable continuous martingale w.r.t.\ $R$.
			\item Its quadratic variation is
			\begin{equation*}
			\big[ I[v] \big]_t = \nu \int_0^t \left\cg |v(s)|^2, M_s \right\cd \D s.
			\end{equation*}
			\item For all $v \in L^2(\XX, \tilde \G, \Pi_R)$, for all $0\leq t_0 < t_1 \leq 1$ and for all bounded real-valued random variable $Z_{t_0}$ that is $\F_{t_0}$-measurable, we have for all $t \in [0,1]$,
			\begin{equation}
			\label{eq:truncation_stochastic_integral}
			I[Z_{t_0} \1_{(t_0,t_1]} v]_t = Z_{t_0} \Big( I[v]_{t_1 \wedge t} - I[v]_{t_0 \wedge t} \Big).
			\end{equation}
		\end{itemize}
	\end{Thm}
	\begin{Rem}
		\begin{itemize}
			\item As we will see in Appendix~\ref{app:stochastic_calculus}, the main subtlety in this theorem is the fact that the r.h.s. in formula~\eqref{eq:def_stochastic_integral} is adapted to the filtration generated by $M^0$, so that it is possible to define $I[v]$ on the canonical space $\Omega$.
			\item Here, we choose $v$ to be predictable to be coherent with the rest of the chapter, but as we are dealing with continuous processes, we could have chosen it to be only progressive in the sense of~\cite[Definition~3.3]{legall2016brownian}. However, we could imagine some generalization allowing particles to jump. Doing so, predictability would be crucial.
		\end{itemize}
	\end{Rem}

	\subsection{A class of simple pure-jump processes}
	\label{subsec:pure_jump}
	As already said, the second kind of processes that we want to introduce are simple pure-jump processes whose jump instants correspond to branching events. Given a predictable scalar fields $a$, we introduce the piecewise constant process which undergoes a jump of size $a(t,x)$ whenever there is a branching event with $k$ descendants at time $t$ and position $x$.
	
	\begin{Def}
		\label{def:jump_processes}
		Let $k \in \N \backslash \{1\}$ and $a$ be a predictable scalar field. For all $t \in [0,1]$, we call
		\begin{equation*}
		J^k[a]_t :=  \sum_{s \leq t} \1_{\Delta M_s(\T^d) = k-1} \left\cg a(s), \frac{\Delta M_s}{k-1} \right\cd,
		\end{equation*}
		where $\Delta M_s = M_s - M_{s-}$.
	\end{Def}
	\begin{Rem}
		Because of the definition of the BBM given in Subsection~\ref{subsec:def_BBM}, we see that \emph{a.s.}, for all $k \neq 1$ and all $s\in[0,1]$, $\Delta M_s(\T^d) = k-1$ if and only if there is a branching event with $k$ descendants at time $s$. In that case there is an $x \in \T^d$ such that $\Delta M_s = (k-1) \delta_x$, and the corresponding term in the definition above is simply $a(t,x)$.
		
		Notice that we excluded the case $k=1$ for the same reason as when we chose $q_1$ to be $0$ in the Definition~\ref{def:branching_mechanism} of a branching mechanism. 
	\end{Rem}
	
	The following theorem gives the main properties of the processes $J^k$, $k \in \N\backslash \{1\}$, and is proved in Appendix~\ref{app:stochastic_calculus}.

	\begin{Thm}
		\label{thm:properties_jump_processes}
		
		Let $(a_k)_{k \in \N\backslash{1}}$ be a family of predictable scalar fields We call $J := \sum_{k \neq 1} J^k[a_k]$. The following properties hold:
		\begin{itemize}
			\item The process $J$
			is a simple pure-jump process whose jump measure is:
			\begin{equation}
			\label{eq:jump_measure_J}
			\eta_t := \sum_{k\neq 1} a_k(t) \pf \Big(q_k M_{t-}\Big), \qquad t \in [0,1],
			\end{equation}
			where $\pf$ denotes the push-forward operation (see Remark~\ref{rem:push_forward} below).
			\item Given a function $f \in C_b(\R)$, two times $0 \leq t_0 < t_1 \leq 1$ and a bounded real-valued random variable $Z_{t_0}$ that is $\F_{t_0}$-measurable, we have for all $t \in [0,1]$,
			\begin{equation}
			\label{eq:truncation_jump_process}
			\sum_{k \neq 1} J^k[b^k]_t = Z_{t_0} \Big( f(J_{t \wedge t_1}) - f(J_{t \wedge t_0}) \Big) ,
			\end{equation}
			where for all $k \in \N\backslash\{1\}$ and $t \in [0,1]$, $b^k(t):= Z_{t_0}\1_{(t_0,t_1]}(t) (f(J_{t-} + a^k(t)) - f(J_{t-}))$. (In particular, $b^k$ is easily seen to be predictable.) 
		\end{itemize} 
	\end{Thm}
	\begin{Rem}
		\label{rem:push_forward}
		By the standard definition of the push-forward operation (see~\cite[Section~1.1]{santambrogio2015optimal}), the measure $\eta_t$ from formula~\eqref{eq:jump_measure_J} satisfies for all bounded measurable test function $\varphi$ on $\R$
		\begin{equation}
		\label{eq:formula_pf}
		\int_{\R}\varphi \D \eta_t = \sum_{k \neq 1} q_k \big\cg \varphi\big( a_k(t) \big) , M_{t-} \big\cd.
		\end{equation}
	\end{Rem}
	
	\subsection{An extended Itô formula}
	\label{subsec:extended_Ito}
	The two types of processes defined in the previous subsections let us derive a generalized Itô formula in the branching case. What we show is that for all smooth function $\varphi$ on $[0,1]\times\T^d$, the process $(\cg\varphi(t), M_t\cd)_{t \in [0,1]}$ is a semi-martingale with simple jumps, and we can express its continuous martingale component and its jump component in terms of processes of type $I$ and $J^k$, $k \in \N\backslash\{1\}$. We provide the proof of this theorem in Appendix~\ref{app:stochastic_calculus}. As a first application of this formula, we are finally ready to prove Proposition~\ref{prop:loc_exp_mart_BBM}.
	
	\begin{Thm}
		\label{thm:branching_Ito}
		Let $\varphi:[0,1]\times\T^d \to \R$ be a smooth function. We have $R$-\emph{a.s.}\ for all $t \in [0,1]$:
		\begin{equation}
		\label{eq:branching_Ito}
		\cg \varphi(t), M_t \cd = \cg \varphi(0), M_0\cd + I[\nabla \varphi]_t + \sum_{k\neq 1} J^k[(k-1)\varphi]_t +  \int_0^t \Big\cg \partial_t \varphi(s) + \frac{\nu}{2} \Delta \varphi(s), M_s \Big\cd \D s.
		\end{equation}
	\end{Thm}
	
	\begin{Rem}
		For those who are already familiar with the standard Itô formula, this generalized Itô formula is easily understandable.

		The continuous martingale part of $(\cg\varphi(t), M_t\cd)_{t \in [0,1]}$ is precisely $I[\nabla \varphi]$, which evolves at time $t$ as the sum of all the stochastic integrals of $\nabla \varphi$ along the trajectories of all the particles alive at time $t$.
		
		If a branching event with $k$ descendants occurs at time $t$ and position $x$, then the process undergoes a jump of size $(k-1)\varphi(t,x)$ at time $t$, so that the jump part of  $(\cg\varphi(t), M_t\cd)_{t \in [0,1]}$ is the sum over $k$ of $J^k[(k-1)\varphi]$.
		
		Finally, as in the non-branching case, every particle alive at time $t$, say at position $x$, contributes to the derivative of the absolutely continuous part of $(\cg\varphi(t), M_t\cd)_{t \in [0,1]}$ as $\partial_t \varphi(t,x) + \frac{\nu}{2}\Delta \varphi(t,x)$.
	\end{Rem}
	
	\begin{proof}[Application of this formula: proof of Proposition~\ref{prop:loc_exp_mart_BBM}]
		Let $\psi$ be as in the statement of Proposition~\ref{prop:loc_exp_mart_BBM}. In view of the definition~\eqref{eq:operator_L} of $\L_{\nu, \boldsymbol q}$ and formula~\eqref{eq:branching_Ito}, we have $R$-\emph{a.s.} for all $t \in [0,1]$:
		\begin{align*}
		\frac{\cg \psi(t), M_t\cd}{\nu} &- \frac{\cg \psi(0), M_0\cd}{\nu}  - \frac{1}{\nu}\int_0^t \big\cg \partial_t \psi(s) + \mathcal L_{\nu, \boldsymbol q}[\psi(s)], M_s \big\cd \D s \\
		&\qquad= I\left[  \frac{\nabla \psi}{\nu}\right]_t + \sum_{k\neq 1}J^k\left[ \frac{(k-1)\psi}{\nu} \right]_t - \frac{1}{\nu} \int_0^t \left\cg \frac{1}{2} |\nabla \psi(s)|^2 + \Psi^*_{\nu, \boldsymbol q}\big( \psi(s) \big) , M_s \right\cd \D s.
		\end{align*}
		Calling
		\begin{equation*}
		X_t :=  I\left[  \frac{\nabla \psi}{\nu}\right]_t + \sum_{k\neq 1}J^k\left[ \frac{(k-1)\psi}{\nu} \right]_t, \qquad t \in [0,1],
		\end{equation*}
		we see that $X$ is a semi-martingale with simple jumps with $X_0 = 0$, with no continuous finite-variation component. According to Theorem~\ref{thm:definition_stochastic_integral}, its bracket is
		\begin{equation*}
		[X]_t = \frac{1}{\nu}\int_0^t \left\cg |\nabla \psi(s)|^2 , M_s \right\cd \D s, \qquad t \in [0,1],
		\end{equation*}
		and according to Theorem~\ref{thm:properties_jump_processes}, its jump measure is
		\begin{equation*}
		\eta_t := \sum_{k \neq 1}\left((k-1) \frac{ \psi(t)}{\nu} \right) \pf \Big( q_k M_{t-}\Big),
		\end{equation*}
		so that with the help of~\eqref{eq:def_Psi*} and our boundedness assumption on $\Psi^*_{\nu, \boldsymbol q}(\psi)$, for all $t\in[0,1]$,
		\begin{align*}
		\int_0^t \hspace{-5pt} \int_{\R}\Big\{ \exp(y) - 1 \Big\} \D \eta_s(y) \D s &= \int_0^t \left\cg\sum_{k\neq 1} q_k \left\{\exp\left( (k-1)\frac{\psi(t)}{\nu} \right) - 1 \right\} , M_s \right\cd \D s \\
		&= \frac{1}{\nu} \int_0^t \Big\cg \Psi^*_{\nu, \boldsymbol q}\big( \psi(s) \big), M_s \Big\cd \D s < + \infty,\qquad R\mbox{-\emph{a.s.}}
		\end{align*}
		(where we use that for any curve in the canonical space $(s\mapsto \mu_s) \in \Omega = \cadlag([0,1]; \M_+(\T^d))$, we have $\int_0^1 \mu_s(\T^d)\D s < + \infty$). Therefore, the result is a direct application of Theorem~\ref{thm:Ito_with_jumps}.
	\end{proof}

	\section{Branching Brownian motions with predictable and space-time dependent parameters}
	\label{sec:construction_modified_BBM}
	
	In this section, we want to use the two processes introduced at Section~\ref{sec:new_processes} to give an explicit formula for the Radon-Nikodym derivative w.r.t.\ $R$ of the law of the BBM where particles undergo a drift, and where the branching mechanism is changed.
	
	Therefore, in Subsection~\ref{subsec:def_modified_BBM} we define what we mean by BBMs with drift and branching mechanism that are space-time dependent and predictable. By analogy with the characterization of the standard BBM given in Proposition~\ref{prop:martingale_problem}, we will give a definition in terms of a martingale problem.
	
	Then, in Subsection~\ref{subsec:construction_modified_BBM}, we provide our generalization of formula~\eqref{eq:RN_derivative_brownian_case}, and we show how it provides easily a wide class of generalized BBMs.
	
	\subsection{Definition of modified branching Brownian motions}
	
	\label{subsec:def_modified_BBM}
	Let us provide a definition for the modified BBMs that we will consider. In the formula, one needs to recall the definition~\eqref{eq:operator_L} of the operator $\L_{\nu, \boldsymbol q}$, \eqref{eq:def_Psi*} of the function $\Psi^*_{\nu, \boldsymbol q}$ and~\eqref{eq:def_lambda_p} of the branching rate $\lambda_{\boldsymbol q}$, associated to a diffusivity $\nu>0$ and a branching mechanism $\boldsymbol q$. Notice that when the deterministic and homogeneous $\boldsymbol q$ is replaced by a predictable field $\tilde{\boldsymbol q}$, these homogeneous and deterministic objects become predictable fields.
	\begin{Def}
		\label{def:modified_BBM}
		A probability measure $P \in \P(\Omega)$ is called a BBM with predictable and space-time dependent parameters if there exist a predictable vector field $\tilde v$ and a predictable field valued in the set of branching mechanisms $\tilde{\boldsymbol q}$ such that:
		\begin{itemize}[leftmargin=10pt]
			\item $P$-\emph{a.s},
			\begin{equation}
			\label{eq:integrability_condition_def}
			\int_0^1 \left\cg |\tilde v(t)| , M_t \right\cd \D t < + \infty \qquad \mbox{and} \qquad \int_0^1 \left\cg  \lambda_{\tilde{\boldsymbol q}(t)} , M_t \right\cd \D t < +\infty.
			\end{equation}
			\item For all smooth nonpositive $\psi:[0,1]\times \T^d \to \R_-$, the process whose value at time $t \in [0,1]$ is
			\begin{equation}
			\label{eq:def_modified_BBM}
			\exp\left( \frac{\cg \psi(t), M_t\cd}{\nu} -\frac{\cg \psi(0), M_0\cd }{\nu} - \frac{1}{\nu}\int_0^t \big\cg \partial_t \psi(s) + \mathcal L_{\nu, \tilde{\boldsymbol q}(s)}[\psi(s)] + \tilde v(s) \cdot \nabla \psi(s), M_s \big\cd \D s \right)
			\end{equation}
			is a local martingale under $P$.
		\end{itemize}
		If these two conditions hold, we simply say that $P$ is a BBM of drift $\tilde v$ and branching mechanism $\tilde{\boldsymbol q}$.
	\end{Def}
	\begin{Rem}
		\begin{itemize}
			\item Adding a drift results in adding the term $\tilde v \cdot \nabla \psi$ in the formulation of the martingale problem, and prescribing a space-time dependent and predictable branching mechanism results in replacing the deterministic and homogeneous operator $\mathcal L_{\nu, \boldsymbol q}$ by the predictable and space-time dependent one $\mathcal L_{\nu, \tilde{\boldsymbol q}}$ (where only the non-linear zero order term $\psi \mapsto \Psi^*_{\nu, \boldsymbol q}(\psi)$ changes).
			\item Notice the several abusive notations in formula~\eqref{eq:def_modified_BBM}: by $\mathcal L_{\nu, \tilde{\boldsymbol q}(s)}[\psi(s)]$, we mean, correspondingly with formula~\eqref{eq:operator_L}, the function (more precisely the random function, as $\tilde{\boldsymbol q}(s,x)$ is random):
			\begin{equation}
			\label{eq:def_random_operator}
			\mathcal L_{\nu, \tilde{\boldsymbol q}(s)}[\psi(s)]:x \longmapsto  \frac{1}{2} |\nabla \psi(s,x)|^2 + \frac{\nu}{2} \Delta \psi(s,x) + \Psi^*_{\nu,\tilde{\boldsymbol q}(s,x)}(\psi(s,x)).
			\end{equation}
			\item The first point in the definition is there to ensure that the process given in the second point is well defined (notice that for nonpositive $\psi$, for all $k \in \N$, $e^{(k-1)\psi} \leq e^{\| \psi\|_{\infty}}$).
			\item When $\tilde v$, $\tilde{\boldsymbol q}$ are deterministic and smooth, this definition is coherent with a hand by hand construction performed by adapting the one done in Subsection~\ref{subsec:def_BBM} in the case of $\BBM(\nu, \boldsymbol{q}, R_0)$.
			\item The predictable fields $\tilde v$ and $\tilde{\boldsymbol q}$ being given \emph{a priori}, we do not want to address the question of uniqueness of a corresponding BBM. So we will not talk about ``the'' BBM corresponding to these parameters, but rather about ``a'' BBM, just as we would talk about ``a'' solution of an SDE in case of non-uniqueness. We can still notice that by the same arguments as in Subsection~\ref{subsec:def_BBM}, this uniqueness holds at least when $\tilde v$ and $\tilde{\boldsymbol q}$ are deterministic and smooth.
		\end{itemize}
	\end{Rem}
	\subsection{Construction of modified branching Brownian motions}
	\label{subsec:construction_modified_BBM}
	In the next theorem, we build BBMs with drift $\tilde v$ and branching mechanism $\tilde{\boldsymbol q}$ that are predictable fields, only working under boundedness assumptions. We do it by providing the Radon-Nikodym derivatives of such processes w.r.t. the more classical BBM built in Subsection~\ref{subsec:def_BBM}. Notice that in this theorem, the exponential bound~\eqref{eq:exponential_moment_q} for $\boldsymbol q$ plays a role. Also recall the definition~\eqref{eq:def_h} of $h$.
	
	\begin{Thm}
		\label{thm:formula_RN_derivative}
		Let $R \sim \BBM(\nu,\boldsymbol q,R_0)$, and assume that $\boldsymbol q$ satisfies~\eqref{eq:exponential_moment_q}. Let us give ourselves:
		\begin{itemize}
			\item A probability measure $P_0 \in \P(\M_+(\T^d))$ that is absolutely continuous w.r.t.\ $R_0$.
			\item A bounded predictable vector field $\tilde v \in L^\infty(\XX, \tilde \G, \Pi_R)$.
			\item A predictable field of branching mechanisms $\tilde{\boldsymbol q}$ such that $h(\tilde{\boldsymbol{q}} | \boldsymbol{q}) \in L^\infty(\XX, \tilde \G, \Pi_R)$.
		\end{itemize}
		The measure $P$ given by the formula
		\begin{equation}
		\label{eq:formula_RN_derivative}
		P := \frac{\D P_0}{\D R_0}(M_0) \exp\left( \frac{1}{\nu} I[\tilde v]_1 + \sum_{k \neq 1} J^k\left[\log \frac{\tilde q_k}{q_k}\right]_1 - \int_0^1 \left\cg \frac{|\tilde v(t)|^2}{2\nu} +  \lambda_{\tilde{\boldsymbol q}(t)} - \lambda_{\boldsymbol q}, M_t \right\cd \D t  \right) \cdot R
		\end{equation}
		(with convention $\log 0/0=0$) is BBM starting from $P_0$ with drift $\tilde v$ and branching mechanism $\tilde{\boldsymbol q}$. 
		
		Moreover, the following entropy bound holds:
		\begin{equation}
		\label{eq:entropy_smaller_than_RUOT}
		\nu H(P|R) \leq \nu H(P_0 | R_0) + \E_P\left[ \int_0^1 \left\cg \frac{|\tilde v(t)|^2}{2} + \nu h(\tilde{\boldsymbol q}(t) | \boldsymbol q) , M_t \right\cd \D t  \right].
		\end{equation}
	\end{Thm}
	\begin{Rem}
		\label{rem:construction_modified_BBM}
		\begin{itemize}
			\item Formula~\eqref{eq:formula_RN_derivative} could have been guessed using the known formulas in the Brownian and pure-jump case respectively. Indeed, if $R$ was the law of a fixed number of independent Brownian particles, that is, if there was no branching event, a slight generalization of~\cite[Theorem~2.3]{leonard2012girsanov} shows that the law $P$ defined by
			\begin{equation*}
			P  := \frac{\D P_0}{\D R_0}(M_0) \exp\left( \frac{1}{\nu} I[\tilde v]_1 - \int_0^1 \left\cg \frac{|\tilde v(t)|^2}{2\nu} , M_t \right\cd \D t  \right) \cdot R
			\end{equation*}
			(with straightforward adaptations of our notations) would be the law starting from $P_0$ of the same number of particles following independent Brownian trajectories plus the additional drift $\tilde v$. On the other hand, if we were in the case of a pure branching process with no spatial trajectories, that is, if we were only interested in the number of particles, then the measure valued process $(M_t)_{t \in [0,1]}$ would be replaced by the integer-valued process $(N_t)_{t \in [0,1]}$ corresponding to this number of particles. In that case, $R \in \P(\cadlag([0,1]; \N))$ would be the law of a pure-jump process, and a direct application of~\cite[Theorem~2.9]{leonard2012girsanov} would show that defining the law $P$ by
			\begin{equation*}
			P  := \frac{\D P_0}{\D R_0}(N_0) \exp\left( \sum_{k \neq 1}  J^k \left[\log \frac{\tilde q_k}{q_k}\right]_1 - \int_0^1 \Big\{ \lambda_{\tilde{\boldsymbol q}(t)} - \lambda_{\boldsymbol q} \Big\} N_t \D t  \right) \cdot R,
			\end{equation*}
			(once again with straightforward adaptations of our notations) then $P$ would be the law starting from $P_0$ of the process where each of the $N_t$ particles alive at time $t$ is replaced by $k \in \N \backslash\{1\}$ particles before time $t + \D t$ with probability $\tilde q_k(t) \D t$. In both cases, formulas similar to~\eqref{eq:entropy_smaller_than_RUOT} would follow from results stated in~\cite{leonard2012girsanov} as well.
			\item Later in Theorem~\ref{thm:representation_competitor_SchBr}, we will prove a converse inequality, so that inequality~\eqref{eq:entropy_smaller_than_RUOT} is in fact an equality. We are more precise about this in the third point of Remark~\ref{rem:representation_competitor_BrSch}.
			\item If we give ourselves predictable fields, and if we want to built a corresponding BBM $P$, we need at least to be sure that the first point in Definition~\ref{def:modified_BBM} holds under $P$, which can be delicate. Here, we avoid this question by choosing $\tilde v$ and $h(\tilde{\boldsymbol q}|\boldsymbol q)$ in $L^\infty( \XX, \tilde \G, \Pi_R)$, so that this first point holds automatically for all $P \ll R$, as we will see in the proof below. In return, predictable fields with these boundedness properties are presumably not the most general ones for which it is possible to build a corresponding BBM.
		\end{itemize}
	\end{Rem}
	\begin{proof}
		\begin{stepc}{Outline of the proof}
			We will perform the proof in the restricted case where the initial configuration is deterministic, that is, with the notations of Subsection~\ref{subsec:def_BBM}, when $R := R^\mu$ for some $\mu \in \M_\delta(\T^d)$. In this case, $P_0 \ll R_0$ ensures $P_0 = \delta_\mu$. Hence, from Step~\ref{step:Z_martingale_existence} to Step~\ref{step:check_martingale_existence}, the initial factor $\D P_0 / \D R_0(M_0)$ in~\eqref{eq:formula_RN_derivative} equals 1 $R$-\emph{a.e}, and the initial term $\nu H(P_0 | R_0)$ in \eqref{eq:entropy_smaller_than_RUOT} cancels. We will relax this condition at the end of the proof. From now on and until then, we chose $\mu \in \M_\delta(\T^d)$, and we assume that $R = R^\mu$.
			
			With the notations of the statement of the theorem, we define the process $Z=(Z_t)_{t \in [0,1]}$ by
			\begin{equation}
			\label{eq:def_Z_proof_RN_derivative}
			Z_t := \exp\left( \frac{1}{\nu} I[\tilde v]_t + \sum_{k \neq 1} J^k\left[\log \frac{\tilde q_k}{q_k}\right]_t - \int_0^t \left\cg \frac{|\tilde v(s)|^2}{2\nu} + \lambda_{\tilde{\boldsymbol q}(s)} - \lambda_{\boldsymbol q}, M_s \right\cd \D s  \right), \quad t \in [0,1].
			\end{equation}
			
			Let us justify why it is well defined $R^\mu$-\emph{a.s.}\ for all $t \in [0,1]$. The only term that is not clearly finite under our assumptions is the one involving $\lambda_{\tilde{\boldsymbol q}}$. We will actually prove that $\lambda_{\tilde{\boldsymbol q}} \in L^\infty(\XX, \tilde \G, \Pi_{R^\mu})$. Notice that as a consequence, not only $Z$ is well defined, but also the bound~\eqref{eq:integrability_condition_def} holds automatically $R^\mu$-\emph{a.s.}, and so $P$-\emph{a.s.} as well for all $P \ll R^\mu$.
			
			To prove that $\lambda_{\tilde{\boldsymbol q}} \in L^\infty(\XX, \tilde \G, \Pi_R)$, observe that because of~\eqref{eq:convex_inequality_l} with $y = 1$,
			\begin{equation*}
			\frac{\tilde q_k}{q_k} \leq l \left( \frac{\tilde q_k}{q_k} \right) + e - 1.
			\end{equation*}
			Multiplying both sides by $q_k$ and summing up in $k$, we get
			\begin{equation}
			\label{eq:tilde_lambda_bounded}
			\lambda_{\tilde{\boldsymbol q}} \leq  h(\tilde{\boldsymbol q} | \boldsymbol q) + \lambda_{\boldsymbol q}(e-1),
			\end{equation}
			where the r.h.s.\ belongs to $L^\infty(\XX, \tilde \G, \Pi_{R^\mu})$.
			
			Therefore, $Z = (Z_t)_{t \in [0,1]}$ defined in~\eqref{eq:def_Z_proof_RN_derivative} is a well defined $\cadlag$ process, 	and with these new notations and assumptions on the initial condition, formula~\eqref{eq:formula_RN_derivative} rewrites $P = Z_1 \cdot R^\mu$. If we want this law to be a probability measure, it suffices to prove that $Z$ is a true nonnegative martingale. Indeed, in that case, as $Z_0 \equiv 1$, we would have $\Em_\mu[Z_1]= 1$.
			
			Showing that $Z$ is a true martingale is the next step of the proof. As we will see in a third step, the inequality~\eqref{eq:entropy_smaller_than_RUOT} in this restricted case will be a direct consequence of the analysis made in the second step. Finally, in a fourth step, we will show that $P$ defined by~\eqref{eq:formula_RN_derivative} is a modified BBM, as announced in the statement of the theorem. Finally, at Step~\ref{step:general_initial_law_existence}, we will get rid of the assumption on the initial condition.
		\end{stepc}
		
		\begin{stepc}{$Z$ is a martingale}
			\label{step:Z_martingale_existence}
			First, $Z$ is a nonnegative local martingale (and hence a supermartingale as well). This is a direct consequence of the second point of Theorem~\ref{thm:Ito_with_jumps}. Indeed, let us call for $t \in [0,1]$
			\begin{equation*}
			X_t := \frac{1}{\nu} I[\tilde v]_t + \sum_{k \in \N} J^k\left[\log \frac{\tilde q_k}{q_k}\right]_t.
			\end{equation*}
			The process $X$ is a semi-martingale with simple jumps, with no continuous finite variation component. By Theorem~\ref{thm:definition_stochastic_integral}, the bracket of $X$ is given by:
			\begin{equation*}
			[X]_t = \frac{1}{\nu} \int_0^t \left\cg |\tilde v(s)|^2, M_s \right \cd \D s.
			\end{equation*}
			By Theorem~\ref{thm:properties_jump_processes} its jump measure at time $t \in [0,1]$ is given by:
			\begin{equation}
			\label{eq:levy_jump_measure_X}
			\eta_t:=\sum_{k \neq 1}\log \frac{\tilde q_k(t)}{q_k} \pf \Big( q_k M_{t-} \Big),
			\end{equation}
			so that, by formula~\eqref{eq:formula_pf},
			\begin{equation*}
			\int_{\R}\Big\{ \exp(y) - 1 \Big\}\D \eta_t(y) = \sum_{k \neq 1} q_k \left\cg \frac{\tilde q_k(t)}{q_k} - 1 , M_{t-} \right \cd = \left\cg \lambda_{\tilde{\boldsymbol q}(t)} - \lambda_{\boldsymbol q} , M_{t-} \right\cd.
			\end{equation*}
			Hence, the fact that $Z$ is a local martingale follows directly from our boundedness assumptions and from formula~\eqref{eq:exponential_martingale}.
			
			To prove that actually $Z$ is a true martingale, we will show that $\Em_\mu[l(Z_1)]$ is finite, where $l$ is defined in~\eqref{eq:def_l}. Clearly, as a consequence of Jensen's inequality, this property is enough to justify that $Z$ is uniformly integrable. To prove our claim, it is useful to introduce the following process:
			\begin{equation*}
			Y_t := \log Z_t = \frac{1}{\nu} I[\tilde v]_t + \sum_{k \in \N} J^k\left[\log \frac{\tilde q_k}{q_k}\right]_t - \int_0^t \left\cg \frac{|\tilde v(s)|^2}{2\nu} +  \lambda_{\tilde{\boldsymbol q}(s)} - \lambda_{\boldsymbol q}, M_s \right\cd \D s, \quad t \in [0,1],
			\end{equation*}
			and to remark that for all $t \in [0,1]$, $l(Z_t) = 1 + (Y_t - 1)\exp(Y_t)  = F(Y_t)$, where for all $y \in \R$, $F(y) := 1 + (y-1)\exp(y) $. The process $Y$ is clearly a semi-martingale with simple jumps, whose decomposition is directly visible from its definition. Notice that its bracket and jump measure coincide with the ones of $X$.
			
			The function $F$ is smooth, with for all $y \in \R$, $F'(y) = y \exp(y)$ and $F''(y) = (y + 1)\exp(y)$. We want to apply to $F$ the first point of Theorem~\ref{thm:Ito_with_jumps}, so let us check condition~\eqref{eq:assumption_F(X)}. A quick computation using formula~\eqref{eq:formula_pf} shows that for a given $t \in [0,1]$ and $x \in \R$,
			\begin{equation*}
			\int_{\R}| F(x + y) - F(x) | \D \eta_t(y) = e^x \Big\cg\Big| h(\tilde{\boldsymbol q}(t)|\boldsymbol q) + x \Big( \lambda_{\tilde{\boldsymbol q}(t)} - \lambda_{\boldsymbol q}\Big) \Big|, M_{t-} \Big\cd.
			\end{equation*}
			Hence, condition~\eqref{eq:assumption_F(X)} follows easily from our boundedness assumptions and~\eqref{eq:tilde_lambda_bounded}. We deduce with the first point of Theorem~\ref{thm:branching_Ito} that the process whose value at time $t \in [0,1]$ is
			\begin{multline}
			\label{eq:application_Ito_with_jump_modified_BBM}
			F(Y_t) - F(Y_0) + \int_0^t F'(Y_s) \left\cg \frac{|\tilde v(s)|^2}{2\nu} + \lambda_{\tilde{\boldsymbol q}(s)} - \lambda_{\boldsymbol q}, M_s \right\cd \D s - \int_0^t F''(Y_s) \left\cg \frac{|\tilde v(s)|^2}{2\nu}, M_s \right\cd \D s\\
			- \int_0^t \hspace{-5pt} \int_{\R}\Big\{ F(Y_s + y) - F(Y_s) \Big\} \D \eta_s(y) \D s
			\end{multline}
			is a local martingale starting from $0$. Let us compute the last term using the definition~\eqref{eq:levy_jump_measure_X} of $\eta$ and formula~\eqref{eq:formula_pf}. We find
			\begin{align*}
			\int_0^t \hspace{-5pt} \int_{\R}\Big\{ F(Y_s + y) - F(Y_s) \Big\} \D \eta_s(y) \D s&= \int_0^t \exp(Y_s) \Big\cg h(\tilde{\boldsymbol q}(s)|\boldsymbol q) + Y_s \Big( \lambda_{\tilde{\boldsymbol q}(s)} - \lambda_{\boldsymbol q}\Big), M_{s} \Big\cd \D s \\
			&= \int_0^t \exp(Y_s) \Big\cg h(\tilde{\boldsymbol q}(s)|\boldsymbol q), M_{s} \Big\cd \D s + \int_0^t F'(Y_s)\Big\cg \lambda_{\tilde{\boldsymbol q}(s)} - \lambda_{\boldsymbol q}, M_{s} \Big\cd \D s.
			\end{align*}
			Plugging this identity into~\eqref{eq:application_Ito_with_jump_modified_BBM} and using $Z_s = \exp(Y_s)= F'(Y_s) - F''(Y_s)$, $l(Z_s)= F(Y_s)$ and $F(Y_0) = 0$, we discover that the process
			\begin{equation*}
			l(Z_t) - \int_0^t Z_s \left\cg \frac{|\tilde v(s)|^2}{2\nu} + h(\tilde{\boldsymbol q}(s) | \boldsymbol q ), M_s \right\cd  \D s, \qquad t \in [0,1],
			\end{equation*}
			is a local martingale starting from $0$.
			
			Let us define the nondecreasing sequence of stopping times $(S_n)_{n\in\N}$ by
			\begin{equation*}
			S_n := \inf\{ s \in [0,1] \ : \ M_s(\T^d) > n\},
			\end{equation*}
			with convention $\inf \emptyset = 1$. $R$-\emph{a.s}, $S_n \to 1$ as $n \to + \infty$, and hence there exists a nondecreasing sequence $(T_n)_{n \in \N}$ of stopping times reducing the previous local martingale, and with $R$-\emph{a.s.}\ for all $n$, $T_n \leq S_n$. 
			Under our boundedness assumptions on $\tilde v$ and $\tilde{\boldsymbol q}$, and by our definition of $(S_n)_{n \in \N}$, for a given $n \in \N$ the random variable
			\begin{equation*}
			\left\cg \frac{|\tilde v(s)|^2}{2\nu} + h(\tilde{\boldsymbol q}(s) | \boldsymbol q ), M_s \right\cd
			\end{equation*}
			is bounded uniformly in $s \in [0,S_n]$, and hence also in $s \in [0,T_n]$. As a consequence, for a given $t \in [0,1]$, we can use the integrability properties of $Z$ as a nonnegative supermartingale to decompose the following expectation and find
			\begin{align*}
			\Em_\mu\bigg[ l(Z_{t \wedge T_n}) - \int_0^{t \wedge T_n} &Z_s \left\cg \frac{|\tilde v(s)|^2}{2\nu}  + h(\tilde{\boldsymbol q}(s) | \boldsymbol q ), M_s \right\cd  \D s \bigg] \\
			&= \Em_\mu[ l(Z_{t \wedge T_n})] - \Em_\mu\left[ \int_0^{t \wedge T_n} Z_s \left\cg \frac{|\tilde v(s)|^2}{2\nu}  + h(\tilde{\boldsymbol q}(s) | \boldsymbol q ), M_s \right\cd  \D s \right]=0.
			\end{align*}
			Using Fatou's lemma and the monotone convergence theorem, we deduce that for all $t \in [0,1]$,
			\begin{equation}
			\label{eq:unif_integrability_bound1}
			\Em_\mu[ l(Z_t)] \leq \Em_\mu\left[ \int_0^t Z_s \left\cg \frac{|\tilde v(s)|^2}{2\nu} + h(\tilde{\boldsymbol q}(s) | \boldsymbol q ), M_s \right\cd  \D s \right].
			\end{equation} 
			It just remains to bound the r.h.s.\ of this inequality.
			
			To achieve this, let us consider $C>0$, a positive number bounding the $L^\infty(\XX, \G, \Pi_R)$ norm of $
			\frac{|\tilde v|^2}{2\nu} + h(\tilde{\boldsymbol q} | \boldsymbol q )$.
			We have
			\begin{equation}
			\label{eq:unif_integrability_bound2}
			\Em_\mu\left[ \int_0^t Z_s \left\cg \frac{|\tilde v(s)|^2}{2\nu} + h(\tilde{\boldsymbol q}(s) | \boldsymbol q ), M_s \right\cd  \D s \right] \leq C \int_0^t \Em_\mu[ Z_s M_s(\T^d) ] \D s.
			\end{equation}
			Now, let $\kappa>0$ be a number so small that for all $s \in [0,1]$, $\Em_\mu[\exp(\kappa M_s(\T^d)) - 1] \leq C$ (use Proposition~\ref{prop:exponential_bounds} to find such a $\kappa$). By~\eqref{eq:convex_inequality_l} applied to $z := Z_s$ and $y := \kappa M_s(\T^d)$, we have for all $s \in [0,1]$,
			\begin{equation*}
			\Em_\mu[Z_s M_s(\T^d)] \leq \frac{\Em_\mu[l(Z_s)] + \Em_\mu[\exp(\kappa M_s(\T^d))-1]}{\kappa} \leq \frac{\Em_\mu[l(Z_s)] + C}{\kappa}.
			\end{equation*}  
			We plug this bound into~\eqref{eq:unif_integrability_bound1}, using~\eqref{eq:unif_integrability_bound2} as well, to find for all $t \in [0,1]$
			\begin{equation*}
			\Em_\mu[l(Z_t)] \leq \frac{C}{\kappa} \left( \int_0^t \Em_\mu[l(Z_s)] \D s + Ct \right).
			\end{equation*}
			We conclude using Gr\"onwall's lemma that $\Em_\mu[l(Z_1)]$ is finite, which closes this step of the proof, namely that formula~\eqref{eq:formula_RN_derivative} indeed defines a probability measure in the case of a deterministic initial configuration.
		\end{stepc}
		
		\begin{stepc}{The inequality~\eqref{eq:entropy_smaller_than_RUOT} when $M_0$ is deterministic}
			Now, we know that $P := Z_1 \cdot R^\mu$ is a probability measure on $\Omega$. Let us compute its relative entropy with respect to $R$. We have
			\begin{align*}
			H(P|R) = \Em_\mu[Z_1 \log Z_1] = \Em_\mu[l(Z_1)],
			\end{align*}
			where the second equality follows from $\Em_\mu[Z_1] = 1$.
			By~\eqref{eq:unif_integrability_bound1}, we have:
			\begin{align*}
			\Em_\mu[ l(Z_1)] &\leq  \int_0^1 \Em_\mu\left[ Z_t \left\cg \frac{|\tilde v(t)|^2}{2\nu}+ h(\tilde{\boldsymbol q}(t) | \boldsymbol q ), M_t \right\cd \right]\D t = \int_0^1 \Em_\mu\left[ \Em_\mu[ Z_1| \F_t ]\left\cg \frac{|\tilde v(t)|^2}{2\nu} + h(\tilde{\boldsymbol q}(t) | \boldsymbol q ), M_t \right\cd  \right]\D t \\
			&=\int_0^1 \Em_\mu\left[ Z_1\left\cg \frac{|\tilde v(t)|^2}{2\nu} + h(\tilde{\boldsymbol q}(t) | \boldsymbol q ), M_t \right\cd \right]  \D t =\E_P\left[\int_0^1 \left\cg \frac{|\tilde v(t)|^2}{2\nu}+ h(\tilde{\boldsymbol q}(t) | \boldsymbol q ), M_t \right\cd  \D t \right].
			\end{align*} 
			Inequality~\eqref{eq:entropy_smaller_than_RUOT} in our restricted case follows.
		\end{stepc}
		\begin{stepc}{$P$ is a modified branching Brownian motion}
			\label{step:check_martingale_existence}
			Let us prove that $P$ defined by formula~\eqref{eq:formula_RN_derivative} is a BBM with drift $\tilde v$ and branching mechanism $\tilde{\boldsymbol q}$. By the Definition~\ref{def:modified_BBM} that we gave of such BBMs , we just need to prove that for all smooth nonpositive function $\psi$ on $[0,1] \times \T^d$, the \emph{càdlàg} process whose value at time $t \in [0,1]$ is
			\begin{equation}
			\label{eq:exp_martingale_nonpositive_varphi_proof_existence_BBM}
			\mathcal{E}^\psi_t := \exp\left( \frac{\cg \psi(t), M_t\cd}{\nu} -\frac{\cg \psi(0), M_0\cd }{\nu} - \frac{1}{\nu}\int_0^t \big\cg \partial_t \psi(s) + \mathcal L_{\nu, \tilde{\boldsymbol q}(s)}[\psi(s)] +  \tilde v(s) \cdot \nabla \psi(s), M_s \big\cd \D s \right)
			\end{equation}
			is a local martingale under $P$. By a classical argument (see for instance the proof of the Girsanov theorem in~\cite{legall2016brownian}), this is the same as proving that $(Z_t \mathcal{E}^\psi_t)_{t \in [0,1]}$, where $Z = (Z_t)_{t \in [0,1]}$ was defined in~\eqref{eq:def_Z_proof_RN_derivative}, is a local martingale under $R^\mu$.
			
			Let us call
			\begin{gather*}
			X_t := \frac{1}{\nu}I\left[ \tilde v + \nabla \psi\right]_t + \sum_{k \neq 1} J^k\left[ \log \frac{\tilde q_k}{q_k} + (k-1)\frac{\psi}{\nu} \right]_t, \quad t \in [0,1].
			\end{gather*}
			This is a semi-martingale with simple jumps, with no continuous finite variation component. Its bracket is 
			\begin{equation*}
			[X]_t = \frac{1}{\nu} \int_0^t \left\cg \left| \tilde v(s) + \nabla \psi(s) \right|^2 , M_s \right\cd \D s, \quad t \in [0,1],
			\end{equation*}
			and by formula~\eqref{eq:formula_pf} and the definition~\eqref{eq:def_Psi*} of $\Psi^*_{\nu, \boldsymbol q}$, its jump measure $(\eta_t)_{t \in [0,1]}$ satisfies for all $s \in [0,1]$
			\begin{equation*}
			\int\left\{ \exp(y) - 1 \right\}\D \eta_s(y) = \left\cg \sum_{k \neq 1} \tilde q_k(s) \exp\left( (k-1) \frac{\psi(s)}{\nu}\right) - q_k, M_{s-} \right \cd =\left\cg \Psi^*_{\nu, \tilde{\boldsymbol q}(s)}(\psi(s)) + \lambda_{\tilde{\boldsymbol q}(s)} - \lambda_{\boldsymbol q}, M_{s-}\right\cd .
			\end{equation*}
			
			With these two ingredients, exploiting~\eqref{eq:branching_Ito} with $\varphi := \psi / \nu$, and the respective definitions of $Z$, $\mathcal{E}^\psi$ and $\mathcal L_{\nu, \tilde{\boldsymbol q}(s)}[\psi(s)]$ (see formula~\eqref{eq:def_random_operator}), we find
			\begin{equation*}
			Z_t \mathcal{E}^\psi_t = \exp\left( X_t - \frac{1}{2} [X]_t - \int_0^t \hspace{-5pt} \int_{\R}\left\{ \exp(y) - 1 \right\}\D \eta_s(y) \D s \right), \qquad \forall t \in [0,1].
			\end{equation*}
			Hence, it is a local martingale and the result follows.
		\end{stepc}
		
		\begin{stepc}{Non-deterministic initial laws}
			\label{step:general_initial_law_existence}
			
			Now we deal with the case where $R \sim \BBM(\nu, \boldsymbol q, R_0)$, where $R_0$ can be any law on $\M_\delta(\T^d)$. The main idea to do so is to deal with formula~\eqref{eq:formula_RN_derivative} conditionally on $\F_0$. Therefore, for a given $\mu \in \M_\delta(\T^d)$, we call
			\begin{equation*}
			P^\mu := \exp\left( \frac{1}{\nu} I[\tilde v]_1 + \sum_{k \in \N} J^k\left[\log \frac{\tilde q_k}{q_k}\right]_1 - \int_0^1 \left\cg \frac{|\tilde v(t)|^2}{2\nu} + \lambda_{\tilde{\boldsymbol q}(t)} - \lambda_{\boldsymbol q}, M_t \right\cd \D t  \right) \cdot R^\mu,
			\end{equation*}
			which is well defined for $R_0$-almost all $\mu$, as explained in the first step. In addition, in this context, formula~\eqref{eq:formula_RN_derivative} rewrites
			\begin{equation}
			\label{eq:rewriting_formula_RN}
			P = \E_{R} \left[ \frac{\D P_0}{\D R_0}(M_0) \cdot P^{M_0} \right] = \E_{P}\left[ P^{M_0} \right],
			\end{equation} 
			which ensures in particular that for $P_0$-almost all $\mu$, $P^\mu = P( \, \cdot \, | M_0 = \mu)$.
			
			We are precisely in the case of the first part of the proof so that for $R_0$-almost all $\mu$:
			\begin{itemize}
				\item For all smooth nonpositive function $\psi$ on $\T^d$, $\mathcal{E}^\psi$ defined in~\eqref{eq:exp_martingale_nonpositive_varphi_proof_existence_BBM} is a local martingale under $P^\mu$.
				\item The following entropy bound holds:
				\begin{equation}
				\label{eq:entropy_smaller_than_RUOT_simple_case}
				H(P^\mu | R^\mu) \leq \Em_{\mu}\left[\int_0^1 \left\cg \frac{|\tilde v(t)|^2}{2\nu} + h(\tilde{\boldsymbol q}(t) | \boldsymbol q ), M_t \right\cd  \D t \right]
				\end{equation}
			\end{itemize}
			
			By the first point, for all smooth nonpositive $\psi$, $\mathcal{E}^\psi$ is a local martingale under $P^\mu$, for $R_0$-almost all $\mu$. As $P_0 \ll R_0$, this implies that it is a local martingale under $P^\mu$, for $P_0$-almost all $\mu$. As a consequence, it is a local martingale under $P$ and $P$ is a BBM with drift $\tilde v$ and branching mechanism $\tilde{\boldsymbol q}$.
			
			To get~\eqref{eq:entropy_smaller_than_RUOT}, it suffices plug~\eqref{eq:entropy_smaller_than_RUOT_simple_case} in formula~\eqref{eq:disintegration_entropy} and to use once again~\eqref{eq:rewriting_formula_RN}.
		\end{stepc}
	\end{proof}
	
	\section{A process with finite entropy w.r.t.\ a BBM is a BBM}
	\label{sec:girsanov}
	
	Here, we show that under the only assumption that $\boldsymbol q$ admits an exponential bound, that is, under the bound~\eqref{eq:exponential_moment_q}, any law $P$ that has finite entropy with respect to a BBM $R \sim \BBM(\nu, \boldsymbol{q}, R_0)$ is a branching Brownian motion with predictable and space-time dependent parameters. The result is written as follows, using the notations introduced in Subsection~\ref{subsec:def_modified_BBM}.
	
	\begin{Thm}
		\label{thm:representation_competitor_SchBr}
		Let $R \sim \BBM(\nu,\boldsymbol q,R_0)$, and assume that $\boldsymbol{q}$ satisfies the bound~\eqref{eq:exponential_moment_q}. Let $P$ be such that $H(P|R)< +\infty$.
		There is a predictable vector field $\tilde v$ and a predictable field of branching mechanisms $\tilde{\boldsymbol q}$ with the following integrability
		\begin{equation}
		\label{eq:integrability_bound_girsanov}
		\E_P\left[ \int_0^1 \left\cg \frac{|\tilde v(t)|^2}{2} + \nu h(\tilde{\boldsymbol q}(t) |\boldsymbol q) , M_t \right\cd \D t  \right] < + \infty,
		\end{equation}
		such that $P$ is a BBM of drift $\tilde v$ and branching mechanism $\tilde{\boldsymbol q}$.
		Moreover, these $\tilde v$ and $\tilde{\boldsymbol q}$ are unique up to a $\Pi_P$-negligible set. Finally, the following inequality holds:
		\begin{equation}
		\label{eq:RUOT_smaller_than_entropy}
		\nu H(P_0|R_0) + \E_P\left[ \int_0^1 \left\cg \frac{|\tilde v(t)|^2}{2} +\nu h(\tilde{\boldsymbol q}(t) |\boldsymbol q) , M_t \right\cd \D t  \right] \leq \nu H(P|R).
		\end{equation}
	\end{Thm}
	\begin{Rem}
		\label{rem:representation_competitor_BrSch}
		\begin{itemize}
			\item The integrability bound~\eqref{eq:integrability_bound_girsanov} implies the integrability condition~\eqref{eq:integrability_condition_def} on $\tilde v$ and $\lambda_{\tilde{\boldsymbol q}}$ appearing in the Definition~\ref{def:modified_BBM} of BBM with predictable and space-time dependent parameters.
			\item The integrability bound~\eqref{eq:integrability_bound_girsanov} also implies that $h(\tilde{\boldsymbol q}|\boldsymbol q)$ is finite $\Pi_P$-\emph{a.e.} In this case, it is clear that $\tilde q_1=0$, $\Pi_P$-\emph{a.e.}, and it is possible to use a slight generalization of~\eqref{eq:convex_ineq_entropy} together with the exponential bound~\eqref{eq:exponential_moment_q} to show that $\tilde{\boldsymbol q}$ satisfies $\Pi_P$-\emph{a.e.}\ the bound~\eqref{eq:q_finite_mean}. Consequently, $\tilde{\boldsymbol q}$ is automatically a predictable field of branching mechanism in the sense of Definition~\ref{def:branching_mechanism}. 
			\item As already announced in Remark~\ref{rem:construction_modified_BBM}, the uniqueness of the different parameters stated in the theorem and the inequality~\eqref{eq:RUOT_smaller_than_entropy} ensures that the inequality in~\eqref{eq:entropy_smaller_than_RUOT} from Theorem~\ref{thm:formula_RN_derivative} is in fact an equality.
		\end{itemize}
	\end{Rem}
	
	\begin{proof}
		\begin{stepd}{Setup and outline of the proof}
			Let us consider $R\sim\BBM(\nu,\boldsymbol q, R_0)$, and $P$ with $H(P|R)<+\infty$. The first steps of this proof consist in obtaining the functions $\tilde v$ and $\tilde q_k$, $k\in\N$ of the statement, using versions of the Riesz representation theorem (\emph{i.e.} following the strategy of~\cite{leonard2012girsanov}). Their properties will be obtained in further steps. The linear forms that we will represent are precisely the expectation with respect to $P$ of the processes built in Section~\ref{sec:new_processes}.
			
			The general strategy is to get bounds thanks to the formula~\eqref{eq:convex_ineq_entropy}, being $\exp(Y)$ the value at time $1$ of an exponential supermartingale involving the processes introduced in Section~\ref{sec:new_processes}.
			
			As we will see, this procedure will be rather straightforward in the case of $\tilde v$, and a bit more unusual in the cases of $\tilde q_k$, as in the latter, and similarly to~\cite[Theorem~2.6]{leonard2012girsanov}, we will have to use a Riesz representation theorem for linear forms bounded by exponential moments. Such forms are continuous on the cone of functions having an exponential moment for a certain topology that we do not want to describe precisely. Therefore, the situation is very similar\footnote{In fact, not exactly because our bound is only one-sided, but we will not enter into the details.} to the representation of continuous linear forms in Orlicz spaces~\cite{rao1991theory}, which are Banach spaces generated by functions satisfying integrability properties that are not necessarily homogeneous (generalizing the homogeneous $L^p$ spaces). The representation theorem that we will use could be deduced from the general theory of Orlicz spaces, but, as we found a rather short and elegant proof of exactly what we need, we have chosen to present it in Appendix~\ref{app:riesz}.

			The proof is organized as follows. As for Theorem~\ref{thm:formula_RN_derivative}, we will write the proof assuming that $M_0$ is deterministic under $R$ (and hence under $P$ as well). In this case, $R_0$ has an exponential moment so that both bound from Assumption~\ref{ass:exponential_bounds} hold. Under this additional assumption, we prove the existence part of the theorem at Steps \ref{step:finite_Pi}-\ref{step:martingale_property}, and the uniqueness result at Step~\ref{step:uniqueness}.
			We relax this assumption at Step~\ref{step:non_deterministic_intial_girsanov}, more or less in the same way as in the proof of Theorem~\ref{thm:formula_RN_derivative}. 
			
			Therefore, for now and until Step~\ref{step:non_deterministic_intial_girsanov} (not included), we assume that there is $\mu \in \M_{\delta}(\T^d)$ such that $R = R^\mu$. In this case, $M_0 = \mu$ $R$-\emph{a.s.}, and hence $P$-\emph{a.s.}\ as well. We recall that the measurable space $(\XX, \tilde\G)$ was defined in Definition~\ref{def:predictable_fields}, and that $\Pi_{R^\mu}$ and $\Pi_P$ stand for the measures on $\XX$ defined in Definition~\ref{def:Pi_Q}.
		\end{stepd}
		
		\begin{stepd}{$\Pi_{R^\mu}$ and $\Pi_P$ are finite}
			\label{step:finite_Pi}	
			Both measures are finite. Indeed, as both bound from Assumption~\ref{ass:exponential_bounds} hold under $R^\mu$, the result directly follows from integrating formula~\eqref{eq:unif_bound_density} in time. 
		\end{stepd}
		
		\begin{stepd}{Getting $\tilde v$}
			\label{step:getting_v}
			
			Let us show that the linear map
			\begin{equation}
			\label{eq:def_Lambda_I}
			\Lambda:	w \in L^\infty(\XX, \tilde \G, \Pi_{R^\mu}) \longmapsto \E_P\big[I[w]_1\big] \in \R,
			\end{equation}
			can be extended as a continuous linear map on $L^2(\XX, \tilde \G, \Pi_P)$, where $I$ is the stochastic integral introduced in Theorem~\ref{thm:definition_stochastic_integral}. Note that when $w \in L^\infty(\XX, \tilde \G, \Pi_{R^\mu})$, the random variable $I[w]_1$ is well defined $R^\mu$-\emph{a.e.}\ and therefore $P$-\emph{a.e.} So up to integrability properties, the formula for $\Lambda$ makes sense. Part of the result asserts that $\Lambda(w)$ does not depend on the values of $w$ on $\Pi_P$-negligible sets.
			
			Take $w \in L^\infty(\XX, \tilde \G, \Pi_{R^\mu})$. We first show that $I[w]_1$ is $P$-integrable, so that the definition~\eqref{eq:def_Lambda_I} of $\Lambda$ makes sense. By Theorem~\ref{thm:definition_stochastic_integral} and Theorem~\ref{thm:Ito_with_jumps}, the process 
			\begin{equation*}
			\mathcal{E}_t := \exp\left(I[ w]_t - \frac{\nu}{2} \int_0^t \cg|w(s)|^2, M_s\cd \D s  \right), \quad t \in [0,1],
			\end{equation*}
			is a nonnegative supermartingale under $R^\mu$, starting from $1$. In particular, $\mathcal{E}_1$ belongs to $L^1(\Omega, \F, R^\mu)$, and its expectation is smaller or equal to $1$. Let us write~\eqref{eq:convex_ineq_entropy} with $Y := (I[ w]_1 - \frac{\nu}{2} \int_0^1 \cg|w(s)|^2, M_s\cd \D s)_+$, the nonnegative part of $\log \mathcal{E}_1$. We get
			\begin{align*}
			\E_P\left[ \left(I[ w]_1 - \frac{\nu}{2} \int_0^1 \cg|w(s)|^2, M_s\cd \D s\right)_+ \right] &\leq H(P|R) + \log \Em_\mu\big[\exp\big((\log \mathcal{E}_1)_+\big)\big]\\
			&\leq H(P|R) + \log\Big( 1 + \Em_\mu[\mathcal{E}_1]\Big) \leq H(P|R) + \log 2.
			\end{align*}
			Finally, the inequality
			\begin{equation*}
			\big(I[ w]_1\big)_+ \leq \left(I[ w]_1 - \frac{\nu}{2} \int_0^1 \cg|w(s)|^2, M_s\cd \D s\right)_+ +  \frac{\nu}{2} \int_0^1 \cg|w(s)|^2, M_s\cd \D s ,
			\end{equation*}
			and the facts that $w$ is bounded and $\Pi_P$ is finite lead to
			\begin{equation*}
			\E_P\big[ \big(I[w]_1\big)_+ \big] \leq H(P|R) + \log 2 + \frac{\nu}{2}\E_P\left[ \int_0^1 \cg|w(s)|^2, M_s\cd \D s \right] < +\infty.
			\end{equation*}
			The $L^1$ bound of $I[w]_1$ w.r.t.\ $P$ comes from combining this inequality with the same one for $-w$.
			
			Hence, $\Lambda(w)$ given in~\eqref{eq:def_Lambda_I} is well defined for all $w \in L^\infty(\XX, \tilde \G, \Pi_{R^\mu})$. Let us bound it. The process indexed by $\gamma \in \R$ defined by
			\begin{equation*}
			\exp\left( \gamma I[ w]_t - \gamma^2 \frac{\nu}{2} \int_0^t \cg|w(s)|^2, M_s\cd \D s  \right), \quad t \in [0,1],
			\end{equation*}
			is a $R^\mu$-supermartingale starting from $1$. Let us apply~\eqref{eq:convex_ineq_entropy}  with $Y_\gamma := \gamma I[w]_1 - \gamma^2 \frac{\nu}{2} \int_0^1 \cg|w(s)|^2, M_s\cd \D s$. Optimizing in $\gamma$, we get:
			\begin{equation*}
			\big|\Lambda(w) \big| \leq \sqrt{ 2\nu  H(P|R) \E_P\left[ \int_0^1 \cg|w(s)|^2, M_s\cd \D s \right]} =  \sqrt{ 2\nu  H(P|R)} \| w \|_{L^2(\XX, \tilde \G,\Pi_P)}, 
			\end{equation*}
			where the r.h.s.\ is finite because $w$ is bounded and $\Pi_P$ is finite. 
			
			By standard arguments relying on the Riesz representation theorem on the Hilbert space $L^2(\XX, \tilde \G, \Pi_P)$, we get a unique $\tilde v \in L^2(\XX, \tilde \G, \Pi_P)$ such that for all $w \in L^\infty(\XX, \tilde \G, \Pi_{R^\mu})$,
			\begin{equation}
			\label{eq:charact_v}
			\Lambda(w) = \E_P\big[I[w]_1\big] = \E_P\left[ \int_0^1 \cg \tilde v(s) \cdot w(s), M_s \cd \D s \right].
			\end{equation}
			Note that here, we took $w$ in $L^\infty(\XX, \tilde \G, \Pi_{R^\mu})$ and not in $L^\infty(\XX, \tilde\G, \Pi_P)$. The reason is that it is not clear yet that $I[w]_1$ does not depend on the values of $w$ on $\Pi_P$-negligible sets. Only its expectation (namely, the r.h.s.) is well defined on $L^\infty(\XX, \tilde \G, \Pi_P)$. An observation will be made at Step~\ref{step:interpretation_tilde_v} below to solve this problem.
		\end{stepd}

		\begin{stepd}{Getting $\tilde q_k$, $k \in \N$}	
			\label{step:getting_q}
			
			First, we set $\tilde q_1 \equiv 0$. Then, we consider $k \in \N$, $k \neq 1$. The procedure to get $\tilde q_k$ is the same as in the previous step, replacing $I$ by $J^k$ and the Riesz representation theorem by Theorem~\ref{thm:riesz}.  Namely, this time, we show that the linear map
			\begin{equation}
			\label{eq:def_Lambda_J}
			\Lambda: a \in L^\infty(\XX, \tilde \G, \Pi_{R^\mu}) \longmapsto \E_P\big[ J^k[a]_1 \big] \in \R,
			\end{equation}
			where $J^k$ is the jump process introduced in Definition~\ref{def:jump_processes}, satisfies the properties of Theorem~\ref{thm:riesz}, so that we can represent it by an $L^1(\XX, \tilde \G, \Pi_P)$ function.
			
			Take $a \in L^\infty(\XX, \tilde \G, \Pi_{R^\mu})$. Because of Theorem~\ref{thm:properties_jump_processes} and Theorem~\ref{thm:Ito_with_jumps}, the process
			\begin{equation*}
			\exp\left(J^k[a]_t - q_k \int_0^t \cg  \exp(a(s))-1  , M_s\cd \D s\right), \quad t \in [0,1],
			\end{equation*}
			is a $R^\mu$-supermartingale starting from $1$. Applying~\eqref{eq:convex_ineq_entropy} with $Y := J^k[a]_1 - q_k \int_0^1 \cg \exp(a(s)) - 1, M_s\cd \D s$, we get
			\begin{equation*}
			\E_P\big[J^k[a]_1\big]\leq H(P|R) + q_k\E_P\left[ \int_0^1 \cg \exp(a(s)) - 1, M_s\cd \D s \right] = H(P|R) + \int_{\mathcal X}\big\{\exp(a) - 1 \big\} \D \big(q_k \Pi_P\big).
			\end{equation*}
			Note that \emph{a priori}, this inequality only holds in $[-\infty, +\infty)$. But applying it separately to $a_+$ and $a_-$, the nonnegative and nonpositive parts of $a$, and exploiting the identities $(J^k[a]_1)_\pm = J^k[a_\pm]_1$ (which are clear from the Definition~\ref{def:jump_processes} of $J^k$), the boundedness of $a$ and the fact that $\Pi_P$ is finite, we get a $L^1$ bound on $J^k[a]_1$. 	Hence, $\Lambda(a)$ given in~\eqref{eq:def_Lambda_J} is well defined for all $a \in L^\infty(\XX, \tilde \G, \Pi_{R^\mu})$ and therefore also for all $a \in \L^\infty(\XX, \tilde \G)$ (up to composing with the quotient map), with the bound
			\begin{equation*}
			\Lambda(a) \leq H(P|R) + \int_{\mathcal X}\big\{\exp(a) - 1 \big\} \D \big(q_k \Pi_P\big).
			\end{equation*}
			
			A direct application of Theorem~\ref{thm:riesz} in $(\XX, \tilde \G, q_k \Pi_P)$ (recall that by Proposition~\ref{prop:predictable_sigma_field}, $\tilde \G$ is countably generated) provides a nonnegative function $g \in L^1(\XX, \tilde \G, q_k\Pi_P) = L^1(\XX, \tilde \G, \Pi_P)$ such that for all function $a \in L^\infty(\XX, \tilde \G, \Pi_{R^\mu})$,
			\begin{equation*}
			\Lambda(a) = q_k \int_{\mathcal X} ag \D \Pi_P.
			\end{equation*}
			Defining $\tilde q_k := q_k g$ (which is nonnegative) and using the formula~\eqref{eq:def_Lambda_J} for $\Lambda$ and the definition of $\Pi_P$, this identity rewrites for all $a \in L^\infty(\XX, \tilde \G, \Pi_{R^\mu})$,
			\begin{equation}
			\label{eq:charact_q}
			\E_P\big[ J^k[a]_1 \big] = \E_P\left[ \int_0^1 \cg \tilde q_k(s) a(s), M_s \cd \D s \right].\\
			\end{equation}
		\end{stepd}
		
		\begin{stepd}{Energy estimate}
			\label{step:ineq_entropy}	
			Before showing that $P$ is a BBM with parameters $\tilde v$ and $\tilde{\boldsymbol q}:= (\tilde q_k)_{k \in \N}$, let us prove~\eqref{eq:RUOT_smaller_than_entropy} in this context where $P_0 = R_0 = \delta_\mu$ (and hence $H(P_0|R_0)=0$). Let $(w_n)_{n \in \N}$ and $(a^k_n)_{k,n \in \N}$ be families of functions of $L^\infty(\XX, \tilde \G, \Pi_{R^\mu})$ satisfying:
			\begin{equation}
			\label{eq:def_w_a}
			\Pi_P\mbox{-a.e.},\ \forall n \in \N, \qquad w_n = \frac{\tilde v}{\nu} \1_{\frac{|\tilde v|}{\nu} \leq n} \qquad \mbox{and} \qquad \forall k \neq 1,\ a^k_n = \log \frac{\tilde q_k}{q_k}\1_{\log\frac{\tilde q_k}{q_k} \leq n}.
			\end{equation}
			Fix $n \in \N$. By Theorem~\ref{thm:Ito_with_jumps}, Theorem~\ref{thm:definition_stochastic_integral} and Theorem~\ref{thm:properties_jump_processes}, the process
			\begin{equation*}
			\exp\left( I[w_n]_t + \sum_{k\neq 1} J^k[a^k_n]_t - \int_0^t \left\cg \frac{\nu}{2} |w_n(s)|^2 + \sum_{k\neq 1} q_k\Big\{ \exp\big(a^k_n(s)\big) - 1 \Big\} ,M_s\right\cd \D s\right), \quad t \in [0,1],
			\end{equation*}
			is a $R^\mu$-supermartingale starting from $1$. We call
			\begin{equation*}
			Y :=  I[w_n]_1 + \sum_{k=0}^{+\infty} J^k[a^k_n]_1 - \int_0^1 \left\cg \frac{\nu}{2} |w_n(s)|^2 + \sum_{k=0}^{+\infty} q_k\Big\{ \exp\big(a^k_n(s)\big) - 1 \Big\} ,M_s\right\cd \D s.
			\end{equation*}
			Applying inequality~\eqref{eq:convex_ineq_entropy} to $Y_+$ and then $Y$, we obtain that $Y_+$ is $P$-integrable, and that
			\begin{equation*}
			\E_P\big[ I[w_n]_1 \big] + \sum_{k\neq 1} \E_P\big[ J^k[a_n^k]_1 \big] - \E_P\left[ \int_0^1 \left\cg \frac{\nu}{2} |w_n(s)|^2 + \sum_{k\neq 1} q_k\Big\{ \exp\big(a^k_n(s)\big) - 1 \Big\} ,M_s\right\cd \D s \right] \leq H(P|R),
			\end{equation*}
			where the l.h.s.\ is possibly $-\infty$ at this stage. By our two formulas~\eqref{eq:charact_v} and~\eqref{eq:charact_q}, we get
			\begin{equation*}
			\E_P\big[ I[w_n]_1 \big] = \E_P\left[ \int_0^1 \cg \tilde v(s) \cdot w_n(s), M_s \cd \D s \right] \quad  \mbox{and} \quad \forall k\neq 1, \ \E_P\big[ J^k[a^k_n]_1 \big] = \E_P\left[ \int_0^1 \cg \tilde q_k(s) a^k_n(s), M_s \cd \D s \right].
			\end{equation*}
			Plugging these identities in the previous inequality and rearranging slightly the terms, we get
			\begin{equation*}
			\E_P\left[ \int_0^1 \left\cg \tilde v(s)\cdot w_n(s) - \frac{\nu}{2} |w_n(s)|^2 + \sum_{k\neq 1} \Big\{\tilde q_k(s) a^k_n(s) + q_k - q_k \exp\big(a^k_n(s)\big) \Big\} ,M_s\right\cd \D s \right]\leq H(P|R).
			\end{equation*}
			If we plug the values~\eqref{eq:def_w_a} of $w_n$ and $(a^k_n)_{k \in \N}$ in this formula, we discover
			\begin{equation*}
			\E_P\left[ \int_0^1 \left\cg  \frac{|\tilde v(s)|^2}{2\nu} \1_{\frac{|\tilde v(s)|}{\nu} \leq n} + \sum_{k\neq 1} \Big\{\tilde q_k(s) \log \frac{\tilde q_k(s)}{q_k} + q_k - \tilde q_k(s) \Big\} \1_{\log \frac{\tilde q_k(s)}{q_k} \leq n} ,M_s\right\cd \D s \right]\leq H(P|R).
			\end{equation*}
			The integrand is nonnegative, and nondecreasing in $n$. Taking $n \to +\infty$, we get by the monotone convergence theorem
			\begin{equation*}
			\E_P\left[ \int_0^1 \left\cg  \frac{|\tilde v(s)|^2}{2\nu}   +h(\tilde{\boldsymbol q}(s) | \boldsymbol q) ,M_s\right\cd \D s \right]\leq H(P|R),
			\end{equation*}
			which is nothing but inequality~\eqref{eq:RUOT_smaller_than_entropy}, up to multiplying it by $\nu$, in the specific case where $P_0 = R_0$.
		\end{stepd}
		
		\begin{stepd}{Interpretation of \eqref{eq:charact_v}}
			\label{step:interpretation_tilde_v}	
			Here, we show that for all $w \in L^\infty(\XX, \tilde \G, \Pi_{R^\mu})$, the process
			\begin{equation*}
			I[w]_t - \int_0^t \cg \tilde v(s) \cdot w(s), M_s\cd \D s , \quad t \in [0,1],
			\end{equation*}
			is a continuous martingale under $P$, whose bracket is
			\begin{equation*}
			\nu \int_0^t \left\cg |w(s)|^2, M_s \right\cd \D s, \quad t \in [0,1].
			\end{equation*}
			Observe that by the same arguments as before, both terms are $P$-integrable.
			
			Take $w\in L^\infty(\XX, \tilde \G, \Pi_{R^\mu})$. Let $t_0 < t_1$ be two times in $[0,1]$, and let $Z_{t_0}$ be a bounded $\F_{t_0}$-measurable random variable on $\Omega$. For any $s \in [0,1]$, we call 
			\begin{equation*}
			\bar w(s) := Z_{t_0} \1_{(t_0 ,t_1]}(s) w(s).
			\end{equation*}
			This leads to $\bar w \in L^\infty(\XX, \tilde \G, \Pi_{R^\mu})$, and~\eqref{eq:charact_v} applies:
			\begin{equation*}
			\E_P\big[I[\bar w]_1\big] = \E_P\left[ \int_0^1 \cg \tilde v(s) \cdot \bar w(s), M_s \cd \D s \right]=\E_P\left[ Z_{t_0}\int_{t_0}^{t_1} \cg \tilde v(s) \cdot w(s), M_s \cd \D s \right].
			\end{equation*}	
			By formula~\eqref{eq:truncation_stochastic_integral} from Theorem~\ref{thm:definition_stochastic_integral}, there holds $R$-\emph{a.e.}\ $I[\bar w]_1 = Z_{t_0}(I[w]_{t_1} - I[w]_{t_0})$. Therefore, the identity above rewrites
			\begin{equation*}
			\E_P\left[ Z_{t_0} \left( I[w]_{t_1} - \int_0^{t_1} \cg \tilde v(s) \cdot w(s), M_s \cd \D s \right) \right] = \E_P\left[ Z_{t_0} \left( I[w]_{t_0} - \int_0^{t_0} \cg \tilde v(s) \cdot w(s), M_s \cd \D s \right) \right].
			\end{equation*}
			This is precisely the martingale property announced. The continuity of the process is obvious, and the value of its bracket follows from the classical fact that the bracket of a continuous local martingale is invariant under an absolutely continuous change of law.
			
			As a remark, note that if $w^1 = w^2$ $\Pi_P$-\emph{a.e.}, being $w^1,w^2$ two predictable vector fields in~$L^\infty(\mathcal X, \tilde \G, \Pi_{R^\mu})$, then the bracket of $I[w^2 - w^1]$ cancels. A consequence of this is that for all $w \in L^\infty(\mathcal X, \tilde \G, \Pi_{R^\mu})$, $P$-\emph{a.s.}, $I[w]_1$ does not depend on the values of $w$ on $\Pi_P$-negligible sets. As we will not use this property, we do not develop further this observation.
		\end{stepd}
		
		\begin{stepd}{Interpretation of \eqref{eq:charact_q}}
			Now, we prove that for all family $(a^k)_{k \neq 1} \in (L^\infty(\XX,\tilde \G,\Pi_{R^\mu}))^{\N \backslash\{1\}}$ and all $k\neq 1$, the process $J := \sum_{k \neq 1}J^k[a^k]$ is a simple pure-jump process under $P$, whose jump measure is the predictable field
			\begin{equation*}
			\eta_t := \sum_k a^k(t)\pf\Big\{ \tilde q_k(t) M_t \Big\}, \quad t \in [0,1].
			\end{equation*}
			Up to a $R^\mu$-negligible set, the process $J$ is $\cadlag$, piecewise constant, and has a finite number of jumps, so it satisfies these properties also up to a $P$-negligible set. Hence, in view of Definition~\ref{def:simple_pure-jump_processes}, we need to check~\eqref{eq:integrability_jump_measure} and~\eqref{eq:martingale_property_def_simple_jump_processes}. For~\eqref{eq:integrability_jump_measure}, we have
			\begin{equation*}
			\int_0^1 \eta_t(\R)\D t = \int_0^1 \sum_k \cg \tilde q_k(t), M_t \cd \D t = \int_0^1 \cg \lambda_{\tilde{\boldsymbol q}(t)}, M_t \cd \D t,
			\end{equation*}
			which is easily seen to be $P$-\emph{a.s.}\ finite exploiting~\eqref{eq:integrability_bound_girsanov} and the inequality~\eqref{eq:tilde_lambda_bounded} given in the previous proof. To prove~\eqref{eq:martingale_property_def_simple_jump_processes}, we use the formula~\eqref{eq:truncation_jump_process} given in Theorem~\ref{thm:properties_jump_processes}: Given $f \in C_b(\R)$, for all $t_0 < t_1$ in $[0,1]$, and for all bounded $\F_{t_0}$-measurable random variable $Z_{t_0}$, $R$-\emph{a.s.}\ (and hence also $P$-\emph{a.s.}), we have:
			\begin{equation*}
			Z_{t_0}(f(J_{t_1}) - f(J_{t_0})) = \sum_{k\neq 1} J^k[b^k]_1,
			\end{equation*}
			where for all $k \neq 1$, $b^k$ is the bounded (uniformly in $k$) predictable scalar field defined by
			\begin{equation*}
			b^k(t) = Z_{t_0}\1_{(t_0, t_1]}(t)\Big(f\big(J_{t-}+ a^k(t)\big) - f\big(J_{t-}\big)\Big), \quad t \in [0,1].
			\end{equation*}
			Using the definition~\eqref{eq:charact_q} of $\tilde q_k$ and writing $b^k = b^k_+ - b^k_-$, we deduce that $\sum_k J^k[b^k]$ is $P$-integrable, with
			\begin{equation*}
			\E_P\left[ \sum_k J^k[b^k]_1 \right] = \sum_k \E_P\left[ Z_{t_0}\int_{t_0}^{t_1} \left\cg \tilde q_k(t) \Big\{ f(J_{t-} + a^k(t)) - f(J_{t-}) \Big\}, M_s \right\cd \D s \right].
			\end{equation*}
			This last equality rewrites
			\begin{equation*}
			\E_P\left[ Z_{t_0} \left( f(J_{t_1}) -\! \int_{0}^{t_1} \hspace{-5pt} \Big\{ f(J_t + y) - f(J_t) \Big\}\D \eta_t(y) \right) \right] = \E_P\left[ Z_{t_0} \left( f(J_{t_0}) -\! \int_{0}^{t_0} \hspace{-5pt} \Big\{ f(J_t + y) - f(J_t) \Big\}\D \eta_t(y) \right) \right],
			\end{equation*}
			and~\eqref{eq:martingale_property_def_simple_jump_processes} follows.
		\end{stepd}
		
		\begin{stepd}{The local martingale property}
			\label{step:martingale_property}
			
			Now that we obtained $\tilde v$ and $\tilde{\boldsymbol q}$, we want to prove that $P$ is the law of a BBM with these parameters, that is that the process defined in~\eqref{eq:def_modified_BBM} is a local martingale under $P$.
			
			Let us consider a smooth nonpositive function $\psi$ on $[0,1] \times \T^d$. As $P \ll R^\mu$, the formula~\eqref{eq:branching_Ito} holds $P$-\emph{a.s.}\ for all $t \in [0,1]$. As a consequence, using the formula~\eqref{eq:def_random_operator} of $\mathcal L_{\nu, \tilde{\boldsymbol q}(s)}[\psi(s)]$, we find that $P$-\emph{a.s.} $\forall t \in [0,1]$,
			\begin{align*}
			\exp\bigg( &\frac{\cg \psi(t), M_t\cd}{\nu} -\frac{\cg \psi(0), M_0\cd }{\nu} - \frac{1}{\nu}\int_0^t \big\cg \partial_t \psi(s) + \mathcal L_{\nu, \tilde{\boldsymbol q}(s)}[\psi(s)] + \tilde v(s) \cdot \nabla \psi(s), M_s \big\cd \D s \bigg)\\
			&=\exp\bigg(\frac{1}{\nu}\bigg\{   I[\nabla \psi]_t + \sum_{k\neq 1} J^k[(k-1)\psi]_t -  \int_0^t \Big\cg \frac{1}{2}|\nabla \psi(s)|^2 + \Psi^*_{\nu, \tilde{\boldsymbol q}(s)}(\psi(s)) + \tilde v(s) \cdot \nabla \psi(s), M_s \Big\cd \D s \bigg\}\bigg).
			\end{align*}
			What we want to do is to apply the second point of Theorem~\ref{thm:Ito_with_jumps} to the $P$-semi-martingale with simple jumps:
			\begin{equation*}
			X_t := \frac{1}{\nu}\left( I[\nabla \psi]_t - \int_0^t \cg \tilde v(s) \cdot \nabla \psi(s), M_s\cd \D s +\sum_{k\neq 1} J^k[(k-1)\psi]_t\right), \quad t \in [0,1].
			\end{equation*}
			By the two previous steps, $X$ has no continuous finite variation component, its bracket is
			\begin{equation*}
			[X]_t = \frac{1}{\nu} \int_0^t \cg |\nabla \psi(s)|^2, M_s \cd \D s, \quad t \in [0,1],
			\end{equation*} 
			and its jump measure is
			\begin{equation*}
			\eta_t := \sum_{k\neq 1} \left\{ (k-1)\frac{\psi(s)}{\nu} \right\}\pf\Big\{ \tilde q_k(t) M_t \Big\}, \quad t \in [0,1],
			\end{equation*}
			so that
			\begin{align*}
			\int_0^t \hspace{-5pt} \int_{\R}\Big\{ \exp(y)-1 \Big\}\D \eta_s(y) \D s &= \int_0^t\left\cg \sum_{k \neq 1} \tilde q_k( s)\left\{ \exp\left((k-1)\frac{\psi(s)}{\nu}\right)-1 \right\}, M_s \right\cd \D s \\
			&= \frac{1}{\nu}\int_0^t\left\cg \Psi^*_{\nu, \tilde{\boldsymbol q}(s)}(\psi(s)), M_s \right\cd \D s.
			\end{align*}
			In view of Theorem~\ref{thm:Ito_with_jumps}, the only thing that we have to check, is that the previous quantity is not $+\infty$, $P$-\emph{a.s.}
			Actually, using the nonpositivity of $\psi$ and the definition~\eqref{eq:def_Psi*} of $\Psi^*_{\nu, \boldsymbol q}$, this is obvious since 
			\begin{equation*}
			\int_0^1\left\cg \Psi^*_{\nu, \tilde{\boldsymbol q}(t)}(\psi(t)), M_t \right\cd \D t \leq e^{\frac{\| \psi \|_{\infty}}{\nu}} \int_0^1\left\cg \tilde q_0(t) , M_t \right\cd \D t \leq  e^{\frac{\| \psi \|_{\infty}}{\nu}} \int_0^1 \cg \lambda_{\tilde{\boldsymbol q}(s)}(t), M_t \cd \D t,
			\end{equation*}
			which is $P$-\emph{a.s.}\ finite by~\eqref{eq:integrability_bound_girsanov}. This concludes the existence part of the first part of this proof.
		\end{stepd}
		
		\begin{stepd}{Uniqueness}
			\label{step:uniqueness}
			
			In this last step of this first part, we show that $\tilde v$ and $\tilde{\boldsymbol q}$ are unique on $\XX$ up to a $\Pi_P$-negligible set.
			
			Let us consider $\hat v$ and $\hat{\boldsymbol q}$ fulfilling the requirements of Definition~\ref{def:modified_BBM}. By~\eqref{eq:def_modified_BBM} and~\eqref{eq:branching_Ito}, for all nonpositive smooth function $\psi$ on the torus, the process
			\begin{equation*}
			\exp\left( \frac{1}{\nu}\left\{ I[\nabla \psi]_t + \sum_{k \neq 1} J^k[(k-1)\psi] - \int_0^t \left\cg \frac{1}{2} |\nabla \psi|^2 + \Psi^*_{\nu, \hat{\boldsymbol q}(s)}(\psi)+ \hat v(s) \cdot \nabla \psi  , M_s \right\cd \D s\right\}\right), \quad t \in [0,1],
			\end{equation*}
			is a $P$-local martingale. Let us call for all $t \in [0,1]$
			\begin{equation*}
			Z^\psi_t := \exp\left( \frac{1}{\nu}\left\{ I[\nabla \psi]_t + \sum_{k \neq 1} J^k[(k-1)\psi] - \int_0^t \left\cg \frac{1}{2} |\nabla \psi|^2 + \Psi^*_{\nu, \tilde{\boldsymbol q}(s)}(\psi)+ \tilde v(s) \cdot \nabla \psi  , M_s \right\cd \D s\right\}\right),
			\end{equation*}
			the process defined in the same way but with tildas instead of hats. This is a $P$-local martingale as well due to the previous step. We deduce that for all nonpositive smooth $\psi$, the process
			\begin{equation*}
			Z^\psi_t \times \exp\left( \frac{1}{\nu}\int_0^t \left\cg (\tilde v(s) - \hat v(s))\cdot \nabla \psi + \sum_{k\neq 1} (\tilde q_k(s) - \hat q_k(s))\left\{ \exp\left( (k-1)\frac{\psi}{\nu}\right) - 1 \right\} , M_s \right\cd \D s \right), \quad t \in [0,1],
			\end{equation*}
			is a $P$-local martingale. By Lemma~\ref{lem:mg_times_finite_var} this is possible if and only if for all nonpositive smooth $\psi$, $P$-\emph{a.s.}\ for almost every $t \in [0,1]$,
			\begin{equation*}
			\left\cg (\tilde v(t) - \hat v(t))\cdot \nabla \psi + \sum_{k\neq 1} (\tilde q_k(t) - \hat q_k(t))\left\{ \exp\left( (k-1)\frac{\psi}{\nu}\right) - 1 \right\} , M_t \right\cd = 0.
			\end{equation*}
			By standard density and easy localization arguments, this implies that $\Pi_P$-almost everywhere, for all nonpositive smooth $\psi$,
			\begin{equation*}
			(\tilde v - \hat v)\cdot \nabla \psi + \sum_{k\neq 1} (\tilde q_k - \hat q_k)\left\{ \exp\left( (k-1)\frac{\psi}{\nu}\right) - 1 \right\}=0.
			\end{equation*}
			Let $(\omega,t,x) \in \XX$ a point where this identity is true. By testing it with constant nonpositive functions $\psi: z \in \T^d \mapsto \nu c \in \R_-$, we get:
			\begin{equation*}
			\sum_{k\neq 1} (\tilde q_k(\omega,t,x) - \hat q_k(\omega,t,x))\left\{ \exp\big( (k-1)c\big) - 1 \right\}, \quad \forall c \in \R_-.
			\end{equation*}
			This clearly implies $\hat q_k(\omega,t,x) = \tilde q_k(\omega,t,x)$, for all $k \neq 1$. then, taking for $\psi$ any nonpositive smooth function whose gradient at $x$ is $\tilde v(\omega,t,x) - \hat v(\omega,t,x)$, we see that $\tilde v(\omega,t,x) = \hat v(\omega,t,x)$. The announced uniqueness property follows.
		\end{stepd}
		
		\begin{stepd}{Non-deterministic initial laws}
			\label{step:non_deterministic_intial_girsanov}
			
			Here, we show that we can relax the condition that $M_0$ is deterministic under $R$. To do this, let us consider $R_0 \in \P(\M_\delta(\T^d))$, $R\sim\BBM(\nu, \boldsymbol q, R_0)$, and $P$ with $H(P|R)<+\infty$. We start from the additivity property of the relative entropy~\eqref{eq:disintegration_entropy}.
			
			Because $H(P|R)$ is finite and $H(P_0|R_0)$ is nonnegative, formula~\eqref{eq:disintegration_entropy} ensures that for $P_0$-almost every $\mu \in \M_\delta(\T^d)$, $H(P^\mu|R^\mu)$ is also finite. Hence, we can apply the results of the previous steps, and find predictable fields $\tilde v^\mu$ and $\tilde{\boldsymbol q}^\mu$ meeting the requirements of the theorem applied to the laws $R^\mu$ and $P^\mu$. These are defined $\Pi_{P^\mu}$-\emph{a.e.} 
			
			Moreover, by the measurability property given by the disintegration theorem, one can check that the families of linear forms
			\begin{equation*}
			w \mapsto \E_{P^\mu}[I[w]_1] \qquad \mbox{and} \qquad 	a \mapsto \E_{P^\mu}[J^k[a]_1], \quad k \in \N\backslash\{1\},
			\end{equation*}
			depend on $\mu$ in a measurable way. Therefore, by definition of the drift an branching mechanism given at steps~\ref{step:getting_v} and~\ref{step:getting_q}, we deduce that $\tilde v^\mu$ and $\tilde{\boldsymbol q}^\mu$ depend on $\mu$ in a measurable way.
			
			The formula
			\begin{equation*}
			\tilde v(\omega, t, x) := \tilde v^{M_0(\omega)}(\omega, t, x)
			\end{equation*}
			provides a definition of the field $\tilde v$, for $\Pi_P := P\otimes \D t \otimes M_t$ almost every $(\omega,t,x) \in \XX$. The field $\tilde{\boldsymbol q}$ is built correspondingly.
			
			Let us check the properties announced in the statement of the theorem. These fields are clearly predictable (as $(\omega,t,x) \in \XX \mapsto M_0(\omega)\in \mathcal M_+(\T^d)$ is predictable). Inequality~\eqref{eq:RUOT_smaller_than_entropy} is a direct consequence of~\eqref{eq:disintegration_entropy}, Step~\ref{step:ineq_entropy} and $\E_{P}[ \E_{P^{M_0}}[\ \cdot \ ]] = \E_P[\ \cdot \ ]$. The local martingale property of the process defined in formula~\eqref{eq:def_modified_BBM} follows from the general fact that for any process $X$, $X$ is a local martingale under $P$ if and only if for $P_0$-almost every $\mu \in \M_\delta(\T^d)$, $X$ is a local martingale under $P^\mu$. 
			
			Concerning uniqueness, Step~\ref{step:uniqueness} provides uniqueness of the fields up to $\Pi_{P^\mu}$-negligible sets, for $P_0$-almost every $\mu$. This is the same as uniqueness up to $\Pi_P$-negligible sets, as $\Pi_P$ is easily seen to be $\E_{P_0}[\Pi_{P^{M_0}}]$.
			
		\end{stepd}
	\end{proof}
	
	We close this section and this chapter with a remark generalizing~\eqref{eq:charact_v} and~\eqref{eq:charact_q} in the case when $R_0$ has an exponential bound, that is, under~\eqref{eq:exponential_moment_R_0}. 
	\begin{Rem}
		\label{rem:general_charact}
		Having a close look at this proof, we observe that except at Step~\ref{step:non_deterministic_intial_girsanov} where we get the term $H(P_0|R_0)$ in the bound~\eqref{eq:RUOT_smaller_than_entropy}, the only reason for which we used $R^\mu$, $\mu \in \M_\delta(\T^d)$, instead of $R$ is that contrarily to $R$, $R^\mu$ satisfies the exponential bound on the initial law~\eqref{eq:exponential_moment_R_0}. Therefore, if we had assumed in the first place that $R$ satisfies this bound, we could have written everything from Step~\ref{step:finite_Pi} to Step~\ref{step:uniqueness} in the exact same way for $R$ directly.
		Therefore, under this additional assumption, we see that~\eqref{eq:charact_v} and~\eqref{eq:charact_q} rewrite
		\begin{align*}
		&\forall w \in L^\infty(\XX, \tilde{\mathcal G}, \Pi_R),   &\E_P\Big[ I[w]_1  \Big] &= \E_P\left[  \int_0^1 \cg \tilde v(t) \cdot w(t) , M_t \cd \D t  \right],\\
		&\forall a \in L^\infty(\XX, \tilde{\mathcal G}, \Pi_R), \, \forall k \in \N\backslash\{1\},&\E_P\Big[ J^k[a]_1 \Big] &= \E_P\left[ \int_0^1 \cg \tilde q_k(t) a(t), M_t \cd \D t \right].
		\end{align*}
	\end{Rem}
	
	\chapter{Equivalence of the competitors}
	\label{chap:equivalence_competitors}

	The goal of this chapter is to use the results of Chapter~\ref{chap:characterization_finite_entropy} to build a correspondence between the competitors of our two problems, namely, of the RUOT problem with $\Psi$ defined through its Legendre transform by~\eqref{eq:def_Psi*}, and the branching Schrödinger problem associated with the reference law $R\sim\BBM(\nu, \boldsymbol{q}, R_0)$. This will shed a new light on Theorem~\ref{thm:equality_values_dyn}, for which we provide another proof by constructing competitors.
	
	The main idea will be to consider the density $\tilde \rho_t := \E_P[M_t]$, $t \in [0,1]$, the drift $\tilde v$ and the growth rate
	\begin{equation*}
	\tilde r := \sum_k (k-1) \tilde q_k
	\end{equation*}
	of a competitor $P$ for the branching Schrödinger problem, defined thanks to Theorem~\ref{thm:representation_competitor_SchBr}, as a competitor for the RUOT problem. The drift and growth rate are not deterministic in general, so one of the points will be to average them. In the other direction, given $(\rho, v, r)$ a competitor for the RUOT problem, we will seek for a competitor $P$ for the branching Schrödinger problem having this density, this drift and this growth rate.
	
	A crucial observation will be to identify $\Psi_{\nu, \boldsymbol q}$ as the optimal value of the entropic cost of changing the branching mechanism when the growth rate is fixed (see Lemma~\ref{lem:Psi_compromise}). In that way, the term involving $h(\tilde{\boldsymbol q}| \boldsymbol q)$ from formulas~\eqref{eq:entropy_smaller_than_RUOT} and~\eqref{eq:RUOT_smaller_than_entropy} is seen to correspond with the one involving $\Psi$ in formula~\eqref{eq:def_ruot}.
	
	In this setting, the fact that the two problems are not strictly equivalent (in the sense that the RUOT problem is only a lower semi-continuous relaxation of the branching Schrödinger problem, as we already saw in Chapter~\ref{chap:duality}) will be visible in the following way. On the one hand, it will always be possible to associate to a competitor $P$ for the branching Schrödinger problem a competitor $(\rho,v,r)$ for the corresponding RUOT problem, with lower value of the objective functional. On the other hand, starting from a competitor $(\rho,v,r)$ for the RUOT problem, we will be able to build a corresponding competitor $P$ for the branching Schrödinger problem only up to regularizing $(\rho,v,r)$. This will be the case both for the static problem (Definition~\ref{def:init}) and for the dynamical one (Definition~\ref{def:BrSch}).
	
	We organize this chapter as follows. In Section~\ref{sec:statement_correspondence}, we state Theorem~\ref{thm:main_result_correspondence}, which is the main result of this chapter. It asserts that there exists a correspondence between the competitors of the RUOT problem and the branching Schrödinger problem. Then, in Section~\ref{sec:proof_main_result_using_equivalence_competitors}, we show how Theorem~\ref{thm:main_result_correspondence} provides a more tractable proof of Theorem~\ref{thm:equality_values_dyn} than the one given in Chapter~\ref{chap:duality}. The rest of the chapter will be devoted to the proof of Theorem~\ref{thm:main_result_correspondence}. In Section~\ref{sec:lemma_Psi}, we show that $\Psi_{\nu, \boldsymbol q}$ appears as the minimal entropic cost of modifying the branching mechanism of the reference BBM at fixed growth rate. In Section~\ref{sec:from_BrSch_to_RUOT}, we give a proof of the easy implication of Theorem~\ref{thm:main_result_correspondence}, namely, that a competitor of the branching Schrödinger problem always provides a competitor of the RUOT problem, with lower value of the objective functional. Finally, the other and harder implication is divided into two parts. In Section~\ref{sec:from_RUOT_to_BrSch_static}, we treat the static case, meaning that we study the equivalence of competitors for the minimization problem of Definition~\ref{def:init}. And lastly in Section~\ref{sec:from_RUOT_to_BrSch_dynamic},  we explain how to regularize a competitor for the RUOT problem in order to find a corresponding competitor for the branching Schrödinger problem.
	
	In the whole Chapter, we will assume both bounds from Assumption~\ref{ass:exponential_bounds}. Therefore, we give a more precise description of the equivalence between the two problems, but assuming a bit more on the parameters of these problem w.r.t.\ what was done with duality arguments in Chapter~\ref{chap:duality}.
	
	\section{Statement of the main result}
	\label{sec:statement_correspondence}
	
	The main result of this section writes as follows. We use the notations of Subsection~\ref{sec:presentation_RUOT}, so that we work in terms of momentums $(\rho, m, \zeta)$ and not in terms of velocity and growth rate $(\rho,v,r)$. Recall that $E_{\Psi}$ denotes the energy functional in RUOT, and that for a given initial law $R_0$, $L^*_{R_0}$ is the Legendre transform of the log-Laplace transform of $R_0$, defined at formula~\eqref{eq:def_L*}.
	
	\begin{Thm}
		\label{thm:main_result_correspondence}
		Let $\nu>0$ be a diffusivity parameter, $\boldsymbol q$ be a branching mechanism, $R_0 \in \P(\M_\delta(\T^d))$ and $R \sim \BBM(\nu,\boldsymbol q,R_0)$. We assume that both condition~\eqref{eq:exponential_moment_q} and condition~\eqref{eq:exponential_moment_R_0} from Assumption~\ref{ass:exponential_bounds} hold, and that $R_0 \neq \delta_0$. Let $\Psi = \Psi_{\nu, \boldsymbol q}$ be defined through its Legendre transform by formula~\eqref{eq:def_Psi*}. Let us call $\mathsf T$ the map that associates to each $P$ with finite entropy w.r.t.\ $R$ the triple 
		\begin{equation*}
		\mathsf T(P) := (\rho, m, \zeta) \in \M_+([0,1] \times \T^d)\times \M([0,1] \times \T^d)^d\times\M([0,1] \times \T^d)
		\end{equation*}
		defined by $\rho := \D t \otimes \rho_t$, $m := \D t \otimes m_t$ and $\zeta := \D t \otimes \zeta_t$, where for almost all $t \in [0,1]$:
		\begin{equation}
		\label{eq:def_rho_m_zeta}
		\rho_t := \E_P[M_t], \quad m_t := \E_P\Big[\tilde v(t, \cdot) M_t\Big], \quad \zeta_t := \E_P\left[  \sum_{k \neq 1} (k-1) \tilde q_k(t, \cdot) M_t \right],
		\end{equation}
		where $\tilde v$ is the drift and $\tilde{\boldsymbol q}$ is the branching mechanism of $P$, obtained using Theorem~\ref{thm:representation_competitor_SchBr}.

		Then, for all $P$ with finite entropy w.r.t.\ $R$, $\mathsf T(P)$ is well defined and it satisfies the equation~\eqref{eq:continuity_linear} with boundary conditions $\rho_0 := \E_P[M_0]$ and $\rho_1:= \E_P[M_1]$ in the sense of Definition~\ref{def:CE}.
		
		Moreover, let us also call $\CE$ the set of $(\rho, m, \zeta)$ satisfying the equation~\eqref{eq:continuity_linear} in a weak sense, $\iota_{\CE}$ the $0/+\infty$ indicator of such set, and recall that $E_\Psi$ is defined in Definition~\ref{def:energy_UOT}. Then the l.s.c.\ envelope of the convex functional
		\begin{equation*}
		(\rho,m,\zeta) \longmapsto \inf\Big\{ \nu H(P|R) \ : \ P \in \P(\Omega) \mbox{ s.t.\ } \mathsf{T}(P) = (\rho, m, \zeta) \Big\},
		\end{equation*}
		(where $\inf \emptyset$ is set to $+\infty$) is $(\rho, m, \zeta) \to \nu L^*(\rho_0) + E_\Psi(\rho,m,\zeta) + \iota_{\CE}(\rho, m, \zeta)$.
	\end{Thm}
	\begin{Rem}
		\begin{itemize}
			\item As already explained in Proposition~\ref{prop:exponential_bounds}, Assumption~\ref{ass:exponential_bounds} is there two ensure that to each competitor for $\BrSch$, we can associate a density.
			\item We already explained at Remark~\ref{rem:R0=delta0} why we have to exlude the case $R_0 = \delta_0$. 
			\item Observe that the initial density $\rho_0$ of a solution $(\rho,m,\zeta)$ of~\eqref{eq:continuity_linear} with finite energy $E_{\Psi}$ is well defined, as a consequence of Lemma~\ref{lem:existence_boundary_CE}.
			\item In the case when $(\rho,m,\zeta)$ comes from a law $P$ with finite entropy w.r.t.\ $R$ through~\eqref{eq:def_rho_m_zeta}, $t \mapsto \rho_t$ needs to be continuous (see Remark~\ref{rem:rho_continuous}), or otherwise stated, $\zeta$ needs to be absolutely continuous w.r.t.\ $\rho$. As developed in Section~\ref{sec:presentation_RUOT}, this is not the case for general solutions of the continuity equation~\eqref{eq:continuity_linear}, even under a finite energy condition, as soon as $\Psi$ is not superlinear at $+\infty$ or $-\infty$. But this is possible for $\Psi = \Psi_{\nu, \boldsymbol q}$ at $+\infty$, even under Assumption~\ref{ass:exponential_bounds}, see Section~\ref{sec:exp_moment}. This is one of the reasons why in the theorem above, the second functional is not equal to the first one, but only its l.s.c.\ envelope.
			\item Given $P$ with $H(P|R) < + \infty$, Remark~\ref{rem:general_charact} provides crucial formulas involving $m$ and $\zeta$ defined in~\eqref{eq:def_rho_m_zeta}, that is, 
			\begin{align}
			\label{eq:charact_m}
			&\forall \xi \in C([0,1]\times \T^d; \R^d),   &\cg \xi,m \cd &= \E_P[I[\xi]_1],\\
			\label{eq:charact_zeta}&\forall \varphi \in C([0,1]\times \T^d),  &\cg \varphi, \zeta \cd &= \E_P\left[ \sum_{k\neq 1} (k-1) J^k[\varphi]_1 \right].
			\end{align}
			(For the latter, we just use the decomposition $\varphi = \varphi_+ - \varphi_-$ to commute the expectation w.r.t.\ $P$ and the sum w.r.t.\ $k$.)
		\end{itemize}
	\end{Rem}
	
	This theorem is a straightforward consequence of Propositions~\ref{prop:from_BrSch_to_RUOT} and~\ref{prop:from_RUOT_to_BrSch} below, so that we decided not to write the proof down completely.
	
	\section{Another proof of Theorem~\ref{thm:equality_values_dyn} under additional assumptions}
	\label{sec:proof_main_result_using_equivalence_competitors}
	Before entering the proof of Theorem~\ref{thm:main_result_correspondence}, let us show that it implies Theorem~\ref{thm:equality_values_dyn} in a straightforward way. Recall that we work in the slightly more restricted case where both bounds from Assumption~\ref{ass:exponential_bounds} hold. In the proof below, we write $\Psi$ as a shortcut for $\Psi_{\nu, \boldsymbol q}$.
	
	\begin{proof}[A proof of Theorem~\ref{thm:equality_values_dyn} using Theorem~\ref{thm:main_result_correspondence}]
		
		First, we prove that for all $\rho_0, \rho_1 \in \M_+(\T^d)$, 
		\begin{equation*}
		\nu L_{R_0}^*(\rho_0) + \ruot_{\nu, \Psi}(\rho_0,\rho_1) \leq \nu \BrSch_{\nu, \boldsymbol q, R_0}(\rho_0,\rho_1).
		\end{equation*}	
		So we take $\rho_0,\rho_1 \in \M_+(\T^d)$, and we assume that $\BrSch_{\nu, \boldsymbol q, R_0}(\rho_0,\rho_1) < + \infty$ (else, there is nothing to prove). In particular, there exists $P$ a competitor for the branching Schrödinger problem, that is a law $P \in \P(\Omega)$ that satisfies $\E_P[M_0] = \rho_0$, $\E_P[M_1] = \rho_1$ and $H(P|R)< + \infty$. Let us call $(\rho,m,\zeta) := \mathsf T(P)$ as defined in Theorem~\ref{thm:main_result_correspondence}. It is a competitor for $\RUOT_{\nu, \Psi}(\rho_0, \rho_1)$. Therefore, we have
		\begin{equation*}
		\nu L_{R_0}^*(\rho_0) + \ruot_{\nu, \Psi}(\rho_0,\rho_1) \leq \nu L_{R_0}^*(\rho_0) + E_\Psi(\rho,m,\zeta) \leq \nu H(P|R),
		\end{equation*}
		where the second inequality is deduced from Theorem~\ref{thm:main_result_correspondence}. Our inequality follows from taking the infimum among all competitors of the branching Schrödinger problem in the r.h.s.
		
		Now, let us show that for all $\rho_0,\rho_1 \in \M_+(\T^d)$, we can find $\rho_0^n \rightharpoonup \rho_0$ and $\rho_1^n \rightharpoonup \rho_1$ such that
		\begin{equation*}
		\limsup_{n \to + \infty} \BrSch_{\nu, \boldsymbol q, R_0}(\rho_0^n, \rho_1^n) \leq \nu L_{R_0}^*(\rho_0) + \ruot_{\nu, \Psi}(\rho_0,\rho_1). 
		\end{equation*}
		So let us take $\rho_0,\rho_1 \in \M_+(\T^d)$, and let us assume that $\nu L^*_{R_0}(\rho_0) + \ruot_{\nu, \Psi}(\rho_0,\rho_1)<+\infty$ (else, there is nothing to prove). Let us consider $(\rho,m,\zeta)$ the solution of $\RUOT_{\nu, \Psi}(\rho_0, \rho_1)$, as given by Theorem~\ref{thm:existence_ruot}. By Theorem~\ref{thm:main_result_correspondence}, we can find $(\rho^n,m^n,\zeta^n)_{n \in \N}$ and $(P_n)_{n\in\N}$ such that $(\rho^n,m^n,\zeta^n) \rightharpoonup (\rho,m,\zeta)$, for all $n \in \N$, $\mathsf T(P_n) = (\rho^n, m^n, \zeta^n)$, and such that
		\begin{equation*}
		\limsup_{n \to + \infty}\nu H(P_n | R) \leq \nu L^*_{R_0}(\rho_0) + E_\Psi(\rho,m,\zeta) = \nu L_{R_0}^*(\rho_0) + \ruot_{\nu, \Psi}(\rho_0,\rho_1).
		\end{equation*}
		For a given $n \in \N$, let us call $\rho_0^n, \rho_1^n$ the boundary conditions of $(\rho^n,m^n,\zeta^n)$. For all $n \in \N$, $P_n$ is a competitor for the branching Schrödinger problem between $\rho_0^n$ and $\rho_1^n$. Hence, we deduce:
		\begin{equation*}
		\limsup_{n \to + \infty}\BrSch_{\nu, \boldsymbol q, R_0}(\rho_0^n, \rho_1^n) \leq \nu L_{R_0}^*(\rho_0) + \ruot_{\nu, \Psi}(\rho_0,\rho_1).
		\end{equation*}
		
		It remains to prove the convergence of $\rho_0^n,\rho^n_1$ towards $\rho_0,\rho_1$, but this is a direct consequence of Lemma~\ref{lem:existence_boundary_CE}.
	\end{proof}

	\section{Optimizing the entropic cost at fixed growth rate}
	\label{sec:lemma_Psi}

	As announced, in this section we show how $\Psi = \Psi_{\nu, \boldsymbol q}$ appears as the optimal entropic cost at fixed growth rate. The result writes as follows.							
	\begin{Lem}
		\label{lem:Psi_compromise}
		Let $\nu>0$, $\boldsymbol q$ be a branching mechanism, and $\Psi_{\nu, \boldsymbol q}$ be defined through its Legendre transform by formula~\eqref{eq:def_Psi*}. We have for all $r \in \R$
		\begin{equation}
		\label{eq:formula_Psi}
		\Psi_{\nu, \boldsymbol q}(r) = \inf\left\{ \nu h(\tilde{\boldsymbol q} | \boldsymbol q)  \ : \  \tilde{\boldsymbol q} \in \M_+(\N), \  \sum_{k} (k-1) \tilde q_k = r \right\} \in \R_+ \cup\{+ \infty \}.
		\end{equation}
	\end{Lem}

	\begin{proof}
		Let us call $\Psi$ the function in the r.h.s.\ of~\eqref{eq:formula_Psi}, it is easy to check that it is convex thanks to the convexity of $h$. We claim that it is also l.s.c. Note that if $\tilde{\boldsymbol q}$ is such that $h(\tilde{\boldsymbol q} | \boldsymbol q) < + \infty$, then using $a \tilde{\boldsymbol q}$ as a competitor for $a \geq 0$, one gets
		\begin{equation*}
		\Psi(ar) \leq a h(\tilde{\boldsymbol q} | \boldsymbol q) + a \log a \lambda_{\tilde{\boldsymbol q}}  + (1-a) \lambda_{\boldsymbol q}.    
		\end{equation*}
		It shows that if $\Psi(r)$ is finite for some $r > 0$ (resp. $r < 0$) then $\Psi$ is finite on $[0, + \infty)$ (resp. $( - \infty, 0]$). Thus the domain where $\Psi$ is finite is either $\R$, $(- \infty, 0]$, $\{ 0 \}$ or $[0, + \infty)$. A convex function is continuous in the interior of the domain where it is finite. Thus, and only in the case where the domain is $(- \infty, 0]$ or $[0, + \infty)$, we also need to check that $\Psi$ is l.s.c.\ at $0$. We will only explain why when the domain is $[0,+\infty)$, the other case being symmetric. Indeed this case is equivalent to $q_0 = 0$. Let us take a sequence $(r_n)_{n \in \N}$ converging to $0$. Without loss of generality we assume $\liminf_n \Psi(r_n) < + \infty$, otherwise there is nothing to prove. For any $n \in \N$, let us find $\tilde{\boldsymbol q}^n$ such that $\Psi(r_n) + 1/n \geq \nu h(\tilde{\boldsymbol q}^n | \boldsymbol q)$ and $\sum_{k} (k-1) \tilde q^n_k = r_n$. In particular the sequence $(h(\tilde{\boldsymbol q}^n | \boldsymbol q))_{n \in \N}$ is bounded, thus up to extraction the sequence of measures $(\tilde{\boldsymbol q^n})_{n \in \N}$ converges weakly to a limit $\tilde{\boldsymbol q}$. As $q_0 = 0$, there holds $\tilde q^n_0 = 0$ for all $n \in \N$ and also $\tilde q_0 = 0$. By Fatou $\sum_{k} (k-1) \tilde q_k \leq 0$, thus it must be equal to $0$. We conclude by l.s.c.\ of $h$ that
		\begin{equation*}
		\Psi(0) \leq \nu h(\tilde{\boldsymbol q} | \boldsymbol q) \leq \liminf_{n \to + \infty}  \nu h(\tilde{\boldsymbol q}^n | \boldsymbol q) \leq \liminf_{n \to + \infty} \Psi(r_n).  
		\end{equation*}
		As a consequence, $\Psi$ is convex and l.s.c.\ thus it is equal to the Legendre transform of its Legendre transform \cite[Proposition 4.1]{ekeland1999convex}, and we only need to show that $\Psi^* = \Psi_{\nu, \boldsymbol q}^*$.
		
		By definition we can rewrite $\Psi^*$ as:
		\begin{equation*}
		\Psi^*(s) = \sup_{r \in \R} rs - \Psi(r) = \sup_{\tilde{ \boldsymbol q} \in \M_+(\N)} s \left( \sum_{k} (k-1) \tilde q_k \right) - \nu h(\tilde{\boldsymbol q} | \boldsymbol q) = \nu \sup_{\tilde{ \boldsymbol q} \in \M_+(\N)} \sum_k q_k \left( \frac{s}{\nu} (k-1) \frac{\tilde q_k}{q_k} - l \left( \frac{\tilde q_k}{q_k} \right) \right).       
		\end{equation*}
		Next, for each $k$ the optimization can be performed directly with equality \eqref{eq:convex_inequality_l}: it yields directly that
		\begin{equation*}
		\sup_{\tilde q_k \geq 0} \left( \frac{s}{\nu} (k-1) \frac{\tilde q_k}{q_k} - l \left( \frac{\tilde q_k}{q_k} \right) \right) = \exp \left(  \frac{s}{\nu} (k-1) \right) -1,
		\end{equation*}
		with equality if and only if $\tilde q_k = q_k \exp((k-1) s/ \nu)$. This $\tilde{\boldsymbol q}$ may not be admissible as $\sum_k (k-1)  q_k$ may be infinite, but we can always use $\tilde q_k = q_k \exp((k-1) s/ \nu) \1_{k \leq K}$ and let $K \to + \infty$. It enables us to sum the suprema in $k$ and get $\Psi^* = \Psi_{\nu, \boldsymbol q}^*$. 
	\end{proof}
	
	\section{From branching Schrödinger to RUOT}
	\label{sec:from_BrSch_to_RUOT}
	In this section, we prove the easy part of Theorem~\ref{thm:main_result_correspondence}, namely that the map $\mathsf T$ sends competitors of the branching Schrödinger problem onto competitors of the RUOT problem with same boundary conditions, and that the cost in the branching Schrödinger problem is greater than the cost in RUOT, up to the functional $L^*_{R_0}$ of the initial datum. This will be a direct consequence of Theorem~\ref{thm:representation_competitor_SchBr} and Proposition~\ref{prop:from_BrSch_to_RUOT}. Here, we do not have to assume that $R_0 \neq \delta_0$.
	\begin{Prop}
		\label{prop:from_BrSch_to_RUOT}
		Let $\nu,\boldsymbol{q},R_0,R,\Psi = \Psi_{\nu, \boldsymbol q}$ be as in Theorem~\ref{thm:main_result_correspondence} (in particular, assume both bounds from Assumption~\ref{ass:exponential_bounds}), and $P \in \P(\Omega)$ be a law with finite entropy w.r.t.\ $R$.
		
		Let us consider $(\rho, m, \zeta) := \mathsf T(P)$ as defined in Theorem~\ref{thm:main_result_correspondence}. Then $(\rho,m,\zeta)$ is well defined, it satisfies equation~\eqref{eq:continuity_linear} with boundary conditions $\rho_0 := \E_P[M_0]$ and $\rho_1 := \E_P[M_1]$, and
		\begin{equation}
		\label{eq:BrSch_to_RUOT}
		\nu L_{R_0}^*(\rho_0) + E_\Psi(\rho,m,\zeta) \leq \nu H(P|R).
		\end{equation}
	\end{Prop}
	\begin{proof}
		Let $\tilde v$, $\tilde{\boldsymbol q}$ be the drift and the branching mechanism of $P$ (they are obtained using Theorem~\ref{thm:representation_competitor_SchBr} thanks to the exponential bound~\eqref{eq:exponential_moment_q} on $\boldsymbol{q}$).
		
		We start by showing that $(\rho,m,\zeta)$ defined by formula~\eqref{eq:def_rho_m_zeta} is well defined in the set
		\begin{equation*}
		\M_+([0,1] \times \T^d)\times \M([0,1] \times \T^d)^d\times\M([0,1] \times \T^d).
		\end{equation*}
		We already observed in Proposition~\ref{prop:exponential_bounds} that $\rho_t = \E_P[M_t]$ is well defined and bounded uniformly in $t$. In particular, $\rho := \D t\otimes \rho_t \in \M_+([0,1]\times \T^d)$ is well defined. Now, let us recall~\eqref{eq:RUOT_smaller_than_entropy}:
		\begin{equation*}
		\nu H(P_0|R_0) + \E_P\left[ \int_0^1 \left\cg \frac{|\tilde v(t)|^2}{2} + \nu h(\tilde{\boldsymbol q}(t)| \boldsymbol q), M_t \right\cd \D t  \right] \leq \nu H(P|R).
		\end{equation*}
		Calling $\tilde r := \sum_k (k-1) \tilde q_k$, we deduce using~\eqref{eq:formula_Psi}:
		\begin{equation}
		\label{eq:ineq_before_average}
		\nu H(P_0|R_0) + \E_P\left[ \int_0^1 \left\cg \frac{|\tilde v(t)|^2}{2} + \Psi_{\nu, \boldsymbol q}(\tilde r), M_t \right\cd \D t  \right] \leq \nu H(P|R).
		\end{equation}
		Let us show that $ m := \D t \otimes m_t \in \M([0,1]\times \T^d)^d$ is well defined. By definition, for all $t$, we clearly have
		\begin{equation*}
		|m_t|(\T^d) \leq \E_P\Big[\cg|\tilde v(t)|,  M_t \cd\Big] \leq \E_P\left[\left\cg\frac{|\tilde v(t)|^2}{2} + \frac{1}{2},  M_t \right\cd\right]  = \E_P\left[\left\cg\frac{|\tilde v(t)|^2}{2} ,  M_t \right\cd\right] + \frac{\rho_t(\T^d)}{2}.
		\end{equation*}
		We get the result by integrating w.r.t.\ $t$, by using~\eqref{eq:ineq_before_average} to bound the first term, and by using the good definition of $\rho$ to bound the second one. The same proof applies in the case of $\zeta$ provided we can find $\kappa>0$ small enough and $K>0$ large enough, depending on the different parameters of the problem, such that for all $\tilde r \in \R$,
		\begin{equation*}
		|\tilde r| \leq \frac{\Psi_{\nu, \boldsymbol q}(r) + K}{\kappa}.
		\end{equation*} 
		To prove this estimate, first take $\kappa:= \kappa_{\boldsymbol q}$ as given by the bound~\ref{eq:exponential_moment_q}. In that way, we have $\Psi^*_{\nu, \boldsymbol q}(\kappa)<+\infty$. Then, observe that by convexity, we have for all $\tilde r$:
		\begin{equation*}
		\kappa \tilde r_+ \leq \Psi_{\nu,\boldsymbol q}(\tilde r_+) + \Psi^*_{\nu, \boldsymbol q}(\kappa) \qquad \mbox{and} \qquad \kappa \tilde r_- \leq \Psi_{\nu,\boldsymbol q}(-\tilde r_-) + \Psi^*_{\nu, \boldsymbol q}(-\kappa).
		\end{equation*}
		Adding up these two estimates, dividing by $\kappa$, and observing that $\Psi_{\nu,\boldsymbol q}(\tilde r_+) + \Psi_{\nu,\boldsymbol q}(-\tilde r_-) = \Psi_{\nu, \boldsymbol q}(r) + \Psi_{\nu, \boldsymbol q}(0)$, we end up with
		\begin{equation*}
		|\tilde r| \leq \frac{\Psi_{\nu, \boldsymbol q}(r) + \Psi_{\nu, \boldsymbol q}(0) + \Psi^*_{\nu, \boldsymbol q}(\kappa) + \Psi_{\nu, \boldsymbol q}^*(-\kappa)}{\kappa}.
		\end{equation*}
		Therefore, the conclusion follows with $K := \Psi_{\nu, \boldsymbol q}(0) + \Psi^*_{\nu, \boldsymbol q}(\kappa) + \Psi_{\nu, \boldsymbol q}^*(-\kappa)$, observing that by definition~\eqref{eq:def_Psi*} of $\Psi_{\nu, \boldsymbol q}^*$, $\Psi_{\nu, \boldsymbol q}^*(-\kappa)< + \infty$, by our choice of $\kappa$, $ \Psi^*_{\nu, \boldsymbol q}(\kappa)<+\infty$, and by a direct computation, $\Psi_{\nu, \boldsymbol q}(0) = - \inf \Psi_{\nu, \boldsymbol q}^* \leq \nu \lambda_{\boldsymbol q}<+\infty$.
		
		Then, we prove~\eqref{eq:BrSch_to_RUOT}. To achieve that, we start from~\eqref{eq:ineq_before_average}, and we first prove that if $\E_{P_0}[M] = \rho_0$, then $L_{R_0}^*(\rho_0) \leq H(P_0|R_0)$. A direct application of~\eqref{eq:convex_ineq_entropy} gives for all $\varphi \in C(\T^d)$:
		\begin{equation*}
		\cg \varphi, \rho_0 \cd = \E_{P_0}[\cg \varphi, M \cd] \leq H(P_0 | R_0) + \log \E_{R_0}\Big[ \exp\Big( \cg \varphi, M \cd \Big) \Big].
		\end{equation*}
		We get the result by putting the log in the l.h.s.\ and by taking the supremum in $\varphi$.
		Finally, let us prove that
		\begin{equation*}
		E_\Psi(\rho,m,\zeta) \leq \E_P\left[ \int_0^1 \left\cg \frac{|\tilde v(t)|^2}{2} + \Psi_{\nu, \boldsymbol q}(\tilde r), M_t \right\cd \D t  \right].
		\end{equation*}
		By Definition~\ref{def:energy_UOT}, it suffices to prove that if $a = a(t,x) \in \R$, $b = b(t,x) \in \R^d$ and $c = c(t,x) \in \R$ are continuous functions satisfying $a + |b|^2/2 + \Psi_{\nu, \boldsymbol q}^*(c) \leq 0$, then
		\begin{equation*}
		\cg a, \rho \cd + \cg b, m \cd + \cg c, \zeta \cd \leq \E_P\left[ \int_0^1 \left\cg \frac{|\tilde v(t)|^2}{2} + \Psi_{\nu, \boldsymbol q}(\tilde r), M_t \right\cd \D t  \right].
		\end{equation*}
		By formula~\eqref{eq:def_rho_m_zeta} defining $(\rho,m,\zeta)$, we have for such $a$, $b$ and $c$:
		\begin{align*}
		\cg a, \rho \cd + \cg b, m \cd + \cg c, \zeta \cd &= \E_P\left[ \int_0^1 \Big\cg a(t) + b(t) \cdot \tilde v(t) + c(t) \tilde r(t) , M_t \Big\cd \D t \right]\\
		&\leq \E_P\left[ \int_0^1 \left\cg  b(t) \cdot \tilde v(t) - \frac{|b(t)|^2}{2} + c(t) \tilde r(t) - \Psi_{\nu, \boldsymbol q}^*(c(t)) , M_t \right\cd \D t \right]\\
		&\leq \E_P\left[ \int_0^1 \left\cg \frac{|\tilde v(t)|^2}{2} + \Psi_{\nu, \boldsymbol q}(\tilde r), M_t \right\cd \D t  \right],
		\end{align*}
		where the second line is obtained using the constraint on $a,b,c$.
		
		It just remains to prove that $(\rho,m,\zeta)$ satisfies~\eqref{eq:continuity_linear} with boundary conditions $\rho_0:= \E_P[M_0]$ and $\rho_1 := \E_P[M_1]$. Let $\varphi = \varphi(t,x)$ be a smooth test function on $[0,1]\times\T^d$. Given the definition of $(\rho,m,\zeta)$ and the weak formulation~\eqref{eq:continuity_weak_form} of~\eqref{eq:continuity_linear}, we need to check
		\begin{equation*}
		\E_P\Big[\cg \varphi(1),M_1\cd\Big] = \E_P\Big[ \cg \varphi(0), M_0\cd \Big] + \E_P\left[ \int_0^1 \left\cg \partial_t \varphi(t) + \frac{\nu}{2} \Delta \varphi(t) + \tilde v(t) \cdot \nabla \varphi(t) + \tilde r(t) \varphi(t) , M_t \right\cd \D t \right]. 	
		\end{equation*}
		But using the Itô-type formula~\eqref{eq:branching_Ito}, we see that we only need to prove
		\begin{equation*}
		\E_P\Big[ I[\nabla \varphi]_1 \Big] = \E_P\left[  \int_0^1 \cg \tilde v(t) \cdot \nabla \varphi(t) , M_t \cd \D t \right] \quad \mbox{and} \quad \E_P\left[ \sum_{k \neq 1} J^k[(k-1)\varphi]_t \right] = \E_P\left[  \int_0^1 \cg \tilde r(t) \varphi(t) , M_t \cd \D t \right].
		\end{equation*}
		But these identities are direct consequences of the formulas~\eqref{eq:charact_m} and~\eqref{eq:charact_zeta}.
	\end{proof}
	
	\section{From RUOT to branching Schrödinger: the static case}
	\label{sec:from_RUOT_to_BrSch_static}
	
	In Theorem~\ref{thm:initial_lsc_envelope} and Section~\ref{sec:static_legendre_transform}, we show that the functional $L^*_{R_0}$ defined by formula~\eqref{eq:def_L*} is the l.s.c.\ envelope of the value of the initial problem given in Definition~\ref{def:init} for the topology of weak convergence. The result that we would like to have in this section would be to find, given a measure $\rho_0 \in \M_+(\T^d)$ satisfying $L^*_{R_0}(\rho_0) < + \infty$, a sequence of measures $(\rho_0^n)_{n \in \N}$ converging weakly towards $\rho_0$ as $n \to + \infty$, and with
	\begin{equation*}
	\limsup_{n \to + \infty} \Init_{R_0}(\rho_0^n) \leq L^*_{R_0}(\rho_0).
	\end{equation*}
	This would be a direct way to prove Theorem~\ref{thm:initial_lsc_envelope}, and good competitors for the problem $\Init_{R_0}(\rho_0^n)$ could be used as initial laws for the approximated competitors for the branching Schrödinger problem that we aim to build out of a competitor of the RUOT problem in the next section.
	
	Yet, the proof we found does not develop exactly in that way. Instead, we prove the following, involving a full BBM, with its dynamical parameters $\nu$ and $\boldsymbol q$, and not only its initial law $R_0$.
	\begin{Prop}
		\label{prop:construction_initial_competitor}
		Let $R \sim \BBM(\nu, \boldsymbol{q},R_0)$ be a BBM and $(\tau_s)_{s>0}$ be the heat kernel. For all $\rho_0 \in \M_+(\T^d)$ and $t \in (0,1]$,
		\begin{equation*}
		\Init_{\bar R_t}(\tau_{\nu t} \ast \rho_0) \leq L^*_{R_0}(\rho_0) + \log R(S(0) > t) + \lambda_{\boldsymbol q} t \rho_0(\T^d),
		\end{equation*}
		where $S(0)$ is the time of the first branching event, defined in Definition~\ref{def:N_S}, and $\bar R_t \in \P(\M_\delta(\T^d))$ is the law of $M_t$ under the conditional law $R(\ \cdot\ | S(0)>t)$.
	\end{Prop}
	Note that we can explicit $R(S(0) > t)$ as it is $\E_{R_0}[ \exp( - t \lambda_{\boldsymbol q} M_0(\T^d) )]$, but we will only use that this quantity converges to $1$ as $t \to 0$. As one can see, not only we regularize $\rho_0$, but we also change slightly the problem that we are looking at: we replace $R_0$ by the law of $M_t$ conditionally on the event $\{S(0)>t\}$. Still, we show that the value of this problem can be controlled by $L^*_{R_0}(\rho_0)$, the actual initial l.s.c.\ relaxation. In Section~\ref{sec:from_RUOT_to_BrSch_dynamic}, we will use the good competitors for the problem $\Init_{\bar R_t}(\tau_{\nu t} \ast \rho_0)$ as laws of the competitors for the branching Schrödinger problem that we build at time $t$, for small value of $t$ and not at time $0$.
	
	We divide the proof of Proposition~\ref{prop:construction_initial_competitor} in two parts, written in separated subsections. In the first part, we show that under a stronger assumption on $R_0$ (namely, the existence of exponential moments at any order), then the functionals $L^*_{R_0}$ and $\Init_{R_0}$ actually coincide. We explain our construction of good competitors for the problem $\Init_{\bar R_t}(\tau_{\nu t} \ast \rho_0)$ in a second part.
	
	\subsection{Equality under an additional assumption}
	\label{subsec:static_no_duality_gap}
	Here, we assume that
	\begin{equation}
	\label{eq:assumption_all_exponential_moments}
	\E_{R_0}\Big[\exp \big(\kappa M(\T^d)\big)\Big] < + \infty, \qquad \forall \kappa>0.
	\end{equation}
	We prove that in that case, $\Init_{R_0} = L^*_{R_0}$ and the infimum in~\eqref{eq:def_init} is achieved. In particular, in that case, $\Init_{R_0}$ is l.s.c. 
	\begin{Thm}
		\label{thm:dual_initial_entropic_minimization}
		Let $R_0 \in \P(\M_\delta(\T^d))$ satisfying~\eqref{eq:assumption_all_exponential_moments}. Then, $\Init_{R_0} = L^*_{R_0}$. Furthermore, for all $\rho_0 \in \M_+(\T^d)$ satisfying $L^*_{R_0}(\rho_0)< +\infty$, the infimum in~\eqref{eq:def_init} is achieved.
	\end{Thm}

	\begin{Rem}
		Here, we prove that the static initial problem, which once again is nothing but an inverse problem in grand canonical classical statistical mechanics, has no duality gap. On the other hand, in~\cite{chayes1984validity}, the authors prove the stronger statement that the dual problem admits a solution (see Remark~\ref{rem:grand_canonical}). To do so, in addition to the specific form~\eqref{eq:grand_canonical_probability} of the reference law, they add the assumption, called ``H-stability'', that there exists $B>0$ such that for all $N \in \N^*$, $w_N \geq -B N$. In particular, as in their setting, the probability to observe exactly $N \in \N^*$ particles is proportional to
		\begin{equation*}
		\frac{1}{N!}\int e^{-w_N} \D \mu^{\otimes N} \leq \frac{e^{BN}}{N!},
		\end{equation*}
		it is easy to see that condition~\eqref{eq:assumption_all_exponential_moments} holds true. Hence, our condition is a generalization of theirs. 
		
		The overall goal of the whole section is to go further and to only assume that $R_0$ has some exponential moments, but not at any order. Up to our knowledge, there is no result for this optimization problem at this level of generality.
	\end{Rem}
	\begin{proof}
		The proof of $L^*_{R_0} \leq \Init_{R_0}$ follows the same lines as the beginning of the proof of Proposition~\ref{prop:legendre_transform_init}.
		
		The other inequality as well as the existence of a minimizer will be an application of the Hahn-Banach theorem in its analytic form, which is the very first result in~\cite{brezis2011functional}, and of a Riesz type theorem in non-homogeneous spaces that we provide in Appendix~\ref{app:riesz}. In this proof and in this proof only, we call $\XX:= \M_+(\T^d)$ and $\G$ its Borel $\sigma$-algebra w.r.t.\ the topology of weak convergence of measures. Notice that it is countably generated.
		
		If $L^*_{R_0}(\rho_0)$ is infinite, then there is nothing to prove. So let us assume that it is finite, and let us call it $C$, so that
		\begin{equation}
		\label{eq:def_C}
		C := \sup_{\sigma \in C(\T^d)} \langle \sigma, \rho_0 \rangle -  \log \E_{R_0} \left[ \exp( \langle \sigma, M \rangle ) \right].
		\end{equation}
		Our goal is to find a probability measure $P_0$ that satisfies:
		\begin{equation}
		\label{eq:condition_P_HB}
		H(P_0|R_0) \leq C \qquad \mbox{and} \qquad  \forall \sigma\in C(\T^d), \quad 	\E_{P_0}\big[ \cg \sigma, M \cd \big] = \cg \sigma, \rho_0 \cd .
		\end{equation}
		For this purpose, let us call $\mathcal{V}\subset L^1(\XX, \G, R_0)$ the following vector space (as in Corollary~\ref{cor:riesz}, with $\pi := R_0$):
		\begin{equation*}
		\mathcal{V} := \Big\{ f:\XX \to \R, \quad \G\mbox{-measurable, s.t. } \forall \kappa >0 \ \E_{R_0}\big[ \exp\big( \kappa |f| \big) \big] < + \infty \Big\},
		\end{equation*}
		and notice that because of~\eqref{eq:assumption_all_exponential_moments}, the vector space
		\begin{equation*}
		\mathcal{W} := \Big\{ f: \XX \to \R \mbox{ of the form } f = \cg \sigma, M \cd  \, : \, \sigma \in C(\T^d) \Big\}.
		\end{equation*}
		is a subspace of $\mathcal{V}$.
		
		As a result of Corollary~\ref{cor:riesz}, finding $P_0$ satisfying~\eqref{eq:condition_P_HB} is the same as finding a linear functional $\Lambda$ on $\mathcal{V}$ that satisfies:
		\begin{itemize}
			\item For all $f \in \mathcal{V}$ the following inequality holds:
			\begin{equation}
			\label{eq:ineq_HB}
			\Lambda(f) \leq \inf_{\kappa > 0} \frac{1}{\kappa} \Big( C + \log \E_{R_0}\big[ \exp\big( \kappa f \big) \big] \Big).
			\end{equation}
			Let us call $\Gamma(f)$ the functional in the r.h.s.
			\item For all $f = \cg \sigma, M \cd  \in \mathcal{W}$, 
			\begin{equation}
			\label{eq:restricted_form}
			\Lambda(f) = \cg \sigma, \rho_0\cd.
			\end{equation}
		\end{itemize} 
		In other terms, we want to extend the linear functional~\eqref{eq:restricted_form} defined on $\mathcal{W}$ to the whole space $\mathcal{V}$ in such a way that inequality \eqref{eq:ineq_HB} holds. By virtue of the Hahn-Banach theorem, this is possible provided inequality~\eqref{eq:ineq_HB} holds on $\mathcal{W}$ and provided $\Gamma$ is real valued, positively homogeneous and subadditive.
		
		Let us check~\eqref{eq:ineq_HB} on $\mathcal{W}$. By definition~\eqref{eq:def_C} of $C$, for all $f = \cg \sigma, M \cd \in \mathcal{W}$ and $\kappa>0$, we have:
		\begin{equation*}
		\cg \sigma, \rho_0\cd = \frac{\cg \kappa\sigma, \rho_0\cd}{\kappa} \leq \frac{1}{\kappa} \Big( C + \log \E_{R_0} \Big[ \exp\Big(\kappa  \cg \sigma, M \cd  \Big)\Big] \Big).
		\end{equation*}
		Taking the infimum among $\kappa$ in the r.h.s. we see that inequality~\eqref{eq:ineq_HB} holds on $\mathcal{W}$.
		
		Let us check the properties of $\Gamma$. By definition of $\mathcal{V}$, we see that for $f \in\mathcal{V}$, $\Gamma(f) < +\infty$, and because for all $\kappa>0$, 
		\begin{equation*}
		\log \E_{R_0}[\exp(\kappa f)] \geq \E_{R_0}[\kappa f ] \geq - \kappa \E_{R_0}[|f|] \geq -\kappa \E_{R_0}\big[\exp\big( |f|\big)\big],
		\end{equation*}
		we also have $\Gamma(f)> - \infty$, so that $\Gamma$ is real valued. The positive homogeneity is obvious. Finally, the subadditivity is a consequence of the convexity of the log-Laplace transform: for all $f_1, f_2 \in \mathcal{V}$, $\kappa_1, \kappa_2>0$, calling $\kappa := \kappa_1 \kappa_2/(\kappa_1 + \kappa_2)$,
		\begin{equation*}
		\frac{1}{\kappa} \Big( C + \log \E_{R_0}\big[ \exp\big( \kappa (f_1 + f_2) \big) \big] \Big) \leq \frac{1}{\kappa_1} \Big( C + \log \E_{R_0}\big[ \exp\big( \kappa_1 f_1 \big) \big] \Big) + \frac{1}{\kappa_2} \Big( C + \log \E_{R_0}\big[ \exp\big( \kappa_2 f \big) \big] \Big).
		\end{equation*}
		Hence, we get the subadditivity of $\Gamma$ by taking the infimum among $\kappa_1,\kappa_2$ in the r.h.s.\ and the result follows.
	\end{proof}
	
	\subsection{Construction of approximated static laws}
	This subsection is devoted to the proof of Proposition~\ref{prop:construction_initial_competitor}. So let us give ourselves a diffusivity parameter $\nu>0$, a branching mechanism $\boldsymbol q$, an initial law $R_0$, the corresponding BBM $R \sim \BBM(\nu, \boldsymbol q, R_0)$ and a measure $\rho_0 \in \M_+(\T^d)$ satisfying $L^*_{R_0}(\rho_0)< + \infty$. Recall that $\bar R_t$ is the law of $M_t$ under the conditional law $R(\ \cdot\ | S(0)>t)$.
	
	\begin{proof}[Proof of Proposition~\ref{prop:construction_initial_competitor}] 
		
		\begin{stepe}{Introduction of auxiliary functionals and structure of the proof}
			In the proof, we will need the following approximated functionals, defined for $t \in (0,1]$ and $n \in \N\cup\{\infty\}$:
			\begin{gather}
			\notag I_t^n(\rho_0) := \inf \Big\{ H(P_t | \bar R_t) \ : \ P_t \in \P(\M_\delta(\T^d)), \ \E_{P_t}[M] = \rho_0 \mbox{ and } M(\T^d) \leq n, \ P_t\mbox{-a.s.} \Big\}, \quad \rho_0 \in \M_+(\T^d),\\
			\label{eq:def_Ltn}
			L_t^n(\alpha) := \log \E_{R_0}\Big[ \exp\cg \log (\tau_{\nu t} \ast \alpha) ,M\cd \1_{M(\T^d) \leq n}\Big],\quad \alpha \in \M_+(\T^d)\backslash\{0\}.
			\end{gather}
			Namely, the main idea of this proof is to restrict ourselves to looking for laws exhibititing a possibly large but uniformly bounded number of particles. As far as the functionals $(L^n_t)_{t,n}$ are concerned, it is convenient for compactness reasons to consider $\alpha$ in a set of measures and not in a set of functions, see Step~\ref{step:Gamma_cv} below. However, we will often mix measure and function notations, so that $\alpha = 1$ and $\alpha = \Leb$ correspond to the same object.  
			
			Given $\rho_0 \in \M_+(\T^d)$ our proof consists in showing the following chain of inequalities:
			\begin{equation}
			\label{eq:chain_of_inequalities}
			\begin{aligned}
			\Init_{\bar R_t}(\tau_{\nu t} \ast \rho_0) &\leq \inf_{n \in \N} I_t^n(\tau_{\nu t} \ast \rho_0) \\
			&\leq \inf_{n \in \N} \sup_{\alpha \in \M_+(\T^d)\backslash\{0\}} \cg \log (\tau_{\nu t} \ast \alpha) , \rho_0 \cd - L_t^n(\alpha) + \log R(S(0)>t) +\lambda_{\boldsymbol q} t \rho_0(\T^d) \\
			&\leq \sup_{\alpha \in \M_+(\T^d)\backslash\{0\}} \cg \log (\tau_{\nu t} \ast \alpha) , \rho_0 \cd - L_t^\infty(\alpha) + \log R(S(0)>t) + \lambda_{\boldsymbol q} t \rho_0(\T^d) \\
			&\leq L_{R_0}^*(\rho_0) + \log R(S(0)>t)+ \lambda_{\boldsymbol q} t \rho_0(\T^d).
			\end{aligned}
			\end{equation}
			The first one is a direct consequence of the trivial inequality $\Init_{\bar R_t} = I^\infty_t \leq I^n_t$, for all $t \in (0,1]$, $n \in \N$. The last one is also direct by definition~\eqref{eq:def_L*} of $L^*_{R_0}$, remarking that for all $\alpha \in \M_+(\T^d) \backslash\{0\}$, $\log (\tau_{\nu t} \ast \alpha) \in C(\T^d)$. The second inequality, proved at Step~\ref{step:no_duality_approximated_problem}, is a consequence of the result proved in Subsection~\ref{subsec:static_no_duality_gap} and of the properties of the branching Brownian motion. Finally, the third inequality, proved at Step~\ref{step:Gamma_cv}, is of $\Gamma$-liminf type, with the subtlety that there is no uniform coercivity in general.
		\end{stepe}
		
		\begin{stepe}{No duality gap and estimate for the approximated problems}
			\label{step:no_duality_approximated_problem}
			Let us fix $t \in (0,1]$, $n \in \N$ large enough for $R_0(M(\T^d) \leq n)$ to be nonzero,  and $\rho_0 \in \M_+(\T^d)$. We prove
			\begin{equation*}
			I_t^n(\tau_{\nu t} \ast \rho_0) \leq \sup_{\substack{\psi \in C(\T^d)\\ \psi>0}} \cg \log (\tau_{\nu t} \ast \psi) , \rho_0 \cd - L_t^n(\psi) + \log R(S(0)>t) +\lambda_{\boldsymbol q} t \rho_0(\T^d).
			\end{equation*}
			The second inequality in~\eqref{eq:chain_of_inequalities} obviously follows. The first thing to observe is
			\begin{equation*}
			I_t^n(\tau_{\nu t} \ast \rho_0) = \inf \Big\{ H(P_t | \bar R_t^n) \ : \ P_t \in \P(\M_\delta(\T^d)), \ \E_{P_t}[M] = \tau_{\nu t} \ast \rho_0 \Big\} - \log R_0(M(\T^d) \leq n),
			\end{equation*}
			where $\bar R_t^n$ stands for the conditional law $\bar R_t(\ \cdot \ | M(\T^d) \leq n) = \bar R_t \1_{M(\T^d) \leq n} / R_0(M(\T^d) \leq n)$ (as there holds $\bar R_t(M(\T^d) \leq n) = R_0(M(\T^d)=n)$). This is just because for all $P_t \in \M_\delta(\T^d)$ whose support is in $\{M(\T^d) \leq n\}$, $H(P_t | \bar R_t^n) = H(P_t | \bar R_t) + \log R_0(M(\T^d) \leq n)$. Therefore, as a consequence of Theorem~\ref{thm:dual_initial_entropic_minimization}, we have
			\begin{align*}
			I_t^n(\tau_{\nu t} \ast \rho_0) &= \sup_{\varphi \in C(\T^d)} \cg \varphi, \tau_{\nu t} \ast \rho_0 \cd - \log \E_{\bar R_t^n}\Big[ \exp\cg \varphi, M\cd \Big] - \log R_0(M(\T^d) \leq n)  \\
			&= \sup_{\varphi \in C(\T^d)} \cg \tau_{\nu t} \ast \varphi , \rho_0 \cd - \log \E_{\bar R_t}\Big[ \exp\cg \varphi, M\cd \1_{M(\T^d) \leq n}\Big].
			\end{align*}
			Now, we change the variables according to $\psi = \exp \varphi$. As $\log$ is concave, $\tau_{\nu t} \ast \varphi \leq \log (\tau_{\nu t} \ast \psi)$, so that $\cg \tau_{\nu t} \ast \varphi , \rho_0 \cd \leq \cg \log(\tau_{\nu t} \ast \psi) , \rho_0 \cd$. We end up with
			\begin{align*}
			I_t^n(\tau_{\nu t} \ast \rho_0) &\leq \sup_{\substack{\psi \in C(\T^d) \\\psi>0 }} \cg \log (\tau_{\nu t} \ast \psi) , \rho_0 \cd -\log \E_{\bar R_t}\Big[ \exp\cg \log \psi, M\cd \1_{M(\T^d) \leq n}\Big] \\
			&= \sup_{\substack{\psi \in C(\T^d) \\\psi>0 }} \cg \log (\tau_{\nu t} \ast \psi) , \rho_0 \cd -\log \E_{\bar R_t}\Big[ \exp\cg \log (e^{\lambda_{\boldsymbol q} t}\psi), M\cd \1_{M(\T^d) \leq n}\Big] + \lambda_{\boldsymbol q} t \rho_0(\T^d),
			\end{align*}
			where the second line is obtained changing the variables according to $\psi \mapsfrom e^{-\lambda_{\boldsymbol q} t}\psi$
			
			Hence, to conclude, we jut need to prove that for all $\psi\in C(\T^d)$ with $\psi>0$, calling $\varphi :=  \log (e^{\lambda_{\boldsymbol q} t}\psi)$,
			\begin{equation}
			\label{eq:conditioning_to_not_branching}
			\E_{\bar R_t}\Big[ \exp\cg \varphi, M\cd \1_{M(\T^d) \leq n}\Big] = \frac{1}{R(S(0)>t)} \E_{R_0}\Big[ \exp\cg \log(\tau_{\nu t} \ast \psi), M\cd \1_{M(\T^d) \leq n}\Big] . 
			\end{equation}
			To do so, we use Definitions~\ref{def:BBM_deterministic_R0} and~\ref{def:BBM_general_R0} of the branching Brownian motion. We have
			\begin{align*}
			\E_{\bar R_t}\Big[ \exp\cg  \varphi, M\cd \1_{M(\T^d) \leq n}\Big] &= \E_R\Big[ \exp\cg \varphi, M_t\cd \1_{M_t(\T^d) \leq n}\Big| S(0)>t\Big]\\
			&= \frac{1}{R(S(0)>t)}\E_R\Big[ \exp\cg \varphi, M_t\cd \1_{M_t(\T^d) \leq n}\1_{S(0)>t}\Big]\\
			&= \frac{1}{R(S(0)>t)} \E_{R_0}\Big[ \Em_M\Big[ \exp\cg \varphi, M_t\cd \1_{M_0(\T^d) \leq n}\1_{S(0)>t}\Big] \Big]\\
			&= \frac{1}{R(S(0)>t)} \E_{R_0}\Big[ \Em_M\Big[ \exp\cg \varphi, M_t\cd \1_{S(0)>t}\Big]\1_{M(\T^d) \leq n}\Big].
			\end{align*}
			(In the third inequality, we used the $R$-\emph{a.s.}\ equality $\1_{M_t(\T^d) \leq n}\1_{S(0)>t} = \1_{M_0(\T^d) \leq n}\1_{S(0)>t}$.) Hence, identity~\eqref{eq:conditioning_to_not_branching} follows from the following identity, valid for all $\mu = \delta_{x_1} + \dots + \delta_{x_p} \in \M_\delta(\T^d)$, asserting that conditionally on the event that there is no branching, $R^\mu$ is the law of a population of independent Brownian particles:
			\begin{align*}
			\Em_\mu\Big[ \exp\cg \varphi, M_t\cd \1_{S(0)>t}\Big] &= R^\mu\big(S(0)>t\big)\Em_\mu\Big[ \exp\cg \varphi, M_t\cd \Big|S(0)>t\Big] \\
			&= \exp(-p \lambda_{\boldsymbol q} t) \prod_{i=1}^p \tau_{\nu t} \ast \big(\exp(\varphi)\big)(x_i) \\
			&= \prod_{i=1}^p \tau_{\nu t} \ast \psi(x_i)=\exp\Big\cg \log(\tau_{\nu t} \ast \psi), \mu\Big\cd .
			\end{align*}
		\end{stepe}
		\begin{stepe}{Taking the limit $n \to + \infty$}
			\label{step:Gamma_cv}
			Here, we prove the third inequality in~\eqref{eq:chain_of_inequalities}, that is, reformulating it, for all $t \in (0,1]$ and $\rho_0 \in \M_+(\T^d)$:
			\begin{equation}
			\label{eq:gamma_liminf}
			\inf_{\alpha \in \M_+(\T^d) \backslash \{ 0 \}}L_t^\infty(\alpha) - \cg \log (\tau_{\nu t} \ast \alpha) , \rho_0 \cd\leq \sup_{n \in \N} \inf_{\alpha \in \M_+(\T^d)\backslash\{0\}} L_t^n(\alpha) -\cg \log (\tau_{\nu t} \ast \alpha) , \rho_0 \cd.
			\end{equation}
			
			As the sequence $(L_t^n)_n$ is nondecreasing, the $\sup$ in the r.h.s.\ is a nondecreasing limit, and we can use a diagonal argument to find a sequence of measures $(\alpha_n)_{n \in \N}$ such that
			\begin{equation*}
			\lim_{n \to + \infty} L_t^n(\alpha_n) - \cg \log (\tau_{\nu t} \ast \alpha_n) , \rho_0 \cd = \sup_{n \in \N} \inf_{\alpha \in \M_+(\T^d)\backslash\{0\}} L_t^n(\alpha) -\cg \log (\tau_{\nu t} \ast \alpha) , \rho_0 \cd.
			\end{equation*}
			Note that for all $n \in \N$, if $\alpha = 1$, $L_t^n(\alpha) = \cg \log (\tau_{\nu t} \ast \alpha), \rho_0 \cd = 0$. Hence, the r.h.s.\ in \eqref{eq:gamma_liminf} is nonpositive, and up to redefining $\alpha_n$ to be $1$ when this is not the case, we can assume that for all $n \in \N$,
			\begin{equation}
			\label{eq:nonpositive_limit}
			L_t^n(\alpha_n) \leq \cg \log (\tau_{\nu t} \ast \alpha_n), \rho_0 \cd.
			\end{equation}
			
			Up to extraction, we can suppose that we are in one of the following cases: (i) $\alpha_n$ converges weakly to some $\alpha \in \M_+(\T^d)$ with $\alpha(\T^d) > 0$, (ii) $\alpha_n(\T^d) \to 0$, and (iii) $\alpha_n(\T^d) \to + \infty$. Let us treat these cases one by one.
			
			The convergence in case (i), ensures that $\log (\tau_{\nu t} \ast \alpha_n)$ converges uniformly towards $\log (\tau_{\nu t} \ast \alpha)$. Consequently, $\cg \log (\tau_{\nu t} \ast \alpha_n) , \rho_0 \cd$ converges towards $\cg \log (\tau_{\nu t} \ast \alpha) , \rho_0 \cd$ and by Fatou's lemma and the definition~\eqref{eq:def_Ltn} of $L^n_t$, $L_t^\infty (\alpha) \leq \liminf_{n} L_t^n(\alpha_n)$. Inequality~\eqref{eq:gamma_liminf} follows.
			
			Let us treat case (ii).	In that case, we use also $\alpha_n$ as competitors in the l.h.s.\ in~\eqref{eq:gamma_liminf}:
			\begin{align*}
			\inf_{\alpha \in \M_+(\T^d) \backslash \{ 0 \}}L_t^\infty(\alpha) - &\cg \log (\tau_{\nu t} \ast \alpha) , \rho_0 \cd\\ &\leq \limsup_{n \to + \infty}L_t^\infty(\alpha_n) - \cg \log (\tau_{\nu t} \ast \alpha_n) , \rho_0 \cd \\
			&\leq \lim_{n \to + \infty}  L_t^n(\alpha_n) - \cg \log (\tau_{\nu t} \ast \alpha_n) , \rho_0 \cd + \limsup_{n \to + \infty} L_t^\infty(\alpha_n) - L_t^n(\alpha_n)\\
			&= \sup_{n \in \N} \inf_{\alpha \in \M_+(\T^d)\backslash\{0\}} L_t^n(\alpha) -\cg \log (\tau_{\nu t} \ast \alpha) , \rho_0 \cd + \limsup_{n \to + \infty} L_t^\infty(\alpha_n) - L_t^n(\alpha_n).
			\end{align*}
			Hence, it suffices to prove that $\limsup_n L_t^\infty(\alpha_n) - L_t^n(\alpha_n) \leq 0$. We call $c_n := \log \alpha_n(\T^d)$, $n \in \N$. By assumption, this sequence tends to $-\infty$ as $n \to + \infty$. Let $K>0$ be such that pointwise, $e^{-K}\leq \tau_{\nu t} \leq e^K$. Observe that $c_n - K\leq \log(\tau_{\nu t} \ast \alpha_n) \leq c_n + K$.
			
			Let us go back to the definition~\eqref{eq:def_Ltn} of $L_t^n$. Let $n_0$ be such that $R_0(M(\T^d)\leq n_0)>0$. For $n \geq n_0$ sufficiently large to have $c_n + K \leq 0$, using $\log b - \log a \leq (b-a)/a$, we have
			\begin{align*}
			L_t^\infty(\alpha_n) - L_t^n(\alpha_n) 
			&\leq \frac{\E_{R_0}\Big[ \exp\Big\cg \log (\tau_{\nu t} \ast \alpha_n) ,M\Big\cd \1_{M(\T^d) > n}\Big]}{\E_{R_0}\Big[ \exp\Big\cg \log (\tau_{\nu t} \ast \alpha_n) ,M\Big\cd \1_{M(\T^d) \leq n}\Big]} \\
			& \leq \frac{\E_{R_0}\Big[ \exp\Big((c_n + K) M(\T^d)\Big) \1_{M(\T^d) > n}\Big]}{\E_{R_0}\Big[ \exp\Big((c_n - K) M(\T^d)\Big) \1_{M(\T^d) \leq n}\Big]} \\
			&\leq \frac{\exp\Big( n(c_n+K) \Big) }{R_0(M(\T^d) = n_0) \exp\Big( n_0(c_n-K) \Big)} \\
			&= \frac{\exp(2n_0 K)}{R_0(M(\T^d) = n_0) } \exp\Big( (n-n_0) (c_n + K) \Big) \underset{n \to + \infty}{\longrightarrow} 0,
			\end{align*}
			and the result follows.
			
			It remains to treat case (iii). Let us take $n_0 \in \N^*$ larger that $\rho_0(\T^d)$. We will prove that $R_0(M(\T^d) \geq n_0) = 0$. If this is true, $L^\infty_t = L^{n_0}_t$, so that~\eqref{eq:gamma_liminf} is obvious. 
			
			We define $c_n$, $n \in \N$ and $K$ as in case (ii), and we exploit~\eqref{eq:nonpositive_limit}. This time $c_n \to + \infty$. Let $n \geq n_0$ be sufficiently large to have $c_n - K \geq 0$. Direct estimates provide 
			\begin{equation*}
			\log \Big(R_0\Big(n_0\leq M(\T^d) \leq n\Big) \Big) + n_0(c_n - K)  \leq L_t^n(\alpha_n) \leq  \cg \log (\tau_{\nu t} \ast \alpha_n), \rho_0 \cd \leq \rho_0(\T^d)(c_n + K). 
			\end{equation*}
			This inequality implies
			\begin{equation*}
			\log \Big(  R_0\Big(n_0 \leq M(\T^d) \leq n\Big) \Big) \underset{n \to + \infty}{\longrightarrow} - \infty,
			\end{equation*}
			so that $R_0(M(\T^d) \geq n_0) = 0$, as announced.	
		\end{stepe}
	\end{proof}
	
	\section{From RUOT to branching Schrödinger: the dynamical case}
	\label{sec:from_RUOT_to_BrSch_dynamic}
	
	In this section, we conclude the proof of Theorem~\ref{thm:main_result_correspondence} by proving that for all competitor $(\rho, m,\zeta)$ of the RUOT problem, one can associate a sequence of competitors for the branching Schrödinger problem whose moments converge towards $(\rho,m,\zeta)$, and with convergence of the corresponding functionals. This is the meaning of the following proposition.
	\begin{Prop}
		\label{prop:from_RUOT_to_BrSch}
		Let $\nu,\boldsymbol q,R_0,R,\Psi = \Psi_{\nu, \boldsymbol q}$ be as in Theorem~\ref{thm:main_result_correspondence} (in particular, assume both bounds from Assumption~\ref{ass:exponential_bounds}, and suppose that $R_0 \neq \delta_0$), and let $(\rho,m, \zeta)$ be a solution of equation~\eqref{eq:continuity_linear} with $\nu L_{R_0}^*(\rho_0) + E_\Psi(\rho,m, \zeta)< + \infty$. Then there exist laws $(P^\eps)_{\eps >0}$ in $\P(\Omega)$ satisfying
		\begin{equation}
		\label{eq:Peps_recovery}
		\mathsf T(P^\eps) \underset{\eps \to 0}{\rightharpoonup} (\rho, m, \zeta) \qquad \mbox{and}\qquad \limsup_{\eps \to 0}\nu H(P^\eps | R) \leq \nu L_{R_0}^*(\rho_0) + E_{\Psi}(\rho,m,\zeta).
		\end{equation}
	\end{Prop}
	
	\begin{Rem}
		\label{rem:from_RUOT_to_BrSch}
		Clearly, what we would like to do to prove this proposition would be to define as a competitor $P$ the law of a BBM starting from the optimizer of~\eqref{eq:def_init}, of drift $v = m / \rho$, and of branching branching mechanism given by the optimizer of~\eqref{eq:formula_Psi} with $r = \zeta/\rho$. But there are several obstacles to achieve this:
		\begin{itemize}
			\item There is no optimizer for~\eqref{eq:def_init} in general, and there is possibly a duality gap. Therefore, we exploit Proposition~\ref{prop:construction_initial_competitor}.
			\item There is no optimizer for~\eqref{eq:formula_Psi} in general, so we have to consider parameters that are almost optimal.
			\item The drift and branching mechanism obtained by the procedure just described do not satisfy the bounds required in Theorem~\ref{thm:formula_RN_derivative} to define a corresponding branching Brownian motion. Hence, we have to mollify them using Lemma~\ref{lem:mollification}.
		\end{itemize}
	\end{Rem}
	\begin{proof}
		\begin{stepf}{Definition of approximated fields}
			We consider a unique regularization parameter $\eps>0$, that will take into account the three obstacles described above in Remark~\ref{rem:from_RUOT_to_BrSch}. 
			
			First, we consider the family $(\rho^\eps, m^\eps, \zeta^\eps)_{\eps>0}$ as given by Lemma~\ref{lem:mollification}, with $\bar \rho_0 := \E_{R_0}[M]$ (which is nonzero because we assumed $R_0 \neq \delta_0$), $\bar r := \sum_k (k-1) q_k$ and $\bar \nu = \nu$. Observe that with these choices, we have indeed $\Psi(\bar r) = 0$ (for instance by choosing $\tilde{\boldsymbol q} = \boldsymbol q$ in~\eqref{eq:formula_Psi}). The fourth point of Lemma~\ref{lem:mollification} provides, with the notations of Proposition~\ref{prop:energy_UOT},
			\begin{equation}
			\label{eq:ineq_mollified}
			\limsup_{\eps \to 0} E_\Psi(\rho^\eps, m^\eps, \zeta^\eps) = \limsup_{\eps \to 0} \int_0^1\hspace{-5pt} \int \left\{ \frac{|v^\eps(t,x)|^2}{2} + \Psi_{\nu, \boldsymbol q}(r^\eps(t,x)) \right\} \D\rho^\eps(t,x) \leq E(\rho,m,\zeta).
			\end{equation}
			
			Now, we define branching mechanisms corresponding to these mollified fields. First, for a given $(t,x)$, we aim to define $\boldsymbol q^\eps(t,x)$ a branching mechanism such that:
			\begin{equation}
			\label{eq:choice_of_q}
			\sum_{k} (k-1) q_k^\eps(t,x) = r^\eps(t,x), \qquad  \nu h( \boldsymbol q^\eps(t,x) |\boldsymbol q) \leq \Psi_{\nu, \boldsymbol q}(r^\eps(t,x))  + \eps.
			\end{equation}
			Because of Lemma~\ref{lem:Psi_compromise}, such a $\boldsymbol q^\eps(t,x)$ is easy to find for all $t$ and $x$, but the subtlety is that we need to do it in a measurable way w.r.t.\ $t$ and $x$. We do it by taking $\boldsymbol q^\eps(t,x) := \hat{\boldsymbol q}(r(t,x))$, where $\hat{\boldsymbol q} = \hat{\boldsymbol q}(r)$ is a measurable function of $r$ that we define now. For a given $r \in \R$ and $n \in \N$, we define:
			\begin{equation*}
			\hat{\boldsymbol q}^n(r) := \argmin\bigg\{ h(\tilde{\boldsymbol q} | \boldsymbol q) \ : \ \tilde{\boldsymbol q} \in \M_+(\N), \, \sum_{k} (k-1)\tilde q_k = r \mbox{ and }\tilde q_k = 0, \, \forall k >n \bigg\}.
			\end{equation*}
			(Such minimizers exists as soon as there exists $k \leq n$ with $q_k \neq 0$.) It is easy to check that $r \mapsto \hat{\boldsymbol q}^n(r)$ is measurable, and exploiting Lemma~\ref{lem:Psi_compromise}, we can see that for all $r \in \R$, $h(\hat{\boldsymbol q}^n(r) | \boldsymbol q)$ is a nonincreasing sequence of limit $\Psi_{\nu, \boldsymbol q}(r)$. Now, for a given $r$, let $n^\eps(r)$ be the integer defined by:
			\begin{equation*}
			n^\eps(r) := \min \Big\{ n \in \N \ : \ h(\hat{\boldsymbol q}^n(r) | \boldsymbol q) \leq \Psi_{\nu, \boldsymbol q}(r)  + \eps\Big\}.
			\end{equation*}
			Once again, $r \mapsto n^\eps(r)$ is measurable. Therefore, we conclude by taking $\hat{\boldsymbol q}(r) := \hat{\boldsymbol q}^{n^\eps(r)}(r)$.

			Observe that the second inequality in~\eqref{eq:choice_of_q} together with the third point of Lemma~\ref{lem:mollification} guarantee that $\boldsymbol{q}^\eps$ satisfies the assumptions of Theorem~\ref{thm:formula_RN_derivative}. As $v^\eps$ is bounded (still by the third point of Lemma~\ref{lem:mollification}), by Theorem~\ref{thm:formula_RN_derivative}, we know how to define BBMs with drift $v^\eps$ and branching mechanism $\boldsymbol q^\eps$.
		\end{stepf}
		
		\begin{stepf}{Definition of corresponding laws}
			Observe that by the fifth point of Lemma~\ref{lem:mollification}, we have $\rho_0^\eps = \tau_{\nu\eps} \ast((1-\eps) \rho_0 + \eps \bar \rho_0)$, for all $\eps >0$. For a given $\eps \in (0,1)$, we call $\bar R^\eps$ the conditional law $R(\ . \ | S(0)>\eps)$, and $\bar R_\eps$ the law of $M_\eps$ under $\bar R^\eps$. According to Proposition~\ref{prop:construction_initial_competitor}, there exists $P^\eps_\eps \in \P(\M_+(\T^d))$ such that
			\begin{gather}
			\label{eq:initial_intensity_Qeps}\E_{P^\eps_\eps}[M] = \rho^\eps_0 = \tau_{\nu \eps} \ast \big((1-\eps)\rho_0 + \eps \bar \rho_0 \big) , \\
			\label{eq:estimate_entropy_initial_law} H(P^\eps_\eps | \bar R_\eps) \leq L_{R_0}^* \big( (1-\eps)\rho_0 + \eps \bar \rho_0 \big) + \log R(S(0)>\eps) + \lambda \eps \rho_0(\T^d) + \eps \leq L^*_{R_0}(\rho_0) + \underset{\eps \to 0}{o}(1),
			\end{gather}
			where the second inequality is obtained by convexity of $L^*_{R_0}$, by the trivial equality $L^*_{R_0}(\bar \rho_0) = 0$ for our choice of $\bar \rho_0 := \E_{R_0}[M]$, and by $R(S(0) > \eps) \to 1$ as $\eps \to 0$. 
			
			Let us call 
			\begin{equation*}
			\begin{aligned}
			\Theta_\eps: \cadlag([0,1]; \M_+(\T^d)) &\longrightarrow \cadlag([0, 1-\eps]; \M_+(\T^d))\\
			\Big(t \in [0,1] \mapsto \omega_t\Big) & \longmapsto \Big( t \in [0, 1-\eps] \mapsto \omega_{t + \eps} \Big)
			\end{aligned}
			\end{equation*}
			the truncation functional, and $\hat R^\eps := \Theta_\eps \pf \bar R^\eps$. Adapting slightly Theorem~\ref{thm:formula_RN_derivative} in order to allow different initial and final times than $0$ and $1$, we can find a BBM $Q^\eps \in \P(\cadlag([0, 1-\eps]; \M_+(\T^d)))$ starting from $P^\eps_\eps$, of drift $v^\eps$ and branching mechanism $\boldsymbol{q}^\eps$ (restricted to times $[0, 1-\eps]$), satisfying the inequality
			\begin{equation}
			\label{eq:entropy_Qeps}
			\nu H(Q^\eps | \hat R^\eps) \leq \nu H(P^\eps_\eps | \bar R^\eps) + \E_{Q^\eps}\left[ \int_0^{1-\eps} \left\cg \frac{| v^\eps(t)|^2}{2} + \nu h(\boldsymbol q^\eps(t)|\boldsymbol q) , M_t \right\cd \D t  \right].
			\end{equation}
			Finally, we define $P^\eps$ as
			\begin{equation}
			\label{eq:def_Peps}
			P^\eps := \left( \frac{\D Q^\eps}{\D \hat R^\eps} \circ \Theta_\eps \right)\cdot \bar R^\eps.
			\end{equation}
			To give an idea of what this law is, one could check that $P^\eps$ is the minimizer of $H(P|\bar R^\eps)$ under constraint $\Theta_\eps \pf P = Q^\eps$.
			
			From now on, the goal is to show~\eqref{eq:Peps_recovery}. We start by estimating the relative entropy of $(P^\eps)_{\eps>0}$, and prove the weak convergence of $(\mathsf T(P^\eps))_{\eps>0}$ towards $(\rho, m, \zeta)$ in further steps. 
		\end{stepf}
		\begin{stepf}{Estimate of the entropic cost}
			\label{step:entropy_estimate}
			A quick computation shows that $H(P^\eps|R) = H(P^\eps | \bar R^\eps) - \log R(S(0)>\eps) = H(Q^\eps | \hat R^\eps) + o_{\eps \to 0}(1)$. So in what follows, we will only be interested in estimating $H(Q^\eps | \hat R^\eps)$, $\eps>0$.
			
			Observe that for all $t \in [0, 1-\eps]$, the density of $Q^\eps$ at time $t$, namely, $\hat \rho^\eps_t :=\E_{Q^\eps}[M_t]$, coincides with $\rho^\eps_t$. The reason is that as the drift and branching mechanism of a process are proved to be unique in Theorem~\ref{thm:representation_competitor_SchBr}, just as $\rho^\eps$, $\hat \rho^\eps$ solves distributionally
			\begin{equation*}
			\left\{
			\begin{gathered}
			\partial_t \hat\rho^\eps + \Div (\hat\rho^\eps v^\eps) = \frac{\nu}{2} \Delta \hat\rho^\eps + \hat\rho^\eps r^\eps,\\
			\hat \rho^\eps_t|_{t = 0} = \rho^\eps_0,
			\end{gathered}
			\right.
			\end{equation*}
			for which there is uniqueness when $v^\eps$ and $r^\eps$ are smooth (the initial condition comes from~\eqref{eq:initial_intensity_Qeps}). In particular, inequality~\eqref{eq:entropy_Qeps} rewrites
			\begin{equation*}
			\nu H(Q^\eps | \hat R^\eps) \leq \nu H(P^\eps_\eps | \bar R^\eps) + \int_0^{1-\eps}\hspace{-5pt}\int \left\{ \frac{| v^\eps(t)|^2}{2} + \nu  h(\boldsymbol{q}^\eps(t) | \boldsymbol{q})\right\} \D \rho^\eps_t \D t .
			\end{equation*}
			Plugging~\eqref{eq:choice_of_q} and~\eqref{eq:estimate_entropy_initial_law} into this last inequality, and as $H(P^\eps|R) = H(Q^\eps|\hat R^\eps)+ \underset{\eps \to 0}{o}(1)$, we get
			\begin{equation*}
			\nu H(P^\eps | R) \leq \nu L^*_{R_0}(\rho_0) + \int_0^{1-\eps}\hspace{-5pt}\int \left\{ \frac{| v^\eps(t)|^2}{2} +  \Big( \Psi_{\nu, \boldsymbol q}(r^\eps(t)) + \eps\Big) \right\} \D \rho^\eps_t \D t + \underset{\eps \to 0}{o}(1) .
			\end{equation*}
			Now, we can use the estimate of the action of mollified fields~\eqref{eq:ineq_mollified}, and get
			\begin{equation*}
			\nu H(P^\eps | R) \leq  \nu L^*_{R_0}(\rho_0) + \int_0^{1-\eps}\hspace{-5pt}\int \left\{ \frac{| v(t)|^2}{2} +   \Psi_{\nu, \boldsymbol q}(r(t))  \right\} \D \rho_t \D t + \eps \sup_{t} \rho^\eps_t(\T^d)+ \underset{\eps \to 0}{o}(1).
			\end{equation*}
			But we can estimate $\sup_{t} \rho^\eps_t(\T^d)$ using Proposition~\ref{prop:control_mass_from_energy} and~\eqref{eq:initial_intensity_Qeps}, so we deduce the inequality in~\eqref{eq:Peps_recovery}.
		\end{stepf}
		\begin{stepf}{Computation of $\mathsf T(P^\eps)$, $\eps>0$}
			We already mentioned that for all $\eps \in (0,1)$, $\Theta_\eps \pf P^\eps = Q_\eps$. Let us prove it. Let us take $\eps \in (0,1)$, and $Z$ be a bounded random variable on $\cadlag([0, 1-\eps]; \M_+(\T^d))$. We have
			\begin{equation*}
			\E_{\Theta_\eps \pf P^\eps}[Z] = \E_{P^\eps}[Z\circ \Theta_\eps] = \E_{\bar R^\eps}\left[ Z\circ \Theta_\eps \frac{\D Q^\eps}{\D \hat R^\eps} \circ \Theta_\eps \right] = \E_{\Theta_\eps \pf \bar R^\eps}\left[ Z \frac{\D Q^\eps}{\D \hat R^\eps} \right] = \E_{\hat R^\eps}\left[ Z \frac{\D Q^\eps}{\D \hat R^\eps} \right] = \E_{Q_\eps}[Z],
			\end{equation*}
			where the second equality follows from the definition~\eqref{eq:def_Peps} of $P^\eps$, and the third one follows from the definition $\hat R^\eps := \Theta_{\eps} \pf \bar R^\eps$. 
			
			On the other hand, it is clear from the definition of $Q_\eps$ and from the observation that we already made about the density associated with $Q_\eps$ that the solution of~\eqref{eq:continuity_linear} associated with $Q^\eps$ by formula~\eqref{eq:def_rho_m_zeta} (with $t \in [0,1-\eps]$ instead of $[0,1]$), and that we still call $\mathsf T(Q^\eps)$, is nothing but $(\rho^\eps, m^\eps, \zeta^\eps)$.
			
			Let us deduce from these two points that calling $\mathsf T(P^\eps) = (\tilde \rho^\eps, \tilde m^\eps, \tilde \zeta^\eps) = \D t \otimes(\tilde \rho^\eps_t, \tilde m^\eps_t, \tilde \zeta^\eps_t)$, we have for all $t \in [\eps,1]$:
			\begin{equation}
			\label{eq:shift}
			(\tilde \rho^\eps_t, \tilde m^\eps_t, \tilde \zeta^\eps_t) = (\rho^\eps_{t-\eps}, m^\eps_{t-\eps}, \zeta^\eps_{t-\eps}).
			\end{equation}
			In the case of the densities, this is direct using the fact that with slightly abusive notations, for all $t \in [\eps, 1]$, $M_{t} = M_{t-\eps} \circ \Theta_{\eps}$. To treat the case of the net source of mass, let us consider $\varphi \in C([0,1]\times \T^d; \R^d)$. Let us call $\varphi^\eps(t,x) := \varphi(t+\eps,x)$ for all $\eps \in (0,1)$, $t \in [0, 1-\eps]$, $x \in \T^d$. It is easy to see that $R$-\emph{a.s.}, and hence $P$-\emph{a.s.}, 
			\begin{equation}
			\label{eq:shift_J}
			\sum_{k \neq 1} J^k[(k-1) \varphi]_1 - \sum_{k\neq 1} J^k[(k-1)\varphi]_\eps = \sum_{k\neq 1} J^k[\xi^\eps]_{1-\eps}\circ \Theta_\eps.
			\end{equation} 
			Therefore, if $\varphi$ has its support in $[\eps, 1]\times \T^d$, integrating the previous formula w.r.t.\ $P$ and making use of~\eqref{eq:charact_zeta}, we find
			\begin{equation*}
			\cg \varphi, \zeta \cd = \E_P\left[ \sum_{k\neq 1} J^k[\xi^\eps]_{1-\eps}\circ \Theta_\eps \right] = \E_{Q^\eps}\left[ \sum_{k\neq 1} J^k[\varphi^\eps]_{1-\eps} \right] = \cg \varphi^\eps, \tilde \zeta^\eps\cd,
			\end{equation*}
			and the result follows. Finally, the case of the momentum is treated in a similar way, even though in this case, the formula corresponding to~\eqref{eq:shift_J} is more delicate to prove (one needs for instance to use approximations similar to the ones in the proof of Proposition~\ref{prop:I_adapted_restricted} below, but we leave the details to the reader.)
			
			Lastly, because of~\eqref{eq:ineq_mollified}, the entropy estimates that we got at Step~\ref{step:entropy_estimate} and Proposition~\ref{prop:from_BrSch_to_RUOT}, we have
			\begin{equation}
			\label{eq:limsup_energy}
			\limsup_{\eps \to 0}E_\Psi( \rho^\eps,  m^\eps, \zeta^\eps) < + \infty \qquad \mbox{and} \qquad
			\limsup_{\eps \to 0}E_\Psi(\tilde \rho^\eps, \tilde m^\eps, \tilde \zeta^\eps) < + \infty.
			\end{equation}
			
			We will see that these finiteness properties together with~\eqref{eq:shift} and the coercivity estimates from Proposition~\ref{prop:control_mass_from_energy} will be enough to prove the convergence of $(\tilde \rho^\eps, \tilde m^\eps, \tilde \zeta^\eps)$ towards $( \rho^\eps,  m^\eps, \zeta^\eps)$. Note that in virtue of Remark~\ref{rem:Psi_nu_q_grows_at_m_infty}, $\Psi = \Psi_{\nu, \boldsymbol q}$ satisfies Assumption~\ref{ass:ruot_weak}.
		\end{stepf}
		
		\begin{stepf}{Convergence}
			In order to prove the convergence of $(\tilde \rho^\eps, \tilde m^\eps, \tilde \zeta^\eps)$ towards $(\rho,m,\zeta)$, the first thing to observe is that both $\limsup_{\eps \to 0} \rho^\eps_1(\T^d)$ and $\limsup_{\eps \to 0} \tilde \rho^\eps_1(\T^d)$ are finite so that all the bounds from Proposition~\ref{prop:control_mass_from_energy} apply. In the case of $(\rho^\eps_1)_{\eps>0}$, this is for instance because of the continuity property stated in Lemma~\ref{lem:existence_boundary_CE}, and in the case of $(\tilde \rho^\eps_1)_{\eps>0}$, this is for instance because of Proposition~\ref{prop:exponential_bounds}.
			
			We start by proving that $\tilde \rho^\eps \rightharpoonup \rho$. For this let us take $\varphi \in C([0,1] \times \T^d)$ and $\eps \in (0,1)$. We have, using the notation $\varphi^\eps(t,x) := \varphi(t+\eps, x)$, $t \in [0, 1-\eps]$ and $x \in \T^d$, and formula~\eqref{eq:shift}
			\begin{equation*}
			\int_0^1 \hspace{-5pt} \int \varphi \D \tilde \rho^\eps = \int_0^1\hspace{-5pt} \int \varphi \D \rho^\eps + \int_0^\eps \hspace{-5pt} \int \varphi \D \tilde \rho^\eps + \int_0^{1-\eps}\hspace{-5pt} \int \{ \varphi^\eps - \varphi \} \D \rho^\eps - \int_{1-\eps}^1  \int \varphi \D \rho^\eps.
			\end{equation*}
			The first term in the r.h.s.\ converges towards $\iint \varphi \D \rho$ by definition of $\rho^\eps$. All the other ones converge towards~$0$, because of the estimate~\eqref{eq:uniform_bound_density} of densities in terms of energies, these energies being bounded thanks to~\eqref{eq:limsup_energy}.
			
			Then, we prove $\tilde m^\eps \rightharpoonup m$. So we take  $\xi \in C([0,1] \times \T^d; \R^d)$, and as before, $\xi^\eps(t,x) := \xi(t+\eps, x)$ for all $\eps \in (0,1)$, $t \in [0, 1-\eps]$ and $x \in \T^d$. By formula~\eqref{eq:shift}, we have for all $\eps \in (0,1)$
			\begin{equation*}
			\int_0^1 \hspace{-5pt} \int \xi \cdot \D \tilde m^\eps = \int_0^1\hspace{-5pt} \int \xi\cdot \D m^\eps + \int_0^\eps \hspace{-5pt} \int \xi\cdot \D \tilde m^\eps + \int_0^{1-\eps}\hspace{-5pt} \int \{ \xi^\eps - \xi \}\cdot \D m^\eps - \int_{1-\eps}^1  \int \xi\cdot \D m^\eps.
			\end{equation*}
			The first term in the r.h.s.\ converges towards $\iint \xi \cdot \D m$ by definition of $m^\eps$, and all the other ones converge towards $0$ because of the estimate~\eqref{eq:bound_velocity_momentum} of the density momentums in terms of energies, that are bounded thanks to~\eqref{eq:limsup_energy}.
			
			Finally, in the case of the net source of mass $\zeta^\eps \rightharpoonup \zeta$, we cannot use~\eqref{eq:bound_growth_momentum_new} instead of~\eqref{eq:bound_velocity_momentum} for the boundary terms, as doing so, we lose the uniform control in $(t_1 - t_0)$ of the corresponding mass (see Remark~\ref{rem:coercivity}). So we will rely on our regularization procedure and our construction. As before, for a given $\varphi \in C([0,1]\times \T^d)$ and $\eps \in (0,1)$, by formula~\eqref{eq:shift}
			\begin{equation*}
			\int_0^1 \hspace{-5pt} \int \varphi \D \tilde \zeta^\eps = \int_0^1\hspace{-5pt} \int \varphi \D \zeta^\eps + \int_0^\eps \hspace{-5pt} \int \varphi \D \tilde \zeta^\eps + \int_0^{1-\eps}\hspace{-5pt} \int \{ \varphi^\eps - \varphi \} \D \zeta^\eps - \int_{1-\eps}^1  \int \varphi \D \zeta^\eps.
			\end{equation*}
			The first term in the r.h.s.\ converges towards $\iint \varphi \D \zeta$ by definition of $\zeta^\eps$, the third one converges towards~$0$ because of the estimate~\eqref{eq:bound_growth_momentum_new} of the net sources of mass in terms of energies, that are bounded thanks to~\eqref{eq:limsup_energy}. The last one cancels as a consequence of Point~\ref{item:fields} from Lemma~\ref{lem:mollification}: we restricted $\zeta^\eps$ to the set of times $[0,1-\eps]$ precisely for that reason. Therefore, our last task is to show that the second term in the r.h.s.\ tends to $0$. To do this, let us observe that because of~\eqref{eq:charact_zeta},
			\begin{equation*}
			\int_0^\eps \hspace{-5pt} \int \varphi \D \tilde \zeta^\eps = \E_{P^\eps}\left[ \sum_{k\neq 1} (k-1) J^k[\varphi]_\eps \right].
			\end{equation*}
			But we were smart enough to choose $P^\eps \ll \bar R^\eps$, which is a law with no branching event up to time $\eps$. Hence, $P^\eps$-\emph{a.s.} for all $k \in \N \backslash\{1\}$, $J^k[\varphi]_\eps=0$, so that actually, the quantity above cancels. The result follows, which closes our proof.
		\end{stepf}
	\end{proof}

	\chapter{Additional results: asymptotics, generalizations, numerics}
	\label{chap:additional_results}

	In this chapter we present results which are not directly about the link between RUOT and the Branching Schrödinger problem but that we still think are interesting and that are rather easy to present given the theory we built above. Specifically we will cover: a small noise asymptotic (Section~\ref{sec:smallNoise}), the generalization to more general measure-valued branching Markov processes (Section~\ref{sec:general_superproc}), and numerical aspects (Section~\ref{sec:numerics}). All the three sections can be read independently. 
	
	\section{Small noise limit}
	\label{sec:smallNoise}
	
	In this section we investigate what happens when one reduces the ``noise intensity'' of the system, that is, the diffusivity as well as the branching rate. To that end, we work only on the optimal transport formulation. We first identify with heuristic arguments the regime of parameters in which we expect a limit with both transport and growth, and then we proceed to prove a rigorous $\Gamma$-limit at the level of the optimal transport problem. We refer the reader to \cite{braides2002gamma} for an introduction to the theory of $\Gamma$-convergence, which deals with convergence of problems of calculus of variations. 
	
	We recall that in the rest of the present work we work with $(q_k)_{k \in \N}$, the temporal rates of having a branching event with $k$ offsprings. To investigate the limit, we switch back to write $q_k = \lambda p_k$ where $\boldsymbol{p} = (p_k)_{k \in \N} \in \P(\N)$ is a probability distribution on $\N$, while $\lambda = \lambda_{\boldsymbol q} > 0$ is the rate of having a branching event. The reason is that $\lambda$ becomes a parameter which will vary, while $\boldsymbol p$ is kept fixed.  
	
	\subsection{Heuristic derivation}
	
In the non-branching case, the small noise limit has been investigated extensively by various authors~\cite{mikami2004monge,leonard2012schrodinger,carlier2017convergence,baradat2020small}, and in the limit $\nu \to 0$, the problem converges to plain optimal transport. In the pure branching case \emph{i.e.}\ in the case where there is no spatial trajectories, the limit $\lambda \to 0$ is less classical, but could be deduced using the arguments developed in~\cite{leonard2016lazy}, and the growth penalization would become linear (here, the analogy is formal, see footnote~\ref{footnote:lazy}). Here, as already said in the introduction, in order to have a small noise limit where both transport and growth play a role, one has to let both $\nu$ and $\lambda$ go to $0$, at a controlled rate. Because of the linear growth penalization, the limiting problem is the partial optimal transport~\cite{caffarelli2010free,figalli2010optimal,chizat2018unbalanced}. A careful study of the aforementioned bibliography provides the regime to consider, as explained in the following remark. In what follows, we think of $\lambda$ as a function of $\nu$, and we analyze the limit $\nu \to 0$.

\begin{Rem}
		\label{rem:regime_small_noise}
		Having a close look into the two papers~\cite{leonard2012schrodinger} and~\cite{leonard2016lazy} we see that:
		\begin{itemize}
			\item If $R^\nu$ is the law of a Brownian motion of diffusivity $\nu>0$, the right quantity to consider to study the limit $\nu \to 0$ is $\nu H( \cdot | R^\nu)$, see~\cite[Corrolary~3.2]{leonard2012schrodinger} with $\nu = 1/k$. In other terms, the entropic cost for perturbing $R^\nu$ in order to modify its final marginal at order $1$ is like $1/\nu$.
			\item If now $R^\lambda$ is the law of a pure branching process of branching rate $\lambda>0$, seen as a pure-jump process on the discrete set $\N$, the right quantity to consider when studying the limit $\lambda \to 0$ is $H(\cdot|R^\lambda)/(-\log \lambda)$, see~\cite[Theorem~2.7]{leonard2016lazy} with $\lambda = 1/k$. This time\footnote{\label{footnote:lazy} Let us stress that the analogy is not perfect, as the results of~\cite{leonard2016lazy} cannot be applied directly for pure branching processes, for instance because for such processes, the graph structure on $\N$ would need to be directed, which is not included in the assumptions of~\cite{leonard2016lazy}. However, most of the arguments are easily adapted, and the analogy is striking: for instance, condition (2.2) from there translates naturally into Assumption~\ref{ass:small_noise} for us.}, the entropic cost for perturbing $R^\lambda$ in order to modify its final marginal at order $1$ is like $-\log \lambda$.
		\end{itemize}
	Therefore, as in our case, we want the entropic cost of perturbing the Brownian trajectories and of changing the branching mechanism to be at the same order, we need to take $\lambda$ and $\nu$ in such a way that $1/\nu \propto -\log \lambda$, that is, we want $\lambda$ to be of the form $\exp(-\varpi / \nu)$, for some $\varpi >0$. 
	\end{Rem}

Let us recover this regime analytically. Let us keep $\boldsymbol{p}$ fixed, call $\Phi_{\boldsymbol{p}}$ the generating function of $\boldsymbol{p}$ and recall that the Legendre transform of the growth penalization $\Psi = \Psi_{\lambda,\nu}$ reads, once one expresses it in terms of $\boldsymbol{p}$ instead of $\boldsymbol{q}$,
\begin{equation*}
\Psi^*_{\lambda,\nu}(s) = \nu \lambda \Big( \Phi_{\boldsymbol p} \big(e^{s/\nu}\big) e^{-s/\nu} - 1 \Big) =  \nu \lambda \left( \sum_{k=0}^{+ \infty}  p_k \exp \left\{ \frac{(k-1)s}{\nu} \right\}  - 1 \right) .
\end{equation*}
We stick to a simple probability distribution $\boldsymbol{p}$ by assuming that $p_0 > 0$ and $\sum_{k \geq 2} p_k > 0$: particles can die and split. In this case, one sees that:
\begin{itemize}
	\item If for any $\varpi>0$, $\nu \lambda e^{\varpi/\nu}$ does not converge towards zero as $\nu \to 0$ (thus it converges to $+ \infty$), then for all $s \neq 0$, 
	\begin{equation*}
	\lim_{\nu \to 0, \lambda \to 0} \Psi^*_{\lambda,\nu}(s) = + \infty. 
	\end{equation*} 
	We expect growth to be free in this limit.
	\item On the other hand, if for any $\varpi>0$, $\nu \lambda e^{\varpi/\nu}$ converges towards zero as $\nu \to 0$, then for all $s \in \R$,
	\begin{equation*}
	\lim_{\nu \to 0, \lambda \to 0} \Psi^*_{\lambda,\nu}(s) = 0.
	\end{equation*}
	We expect growth to be forbidden in this limit.
\end{itemize}	
	Thus, at least heuristically, the interesting regime is the one where $\nu \lambda \sim e^{-\varpi/\nu}$ for some $\varpi > 0$. Let us denote by $k_l$ the largest integer such that $p_{k_l} > 0$ and assume that $k_l < + \infty$. In this case, as represented in Figure~\ref{fig:bar_Psi}
	\begin{equation}
	\label{eq:small_noise_Psi*} 
	\lim_{\nu \to 0} \Psi^*_{\nu^{-1}\exp(-\varpi/\nu),\nu}(s) = \bar{\Psi}^*(s) = 
	\begin{cases}
	+\infty & \text{if } s < -\varpi,\\
	p_0 & \text{if } s = -\varpi,\\
	0 & \text{if } -\varpi < s < \frac{\varpi}{k_l-1},\\
	p_{k_l} & \text{if } s = \frac{\varpi}{k_l-1},\\
	+\infty &\text{if } s > \frac{\varpi}{k_l-1}.
	\end{cases} 
	\end{equation}	
	\begin{figure}
		\centering
		\begin{tikzpicture}
		\draw[->] (-3.3,0) -- (2,0)node[below right]{$s$};
		\draw[->] (0,-.3) -- (0,2.5)node[right]{$\bar \Psi^*(s)$};
		\draw (1.5, -2pt) -- (1.5, 2pt) node[below=5pt]{$\frac{\varpi}{k_l-1}$};
		\draw (-3, -2pt) -- (-3, 2pt) node[below=5pt]{$-\varpi$};
		\draw (-2pt, 1.5) -- (2pt, 1.5) node[left]{$p_{k_l}$};
		\draw (-2pt, .8) -- (2pt, .8) node[right]{$p_0$};
		\draw[very thick, blue] (-3,0) -- (1.5,0);
		\fill[blue] (1.5,1.5) circle (3pt);
		\fill[blue] (-3,.8) circle (3pt);
		\draw[thick, dashed, blue] (-3,0) -- (-3,.8) -- (0,.8);
		\draw[thick, dashed, blue] (1.5,0) -- (1.5,1.5) -- (0,1.5);
		\draw (0,-3pt) node[below left]{$0$};
		\end{tikzpicture}
		\hspace{50pt}
		\begin{tikzpicture}
		\draw[->] (-2.3,0) -- (2.3,0)node[below right]{$r$};
		\draw[->] (0,-.3) -- (0,2.5)node[right]{$(\bar \Psi^*)^*(r)$};
		\draw (2, -2pt) -- (2, 2pt) node[below=5pt]{$1$};
		\draw (-2, -2pt) -- (-2, 2pt) node[below=5pt]{$-1$};
		\draw (-2pt, 1) -- (2pt, 1) node[left]{$\frac{\varpi}{k_l-1}$};
		\draw (-2pt, 2) -- (2pt, 2) node[right]{$\varpi$};
		\draw[very thick, blue] (-2.2,2.2) -- (0,0) -- (2.2,1.1);
		\draw[thick, dashed, blue] (-2,0) -- (-2,2) -- (0,2);
		\draw[thick, dashed, blue] (2,0) -- (2,1) -- (0,1);
		\draw (0,-3pt) node[below left]{$0$};
		\end{tikzpicture}
		\caption{\label{fig:bar_Psi} Representation of the limiting $\bar \Psi^*$ on the left, and of its Legendre transform $(\bar\Psi^*)^*$ on the right.}
	\end{figure}
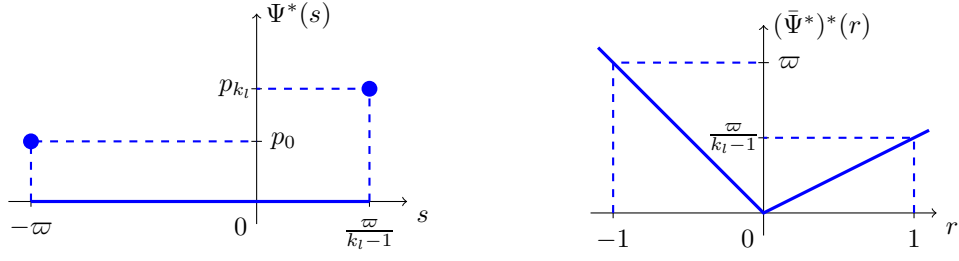
	As the reader can see, this function is not lower semi-continuous. However, as proved in Lemma \ref{lem:cv_Psi} below, what is important to understand the limit of $\Psi$ is to compute the $\Gamma$-limit of $\Psi^*_{\nu^{-1}\exp(-\varpi/\nu),\nu}$, and the latter is the lower semi continuous envelope of the function above. In any case, the Legendre transform of the limit function $\bar{\Psi}^*$ is 
	\begin{equation*}
	(\bar{\Psi}^*)^* : r \mapsto  \begin{cases}
	\frac{\varpi}{k_l-1} r & \text{if } r \geq 0, \\
	- \varpi r & \text{if } r \leq 0.
	\end{cases}
	\end{equation*}	
	In particular, if $k_l = 2$, that is if $p_0, p_2 > 0$ but $p_k = 0$ for $k \geq 3$, then there holds $(\bar{\Psi}^*)^*(r) = \varpi |r|$. 
	
	Sticking to this case $k_l = 2$, we expect that in the limit $\nu \to 0$ and $\lambda = \frac{1}{\nu}\exp(-\varpi/\nu)$ the dynamical formulation of RUOT should converge to 	
	\begin{equation*}
	\min_{\rho, v, r} \left\{ \int_0^1 \hspace{-5pt} \int \left[ \frac{1}{2} |v(t,x)|^2  +  \varpi|r(t,x)| \right] \rho(t,x) \D x \D t \ : \ \dr_t \rho + \Div(\rho v) =  r \rho  \right\},
	\end{equation*}	
	that is a (unregularized) model of unbalanced optimal transport. We emphasize that this is \emph{not} unbalanced optimal transport with quadratic penalization of the growth rate. On the other hand, such model with a $L^1$ penalization of $r$ is equivalent to \emph{partial optimal transport} \cite[Section 5.1]{chizat2018unbalanced}.
	
	In the rest of this section	we prove rigorously the $\Gamma$-convergence of the dynamical formulation of RUOT to the limit formulation that we have just guessed when $\lambda =\frac{1}{\nu} \exp(-\varpi/\nu)$. It is unclear for us if one can make sense of the convergence at a probabilistic level like what was done in the balanced case in \cite{baradat2020small}, thus we stick to the PDE formulation.

	\subsection{Statement of the convergence result}
	
	In the rest of this section we stick to the following assumptions.
	
	\begin{Ass}
		\label{ass:small_noise}
		We assume that the law $\boldsymbol p = (p_k)_{k \in \N}$ of the offsprings, which is a probability distribution over $\N$, has \emph{bounded support}. We assume that $p_0 > 0$, we denote by $k_l$ the largest integer $k$ such that $p_{k_l} > 0$, and we assume that $k_l \geq 2$.    
	\end{Ass}
	
	We fix $\varpi > 0$ and a sequence $(\nu_n)_{n \in \N}$ decreasing to $0$, and we define for all $n \in \N$ the branching rate $\lambda_n := \nu_n^{-1} \exp(-\varpi/\nu_n)$. Let us define on $\R$ the growth penalization $\Psi_n = ( \Psi_n^* )^*$ by
	\begin{equation*}
	\Psi_n^* : s \mapsto \nu_n \lambda_n \left( \Phi_{\boldsymbol p}(e^{s/\nu_n}) e^{-s/\nu_n} - 1 \right) = \sum_{k=0}^{k_l} p_k \left( \exp \left\{ \frac{(k-1)s - \varpi}{\nu_n} \right\} - \exp \left\{ - \frac{ - \varpi}{\nu_n} \right\} \right).
	\end{equation*}
	We have to identify the limit. The pointwise computation has already been done in~\eqref{eq:small_noise_Psi*}. 
	
	\begin{Lem}
		\label{lem:cv_Psi_star}
		Let $\Psi_\infty^* : \R \to [0, + \infty]$ be the convex l.s.c.\ function defined by 
		\begin{equation*}
		\Psi^*_\infty(s) = \begin{cases}
		0 & \text{if } s \in \displaystyle{\left[ -\varpi, \frac{\varpi}{k_l - 1} \right]}, \\
		+ \infty & \text{otherwise}.
		\end{cases}
		\end{equation*}
		Then there holds 
		\begin{equation*}
		\Gamma-\lim_{n \to + \infty} \Psi^*_n = \Psi^*_\infty.
		\end{equation*}
		Moreover, the convergence of $(\Psi_n^*)_{n \in \N}$ is uniform on any compact included in the interior of $\{ \Psi_\infty^* = 0 \}$.
	\end{Lem} 
	
	\noindent We already identified the simple limit of $\Psi_n^*$, so the new parts of this lemma are the statement of $\Gamma$-convergence and the uniform convergence on the compacts included in the interior of $\{ \Psi_\infty^* = 0 \}$. Both are left as an exercise to the reader. From this we can define and identify the limit of the sequence $(\Psi_n)_{n \in \N}$. 
	
	\begin{Lem}
		\label{lem:cv_Psi}
		Let $\Psi_\infty : \R \to [0, + \infty]$ be the function defined by
		\begin{equation*}
		\Psi_\infty(r) = \begin{cases}
		\displaystyle{\frac{\varpi}{k_l - 1} r} & \text{if } r \geq 0, \\
		- \varpi r  & \text{if } r \leq 0.
		\end{cases}
		\end{equation*}
		Then $\Psi_\infty = \left( \Psi_\infty^* \right)^*$, and, $\lim_n \Psi_n = \Psi_\infty$ pointwise.
		In addition, for all sequence $(r_n)_{n \in \N}$ converging $r \in \R$,
		\begin{equation}
		\label{eq:zz_gammalimsup_psi}
		\limsup_{n \to + \infty} \Psi_n(r_n) \leq \Psi_\infty(r).
		\end{equation}
		Eventually, there exists a constant $C > 0$ such that for all $n \geq 1$ and all $r \in \R$,
		\begin{equation}
		\label{eq:control_Psi_n_uniform}
		\frac{1}{C}(|r|-C)_+ \leq \Psi_n(r) \leq C(1 + r^2).
		\end{equation}
	\end{Lem} 
	
	\begin{proof}
		The fact that $\Psi_\infty = \left( \Psi_\infty^* \right)^*$ is a direct computation. 
		
		Then, if $r$ is fixed, note that $\Psi_n(r) = - \inf_{s} ( \Psi_n^*(s) - rs)$. But for a fixed $r$, the functions $s \mapsto \Psi_n^*(s) - rs$ $\Gamma$-converge to $s \mapsto \Psi_\infty^*(s) - rs$ thanks to the previous lemma, thus their infima converge \cite[Section~1.5]{braides2002gamma}. It yields pointwise convergence .
		
		Then, we move first to the bounds~\eqref{eq:control_Psi_n_uniform}. Let us prove that $\Psi_n$ is at most quadratic at infinity, uniformly in $n \in \N$. Observe the following uniform below bound for $\Psi^*_n$, $n\in \N$: 
		\begin{equation*}
		\forall s \in \R, \qquad \Psi_n^*(s) \geq - 1 +
		\begin{cases}
		\displaystyle{p_0 \exp \left( -\frac{s+\varpi}{ \nu_0} \right)} & \text{if } s \leq -\varpi, \\
		0 & \text{if } s \in \displaystyle{\left( -\varpi, \frac{\varpi}{k_l - 1} \right)}, \\
		p_{k_l}  \displaystyle{\exp\left( \frac{(k_l-1)s-\varpi}{ \nu_0} \right)} & \text{if } \displaystyle{s \geq \frac{\varpi}{k_l - 1}},
		\end{cases}
		\end{equation*}	
		Using the fact that for all $x \geq 0$, $\exp x \geq x^2/2$, we find that
		\begin{equation*}
		\forall s \in \R, \qquad\Psi_n^*(s) \geq - 1 +
		\begin{cases}
		\displaystyle{\frac{p_0}{2\nu_0^2} \Big(s+\varpi\Big)^2} & \text{if } s \leq -\varpi, \\
		0 & \text{if } s \in \displaystyle{\left( -\varpi, \frac{\varpi}{k_l - 1} \right)}, \\
		\displaystyle{\frac{p_{k_l}}{2\nu_0^2} \Big((k_l-1)s-\varpi \Big)^2} & \text{if } \displaystyle{s \geq \frac{\varpi}{k_l - 1}},
		\end{cases}
		\end{equation*}	
		The r.h.s.\ is a convex function which does not depend on $n$ and whose Legendre transform is quadratic at infinity. As each $\Psi_n$ is smaller than this Legendre transform, we obtain the upper bound in \eqref{eq:control_Psi_n_uniform}. 
		
		Now we aim to prove the lower bound in~~\eqref{eq:control_Psi_n_uniform}: that $\Psi_n$ is at least linear at infinity, uniformly in $n\in \N$. Consider the following uniform upper bound for $\Psi^*_n$, $n \in \N$:
		\begin{equation*}
		\forall s \in \R, \qquad \Psi_n^*(s) \leq
		\begin{cases}
		+ \infty & \text{if } s \leq -\varpi, \\
		1 & \text{if } s \in \displaystyle{\left( -\varpi, \frac{\varpi}{k_l - 1} \right)}, \\
		+ \infty & \text{if } \displaystyle{s \geq \frac{\varpi}{k_l - 1}}.
		\end{cases}
		\end{equation*}	
		The Legendre transform of the r.h.s.\ grows linearly, thus $\Psi_n$ grows at least linearly, uniformly in $n$. That is, there exists $C > 0$ such that $\Psi_n(r) \geq C^{-1}(|r|-C)$ for all $r$. We then take the positive part in this inequality, by positivity of $\Psi_n$ (as $\Psi_n^*(0) = 0$) and we end up wit the lower bound of~\eqref{eq:control_Psi_n_uniform}.
		
		Eventually, we prove~\eqref{eq:zz_gammalimsup_psi}. Let us take $(r_n)_{n \in \N}$ a sequence converging to $r \in \R$. Thanks to~\eqref{eq:control_Psi_n_uniform} the sequence $(\Psi_n(r_n))_{n \in \N}$ must be bounded. For each $n\in \N$ let us consider a number $s_n \in \R$ such that $\Psi_n(r_n) = r_n s_n - \Psi_n^*(s_n)$: such a number exists (and is unique) thanks to the smoothness of $\Psi_n^*$. The upper bound in~\eqref{eq:control_Psi_n_uniform} implies that for all $n$ and all $s \in \R$, $\Psi_n^*(s) \geq \frac{s^2}{4C} - C$. Therefore, $(s_n)_{n \in \N}$ needs to be bounded, and up to extraction, we assume that it converges to some $s \in \R$. As $\Psi_n^*(s_n) = r_n s_n - \Psi_n(r_n)$, it is also bounded thus  $\Psi^*_{\infty}(s) < + \infty$ by the $\Gamma$-liminf property. Given its expression, it means that $\Psi^*_\infty(s) = 0$. Thus:
		\begin{equation*}
		\Psi_\infty(r) \geq r s = (rs - r_n s_n) + r_n s_n =   (rs - r_n s_n) + \Psi_n(r_n) + \Psi^*_n(s_n) \geq (rs - r_n s_n) + \Psi_n(r_n)  + \inf \Psi_n^*. 
		\end{equation*}
		In the right hand side, the first term goes to $0$ as $n \to + \infty$, and so does the last one as $\inf \Psi_n^* \geq - \exp(- \varpi/\nu_n)$. Thus sending $n \to + \infty$ in the inequality above yields the result.
	\end{proof}
	
	The properties proven in Lemmas \ref{lem:cv_Psi_star} and \ref{lem:cv_Psi} are the only ones needed in the rest of this section. The particular expressions for $\Psi_n$ and $\Psi_n^*$ are not relevant in the sequel. 
	
	\bigskip
	
	Let us introduce some notations in order to state our convergence result. If $n \in \N \cup \{ + \infty \}$ and if $\rho_0, \rho_1 \in \M_+(\T^d)$, we define $\CE_n(\rho_0, \rho_1)$ as the set of all triples $(\rho,m,\zeta)$ satisfying in the weak sense of Definition \ref{def:CE} the continuity equation with diffusivity $\nu_n$ and boundary conditions $\rho_0, \rho_1$. We recall that $\iota_A$ denotes the $0/+\infty$ indicator of a set, see Section~\ref{sec:notations}.

	\begin{Thm}
		\label{Thm:small_noise_limit}
		Let $\rho_0, \rho_1 \in \M_+(\T^d)$ with $h(\rho_1|\Leb) < + \infty$. Under Assumption~\ref{ass:small_noise} and with the notations above, on the space $\M([0,1] \times \T^d)\times \M([0,1] \times \T^d)^d \times \M([0,1] \times \T^d)$ endowed with the topology of weak convergence, 
		\begin{equation*}
		\Gamma-\lim_{n \to + \infty} \left( E_{\Psi_n} + \iota_{\CE_n(\rho_0, \rho_1)} \right) = E_{\Psi_\infty} + \iota_{\CE_\infty(\rho_0, \rho_1)}.
		\end{equation*}
	\end{Thm}
	
	\begin{Rem}
		The estimate~\eqref{eq:control_Psi_n_uniform} combined with the coercivity studied in Proposition~\ref{prop:control_mass_from_energy} yields that the functionals $E_{\Psi_n} + \iota_{\CE_n(\rho_0, \rho_1)}$ are equi-coercive for the topology of weak convergence. In particular, the theory of $\Gamma$-convergence yields that if $(\rho^n, m^n, \zeta^n)$ is optimal for $\RUOT_{\nu_n,\Psi_n}(\rho_0, \rho_1)$, then the sequence $(\rho^n, m^n, \zeta^n)_{n \in \N}$ has at least an accumulation point, and these accumulation points are optimal for $\RUOT_{0,\Psi_\infty}(\rho_0, \rho_1)$. 
	\end{Rem}
	
	The rest of this section is devoted to the proof of this theorem. Though the $\Gamma-\liminf$ is rather straightforward, the explicit computation of a competitor for the $\Gamma-\limsup$ requires some additional estimates in particular to handle temporal boundary conditions.

	At some point we will need to consider the RUOT problem defined between an initial and a final instant $t_0 \leq t_1$ which are not $0$ and $1$, as defined in Definition~\ref{def:RUOT}.

\subsection{Proof of the \texorpdfstring{$\Gamma$}{G}-convergence result}
	
	Before diving into the proof of Theorem~\ref{Thm:small_noise_limit}, let us first prove a technical estimate. We will need a competitor for the value of the RUOT problem whose energy is controlled by more standard distances. This is the object of the next result, which can be seen in fact as a collection already known results from regularized (balanced) optimal transport and unbalanced optimal transport. We denote by $W_2$ the quadratic Wasserstein distance on $\P(\T^d)$, see for instance \cite[Chapter 7]{ambrosio2008gradient}. It metrizes the weak convergence of probability measures \cite[Proposition 7.1.5]{ambrosio2008gradient}.
	
	\begin{Prop}
		\label{prop:control_ruot}
		Take $\Psi$ a convex function such that $\Psi(r) \leq  C(1+r^2)$ for some $C > 0$.
		Let $\alpha, \beta \in \M_+(\T^d)$. Then for any $t_0 < t_1$, there exists $(\rho,m,\zeta) \in \M_+([t_0,t_1] \times \T^d) \times \M([t_0,t_1] \times \T^d)^d \times \M([t_0,t_1] \times \T^d)$ a competitor in $\RUOT_{\nu, \Psi,t_0,t_1}(\alpha,\beta)$ such that 
		\begin{multline}
		\label{eq:control_ruot_boundary_energy}
		E_\Psi(\rho,m,\zeta)
		\leq  C(t_1-t_0) \max(\alpha(\T^d), \beta(\T^d)) + \frac{2C}{t_1 - t_0} \left|\alpha(\T^d)^{1/2} - \beta(\T^d)^{1/2} \right|^2 \\ + \frac{2 \beta(\T^d)}{t_1-t_0} W_2^2 \left( \frac{\tau_{\nu(t_1-t_0)/2} \ast \alpha}{\alpha(\T^d)} , \frac{\beta}{\beta(\T^d)} \right) + \frac{2\nu}{t_1-t_0}  h(\beta| \Leb) ,  
		\end{multline}
		with the convention that $\tilde \alpha / \tilde \alpha(\T^d) = \Leb$ in the case where the measure $\tilde \alpha$ is the zero measure. In addition, the measure $\zeta$ satisfies
\begin{equation}
\label{eq:control_ruot_boundary_zeta}
|\zeta|([t_0,t_1] \times \T^d) \leq |\alpha(\T^d) - \beta(\T^d)|. 		
\end{equation}		 
	\end{Prop}
	
	\begin{proof}
		Let $t_{1/2} = (t_0+t_1)/2$. We will construct a competitor separately on $[t_0, t_{1/2}]$ and $[t_{1/2}, t_1]$, joining first $\alpha$ to $\frac{\beta(\T^d)}{\alpha(\T^d)} \tau_{\nu(t_1-t_0)/2} \ast \alpha$, and then to $\beta$.
		
		On $[t_0, t_{1/2}]$ we set $m = 0$ (no transport) and consider growth rates $r$ (such that $\zeta = r \rho$) that are constant in space. We then choose our $r$ by taking the one which would be optimal if $\Psi$ were quadratic. Specifically, let
\begin{equation*}
c(t) = \left[ \frac{1}{t_{1/2} - t_0} \left( (t_{1/2} - t) \alpha(\T^d)^{1/2} + (t - t_0) \beta(\T^d)^{1/2} \right) \right]^2.
\end{equation*}
It is a monotone function such that $c(t_0) = \alpha(\T^d)$ and $c(t_{1/2}) = \beta(\T^d)$. We then define:
\begin{equation*}
r(t) = \frac{\dot{c}(t)}{c(t)} = \frac{\beta(\T^d)^{1/2} - \alpha(\T^d)^{1/2}}{(t_{1/2} - t) \alpha(\T^d)^{1/2} + (t - t_0) \beta(\T^d)^{1/2}},
\end{equation*}		
and the competitor $\rho = \D t \otimes \rho_t$ where $\rho_t = c(t) \tau_{\nu(t-t_0)} \ast \frac{\alpha}{\alpha(\T^d)}$. Then it is straightforward that
\begin{equation*}
\dr_t \rho = \frac{\nu}{2} \Delta \rho + r \rho,
\end{equation*}		
and $\rho_{t_0} = \alpha$ while $\rho_{t_{1/2}} = \beta(\T^d) \tau_{\nu(t_1-t_0)/2} \ast \frac{\alpha}{\alpha(\T^d)}$. In addition, when estimating the energy we find
\begin{multline*}
\int_{t_0}^{t_{1/2}} \hspace{-5pt} \int_{\T^d}  \Psi(r(t)) \D \rho_t (x) \D t \leq   C \int_{t_0}^{t_{1/2}} \left( c(t) + \frac{\dot{c}(t)^2}{c(t)} \right) \D t \\
 \leq \frac{C(t_{1} - t_0)}{2} \max(\alpha(\T^d), \beta(\T^d)) + \frac{2C}{t_1 - t_0} \left|\alpha(\T^d)^{1/2} - \beta(\T^d)^{1/2} \right|^2.   
\end{multline*}
Moreover, on this part it is clear that the measure $\zeta = r \rho = \D t \otimes (r(t) \rho_t)$ satisfies, as $c$ is a monotone function,
\begin{equation*}
|\zeta|([t_0,t_{1/2}] \times \T^d) = \int_{t_0}^{t_{1/2}} \hspace{-5pt} \int_{\T^d} |r(t)| \D \rho_t(x) \D t = \int_{t_0}^{t_{1/2}} |\dot{c}(t)| \D t = |\beta(\T^d) - \alpha(\T^d)|
\end{equation*}
		
		On the other hand on $[t_{1/2}, t_1]$ we use $\zeta = 0$, that is, no growth, as we have a balanced problem. To that end, and first with the case $[t_{1/2}, t_1] = [0,1]$, we can use \cite[Lemma 4.1]{baradat2020small} which gives the following estimate: for $\tilde \alpha, \tilde \beta$ probability measures,
		\begin{multline*}
		\min_{\rho, w} \left\{ \int_0^1 \hspace{-5pt} \int_{\T^d} \left[ \frac{1}{2} |w(t,x)|^2  + \frac{1}{2}\left|\frac{\nu}{2} \nabla \log \rho(t,x) \right|^2 \right]\rho(t,x) \D x \D t \ : \ \begin{gathered}\dr_t \rho + \Div(\rho w) = 0  \\ \rho(0) = \tilde \alpha, \, \rho(1) = \tilde \beta\end{gathered}\right\} \\ 
		\leq W_2^2(\tilde \alpha, \tilde \beta) + \frac{\nu}{2}(H(\tilde \alpha|\Leb)+H(\tilde \beta|\Leb)).    
		\end{multline*}
		The left hand side is connected to the regularized optimal transport problem, see \eqref{eq:ROT_symmetric}, thus this estimate reads: we can always find $(\tilde{\rho}, \tilde{m}, 0)$ a competitor in $\RUOT_{\nu, \Psi,0,1}(\tilde \alpha,\tilde \beta)$ such that $\tilde{m} = \tilde{v} \tilde{\rho}$ and
		\begin{equation*}
	  \int_{t_{1/2}}^{t_1} \hspace{-5pt} \int_{\T^d} \frac{1}{2} |\tilde v(t,x)|^2 \D \tilde \rho(t,x)  \leq W_2^2(\tilde \alpha, \tilde \beta) + \nu H(\tilde \beta|\Leb).  
		\end{equation*}
		Then we use for $\tilde \alpha, \tilde \beta$ the normalized version of $\tau_{\nu(t_1-t_0)/2} \ast \alpha$ and $\beta$. Moreover we do a temporal scaling: specifically, writing $\tilde{\rho} = \D t \otimes \tilde{\rho}_t$ and $\tilde{m} = \D t \otimes \tilde{m}_t$, and denoting $s : t \in [t_{1/2}, t_1] \to (t-t_{1/2})/(t_1 - t_{1/2}) \in [0,1]$ the affine change of variables, we use on $[t_{1/2},t_1] \times \T^d$: 
\begin{equation*}
\rho = \beta(\T^d) \D t \otimes \tilde \rho_{s(t)}, \qquad m = s'(t) \beta(\T^d) \D t \otimes \tilde{m}_{s(t)}, 
\end{equation*}		
as well as $\zeta = 0$. Then we estimate, using that $\Psi(0) \leq C$, doing a temporal change of variables and using that $\beta(\T^d) H(\beta / \beta(\T^d)|\Leb) = h(\beta|\Leb) - l(\beta(\T^d)) \leq h(\beta|\Leb)$), that on $[t_{1/2},t_1] \times \T^d$,		
\begin{equation*}
		E_\Psi(\rho,m,\zeta) 
		\leq C \frac{t_1 - t_0}{2} \beta(\T^d)   + \frac{2 \beta(\T^d)}{t_1-t_0} W_2^2 \left( \frac{\tau_{\nu(t_1-t_0)/2} \ast \alpha}{\alpha(\T^d)} , \frac{\beta}{\beta(\T^d)} \right) + \frac{2\nu}{t_1-t_0} h(\beta| \Leb) .   
		\end{equation*}		
    Gluing the two competitors together yields the desired estimates.		
	\end{proof}
	
	We are now in position to prove Theorem \ref{Thm:small_noise_limit}.
	
	\begin{proof}[Proof of Theorem \ref{Thm:small_noise_limit}: the $\Gamma-\liminf$]
		
		This part is the easiest. Let $(\rho^n, m^n, \zeta^n)_n$ be a sequence which converges to $(\rho,m,\zeta)$. We assume without loss of generality that 
		\begin{equation*}
		\liminf_{n \to + \infty}
		\Big\{ E_{\Psi_n}(\rho^n,m^n,\zeta^n) + \iota_{\CE_n(\rho_0, \rho_1)}(\rho^n,m^n,\zeta^n)\Big\} < + \infty.
		\end{equation*}
		
		In particular, for each $n$ the triple $(\rho^n, m^n, \zeta^n)$ satisfies the continuity equation with boundary conditions $\rho_0, \rho_1$ and diffusivity $\nu_n$. That is, if $\phi \in C^2([0,1] \times \T^d)$ is a test function,
		\begin{equation*}
		\left\cg \partial_t \phi + \frac{\nu_n}{2} \Delta \phi, \rho^n \right\cd + \cg \nabla \phi, m^n \cd + \cg \phi, \zeta^n \cd =  \langle \phi(1), \rho_1 \rangle - \langle \phi(0), \rho_0 \rangle.
		\end{equation*}
		Next, we send $n \to + \infty$ and use the weak convergence of the sequence $(\rho^n, m^n, \zeta^n)_n$. As $\phi$ is smooth, it is clear that $\cg \frac{\nu_n}{2} \Delta \phi, \rho^n \cd$ converges to $0$. Thus we are left with
		\begin{equation*}
		\cg \partial_t \phi , \rho \cd + \cg \nabla \phi, m \cd + \cg \phi, \zeta \cd =  \langle \phi(1), \rho_1 \rangle - \langle \phi(0), \rho_0 \rangle.
		\end{equation*}
		This shows that $\iota_{\CE_\infty(\rho_0, \rho_1)}(\rho,m,\zeta) = 0$. 
		
		Next, we want to pass to the limit in the cost function. We write $\mathcal{K}_n$ for the set of triples $(a,b,c)$ of continuous functions such that $a +  |b|^2/2 +  \Psi_n^*(c) \leq 0$, as in Definition~\ref{def:energy_UOT}. We fix $\varepsilon > 0$. Let $(a,b,c) \in \mathcal{K}_\infty$ be almost optimal for $(\rho,m,\zeta)$, that is such that
		\begin{equation*}
		E_{\Psi_\infty}(\rho,m,\zeta) \leq  \cg a, \rho \cd +  \cg b, m \cd +  \cg c, \zeta \cd + \varepsilon.
		\end{equation*}  
		Note that due to the form of $\Psi^*$ given in Lemma~\ref{lem:cv_Psi_star}, it means that $\Psi^*_\infty(c)=0$ everywhere. As we have an $\varepsilon$ of room, up to replacing $c$ by $(1-\eta)c$ for some small $\eta > 0$, we can assume that $c$ is valued in the interior of $\{ \Psi_\infty^* = 0 \}$. Using the uniform convergence described in Lemma \ref{lem:cv_Psi_star}, if we define $S_n = \sup_{[0,1] \times \T^d} \Psi^*_n(c)$, there holds $S_n \to 0$ as $n \to + \infty$. Thus, using $(a-S_n, b,c) \in \mathcal{K}_n$ as a test function in the definition of $E_{\Psi_n}(\rho^n, m^n, \zeta^n)$, we get 
		\begin{equation*}
		E_{\Psi_n}(\rho^n,m^n,\zeta^n) \geq \cg a - S_n, \rho^n \cd +  \cg b, m^n \cd +  \cg c, \zeta^n \cd = - S_n \rho^n([0,1] \times \T^d) + \cg a, \rho^n \cd +  \cg b, m^n \cd +  \cg c, \zeta^n \cd.
		\end{equation*} 
		Using the weak convergence and passing to the limit, 
		\begin{equation*}
		\liminf_{n \to + \infty} E_{\Psi_n}(\rho^n,m^n,\zeta^n) \geq \cg a, \rho \cd +  \cg b, m \cd +  \cg c, \zeta \cd \geq E_{\Psi_\infty}(\rho,m,\zeta) - \varepsilon.
		\end{equation*}
		As $\varepsilon$ is arbitrary, this shows that $\liminf_{n \to + \infty} E_{\Psi_n}(\rho^n,m^n,\zeta^n) \geq  E_{\Psi_\infty}(\rho,m,\zeta)$, which concludes the proof of the $\Gamma-\liminf$.
	\end{proof}
	
	\begin{proof}[Proof of Theorem \ref{Thm:small_noise_limit}: the $\Gamma-\limsup$]	
		
		This part is more involved. Let us fix $(\rho,m,\zeta)$ which satisfies the continuity equation with $\nu = 0$ and boundary conditions $\rho_0, \rho_1$, and such that $E_{\Psi_\infty}(\rho,m,\zeta) < + \infty$. 
		As usual, we identify a measure with its density with respect to the Lebesgue measure, and we use $a(t)$ to denote the function $a(t,\cdot)$, being $a$ defined on $[0,1] \times \T^d$.

		\begin{stepg}{Building a recovery sequence}	
			First, for a given $p \in \N^*$, we define $(\tilde \rho^p, \tilde m^p, \tilde \zeta^p)$ the triple given by Lemma~\ref{lem:mollification} at $\eps:= 1/p$ with $\Psi := \Psi_\infty$, $\nu := 0$, $\bar \nu := 1$ (any $\bar \nu > 0$ would work) and $\bar \rho_0 := \Leb$. Recall that
			\begin{equation}
			\label{eq:limsup_energy_small_noise}
			\limsup_{p \to + \infty} E_{\Psi_\infty}(\tilde \rho^p, \tilde m^p, \tilde \zeta^p) \leq E_{\Psi_\infty}(\rho, m, \zeta).
			\end{equation}	
			In addition, following the first point of Lemma~\ref{lem:mollification} and Lemma~\ref{lem:existence_boundary_CE}, the measures $\tilde \rho^p(0)$ and $\tilde \rho^p(1)$ converge weakly on $\M(\T^d)$ to respectively $\rho_0$ and $\rho_1$.   
			
			Next, we ``squeeze'' in time and do a convolution with the heat kernel to put a Laplacian in the continuity equation. For $\varepsilon \in (0,1/2)$ we define $s_\varepsilon : [\varepsilon, 1-\varepsilon] \to [0,1]$ the affine reparametrization
			\begin{equation*}
			s_\varepsilon(t) = \frac{t - \varepsilon}{1 - 2 \varepsilon}.
			\end{equation*} 
			For $n \in \N \cup \{ + \infty \}$ we define the new triple of competitors on $[\varepsilon, 1-\varepsilon] \times \T^d$ by
			\begin{equation*}
			\hat{\rho}^{n,p,\varepsilon}(t) :=  \tau_{\nu_n t} \ast \tilde \rho^p(s_\varepsilon(t)) , \qquad
			\hat{m}^{n,p,\varepsilon}(t) := \frac{1}{1-2 \varepsilon}  \tau_{\nu_n t} \ast \tilde m^p(s_\varepsilon(t)), \qquad 
			\text{and } \hat{\zeta}^{n,p,\varepsilon}(t) := \frac{1}{1-2 \varepsilon}  \tau_{\nu_n t} \ast \tilde \zeta^p(s_{\varepsilon}(t)) .
			\end{equation*}
			(Notice that this has still a meaning when $n = + \infty$ with the convention $\tau_0 = \delta_0$.)
			
			Given the definition:
			\begin{itemize}
				\item They satisfy the continuity equation $\partial_t \hat\rho^{n,p,\varepsilon} + \Div \hat m^{n,p,\varepsilon} = \frac{\nu_n}{2} \Delta \hat{\rho}^{n,p,\varepsilon} + \hat \zeta^{n,p,\varepsilon}$ on $[\varepsilon, 1-\varepsilon] \times \T^d$. This can be checked by standard calculus, as the effect of the convolution with the time dependent kernel $\tau_{\nu_n t}$ is to add diffusion in the evolution equation. Moreover, $\hat{\rho}^{n,p,\varepsilon}(\eps)$ and $\hat{\rho}^{n,p,\varepsilon}(1-\eps)$ converge respectively to $\rho_0$ and $\rho_1$ when $n\to + \infty$, $p \to + \infty$ and $\eps \to 0$.  
				\item They all have a density with respect to $\D t \otimes \D x$ (that we still denote by $\hat\rho^{n,p,\varepsilon}, \hat m^{n,p,\varepsilon}$ and $\hat\zeta^{n,p,\varepsilon}$), and, on $[\varepsilon, 1-\varepsilon] \times \T^d$, these densities are smooth while $\hat\rho^{n,p,\varepsilon}$ is strictly positive.
				\item For the estimation of the energy we will only need it for the case $n = + \infty$: for that case we start from~\eqref{eq:limsup_energy_small_noise}, do a temporal change of variables $t \leftrightarrow s_\varepsilon(t)$ and use the $1$-homogeneity of $\Psi_\infty$ to find:
				\begin{equation}
				\label{eq:zz_small_noise_energy_np}
				\limsup_{p \to + \infty}\int_{\varepsilon}^{1-\varepsilon} \hspace{-5pt} \int_{\T^d} \left(  \frac{1}{2} \left|\frac{\D \hat m^{\infty,p,\varepsilon}}{\D \hat \rho^{\infty,p,\varepsilon}} \right|^2 + \Psi_\infty \left( \frac{\D \hat \zeta^{\infty,p,\varepsilon}}{\D \hat \rho^{\infty,p,\varepsilon}} \right) \right) \D \hat \rho^{\infty,p,\varepsilon} \leq  \frac{1}{1 - 2 \varepsilon} E_{\Psi_\infty}(\rho,m,\zeta).
				\end{equation}
				
			\end{itemize}
			However, we still have a problem with the boundary conditions, as $\hat \rho^{n,p,\varepsilon}$ does not have the right boundary conditions. So for a given $\varepsilon$ we define $(\rho^{n,p,\varepsilon}, m^{n,p,\varepsilon}, \zeta^{n,p,\varepsilon})$ as follows:
			\begin{itemize}
				\item On $[0,\varepsilon] \times \T^d$ we take the competitor of $\RUOT_{\nu_n, \Psi_n, 0, \varepsilon}(\rho_0, \hat \rho^{n,p,\varepsilon}(\varepsilon))$ given by Proposition~\ref{prop:control_ruot}.
				\item On $[\varepsilon, 1-\varepsilon] \times \T^d$ it coincides with $(\hat \rho^{n,p,\varepsilon},\hat m^{n,p,\varepsilon},\hat \zeta^{n,p,\varepsilon})$.
				\item On $[1-\varepsilon, 1] \times \T^d$ we take the competitor of $\RUOT_{\nu_n, \Psi_n, 1-\varepsilon, 1}(\hat \rho^{n,p,\varepsilon}(1-\varepsilon), \rho_1)$ given by Proposition~\ref{prop:control_ruot}.
			\end{itemize}
			We end up with a competitor $(\rho^{n,p,\varepsilon}, m^{n,p,\varepsilon}, \zeta^{n,p,\varepsilon})$ which belongs to $\CE_n(\rho_0,\rho_1)$.
			Informally, $p$ controls the smoothness, $n$ indicates the level of noise and $\varepsilon$ is the time-lapse of our surgery at the boundary. We will take the limit $n \to + \infty$, followed by $p \to + \infty$ and then $\varepsilon \to 0$. 
			
		\end{stepg}
		
		\begin{stepg}{Estimation of the energy}
			Let us first estimate what happens in the interior of the time interval. We define for $n \in \N \cup \{ + \infty \}$ 
			\begin{equation*}
			v^{n,p,\varepsilon} = \frac{\D m^{n,p,\varepsilon}}{\D \rho^{n,p,\varepsilon}} , \qquad \text{and} \qquad r^{n,p,\varepsilon} = \frac{\D \zeta^{n,p,\varepsilon}}{\D \rho^{n,p,\varepsilon}}.
			\end{equation*}
			Note that $\rho^{p,n,\varepsilon}$, $v^{n,p,\varepsilon}$ and $\zeta^{n,p,\varepsilon}$ are all smooth on $[\varepsilon, 1-\varepsilon] \times \T^d$ uniformly in $n \in \N \cup \{ + \infty \}$, and $\rho^{p,n,\varepsilon}$ strictly bounded from below, again uniformly in $n$ (as we reason for $p$ fixed). Moreover $v^{n,p,\varepsilon}$ (resp. $r^{n,p,\varepsilon}$) converges to $v^{\infty,p,\varepsilon}$ (resp. $r^{\infty,p,\varepsilon}$) for the norm of uniform convergence in $C([\varepsilon,1-\varepsilon] \times \T^d; \R^d)$ (resp. $C([\varepsilon,1-\varepsilon] \times \T^d)$). In particular, thanks to~\eqref{eq:zz_gammalimsup_psi} we see that the $\limsup$ of $\Psi_n(r^{n,p,\varepsilon})$ on $[\varepsilon,1-\varepsilon] \times \T^d$ is less than $\Psi_\infty(r^{\infty,p,\varepsilon})$, while being uniformly bounded from above, in virtue of \eqref{eq:control_Psi_n_uniform}. Combined with the strong convergence of $\rho^{n,p,\varepsilon}$ to $\rho^{\infty,p,\varepsilon}$ as $n \to + \infty$, using for instance dominated convergence theorem, 
			\begin{multline*}
			\limsup_{n \to + \infty}  \int_\varepsilon^{1-\varepsilon} \hspace{-5pt}\int_{\T^d}\left\{ \frac{1}{2} |v^{n,p,\varepsilon}(t,x)|^2  +  \Psi_n(r^{n,p,\varepsilon}(t,x)) \right\} \D \rho^{n,p,\varepsilon}(t,x) \\
			\leq \int_\varepsilon^{1-\varepsilon} \hspace{-5pt}\int_{\T^d} \left\{ \frac{1}{2} |v^{\infty,p,\varepsilon}(t,x)|^2 + \Psi_{\infty}(r^{\infty,p,\varepsilon}(t,x))\right\} \D \rho^{\infty,p,\varepsilon}(t,x).
			\end{multline*}
			Moreover, formula~\eqref{eq:zz_small_noise_energy_np} controls the limit $p \to + \infty$ of the right hand side: 
			\begin{equation}
			\label{eq:zz_aux_gamma_limsup_interior}
			\limsup_{p \to + \infty} \limsup_{n \to + \infty}  \int_\varepsilon^{1-\varepsilon} \hspace{-5pt}\int_{\T^d}\left\{ \frac{1}{2} |v^{n,p,\varepsilon}(t,x)|^2  +  \Psi_n(r^{n,p,\varepsilon}(t,x)) \right\} \D \rho^{n,p,\varepsilon}(t,x)
			\leq \frac{1}{1-2 \varepsilon} E_{\Psi_\infty}(\rho,m,\zeta).
			\end{equation}
			
			It remains to control what happens at the boundary. To that end we rely on Proposition~\ref{prop:control_ruot}: with~\eqref{eq:control_ruot_boundary_energy} we find
			\begin{align*}
			& \int_{[0,\varepsilon] \cup [1-\varepsilon,1]}  \int_{\T^d} \left( \frac{1}{2} \left| \frac{\D m^{n,p,\varepsilon}}{\D \rho^{n,p,\varepsilon}} \right|^2 + \Psi \left( \frac{\D \zeta^{n,p,\varepsilon}}{\D \rho^{n,p,\varepsilon}} \right) \right) \D \rho^{n,p,\varepsilon} \\ 
			& \leq 2 C \varepsilon \left( \rho_0(\T^d) + \rho_1(\T^d) +  \max_{t \in [\varepsilon,1-\varepsilon]}  \rho_t^{n,p,\varepsilon}(\T^d) \right) \\
			& \qquad + \frac{2C}{\varepsilon} \left\{ \left| \rho_0(\T^d)^{1/2} - \rho^{n,p,\varepsilon}(\varepsilon)(\T^d)^{1/2}  \right|^2 + \left| \rho_1(\T^d)^{1/2} - \rho^{n,p,\varepsilon}(1-\varepsilon)(\T^d)^{1/2}  \right|^2 \right\} \\
			& \qquad + \frac{2 }{\varepsilon} \left\{ \rho^{n,p,\varepsilon}(\varepsilon)(\T^d) W_2^2 \left( \frac{\tau_{\nu_n \varepsilon / 2} \ast \rho_0 }{\rho_0(\T^d)} , \frac{\rho^{n,p,\varepsilon}(\varepsilon)}{\rho^{n,p,\varepsilon}(\varepsilon)(\T^d)} \right) + \rho_1(\T^d) W_2^2 \left( \frac{\tau_{\nu_n \varepsilon / 2} \ast \rho^{n,p,\varepsilon}(1-\varepsilon)}{\rho^{n,p,\varepsilon}(1-\varepsilon)(\T^d)} , \frac{\rho_1}{\rho_1(\T^d)} \right) \right\} \\ 
			& \qquad + \frac{2 \nu_n}{\varepsilon} \left\{ h(\rho^{n,p,\varepsilon}(\varepsilon)| \Leb) + h(\rho_1|\Leb)  \right\}
			\end{align*}
			We then take the limit $n \to + \infty$, it makes the terms of the last line vanish: this is where we use the assumption $h(\rho_1|\Leb) < + \infty$. Then we take $p \to + \infty$ and observe that $\rho^{n,p,\varepsilon}(t)$ converges weakly, when $n \to + \infty$ followed by $p \to + \infty$, to respectively $\rho_0$ and $\rho_1$ for $t = \varepsilon$ and $t = 1-\varepsilon$ as observed above. As the Wasserstein distance metrizes the weak convergence, the terms containing this distance vanish in the limit. Moreover, for the first line, the term $\max_{t \in [\varepsilon,1-\varepsilon]}  \rho_t^{n,p,\varepsilon}(\T^d)$ is bounded uniformly in $n,p,\varepsilon$: we have a uniform control of the energy~\eqref{eq:zz_aux_gamma_limsup_interior} and it yields a control of the mass thanks to Proposition~\ref{prop:control_mass_from_energy} and Lemma~\ref{lem:cv_Psi} (see~\eqref{eq:control_Psi_n_uniform}). Let's call $\tilde C$ an upper bound. We end up with
			\begin{equation}
			\label{eq:small_noise_aux_boundary_control}
			\limsup_{p \to + \infty} \limsup_{n \to + \infty} \int_{[0,\varepsilon] \cup [1-\varepsilon,1]}  \int_{\T^d} \left( \frac{1}{2} \left| \frac{\D m^{n,p,\varepsilon}}{\D \rho^{n,p,\varepsilon}} \right|^2 + \Psi \left( \frac{\D \zeta^{n,p,\varepsilon}}{\D \rho^{n,p,\varepsilon}} \right) \right) \D \rho^{n,p,\varepsilon}
			\leq 2 C \tilde C \varepsilon 
			\end{equation}
			which means that the left hand side converges to $0$ once we take the limit $\varepsilon \to 0$.
			
			Gluing this estimate with~\eqref{eq:zz_aux_gamma_limsup_interior}, we conclude that
			\begin{equation}
			\label{eq:small_noise_aux_center_control}
			\limsup_{\varepsilon \to 0}  \limsup_{p \to + \infty} \limsup_{n \to + \infty} E_{\Psi_n}(\rho^{n,p,\varepsilon}, m^{n,p,\varepsilon}, \zeta^{n,p,\varepsilon}) \leq E_{\Psi_\infty}(\rho,m,\zeta).
			\end{equation}
		\end{stepg}
		
		\begin{stepg}{Convergence of the recovery sequence}
			It remains to prove the convergence of $(\rho^{n,p,\varepsilon}, m^{n,p,\varepsilon}, \zeta^{n,p,\varepsilon})$ to $(\rho,m,\zeta)$. Given the way it was defined (and in particular the first point of Lemma~\ref{lem:mollification}), it is clear that $(\hat \rho^{n,p,\varepsilon}, \hat m^{n,p,\varepsilon}, \hat \zeta^{n,p,\varepsilon})$ (that is, the restriction of $(\rho^{n,p,\varepsilon}, m^{n,p,\varepsilon}, \zeta^{n,p,\varepsilon})$ to $[\eps, 1-\eps] \times \T^d$) converges weakly on $[0,1] \times \T^d$ to the original triplet $(\rho,m,\zeta)$ when $n \to + \infty$, followed by $p \to + \infty$ and then $\eps \to 0$.
			
			On the other hand, the energy of the restriction of $(\rho^{n,p,\varepsilon}, m^{n,p,\varepsilon}, \zeta^{n,p,\varepsilon})$ to $\{ [0,\eps] \cup [ 1-\eps] \} \times \T^d$ converges to $0$, this is \eqref{eq:small_noise_aux_boundary_control}. Thus, thanks to Proposition~\ref{prop:control_mass_from_energy} together with the bound~\eqref{eq:control_Psi_n_uniform} (uniform in $n$), one can see that the restriction of $(\rho^{n,p,\varepsilon}, m^{n,p,\varepsilon})$ to $\{ [0,\eps] \cup [ 1-\eps] \} \times \T^d$ converges weakly to the zero measure when $n \to + \infty$, followed by $p \to + \infty$ and then $\eps \to 0$. For the net mass balance $\zeta$ we use \eqref{eq:control_ruot_boundary_zeta} to see that 
\begin{equation*}
|\zeta^{n,p,\varepsilon}|(\{ [0,\eps] \cup [ 1-\eps] \} \times \T^d) \leq  \left| \rho_0(\T^d) - \rho^{n,p,\varepsilon}(\varepsilon)(\T^d)  \right| + \left|  \rho^{n,p,\varepsilon}(1-\varepsilon)(\T^d) - \rho_1(\T^d)  \right|
\end{equation*}			
and the right hand side converges to $0$ when taking $n \to + \infty$, followed by $p \to + \infty$ and $\eps \to 0$.	

			To eventually get a recovery sequence, it is enough to take $p = p_n$ and $\varepsilon = \varepsilon_n$ which go to respectively $+ \infty$ and $0$ slowly enough.
		\end{stepg}	
	\end{proof}

	\section{More general measure-valued branching Markov processes}
	\label{sec:general_superproc}
	
	In this section, we investigate what happens if the reference law $R$ is still a probability distribution on $\cadlag([0,1], \M_+(\T^d))$, but is no longer the law of a BBM. We will not aim at proving rigorous results and we will always assume that all the objects that we manipulate are smooth enough to perform computations. Note, as can be seen in what we have done until now, that this assumption is usually false in most cases. A detailed study would involve some work to justify and expand our findings. However, we still think that the formal computations shed a new light on the problem, and they also explain why the law of the BBM, and not of another measure-valued branching Markov processes, turns out to be the one connected to optimal transport.
	
	So let us take $R \in \P(\cadlag([0,1]; \M_+(\T^d)))$. We assume that there exists a family of differential operator $\L : [0,1] \times  C^\infty(\T^d) \to C^\infty(\T^d)$ such that, for all $\varphi : \T^d \to \R$ and $s \in [0,1)$,
	\begin{equation}
	\label{eq:def_L_superprocess}
	\frac{\D}{\D t+} \E_R \Big[ \exp( \langle \varphi, M_t \rangle ) \Big| \F_s \Big] \Big|_{t=s} = \langle \L[s,\varphi], M_s \rangle  \exp( \langle \varphi, M_s \rangle ).
	\end{equation}
	It amounts essentially to asking that $R$ satisfies the (inhomogeneous) Markov property and the branching property. In the case of the BBM, the family of operators $(\L[t,\cdot])_{t \in [0,1]}$ does not depend on $t$ and coincide, up to a scaling factor $\nu$, with $\L_{\nu, \boldsymbol q }$ as defined in \eqref{eq:operator_L}: this is a consequence of Proposition~\ref{prop:generator} together with the branching property, Proposition~\ref{prop:branching_property}. 
	
	However, this framework encompasses more general cases than the BBM. In particular, as we will explain now in \eqref{eq:exp_mart_generic}, if the process $R$ solves a martingale problem, there is a good chance to identify (at least formally) the generator $\L$.
	
	More precisely and similarly to the case of the BBM (see Proposition~\ref{prop:martingale_problem}), at least formally, $R$ satisfies~\eqref{eq:def_L_superprocess} for all smooth $\varphi$ and $s\in[0,1)$ if and only if for any smooth $\psi : [0,1] \times \T^d \to \R$, the process    
	\begin{equation}
	\label{eq:exp_mart_generic}
	\exp\left(\cg \psi(t) , M_t \cd - \cg \psi(0), M_0 \cd - \int_0^t \big\cg \partial_t \psi(s) + \mathcal \L[s,\psi(s)], M_s \big\cd \D s \right), \qquad t \in [0,1],
	\end{equation}
	is a martingale under $R$. Usually one would restrict to $\psi$ nonpositive in order to guarantee integrability, but this is the kind of technical assumptions that we will not discuss as they depend on the precise process and would likely involve a fine analysis. Note that a corollary is that, if $\psi$ is a solution of the (backward) PDE $\partial_t \psi(s) + \mathcal \L[s,\psi(s)] = 0$, then the process $(\exp(\cg \psi(t), M_t \cd))_{t \in [0,1]}$ is a martingale under $R$.
	
	\begin{Ex}
		The Dawson-Watanabe superprocess \cite[Chapter 1]{etheridge2000introduction} is the limit of the branching Brownian motion when the initial number of particles tends to infinity, the mass is rescaled correspondingly, and the branching rate is proportional to the initial number of particles. More precisely, start from a branching mechanism $\boldsymbol q$ such that:
		\begin{equation*}
		\sum_k k q_k = 1 \qquad \mbox{and} \qquad \sum_k (k-1)^2 q_k = \gamma.
		\end{equation*}
		Then consider $R^n \sim \BBM(\nu,  n \boldsymbol{q}, R^n_0)$. If the law of $M_0/n$ under $R^n$ weakly converges to $R_0 \in \P(\M_+(\T^d))$, then the law of $(M_t/n)_{t \in [0,1]}$ weakly converges to the law of the Dawson-Watanabe superprocess of parameters $\nu$ and $\gamma$, starting from $R_0$. It is characterized by the (time independent) generator $\L : C^\infty(\T^d) \to C^\infty(\T^d)$ which reads (see \cite[Section 1.5]{etheridge2000introduction}):
		\begin{equation}
		\label{eq:operator_DW}
		\mathcal{L}[\phi] =  \frac{\nu}{2} \Delta \phi + \frac{\gamma}{2} \phi^2. 
		\end{equation}
	\end{Ex}
	
	\subsection{Equivalence of the values}
	
	Given $R \in \P(\cadlag([0,1], \M_+(\T^d)))$ characterized by an operator $\L : [0,1] \times  C^\infty(\T^d) \to C^\infty(\T^d)$, we can consider the same entropy minimization problem, the ``General'' branching Schrödinger problem:
	\begin{equation}
	\label{eq:BrSch_general}
	\GBrSch_R(\rho_0,\rho_1) = \inf_{P} \left\{ H(P|R) \ : \ \E_P[M_0] = \rho_0 \text{ and } \E_P[M_1] = \rho_1 \right\},
	\end{equation}
	where now we do not have a scaling factor $\nu$ in front of the entropy: at this level of generality, there is not necessarily a well-defined diffusivity.
	
	Similarly to the informal inf-sup exchange done in the introduction (Section~\ref{sec:duality_informal}) yielding \eqref{eq:dual_generic_R}, and in fact as justified in the proof of Theorem~\ref{thm:equality_values_dyn}, one can write the l.s.c.\ envelope of $\GBrSch_R$ as:
	\begin{equation*}
	\overline{\GBrSch}_R(\rho_0,\rho_1) = \sup_{\sigma,\theta} \cg \sigma, \rho_0 \cd +\cg \theta, \rho_1 \cd  - \log \E_R \left[  \exp \left( \langle \sigma,M_0 \rangle + \langle \theta,M_1 \rangle \right) \right],
	\end{equation*} 
	where now $\sigma, \theta$ are continuous functions on $\T^d$. Still following the formal computations of Section~\ref{sec:duality_informal}, but now using \eqref{eq:exp_mart_generic} by taking $\varphi$ to be the solution of $\partial_t \varphi(t) +  \L[t,\varphi(t)] = 0$ with terminal conditin $\varphi(1) = \theta$, we find that 
	\begin{equation*}
	\overline{\GBrSch}_R(\rho_0,\rho_1) = L_{R_0}^*(\rho_0) + \sup_{\varphi} \left\{ \cg \varphi(1), \rho_1 \cd - \cg \varphi(0), \rho_0 \cd \ : \  \partial_t \varphi(t) +  \L[t,\varphi(t)] = 0 \right\}.
	\end{equation*}
	Note however that establishing this identity in the case of the BBM required some work (see Section~\ref{sec:linkFKPP_BBM} and the proof of Theorem~\ref{thm:equality_values_dyn}) because we had to analyze in details different notions of solutions for the PDE $\partial_t \varphi(t) + \mathcal \L[t,\varphi(t)] = 0$. Establishing it for a different $R$ could be difficult and depends on the structure of $\L$. Next we can try to push further the computation and work on the supremum in $\varphi$. Introducing a Lagrange multipler $\rho : [0,1] \times \T^d \to \R$ for the constraint $\partial_t \varphi(t) + \mathcal \L[t,\varphi(t)] = 0$ and assuming that we can exchange the infimum and the supremum, an integration by parts leads to 
	\begin{align}
	\notag \sup_{\varphi} & \left\{ \cg \varphi(1), \rho_1 \cd - \cg \varphi(0), \rho_0 \cd \ : \  \partial_t \varphi(t) + \mathcal \L[t,\varphi(t)] = 0 \right\} \\
	\notag & = \sup_{\varphi} \inf_{\rho} \left\{ \cg \varphi(1), \rho_1 \cd - \cg \varphi(0), \rho_0 \cd - \int_0^1 \cg \partial_t \varphi(t) + \L[t,\varphi(t)],   \rho(t) \cd \D t \right\} \\
	\label{eq:transport_problem_inf_sup}& = \inf_{\rho} \sup_{\varphi}  \left\{ \cg \varphi(1), \rho_1 - \rho(1) \cd - \cg \varphi(0), \rho_0 - \rho(0) \cd + \int_0^1 \Big( \cg \varphi(t), \partial_t \rho(t) \cd - \cg \mathcal L[t, \varphi(t)], \rho(t) \cd \Big)\D t \right\}.
	\end{align} 
	(Once again, these computations are formal and one should not pay attention to regularity issues to define the different terms in these formulas.) Taking the supremum in $\varphi$ imposes boundary constraint on $\rho$ as well as the evolution equation 
	\begin{equation}
	\label{eq:general_transport_equation}
	\partial_t \rho(t)  = \mathrm{D} \L[t,\varphi(t)]^\top \cdot \rho(t),
	\end{equation}
	where the latter is defined by duality as follows: for all test function $a \in C^\infty(\T^d)$,
	\begin{equation*}
	\cg a, \mathrm{D} \L[t,\varphi(t)]^\top \cdot \rho \cd = \langle \mathrm{D} \L[t,\varphi(t)] \cdot a, \rho \rangle = \frac{\D}{\D s}  \Big\langle \mathrm{D} \L[t,\varphi(t)+ s a] , \rho \Big\rangle \Big|_{s=0}.
	\end{equation*}  
	Now, imagine that for all $t \in [0,1]$ and $\rho: \T^d \to \R_+$, the map $\varphi \mapsto \mathrm D \mathcal L[t, \varphi]^\top \cdot \rho$ is injective. It means that for all curve $t \mapsto \rho(t)$, there is at most one $\varphi = \varphi(t)$ such that~\eqref{eq:general_transport_equation} holds. In that case, provided we add~\eqref{eq:general_transport_equation} as a constraint in~\eqref{eq:transport_problem_inf_sup}, we can replace $\sup_\varphi$ by $\inf_\varphi$, as anyway, there is at most one admissible $\varphi$. We end up with
	\begin{align}
	\notag &\overline{\GBrSch}_R(\rho_0,\rho_1) \\
	\label{eq:equiv_value_BSG}&\qquad = L_{R_0}^*(\rho_0)
	+
	\inf_{\rho, \varphi}  \left\{  \int_0^1 \Big\cg \mathrm{D} \L[t,\varphi(t)] \cdot \varphi(t) - \L[t,\varphi(t)] , \rho(t) \Big\cd \D t \ : \
	\begin{gathered}\partial_t \rho(t) = \mathrm{D} \L[t,\varphi(t)]^\top \cdot \rho(t)\\
	\rho(0) = \rho_0, \; \rho(1) = \rho_1 \end{gathered} \right\}.
	\end{align} 
	Then, one can wonder whether the r.h.s.\ has an interesting interpretation, and the answer depends on the operator $\L$. This can be dealt on a case by case basis. Here, we only comment what one obtains for the Dawson-Watanabe superprocess.
	
	\begin{Ex}
		We illustrate the formula above on the Dawson-Watanabe superprocess, that is when $\L$ is defined through \eqref{eq:operator_DW}. Indeed, in this case one has
		\begin{equation*}
		\mathrm{D} \L[\varphi] \cdot a = \frac{\nu}{2} \Delta a + \gamma a \phi,
		\end{equation*}
		in such a way that $\mathrm{D} \L[\varphi] = \mathrm{D} \L[\varphi]^\top$ and \eqref{eq:equiv_value_BSG} reads 
		\begin{align*}
		&\overline{\GBrSch}_R(\rho_0,\rho_1) \\
		&\qquad= L_{R_0}^*(\rho_0) +
		\inf_{\rho, \varphi}   \Bigg\{ \frac{\gamma}{2} \int_0^1\hspace{-5pt}\int \varphi(t,x)^2  \rho(t,x) \D x \D t  \ : \ \partial_t \rho = \frac{\nu}{2} \Delta \rho + \varphi \rho \text{ and } \rho(0,\cdot) = \rho_0, \; \rho(1,\cdot) = \rho_1 \Bigg\}.
		\end{align*}
		We recognize a problem \emph{without} transport, as there is no velocity field in the evolution equation for $\rho$. On the other hand, $\varphi$ plays here the role of a growth rate, with a quadratic penalization. In the Dawson-Watanabe regime, it is infinitely more expensive to add a drift to each of the infinite number of particles than to create or destroy mass, so that in the limit, transport is forbidden. When $\gamma = 1$ and $\nu \to 0$, we recognize formally the so-called Fischer-Rao metric, also known as Hellinger distance, see \cite{chizat2018unbalanced} and \cite[Chapter 8]{peyre2019computational}. That is, we expect the small noise limit of the value of the entropy minimization with respect to the Dawson-Wanatabe superprocess to be the Fischer-Rao metric. 
	\end{Ex}
	
	\subsection{Structure of the optimizers} 
	
	Still in the case of a measure-valued branching Markov processes characterized by $\L : [0,1] \times C^\infty(\T^d) \to C^\infty(\T^d)$, let us characterize formally the structure of the optimizers.
	
	Actually, from the formal inf-sup exchange that we did in the introduction in order to get~\eqref{eq:dual_generic_R}, we expect the optimal $P$ in~\eqref{eq:BrSch_general} to have a density with respect to $R$ which reads
	\begin{equation*}
	\frac{\D P}{\D R}(M) = \exp \Big( \langle \sigma, M_0 \rangle + \langle \theta, M_1 \rangle \Big)
	\end{equation*}
	for some continuous functions $\sigma, \theta : \T^d \to \R$. As the problem is ill-posed, even in the case of the BBM we have no hope that, for a generic $\rho_0, \rho_1$, one can write $P$ as above with $\sigma, \theta$ continuous. However, for the sake of formal computations we will stick to this case. Actually, we will allow for more freedom by also choosing $\psi : [0,1] \times \T^d \to \R$ and look at the probability measure $P^{\sigma,\psi}$ defined on $\cadlag([0,1], \M_+(\T^d))$ as the one whose density with respect to $R$ is
	\begin{equation}
	\label{eq:shape_solution_general}
	\frac{\D P^{\sigma,\psi}}{\D R}(M) = \exp \left( \langle \sigma, M_0 \rangle + \langle \psi(1), M_1 \rangle - \langle \psi(0), M_0 \rangle - \int_0^1 \left\cg \partial_t \psi(t) + \L[t,\psi(t)], M_t \right\cd \D t \right).
	\end{equation}
	Note that if $\psi$ is a solution of $\partial_t \psi(t) + \L[t,\psi(t)] = 0$ then we retrieve the previous form.
	
	\begin{Rem}
		In the case of the BBM, we rather defined a modified BBM via its Radon-Nikodym density in Theorem~\ref{thm:formula_RN_derivative}. In such a case, we could freely choose a new drift $\tilde{v}$ and a new branching mechanism $\boldsymbol{\tilde{q}}$. If these objects $\tilde{v}$ and $\boldsymbol{\tilde{q}}$ are chosen according to Example~\ref{ex_general_link_v_q_BBM} below, then the Radon-Nikodym derivative in Theorem~\ref{thm:formula_RN_derivative} reduces to \eqref{eq:shape_solution_general}: this a consequence of the extended Itô formula, see Theorem~\ref{thm:branching_Ito}. 
	\end{Rem}
	
	Formally, thanks to the martingale characterization \eqref{eq:exp_mart_generic}, we see that $P^{\sigma,\psi}$ is indeed a probability distribution as soon as $\E_R[\exp(\cg\sigma,M_0\cd)] = 0$. In this case, let us explain why $P^{\sigma,\psi}$ is automatically a solution of the problem:
	\begin{equation*}
	\min_{Q} \left\{ H(P|R) \ : \ \forall t \in [0,1], \, \E_Q[M_t] = \E_{P^{\sigma,\psi}}[M_t] \right\},
	\end{equation*}
	and why in addition, if $\partial_t \psi(t) + \L[t,\psi(t)] = 0$ then $P^{\sigma,\psi}$ is a solution of the problem \eqref{eq:BrSch_general} with its own marginal constraints, that is, $\rho_0=\E_{P^{\sigma,\psi}}[M_0]$ and $\rho_1 = \E_{P^{\sigma,\psi}}[M_1]$. 
	
	Let us take $Q \in \P(\cadlag([0,1], \M_+(\T^d)))$. Denoting by $Z_P$ (resp. $Z_Q$) the density of $P^{\sigma,\psi}$ (resp. $Q$) with respect to $R$, and using the convex inequality $Z_Q \log Z_Q - Z_P \log Z_P \geq ( Z_Q - Z_P) (1+ \log Z_P)$, we see that
	\begin{equation*}
	H(Q|R) - H(P^{\sigma,\psi}|R) \geq (\E_{Q} - \E_{P^{\sigma,\psi}})[1+\log Z_P].
	\end{equation*} 
	Thanks to the explicit expression of $Z_P$ in \eqref{eq:shape_solution_general}, one can see that $\E_{Q}[1+\log Z_P]$ depends only on $\E_Q[M_t]$ for $t \in [0,1]$, and only on $\E_Q[M_0]$ and $\E_Q[M_1]$ if $\partial_t \psi(t) + \L[t,\psi(t)] = 0$. Thus the optimality of $P^{\sigma,\psi}$ follows. 
	
	Again, already in the case of the BBM it is not true that any solution of \eqref{eq:BrSch_general} can be written as \eqref{eq:shape_solution_general} for some smooth $\sigma, \psi$. However, any $P^{\sigma,\psi}$, provided it is actually a probability distribution and $\partial_t \psi(t) + \L[t,\psi(t)] = 0$, is a solution of \eqref{eq:BrSch_general} with its own marginals.
	
	\bigskip
	
	Eventually, we show that, for some smooth $\sigma, \psi$, provided $P^{\sigma,\psi}$ is a probability distribution, it is a measure-valued branching Markov processes characterized by an operator $\L^{\psi}$ in the sense of~\eqref{eq:def_L_superprocess} (or equivalently~\eqref{eq:exp_mart_generic}), which is defined by
	\begin{equation*}
	\varphi \mapsto \L^{\psi}[t,\varphi] = \L[t,\varphi + \psi(t)] - \L[t,\psi(t)].
	\end{equation*}  
	Indeed, for a given smooth $\phi:[0,1] \times \T^d \to \R$ denoting by $\mathcal{E}^\phi_t$ the process
	\begin{equation*}
	\mathcal{E}^\phi_t := \exp\left(\cg \phi(t) , M_t \cd - \cg \phi(0), M_0 \cd - \int_0^t \big\cg \partial_t \phi(s) + \mathcal \L^\psi[s,\phi(s)], M_s \big\cd \D s \right),  \qquad t \in [0,1],  
	\end{equation*}
	we want to prove that $(\mathcal{E}^\phi_t)_{t \in [0,1]}$ is a martingale under $P^{\sigma, \psi}$. With $(Z_t)_{t \in [0,1]}$ the process defined at time $t$ by
	\begin{equation*}
	Z_t = \exp \left( \langle \sigma, M_0 \rangle + \langle \psi(t), M_t \rangle - \langle \psi(0), M_0 \rangle - \int_0^t \left\cg \partial_s \psi(s) + \L[s,\psi(s)], M_s \right\cd \D t \right),    
	\end{equation*}
	which is a martingale under $R$, and whose value at time $1$ gives the Radon-Nikodym density of $P^{\sigma,\psi}$ w.r.t.\ $R$, we only need to show (see for instance the proof of Girsanov's theorem in \cite{legall2016brownian}) that $(Z_t \mathcal{E}^\phi_t)_{t \in [0,1]}$ is a martingale under $R$. But by definition of $\L^\psi$, 
	\begin{equation*}
	Z_t \mathcal{E}^\phi_t = \exp \left( \langle \sigma, M_0 \rangle + \langle  (\psi + \phi)(t), M_t \rangle - \langle (\psi + \phi)(0), M_0 \rangle - \int_0^t \left\cg \partial_s (\psi + \phi)(s) + \L[s,(\psi + \phi)(s)], M_s \right\cd \D t \right)  
	\end{equation*}
	and it is a martingale under $R$ thanks to \eqref{eq:exp_mart_generic} applied to the test function $\psi + \phi$.
	
	\begin{Ex}
		\label{ex_general_link_v_q_BBM}
		To illustrate the relevance of the formal approach, let us stick to the case of BBM. We consider $\L = \L_{\nu, \boldsymbol q}$ the generator given in \eqref{eq:operator_L}, and a function $\psi : [0,1] \times \T^d \to \R$. We look at the measure $P^\psi$ where $R \sim \BBM(\nu, \boldsymbol q, R_0)$. According to the computations above, it is characterized by the operator
		\begin{align}
		\notag
		\varphi \mapsto \L^{\psi}[t,\varphi] &= \L_{\nu, \boldsymbol q}[\varphi + \psi(t)] - \L_{\nu, \boldsymbol q}[t,\psi(t)] \\
		\label{eq:new_L_BBM} &= \frac{1}{2} |\nabla \varphi|^2 + \nabla \varphi \cdot \nabla \psi(t) + \frac{\nu}{2} \Delta \varphi + \Psi^*_{\nu, \boldsymbol q}(\varphi + \psi(t)) - \Psi^*_{\nu, \boldsymbol q}(\psi(t)).
		\end{align} 
		Note that given the definition~\eqref{eq:def_Psi*} of $\Psi_{\nu, \boldsymbol q}^*$ the last part can be rewritten:
		\begin{align*}
		\Psi^*_{\nu, \boldsymbol q}(\varphi + \psi(t)) - \Psi^*_{\nu, \boldsymbol q}(\psi(t)) &= \nu \left[ \Phi_{\boldsymbol q}\left(e^{(\varphi+\psi(t))/\nu}\right)e^{-(\varphi+\psi(t))/\nu} - \Phi_{\boldsymbol q}\left(e^{\psi(t)/\nu}\right)e^{-\psi(t)/\nu} \right] \\ 
		&= \nu \left[ \Phi_{\tilde{\boldsymbol{q}}(t)}\left( e^{\varphi/\nu} \right) e^{- \varphi/\nu} - \Phi_{\tilde{\boldsymbol{q}}(t)}(1) \right] = \Psi^*_{\nu, \tilde{\boldsymbol q}(t)}(\varphi),
		\end{align*}
		provided we define $\tilde{\boldsymbol{q}}$ as the (space-time) dependent $(\tilde{q}_k(t,x))_{k \in \N}$ with
		\begin{equation}
		\label{eq:new_tilde_q_formal}
		\tilde{q}_k(t,x) = q_k \exp \left( (k-1) \frac{\psi(t,x)}{\nu} \right).
		\end{equation}
		Equipped with this, we can reinterpret \eqref{eq:new_L_BBM}. Once $\psi$ is given, $P^\psi$ is characterized by the operator \eqref{eq:new_L_BBM}. In this generator, the term $\nabla \varphi \cdot \nabla \psi(t)$ means that particles experience a drift $\tilde{v} = \nabla \psi$, while the second part yields that the new branching mechanism is $\tilde{\boldsymbol{q}}$ as defined in \eqref{eq:new_tilde_q_formal}. In this case, the corresponding growth rate is $r(t,x) = (\Psi^*_{\nu, \boldsymbol q})'(\psi(t,x))$.
		
		At this point the link has to be made with Chapter 5, where an expression similar to \eqref{eq:new_tilde_q_formal} is found, though it may be less apparent because of technical difficulties. 
	\end{Ex}

	\section{Numerics}
	\label{sec:numerics}
	
	In this section, we discuss numerical methods to solve the RUOT problem or the branching Schrödinger problem. We first argue that, contrarily to the case of the Schrödinger problem, the Sinkhorn algorithm is not available in this case. Note that we do not claim that one cannot solve the branching Schrödinger problem by trying to solve its dual formulation, but rather that it should require additional ideas and techniques compared to the Schrödinger problem. On the other hand, the core of this section explains why solving the RUOT problem via its dynamical formulation does not require additional ideas and techniques compared to solving the dynamical formulation of the OT problem.
	
	\subsection{Obstructions for a Sinkhorn-like algorithm}
	\label{subsec:no_sinkhorn}
	
	Let us now explain why an algorithm similar to Sinkhorn's algorithm commonly used to solve regularized optimal transport is not available. In the case of the Schrödinger problem, Sinkhorn's algorithm is a fast and efficient way to solve numerically the problem \cite{peyre2019computational}. One way to understand it is as a \emph{block update} on the dual problem.

	Let $(\tau_s)_{s > 0}$ be the heat kernel on the torus starting from a Dirac at $0$ whose definition is given by formula~\eqref{eq:def_heat_kernel}.
	Even though we work on the torus for convenience, Sinkhorn's algorithm is usually used with $(\tau_s)$ being the heat kernel on the whole space $\R^d$, whose expression is simpler than the one on the torus.
	
	Let us take $\rho_0, \rho_1 \in \P(\T^d)$ two probability measures. The dual of the Schrödinger problem with diffusivity $\nu$ can be written as the maximization over all $\sigma,\theta : \T^d \to \R$ of the quantity
	\begin{equation}
	\label{eq:dual_schrodinger}
	\langle \sigma, \rho_0 \rangle + \langle \theta, \rho_1 \rangle  - \nu \log \E_W \left[  \exp \left( \frac{1}{\nu} \big\{ \sigma(X_0) + \theta(X_1) \big\}\right) \right],
	\end{equation}
	being $W$ the law of the Brownian motion with diffusivity $\nu$, starting from the Lebesgue measure, that we call the reversible Brownian motion. This formula, already well known in the Schrödinger problem's community, was mentioned in Remark~\ref{rk:duality_Schrodinger_pb}. In particular, we can use the explicit formula for the joint law of $W$ at the instants $0,1$ to rewrite~\eqref{eq:dual_schrodinger} as
	\begin{equation}
	\label{eq:dual_schrodinger_explicit}
	\sup_{\sigma,\theta} \int \sigma(x) \rho_0(x)\D x + \int \theta(x) \rho_1(x) \D x  - \nu \log \left( \iint_{\T^d \times \T^d} \frac{1}{c_\nu} \exp \left( \frac{\sigma(x) + \theta(y)}{\nu} - \log \tau_\nu(x-y) \right) \D x \D y \right),
	\end{equation}
	being $c_\nu$ a normalizing constant that modifies the objective only by an additive constant. Note that in $\R^d$, the expression $\log \tau_\nu(x-y)$ boils down to $|x-y|^2/2\nu$. As one can see, the roles of $\sigma$ and $\theta$ are symmetric, up to the exchange of $\rho_0$ and $\rho_1$. Moreover, for a fixed and smooth $\theta$, there exists a closed formula for the $\sigma$ which maximizes~\eqref{eq:dual_schrodinger_explicit}. This is
	\begin{equation*}
	\sigma(x) = \nu \log \rho_0(x) - \nu \log \left( \int_{\T^d} \exp \left( \frac{\theta(y)}{\nu} - \log \tau_\nu(x-y) \right) \D y \right) + C,
	\end{equation*} 
	where the constant $C$ is chosen so that $(x,y) \mapsto \exp(\{\sigma(x)+\theta(y)\}/\nu - \log \tau_\nu(x-y))$ integrates to $1$ over $\T^d \times \T^d$.
	Moreover, there is a symmetric expression for the $\theta$ which maximizes the objective \eqref{eq:dual_schrodinger} when $\sigma$ is fixed. Sinkhorn's algorithms consists in alternatively updating $\sigma$ and $\theta$ according to these expressions and can be proven to converge to a maximizer of the objective functional \eqref{eq:dual_schrodinger_explicit} with a linear rate (see for instance \cite{chen2016entropic} and references therein).
	
	\bigskip
	
	We have seen in formula~\eqref{eq:dual_generic_R} and in the proof of Theorem~\ref{thm:equality_values_dyn} that in our case, we want the functions $\sigma,\theta :\T^d \to \R$ to maximize, instead of~\eqref{eq:dual_schrodinger}, the quantity
	\begin{equation*}
	\langle \sigma, \rho_0 \rangle + \langle \theta, \rho_1 \rangle - \nu \log \E_R \left[  \exp \left( \frac{1}{\nu} \big\{\langle \sigma,M_0 \rangle + \langle \theta,M_1 \rangle \big\}\right) \right]
	\end{equation*}
	where $R \sim \BBM(\nu, \boldsymbol{q}, R_0)$. There are now two main issues which prevent to use the same alternate maximization as for the Schrödinger problem:
	\begin{itemize}
		\item As we have seen above, the computation of $\E_R \left[  \exp \left( \frac{1}{\nu} \big\{\langle \sigma,M_0 \rangle + \langle \theta,M_1 \rangle \big\}\right) \right]$ amounts to solve the PDE \eqref{eq:constraint_phi_BBM}. However, to the best of our knowledge, there is no closed formula for such a solution. Equivalently, by the change of variables described in Remark \ref{rem:FKPP_log_exp}, there is no general closed formula for a solution of the FKPP equation \eqref{eq:FisherKPP}.
		\item Both the RUOT model and the BBM are not symmetric in time, in sharp contrast with what happens for the ROT model and the reversible Brownian motion. Let us justify a bit more this fact.
		
		First, at the level of optimal transport, to make the problem symmetric with respect to time in the case of \emph{balanced} regularized optimal transport, one changes the variables from $(\rho,v)$ to $(\rho, w) = (\rho, v - \frac{\nu}{2} \nabla \log \rho)$ as seen in formula~\eqref{eq:ROT_symmetric}. Indeed, the differential constraint becomes
		\begin{equation*}
		\dr_t \rho + \Div( \rho w ) = r \rho,
		\end{equation*} 
		which removes the asymmetry. Expanding a square, the objective functional becomes
		\begin{equation*}
		\int_0^1 \hspace{-5pt} \int \left\{ \frac{1}{2} |v|^2 + \Psi(r) \right\} \rho = \int_0^1 \hspace{-5pt} \int \left\{ \frac{1}{2} |w|^2 + \frac{\nu}{2} w \cdot \nabla \log \rho + \frac{\nu^2}{8} |\nabla \log \rho|^2 + \Psi(r) \right\} \rho. 
		\end{equation*}
		To analyze the cross term, notice that, thanks to the continuity equation 
		\begin{equation*}
		\int w \cdot \nabla \log \rho = \left( \frac{\D }{\D t}  \int \rho \log \rho \right) - \int r \rho \log \rho. 
		\end{equation*}
		If $r = 0$, then integrating the left hand side in time gives a term which depends only on the temporal boundary conditions, and we recover~\eqref{eq:ROT_symmetric}. However, because there is now a source term in the continuity equation, this strategy breaks and we see no easy way to make this problem symmetric in time.  
		
		On the other hand, at the level of the stochastic process, the absence of temporal symmetry can be seen as follows: there is no $R_0$ (except $R_0 = \delta_0$) and $\boldsymbol q$ (except $\boldsymbol q = 0$) such that the process $R \sim \BBM(\nu,\boldsymbol{q},R_0)$ is reversible, that is such that $(M_{1-t})_{t \in [0,1]}$ has the same law as $(M_t)_{t \in [0,1]}$ under $R$. To see that, note that for any smooth solution $(\psi_t)_{t \in [0,1]}$ of the ODE $\dot{\psi} + \Psi^*_{\nu,\boldsymbol{q}}(\psi) = 0$, with the martingale property (Proposition~\ref{prop:martingale_problem}), we know that $\E_R[\exp(\psi_1 M_1(\T^d)/\nu)] = \E_R[\exp(\psi_0 M_0(\T^d)/\nu)]$. Thus if $M_0$ has the same law as $M_1$, there holds $\E_R[\exp(\psi_1 M_0(\T^d)/\nu)] = \E_R[\exp(\psi_0 M_0(\T^d)/\nu)]$. This is clearly not possible unless $\psi_1 = \psi_0$, for all $\psi_0$ (which enforces $\boldsymbol q$ to be $0$) or $M_0(\T^d) = 0$ $R$-\emph{a.s.}\ (which enforces $R_0$ to be $\delta_0$). This is in contrast with the law of the reversible Brownian motion.  
	\end{itemize}
	
	In the rest of this section, we explain how still the RUOT problem can be discretized in a very similar way as what is usually done for the OT problem and the UOT problem.

	\subsection{Discretization of the dynamical formulation}
	
	It is still possible to discretize the dynamical formulation of RUOT (Definition \ref{def:RUOT}) and to solve the resulting problem by proximal splitting. Such a method to solve the optimal transport problem (without regularization or unbalanced effects) was started by Benamou and Brenier \cite{benamou2000computational} and several convex optimization methods were proposed and analyzed in \cite{Papadakis2014}. This kind of discretizations were applied successfully to optimal transport on surfaces \cite{Lavenant2018}, on graphs \cite{Erbar2017}, to unbalanced optimal transport \cite{chizat2018unbalanced}, to optimal transport of matrix-valued measures \cite{li2020general}, but also to Wassertsein gradient flows \cite{Benamou2016gradient} and variational Mean Field Games \cite{Benamou2015}.
	
	The discretization and convex optimization proposed below follows closely the works mentioned above and we do not claim new innovative ideas. We rather want to emphasize that such method can be applied to our problem with only minor changes. In particular the two following aspects differ from the usual case. 
	\begin{itemize}
		\item Including a $\nu/2 \Delta \rho$ in the continuity equation can be done in a straightforward way and the problem does not become degenerate as $\nu \to 0$.
		\item Having a growth penalization $\Psi$ which is not quadratic is not a severe issue. One needs only to be able to compute the proximal operator of $\Psi^*$ (or equivalently, of $\Psi$). But as $\Psi^*$ is a one-dimensional function one can always rely on Newton's method in the absence of explicit formula. This leads in practice to computations that are slightly slower than for the quadratic case, but one may remedy this issue by \emph{ad hoc} optimization on a specific $\Psi$.   
	\end{itemize}
	
	For simplicity and to avoid cumbersome formulas, we restrict to $d= 1$, that is one space dimension. Extension to $d > 1$ is tedious but does not contain new conceptual difficulty. As we are simply interested in a ``proof of concept'', we stick to this simple setting. The code to reproduce our figure can be accessed at
	\begin{center}
		\label{adress_code}
		\url{https://github.com/HugoLav/RegUnOT}
	\end{center}
	
	\begin{figure}
		\begin{center}
			\begin{tikzpicture}[scale = 1.9]
			
			\draw[dashed] (-0.5,0) -- (3,0) ;
			\draw[dashed] (-0.5,3) -- (3,3) ;
			\draw[dashed] (0,0) -- (0,3) ;
			\draw[dashed] (2.5,0) -- (2.5,3) ;
			
			\draw[->, line width=0.8pt] (-0.3,-0.5) -- (2.8,-0.5) ;
			\draw (2.8,-0.6) node[below]{$x \in \T^d$} ;
			
			\draw[->, line width=0.8pt] (-0.7,0) -- (-0.7,3.4) ;
			\draw (-0.8,3.4) node[above]{$t \in [0,1]$} ;
			
			\draw[fill = black] (-0.7,0) circle (0.03) ;
			\draw[fill = black] (-0.7,3) circle (0.03) ;
			\draw (-0.8,0) node[left]{\small{$t=0$}} ;
			\draw (-0.8,3) node[left]{\small{$t=1$}} ;
			
			\foreach \i in {0,1,...,2}
			{\foreach \j in  {0,1,...,3}
				{\draw[fill = gray!80] (\i, \j+0.1) -- (\i+0.1, \j-0.1) -- (\i-0.1, \j-0.1) -- cycle  ;}}
			
			\foreach \i in {0,1,...,2}
			{\foreach \j in  {0,1,...,2}
				{\draw[fill = black] (\i+0.5, \j+0.5) circle (0.1);}}
			
			\foreach \i in {0,1,...,2}
			{\foreach \j in  {0,1,...,2}
				{\draw[fill = gray!20] (\i-0.1, \j+0.4) rectangle (\i+0.1, \j+0.6);}}
			
			\foreach \i in {0,1,...,2}
			{\foreach \j in  {0,1,...,3}
				{\draw[fill = white, draw = gray!50] (\i+0.5, \j) circle (0.1);}}
			
			
			\draw[fill = black]  (6.5, 3) circle (0.1) ;
			\draw (3.3,3) node[right]{Fully centered grid $\Gc_t \times \Gc_x$} ;
			\draw (3.5,2.75) node[right]{$\zeta$ and $(\tilde{\rho}, \tilde{m}, \tilde{\zeta})$  are defined on it} ;
			
			\draw[fill = white, draw = gray!50] (6.5, 2) circle (0.1); 
			\draw (3.3,2) node[right]{Fully staggered grid $\Gst_t \times \Gst_x$} ;
			\draw (3.5,1.75) node[right]{Nothing is defined on it} ;
			
			\draw[fill = gray!20] (6.4,0.9) rectangle (6.6, 1.1);
			\draw (3.3,1) node[right]{Centered-Staggered grid $\Gc_t \times \Gst_x$} ;
			\draw (3.5,0.75) node[right]{$m$ is defined on it} ;
			
			\draw[fill = gray!80] (6.5, 0.1) -- (6.6, -0.1) -- (6.4, -0.1) -- cycle  ;
			\draw (3.3,0) node[right]{Staggered-centered grid $\Gst_t \times \Gc_x$} ;
			\draw (3.5,-0.25) node[right]{$\rho$ is defined on it} ;
			
			\end{tikzpicture}
		\end{center}
		\caption{Schematic representations of the different space-time grids over which the different unknowns are discertized in the case $N_t=N_x = 3$.}
		\label{fig:grid}
	\end{figure}
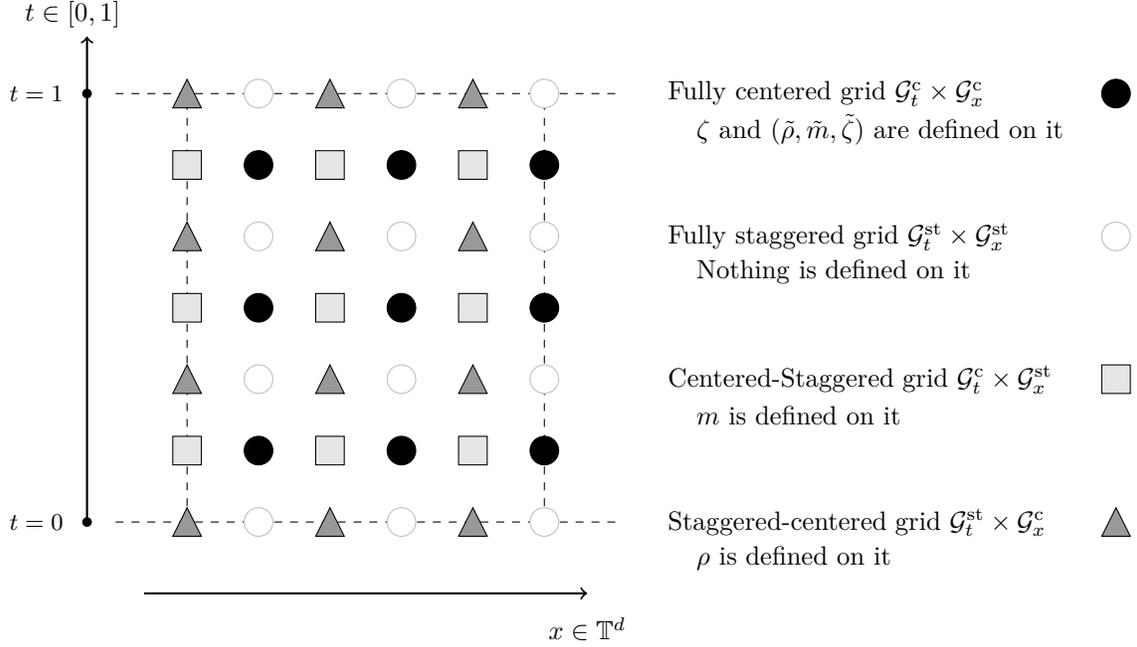
	
	Following \cite{Papadakis2014}, we introduce two kind of grids, the ``centered'' ones $\Gc$ and the ``staggered'' ones $\Gst$. Let $N_t$ and $N_x$ be respectively the number of discretization points in time and space. In space, we define 
	\begin{equation*}
	\Gc_x := \left\{ \frac{j}{N_x} \ : \ 0 \leq j \leq N_x -1 \right\} \subset \T \hspace{1cm} \text{and} \hspace{1cm} 
	\Gst_x := \Gc + \frac{1}{2N_x} \subset \T.
	\end{equation*}  
	As we work on the torus, the grids $\Gst$ and $\Gc$ have the same number of elements ($N_x$) and they are the same up to a translation.  On the other hand, for the temporal interval we define 
	\begin{equation*}
	\Gc_t := \left\{ \frac{i+1/2}{N_x} \ : \ 0 \leq i \leq N_t -1 \right\} \subset [0,1] \hspace{1cm} \text{and} \hspace{1cm} 
	\Gst_x := \left\{ \frac{i}{N_x} \ : \ 0 \leq i \leq N_t \right\} \subset [0,1].
	\end{equation*}
	The centered grid has $N_t$ elements while the staggered one has $N_t+1$. We will also use the notations $\delta_x = 1/N_x$ and $\delta_t = 1/N_t$ for the spatial and temporal steps. In particular, notice that if $x \in \Gc_x$, then $x + \delta_x \in \Gc_x$ while $x + \delta_x/2 \in \Gst_x$, similar statements can be made for the grids in time (modulo some boundary effects around $t=0$ and $t=1$). We emphasize that below, we see $\rho$, $m$ and $\zeta$ as vectors indexed by grid points (and not integers). We refer the reader to Figure \ref{fig:grid} for an example of these grids in the case $N_t=3,N_x=3$. 
	
	We assume that the initial and final measures $\rho_0$ and $\rho_1$ are simply nonnegative functions defined on $\Gc_x$, \emph{i.e.}\ elements of $\R_+^{\Gc_x}$.
	
	\bigskip
	
	The key idea is that the staggered grids will be used to define the continuity equation, while the energy will be defined on the centered grid. Specifically, let
	\begin{equation*}
	\Est = \left\{ U = (\rho,m,\zeta) \in \R^{\Gst_t \times \Gc_x} \times \R^{\Gc_t \times \Gst_x} \times \R^{\Gc_t \times \Gc_x} \right\}
	\end{equation*} 
	the space of ``staggered'' unknowns. The rule of thumb is that, if a derivative is taken with respect to time (resp.\ space) in the continuity equation, then the unknown is defined on the staggered grid in time (resp.\ space). We say that $(\rho,m,\zeta) = U \in \Est$ satisfies the discrete continuity equation with boundary conditions $\rho_0, \rho_1$, and we write $U \in \CE(\rho_0,\rho_1)$ if $\rho_{0,x} = (\rho_0)_x$ and $\rho_{1,x} = (\rho_1)_x$ for all $x \in \Gc_x$ and if for all $(t,x) \in \Gc_t \times \Gc_x$
	\begin{multline*}
	\frac{\rho_{t+\delta_t/2,x} - \rho_{t-\delta_t/2,x}}{\delta_t} + \frac{m_{t,x + \delta_x/2} - m_{t,x - \delta_x/2}}{\delta_x} \\
	= \frac{\nu}{4} \frac{\rho_{t+\delta_t/2,x+\delta_x} + \rho_{t+\delta_t/2,x-\delta_x} - 2 \rho_{t+\delta_t/2,x}}{\delta_x^2} + \frac{\nu}{4} \frac{\rho_{t-\delta_t/2,x+\delta_x} + \rho_{t-\delta_t/2,x-\delta_x} - 2 \rho_{t-\delta_t/2,x}}{\delta_x^2} + \zeta_{t,x}.
	\end{multline*}  
	The only term which is not a clear discretization of $\partial_t \rho + \Div m = \frac{\nu}{2} \Delta \rho + \zeta$ is the one involving the Laplacian. The idea is as follows: as $\rho$ is defined on $\Gst_t \times \Gc_x$, when one takes its discrete space Laplacian, it stays on the grid $\Gst_t \times \Gc_x$. As we want to enforce the continuity equation on $\Gc_t \times \Gc_x$, we do an average in time before plugging it into the discrete continuity equation. Satisfying the continuity equation with the temporal boundary constraints is an affine constraint: that is, it can be rewritten 
	\begin{equation*}
	AU = U_0,
	\end{equation*}
	where $A : \R^{(N_t+1)N_x + 2N_tN_x} \to \R^{N_tN_x + 2N_x}$ is a linear operator (thus it can be represented by a matrix) and $U_0 \in \R^{N_tN_x + 2N_x}$ contains the boundary data $\rho_0,\rho_1$.

	On the other hand we define the ``centered'' space 
	\begin{equation*}
	\Ec = \left\{ V = (\tilde{\rho},\tilde{m},\tilde{\zeta}) \in \R^{\Gc_t \times \Gc_x} \times \R^{\Gc_t \times \Gc_x} \times \R^{\Gc_t \times \Gc_x} \right\}.
	\end{equation*} 
	As everything is defined on the centered grid, the energy can be defined for elements of such a set: if $(\tilde{\rho},\tilde{m},\tilde{\zeta}) = V \in \Ec$ then 
	\begin{equation*}
	E(V) = \sum_{(t,x) \in \Gc_t \times \Gc_x} \delta_t \delta_x \left( \frac{\tilde{m}_{t,x}^2}{\tilde{\rho}_{t,x}} + \Psi \left( \frac{\tilde{\zeta}_{t,x}}{\tilde{\rho}_{t,x}} \right) \tilde{\rho}_{t,x} \right).
	\end{equation*}
	The energy $E$ is a convex function on $\Ec$, which may take the value $+ \infty$. 
	
	Eventually, we need to make a link between the vector $U$ and the vector $V$. Specifically, we enforce $V$ to be the average of $U$. That is, we define the operator $\I : \Est \to \Ec$ by 
	\begin{equation*}
	\I \begin{pmatrix}
	\rho \\ m \\ \zeta 
	\end{pmatrix}_{t,x} = 
	\begin{pmatrix}
	(\rho_{t+\delta t/2,x}+\rho_{t-\delta t/2,x})/2 \\
	(m_{t,x+\delta_x/2}+\rho_{t,x-\delta_x/2})/2 \\
	\zeta_{t,x}
	\end{pmatrix} 
	\end{equation*}
	and we will impose the constraint $\I(U) = V$. 
	
	\begin{Def}
		\label{def:RUOT_discrete}
		With the notations above, if $\rho_0, \rho_1 \in \R_+^{\Gc_x}$, we define the fully discretized regularized unbalanced optimal transport problem as the minimization of $E(V)$ among all $(U,V) \in \Est \times \Ec$ such that $V = \I(U)$ and $U$ satisfies the discrete continuity equation with boundary conditions $\rho_0, \rho_1$. 
	\end{Def}
	
	\noindent As mentioned above, the constraints that $V = \I(U)$ and that $U$ satisfies the discrete continuity equation with boundary conditions $\rho_0, \rho_1$ are ultimately affine constraints in a finite dimensional space. 
	
	\begin{Rem}
		If $\nu=0$, we get exactly the same discretization for unbalanced optimal transport as presented in \cite[Section 5]{chizat2018unbalanced}. If in addition we drop the $\zeta$ variable, then it boils down to the discretization introduced in \cite{Papadakis2014}.
	\end{Rem}  
	
	\begin{Rem}
		A natural question is to understand if the discretization is convergent, that is if one recovers a solution of the original problem in the limit $N_t, N_x \to + \infty$. This is a delicate question already in the case of plain dynamical optimal transport which has been investigated in \cite{Erbar2017} (for the temporal discretization), \cite{gladbach2020scaling} (for the spatial discretization) and the second author in \cite{lavenant2020unconditional} (for both the temporal and spatial discretization). The framework of the latter article has been adapted to a mixed finite element discretization in \cite{natale2020mixed} and extended to matrix-valued unbalanced optimal transport in \cite{li2020general}. As it would take us too far, we will not try to answer this question here. Although adding the unbalanced term should be a straightforward adaptation, the presence of diffusion in the continuity equation seems to be harder to handle, especially for the ``controllability'' arguments \cite[Assumption (A7) and Remark 2.11]{lavenant2020unconditional}.
	\end{Rem}
	
	\subsection{Proximal splitting and the Douglas-Rachford algorithm}
	
	This subsection aims at presenting tools in non-smooth convex optimization that we use to solve the discretized version of RUOT. We work in the framework of a (finite dimensional) Euclidean space $\Hil$ endowed with a scalar product $\langle \cdot, \cdot \rangle$ and the corresponding norm $\| \cdot \|$. 
	
	\begin{Def}
		If $F : \Hil \to \R \cup \{ + \infty \}$ is a proper l.s.c.\ convex functional, we define the ``proximal'' operator $\prox_F : \Hil \to \Hil$ by 
		\begin{equation*}
		\prox_F(h) = \argmin_{h' \in \Hil} \, \frac{\| h - h' \|^2}{2} + F(h').
		\end{equation*}
	\end{Def} 
	
	One can show that this operator is well defined (that is, there exists a unique minimizer), that it is nonexpansive \cite[Theorem 6.42]{beck2017first} and that $\prox_{\gamma F} \to \mathrm{Id}$ when $\gamma \to 0$. The point $h' = \prox_F(h)$ is characterized by the optimality condition 
	\begin{equation*}
	h' - h \in \, - \partial F(h'),
	\end{equation*}
	being $\partial F$ the subdifferential of $F$. This enables to interpret $\prox_{\gamma F}(h)$ as one step of an implicit Euler discretization of the gradient flow $\dot{h}_t \in \, - \nabla F(h_t)$ starting from $h$ with time step $\gamma > 0$.
	
	If $C \subset \Hil$ is a convex set and $\iota_C : \Hil \to \{ 0, + \infty \}$ is the $0/+ \infty$ indicator of $C$, then for all $\gamma > 0$, $\prox_{\gamma \iota_C} = \proj_C$ is the Euclidean projection on $C$.
	
	Eventually, let us mention Moreau's identity (see for instance \cite[Section 2.5]{parikh2014proximal}): if $F^*$ is the Legendre transform of $F$ then 
	\begin{equation*}
	\prox_F + \prox_{F^*} = \mathrm{Id},
	\end{equation*}
	which implies in particular by an easy scaling argument that for all $\gamma > 0$ and $h \in \Hil$, 
	\begin{equation}
	\label{eq:Moreau_general}
	\prox_{\gamma F}(h) = h - \gamma \prox_{F^*/\gamma} \left( \frac{h}{\gamma} \right).
	\end{equation}
	
	The Douglas-Rachford algorithm is an iterative algorithm to compute the minimum of the sum of two convex functions $F$ and $G$, if ones knows how to compute the proximal operators for $F$ and $G$ separately. We directly state the following convergence result which can be found in \cite[Theorem 20]{combettes2007douglas}.
	
	\begin{Prop}
		\label{prop:Douglas_Rachford}
		Let $F,G : \Hil \to \R \cup \{ + \infty \}$ be two proper l.s.c convex functions bounded from below, defined on a (finite dimensional) Euclidean space $\Hil$. Let $\alpha \in (0,2)$ and $\gamma > 0$ two parameters. 
		
		Let $((w_n,z_n))_{n \in \N}$ be a sequence in $\Hil^2$ satisfying the following recurrence relation: for all $n \in \N$,
		\begin{equation*}
		\begin{cases}
		w_{n+1} &= w_n + \alpha \left( \prox_{\gamma F}(2z_n - w_n) - z_n \right), \\
		z_{n+1} & = \prox_{\gamma G}(w_{n+1}).
		\end{cases}
		\end{equation*}
		Then whatever the initial value $(w_0,z_0)$ is, the sequence $(z_n)_{n \in \N}$ converges in $\Hil$ to a point $\bar{z}$ which minimizes the function $F+G$ over $\Hil$.
	\end{Prop} 	
	
	This result is striking by its level of generality, as almost no assumptions are made on $F,G$ besides convexity. On the other hand in practice the convergence can be very slow. The parameters $\alpha$ and $\gamma$ can be freely chosen, usually one can tune them to get the fastest possible convergence.

	\bigskip
	
	In our case we follow the ``Asymmetric--DR'' splitting scheme of \cite{Papadakis2014}. Specifically we define $\Hil = \Est \times \Ec$, that is a point $h \in \Hil$ is $h = (U,V)$ with $U \in \Est$ and $V \in \Ec$. For the scalar product, for scaling reasons, it is better to not take the canonical scalar product on $\R^{\mathrm{dim}(\Hil)}$, but to multiply it by the scalar $\delta_t \delta_x$. Let us call $\CE$ the set of $U \in \Est$ that satisfy the discrete continuity equation with boundary conditions $\rho_0,\rho_1$. We define the functional $F$ as 
	\begin{equation*}
	F(U,V) := \iota_{\CE(\rho_0,\rho_1)}(U) + E(V),
	\end{equation*}
	while the functional $G$ is 
	\begin{equation*}
	G(U,V) := \iota_{\I(U) = V}.
	\end{equation*}
	With these notations, it is clear that minimizing $F+G$ over $\Hil$ is the same as solving the discrete regularized unbalanced optimal transport problem as defined in Definition \ref{def:RUOT_discrete}.

	\subsection{Computation of the proximal operators}

	Thus, to be able to solve the problem numerically, using Proposition \ref{prop:Douglas_Rachford} and the decomposition above, we just need to have an efficient way to compute $\prox_{\gamma F}$ and $\prox_{\gamma G}$. The proximal operator of $F$ decomposes as  
	\begin{equation*}
	\prox_{\gamma F}(U,V) = \begin{pmatrix}
	\proj_{\CE(\rho_0,\rho_1)}(U), \\
	\prox_{\gamma E}(V).
	\end{pmatrix}
	\end{equation*}
	Note also that $\prox_{\gamma G} = \proj_{\{ \I(U) = V \}}$. 
	
	\bigskip
	
	The computation of $\proj_{\CE(\rho_0,\rho_1)}(U)$ and $\proj_{\{ \I(U) = V \}}$ amounts to project onto an affine subspace of a Euclidean space. In general, if ones looks at the affine space
	\begin{equation*}
	\Hil_a = \{ h \in \Hil \ : \ Bh = h_0 \}
	\end{equation*}
	for some matrix $B$, then 
	\begin{equation*}
	\proj_{\Hil_a}(h) = h - B^\top (B B^\top )^{-1} (B h - h_0).
	\end{equation*}
	Notice that if $B B^\top$ is not invertible, provided $\Hil_a$ is not empty the vector $B h- h_0$ belongs to the range of $B B^\top$, and one can replace $(B B^\top )^{-1} (B h - h_0)$ by any solution $h'$ of the system $B B^\top h' = Bh- h_0$. In short, one needs only to be able to invert the matrix $B B^\top$. 
	
	In the case of the set $\I(U) = V$, the matrix $B$ is nothing else than 
	\begin{equation*}
	\begin{pmatrix}
	\I & 0 \\
	0 & - \mathrm{Id}
	\end{pmatrix},
	\end{equation*}
	and this matrix is well conditioned independently of $N_t$ and $N_x$, and can be factorized independently on the temporal and the spatial dimension.
	
	The marginal constraint $U \in \CE(\rho_0, \rho_1)$ can be ultimately expressed as $AU = U_0$ for some matrix $A \in\R^{[N_tN_x + 2N_x] \times [(N_t+1)N_x + 2N_tN_x]}$ and some vector $U_0 \in \R^{N_tN_x + 2N_x}$. Moreover, the matrix $A A^\top$ corresponds to the discretization of the space-time elliptic operator 
	\begin{equation*}
	\phi \mapsto - \partial^2_{tt} \phi + \frac{\nu^2}{4} \Delta^2 \phi - \Delta \phi + \phi  
	\end{equation*}
	with Neumann temporal boundary conditions. In particular, as one can see, the system is not degenerate if $\nu \to 0$. The only issue is the conditioning number of $A A^\top$ that we expect to behave like $O(N_x^2 + \nu^2 N_x^4 + N_t^2)$. 
	
	For both constraints $\I(U) = V$ and $U \in \CE(\rho_0, \rho_1)$, we compute a Cholesky factorization of the relevant matrices before the Douglas-Rachford iterations.

	\bigskip
	
	The less trivial case may be the one of $\prox_{\gamma E}$. First, we define $e : \R^3 \to \R$ as 
	\begin{equation*}
	e(\rho,m,\zeta) = \frac{|m|^2}{2 \rho} + \Psi \left( \frac{\zeta}{\rho} \right) \rho,
	\end{equation*}
	in such a way that if $(\tilde{\rho},\tilde{m},\tilde{\zeta}) = V \in \Ec$ then 
	\begin{equation*}
	E(V) = \sum_{(t,x) \in \Gc_t \times \Gc_x} \delta_t \delta_x e \left( \tilde{\rho}_{t,x}, \tilde{m}_{t,x}, \tilde{z}_{t,x} \right).
	\end{equation*}
	As the scalar product is weighted by $\delta_t \delta_x$, using first an additive decomposition and then Moreau's identity~\eqref{eq:Moreau_general}, for every $(t,x) \in \Gc_t \times \Gc_x$ 
	\begin{equation*}
	\prox_{\gamma E}(V)_{t,x} = \prox_{\gamma e}\left( \tilde{\rho}_{t,x}, \tilde{m}_{t,x}, \tilde{\zeta}_{t,x} \right) = \left( \tilde{\rho}_{t,x}, \tilde{m}_{t,x}, \tilde{\zeta}_{t,x} \right) - \gamma \prox_{e^* / \gamma} \left( \frac{1}{\gamma} \tilde{\rho}_{t,x}, \frac{1}{\gamma}\tilde{m}_{t,x}, \frac{1}{\gamma}\tilde{\zeta}_{t,x} \right) .
	\end{equation*}
	On the other hand, $e^*$ is the $0/+ \infty$ indicator of the set $\mathcal{K} \subset \R^3$ defined by 
	\begin{equation*}
	\mathcal{K} = \left\{ (a,b,c) \in \R^3 \ : \  a + \frac{b^2}{2} + \Psi^*(c) \leq 0 \right\}.
	\end{equation*}
	Thus we just need to be able to project on $\mathcal{K}$.  
	
	\begin{Lem}
		Let $(a,b,c) \notin \mathcal{K}$. We define $r > 0$ as the unique positive solution of 
		\begin{equation}
		\label{eq:prox_k_to_solve}
		a - r + \frac{b^2}{2(1+r)^2} + \Psi^* \left( \prox_{r \Psi^*}(c) \right) = 0 
		\end{equation} 
		Then there holds
		\begin{equation*}
		\proj_{\mathcal{K}}(a,b,c) = \left( a - r, \frac{b}{1+r}, \prox_{r \Psi^*}(c) \right).
		\end{equation*}
	\end{Lem}
	
	\begin{proof}
		Let $(\overline{a},\overline{b},\overline{c})$ the projection of $(a,b,c)$ on $\mathcal{K}$. This projection is characterized by $(\overline{a},\overline{b},\overline{c}) \in \dr \mathcal{K}$, and in addition the vector $(a-\overline{a},b-\overline{b},c-\overline{c})$ is colinear to the outward normal to $\mathcal{K}$ at $(\overline{a},\overline{b},\overline{c})$. Calling $r > 0$ the coefficient of colinearity, we are looking at the system
		\begin{equation*}
		\overline{a} + \frac{\overline{b}^2}{2} + \Psi^*(\overline{c}) = 0  \hspace{1cm} \text{and} \hspace{1cm}
		\begin{cases}
		a - \overline{a} & = r, \\
		b - \overline{b} & = r \overline{b}, \\
		c - \overline{c} & = r (\Psi^*)'(\overline{c}).
		\end{cases}
		\end{equation*}  
		Notice that the equation $c - \overline{c}  = r (\Psi^*)'(\overline{c})$ can be rewritten $\overline{c} = (\mathrm{Id} + r (\Psi^*)')^{-1}(c)$, which exactly reads $\overline{c} = \prox_{r \Psi^*}(c)$, as $\Psi^*$ is smooth and convex. Expressing $(\overline{a},\overline{b},\overline{c})$ as a function of $r \in (0, + \infty)$ and plugging it in the equation characterizing $(\overline{a},\overline{b},\overline{c}) \in \dr \mathcal{K}$, we get \eqref{eq:prox_k_to_solve}. 
		
		Eventually, it is left as an easy exercise to the reader to check that $r \mapsto \Psi^*(\prox_{r \Psi^*}(c))$ is nonincreasing (this would work for any convex function $\Psi^*$), thus the left hand side of \eqref{eq:prox_k_to_solve} is a strictly nondecreasing function of $r$. This justifies the uniqueness of the solution $r$ of this equation. 
	\end{proof}
	
	If $\Psi^*$ is quadratic then solving \eqref{eq:prox_k_to_solve} is equivalent to finding the root of a polynomial of degree $3$. In the general case we use Newton's method. Indeed, as soon as $\Psi^*$ is smooth then
	\begin{equation*}
	\frac{\D}{\D r} \Psi^*(\prox_{r \Psi^*}(c)) = - \frac{(\Psi^*)'( \prox_{r \Psi^*}(c) )}{1+ r (\Psi^*)''( \prox_{r \Psi^*}(c) )},
	\end{equation*}
	thus once we have access to $\prox_{r \Psi^*}(c)$ computing the right hand side is cheap. 
	
	\subsection{Experiments}
	
	We implemented the method described above and we tested in the case
	\begin{equation}
	\label{eq:psi_for_sim}
	\Psi^*(s) = C \left( \cosh\left( s \right) - 1 \right).
	\end{equation}
	for some constant $C$ that was chosen in order to have a visually satisfying output. Its proximal operator cannot be computed in closed form, so we resort to Newton's method to compute it. 
	Note that it mimics the case $p_0 = p_2 = 1/2$, that is, when at a branching event, each particle dies or splits with uniform probability. As pointed out in Remark~\ref{rk:psistar_cosh} it would amount to take
	\begin{equation*}
	\Psi^*(s) = \nu \lambda \left( \cosh\left( \frac{s}{\nu} \right) - 1 \right),
	\end{equation*}
	however we found that we such expression, in the range of value of $\nu \sim 10^{-2}$ giving a diffusion effect not too strong nor too light, was leading to a very degenerate function $\Psi^*$ (in some sense we are already in the ``small noise'' regime) thus we preferred to stick to \eqref{eq:psi_for_sim} which is more well behaved. The outputs of the algorithm are plotted in the introduction, see Figure~\ref{fig:numerics}. By running the same code with different values of $C$ and $\nu$, we were able to simulate Optimal Transport, Regularized Optimal Transport, Unbalanced Optimal Transport and Regularized Optimal Transport. 
	
	We chose a grid with $24$ time points and $100$ space point, and the initial and final conditions is a balanced mixture of two Gaussians. We ran a fix number of iterations ($10^4$), and we controlled the accuracy of the solution by looking at the norm of $\gamma^{-1}(z_{n+1} - z_n, w_{n+1} - w_n)$, as well as how much the $U$ component of $z_n$ satisfies the continuity equation. As usual with dynamical optimal transport, computations can be a bit slow and took of the order of a few minutes on a regular desktop with 4-core 3.60GHz Intel i3 processor with 16GB RAM.
	
	We refer the reader to the address displayed Page~\pageref{adress_code} to access the code, where they can find in particular all the parameters used to produce the figures.
	
	\appendix
	
	\chapter{Illustrations and properties of the growth penalization}
	\label{app:plots}
	In this appendix, we provide a few illustrations and describe some of the interesting properties of the growth penalization $\Psi_{\nu, \boldsymbol q}$ that we get in the RUOT problem, depending on the diffusivity coefficient $\nu>0$ and the branching mechanism $\boldsymbol q$ that we choose for the reference BBM in the corresponding branching Schrödinger problem. Recall that calling $\lambda = \lambda_{\boldsymbol q} := \sum_k q_k$ and $\boldsymbol p = \boldsymbol p_{\boldsymbol q} := \boldsymbol q / \lambda_{\boldsymbol q}$, the map $\Psi_{\nu, \boldsymbol q} = \Psi_{\nu, \lambda \boldsymbol p}$ is defined through its Legendre transform $\Psi^*_{\nu, \lambda \boldsymbol p}$, the latter being defined in terms of $\nu$, $\lambda$ and~$\boldsymbol p$ by
	\begin{equation}
	\label{eq:def_Psi*_appendix}
	\Psi^*_{\nu, \lambda \boldsymbol p}(s) := \nu \lambda \sum_k p_k \left\{ e^{(k-1)\frac{s}{\nu}} - 1 \right\}, \qquad s \in \R.
	\end{equation}
	Therefore, the purpose of this section is to plot the graph and to give some properties of $\Psi_{\nu, \lambda \boldsymbol p}$ for different choices of $\nu>0$, $\lambda >0$, and $\boldsymbol p \in \P(\N)$.  
	
	In a first section, we will study the effect of the scaling coefficients $\nu$ and $\lambda$. Then, we will set $\nu = \lambda = 1$, and we will have a look at the different types of penalization that we get depending on $\boldsymbol p$.
	
	To make it short, we will see that:
	\begin{itemize}
		\item  $\Psi_{\nu, \boldsymbol q}$ cancels at $\bar r = r_{\boldsymbol q} := \sum_k (k-1) q_k$.
		\item  When $q_0 = 0$, then $\Psi_{\nu, \boldsymbol q}$ is infinite on $\R_-^*$, and when $\sum_{k\geq 2} q_k = 0$, then $\Psi_{\nu, \boldsymbol q}$ is infinite on $\R_+^*$.
		\item When $\boldsymbol q$ has no exponential moment, then $\Psi_{\nu, \boldsymbol q}$ cancels on $(\bar r, + \infty)$, and hence is not coercive.
		\item When $\boldsymbol q$ has exponential moments up to the order $\theta_{\boldsymbol q}>0$, then $\Psi_{\nu, \boldsymbol q}$ has a linear growth of slope $\theta_{\boldsymbol q}$ at $r \to + \infty$. Hence in this case, it is not superlinear.
	\end{itemize}  
	
	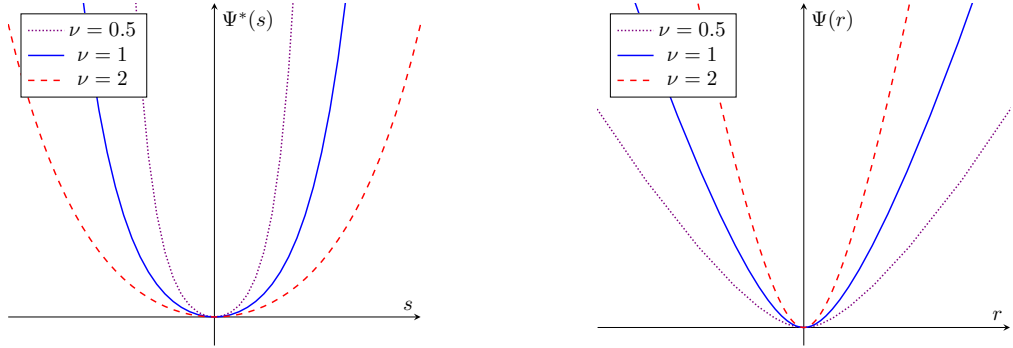
\begin{figure}
		\begin{center}
			\begin{tikzpicture}[scale = .8]
			\begin{axis}[ticks=none,
			axis x line=center,
			axis y line =center,
			xmin=-5, xmax=5,
			ymin = -1, ymax = 11,
			restrict x to domain*=-10:10,
			restrict y to domain*=-2:50,
			xlabel={$s$},ylabel={$\Psi^*(s)$},
			legend pos=north west]
			\addplot[color=violet, densely dotted ,line width = 0.7pt] table [x=s, y=p]{psi_symmetric_nu=0.5.csv};
			\addplot[color=blue,line width = 0.7pt] table [x=s, y=p]{psi_symmetric.csv};
			\addplot[color=red, dashed ,line width = 0.7pt] table [x=s, y=p]{psi_symmetric_nu=2.csv};
			\legend{$\nu = 0.5$, $\nu = 1$, $\nu = 2$};
			\end{axis}
			\end{tikzpicture}
			\hspace{60pt}
			\begin{tikzpicture}[scale = .8]
			\begin{axis}[ticks=none,
			axis x line=center,
			axis y line =center,
			xmin=-11, xmax=11,
			ymin = -1, ymax = 18,
			restrict x to domain*=-50:50,
			restrict y to domain*=-2:25,
			xlabel={$r$},ylabel={$\Psi(r)$},
			legend pos=north west]
			\addplot[color=violet, densely dotted ,line width = 0.7pt] table [x=r, y=q]{psi_symmetric_nu=0.5.csv};
			\addplot[color=blue,line width = 0.7pt] table [x=r, y=q]{psi_symmetric.csv};
			\addplot[color=red, dashed ,line width = 0.7pt] table [x=r, y=q]{psi_symmetric_nu=2.csv};
			\legend{$\nu = 0.5$, $\nu = 1$, $\nu = 2$};
			\end{axis}
			\end{tikzpicture}
		\end{center}
		\caption{
			\label{fig:effect_nu} Illustration of $\Psi^*_{\nu, \lambda \boldsymbol p}$ and $\Psi_{\nu, \lambda \boldsymbol p}$ in the case where $\lambda = 1$, $p_0 = p_2 =1/2$, $p_k = 0$ for $k \neq 0,2$, and $\nu = 1$ (plain line), $\nu = 0.5$ (dotted line) and $\nu = 2$ (dashed line).
		}
	\end{figure}
	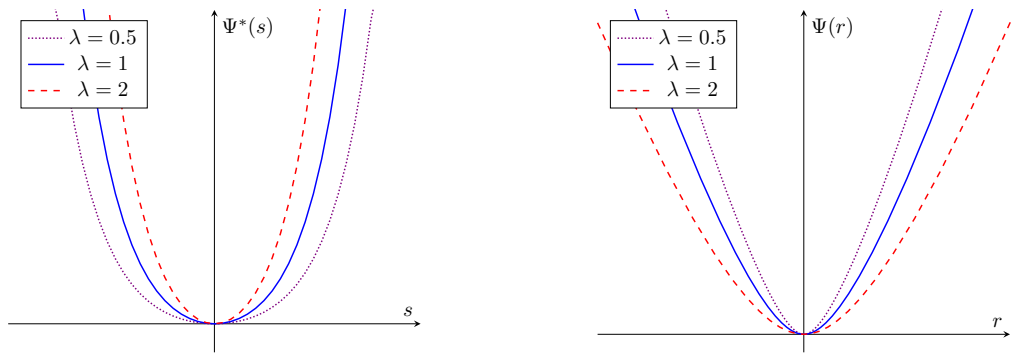
\begin{figure}
		\begin{center}
			\begin{tikzpicture}[scale = .8]
			\begin{axis}[ticks=none,
			axis x line=center,
			axis y line =center,
			xmin=-5, xmax=5,
			ymin = -1, ymax = 11,
			restrict x to domain*=-10:10,
			restrict y to domain*=-2:12,
			xlabel={$s$},ylabel={$\Psi^*(s)$},
			legend pos=north west]
			\addplot[color=violet, densely dotted ,line width = 0.7pt] table [x=s, y=p]{psi_symmetric_lambda=0.5.csv};
			\addplot[color=blue,line width = 0.7pt] table [x=s, y=p]{psi_symmetric.csv};
			\addplot[color=red, dashed ,line width = 0.7pt] table [x=s, y=p]{psi_symmetric_lambda=2.csv};
			\legend{$\lambda = 0.5$, $\lambda = 1$, $\lambda = 2$};
			\end{axis}
			\end{tikzpicture}
			\hspace{60pt}
			\begin{tikzpicture}[scale = .8]
			\begin{axis}[ticks=none,
			axis x line=center,
			axis y line =center,
			xmin=-11, xmax=11,
			ymin = -1, ymax = 18,
			restrict x to domain*=-50:50,
			restrict y to domain*=-2:25,
			xlabel={$r$},ylabel={$\Psi(r)$},
			legend pos=north west]
			\addplot[color=violet, densely dotted ,line width = 0.7pt] table [x=r, y=q]{psi_symmetric_lambda=0.5.csv};
			\addplot[color=blue,line width = 0.7pt] table [x=r, y=q]{psi_symmetric.csv};
			\addplot[color=red, dashed ,line width = 0.7pt] table [x=r, y=q]{psi_symmetric_lambda=2.csv};
			\legend{$\lambda = 0.5$, $\lambda = 1$, $\lambda = 2$};
			\end{axis}
			\end{tikzpicture}
		\end{center}
		\caption{
			\label{fig:effect_lambda} Illustration of $\Psi^*_{\nu, \lambda \boldsymbol p}$ and $\Psi_{\nu, \lambda \boldsymbol p}$ in the case where $\nu = 1$, $p_0 = p_2 =1/2$, $p_k = 0$ for $k \neq 0,2$, and $\lambda = 1$ (plain line), $\lambda = 0.5$ (dotted line) and $\lambda = 2$ (dashed line).
		}
	\end{figure}
	\section{The effect of the diffusivity coefficient and of the branching rate}
	
	Remark the following scaling invariance for $\Psi^*_{\nu, \lambda \boldsymbol p}$: for all $\boldsymbol p \in \P(\N)$, $\nu>0$ and $\lambda>0$, we have, thanks to~\eqref{eq:def_Psi*_appendix},
	\begin{equation*}
	\Psi^*_{\nu, \lambda \boldsymbol p}(s) = \nu \lambda \Psi_{1, \boldsymbol p}^*\left( \frac{s}{\nu} \right), \qquad s \in \R.
	\end{equation*}
	Easy computations using this formula and the definition of the Legendre transform show that for all $\boldsymbol p \in \P(\N)$, $\nu>0$ and $\lambda>0$,
	\begin{equation*}
	\Psi_{\nu, \lambda \boldsymbol p}(r) = \nu \lambda \Psi_{1, \boldsymbol p}\left( \frac{r}{\lambda} \right), \qquad r \in \R.
	\end{equation*}
	We illustrate these scaling properties in Figure~\ref{fig:effect_nu} and Figure~\ref{fig:effect_lambda} in the symmetric case where $p_0 = p_2 = 1/2$, and $p_k=0$ for $k \neq 0,2$. In the next sections, we set $\lambda = \nu = 1$, and we only study the dependence of $\Psi_{1, \boldsymbol p}$ w.r.t.\ $\boldsymbol p$.
	
	\begin{figure}
		\begin{center}
			\begin{tikzpicture}[scale = .8]
			\begin{axis}[ticks=none,
			axis x line=center,
			axis y line =center,
			xmin=-3.5, xmax=6.5,
			ymin = -1, ymax = 11,
			restrict x to domain*=-10:10,
			restrict y to domain*=-2:12,
			xlabel={$s$},ylabel={$\Psi^*(s)$}]
			\addplot[color=blue,line width = 0.7pt] table [x=s, y=p]{psi_shift_gauche.csv};
			\draw[dashed, thick] (.893,-1) -- (-2.5,2.8) node[above=-2pt]{$y = \bar r s$};
			\end{axis}
			\end{tikzpicture}
			\hspace{60pt}
			\begin{tikzpicture}[scale = .8]
			\begin{axis}[ticks=none,
			axis x line=center,
			axis y line =center,
			xmin=-6.2, xmax=3.8,
			ymin = -1, ymax = 11,
			restrict x to domain*=-50:50,
			restrict y to domain*=-2:25,
			xlabel={$r$},ylabel={$\Psi(r)$}]
			\addplot[color=blue,line width = 0.7pt] table [x=r, y=q]{psi_shift_gauche.csv};
			\draw[thick] (-1.13,3pt) -- (-1.13,-3pt) node[below]{$\bar r$};
			\end{axis}
			\end{tikzpicture}
		\end{center}
		\caption{\label{fig:shift_left} Illustration of $\Psi^*_{1, \boldsymbol p}$ and $\Psi_{1,\boldsymbol p}$ in the case where $p_0 = 0.95$, $p_2 = 0.05$ and $p_k = 0$ for $k\neq 0,2$.}
	\end{figure}
	\begin{figure}
		\begin{center}
			\begin{tikzpicture}[scale = .8]
			\begin{axis}[ticks=none,
			axis x line=center,
			axis y line =center,
			xmin=-5.5, xmax=1.7,
			ymin = -1, ymax = 11,
			restrict x to domain*=-10:10,
			restrict y to domain*=-2:12,
			xlabel={$s$},ylabel={$\Psi^*(s)$}]
			\addplot[color=blue,line width = 0.7pt] table [x=s, y=p]{psi_shift_droite.csv};
			\draw[dashed, thick] (-0.446,-1) -- (1.2,2.688) node[above=-4pt]{$y = \bar r s$};
			\end{axis}
			\end{tikzpicture}
			\hspace{60pt}
			\begin{tikzpicture}[scale = .8]
			\begin{axis}[ticks=none,
			axis x line=center,
			axis y line =center,
			xmin=-2.5, xmax=17,
			ymin = -.5, ymax = 5.5,
			restrict x to domain*=-50:50,
			restrict y to domain*=-2:25,
			xlabel={$r$},ylabel={$\Psi(r)$}]
			\addplot[color=blue,line width = 0.7pt] table [x=r, y=q]{psi_shift_droite.csv};
			\draw[thick] (2.24,3pt) -- (2.24,-3pt) node[below]{$\bar r$};
			\end{axis}
			\end{tikzpicture}
		\end{center}
		\caption{\label{fig:shift_right} Illustration of $\Psi^*_{1, \boldsymbol p}$ and $\Psi_{1,\boldsymbol p}$ in the case where $p_0 = 0.1$, $p_4 = 0.9$ and $p_k = 0$ for $k\neq 0,4$.}
	\end{figure}
	
	\section{The case where the branching mechanism is non-degenerate}
	So now, $\lambda = \nu = 1$. We say that $\boldsymbol p$ is non-degenerate if:
	\begin{itemize}
		\item The probability of dying with no descendant is positive, that is, $p_0>0$.
		\item The probability of having descendants is positive, that is, $\sum_{k \geq 2} p_k >0$.
		\item $\boldsymbol p$ has exponential moments at any order, so that in particular, $\Psi^*_{1, \boldsymbol p}(s)$ is finite for all $s \in \R$.
	\end{itemize}
	The consequence of dropping one of these assumptions is studied in further sections (when dropping several of them, the effects would simply add up). In the case of a non-degenerate $\boldsymbol p$, it is possible to check that $\Psi_{1, \boldsymbol p}$ is superlinear at $\pm \infty$. Moreover, $\Psi_{1, \boldsymbol p}$ cancels at $r = \bar r := \sum_k (k-1)p_k$: in terms of our optimization problems, it means that it costs nothing to consider a competitor of the RUOT problem having the same growth rate as the reference BBM of the corresponding branching Schrödinger problem. Depending on $\boldsymbol p$, the value of $\bar r$ can be nonpositive (Figure~\ref{fig:shift_left}) or nonnegative (Figure~\ref{fig:shift_right}). Notice that it corresponds to the derivative of $\Psi^*_{1, \boldsymbol p}$ at $s = 0$.
	
	It can be shown that whenever $p_0>0$, as $r \to - \infty$, then $\Psi_{1, \boldsymbol p}(r) \sim |r| \log |r|$. Moreover, in this section, we took for our illustrations $\boldsymbol p$ with bounded support. In that case, as $r \to + \infty$, then $\Psi_{1, \boldsymbol p}(r) \sim |r| \log |r|/ k_l$, where $k_l$ is the largest integer in the support of $\boldsymbol p$. However, the growth can be slower when $\boldsymbol p$ has an unbounded support while being non-degenerate.

	\begin{figure}
		\begin{center}
			\begin{tikzpicture}[scale = .8]
			\begin{axis}[ticks=none,
			axis x line=center,
			axis y line =center,
			xmin=-2, xmax=1.2,
			ymin = -1, ymax = 11,
			restrict x to domain*=-10:10,
			restrict y to domain*=-2:12,
			xlabel={$s$},ylabel={$\Psi^*(s)$}]
			\addplot[color=blue,line width = 0.7pt] table [x=s, y=p]{psi_creation_only.csv};
			\end{axis}
			\end{tikzpicture}
			\hspace{60pt}
			\begin{tikzpicture}[scale = .8]
			\begin{axis}[ticks=none,
			axis x line=center,
			axis y line =center,
			xmin=-2, xmax=11,
			ymin = -.3, ymax = 3.5,
			restrict x to domain*=-50:50,
			restrict y to domain*=-2:25,
			xlabel={$r$},ylabel={$\Psi(r)$}]
			\addplot[color=blue,line width = 0.7pt] table [x=r, y=q]{psi_creation_only.csv};
			\addplot[only marks,mark=*,mark size=2.9pt,color=blue] coordinates {(0,1)};
			\end{axis}
			\end{tikzpicture}
		\end{center}
		\caption{\label{fig:creation_only} Illustration of $\Psi^*_{1, \boldsymbol p}$ and $\Psi_{1,\boldsymbol p}$ in the case where $p_2 = 0.2$, $p_4 = 0.8$, and $p_k = 0$ for $k\neq 2,4$. The value of $\Psi_{1, \boldsymbol p}$ is infinite on $\R_-^*$.}
	\end{figure}
	
	\begin{figure}
		\begin{center}
			\begin{tikzpicture}[scale = .8]
			\begin{axis}[ticks=none,
			axis x line=center,
			axis y line =center,
			xmin=-3, xmax=5,
			ymin = -1, ymax = 11,
			restrict x to domain*=-10:10,
			restrict y to domain*=-2:12,
			xlabel={$s$},ylabel={$\Psi^*(s)$}]
			\addplot[color=blue,line width = 0.7pt] table [x=s, y=p]{psi_death_only.csv};
			\end{axis}
			\end{tikzpicture}
			\hspace{60pt}
			\begin{tikzpicture}[scale = .8]
			\begin{axis}[ticks=none,
			axis x line=center,
			axis y line =center,
			xmin=-7, xmax=.8,
			ymin = -.5, ymax = 7,
			restrict x to domain*=-50:50,
			restrict y to domain*=-2:25,
			xlabel={$r$},ylabel={$\Psi(r)$}]
			\addplot[color=blue,line width = 0.7pt] table [x=r, y=q]{psi_death_only.csv};
			\addplot[only marks,mark=*,mark size=2.9pt,color=blue] coordinates {(0,1)};
			\end{axis}
			\end{tikzpicture}
		\end{center}
		\caption{\label{fig:death_only} Illustration of $\Psi^*_{1, \boldsymbol p}$ and $\Psi_{1,\boldsymbol p}$ in the case where $p_0 = 1$ and $p_k = 0$ for $k\neq 0$. The value of $\Psi_{1, \boldsymbol p}$ is infinite on $\R_+^*$.}
	\end{figure}
	
	\section{The case where there is either never or always descendants at a branching event}
	\label{sec:psi_infinite}
	In this section, we still set $\lambda = \nu = 1$, and we study what happens when $p_0 = 0$ (there are always descendants at a branching event), or when $\sum_{k\geq 2} p_k = 0$ (there is never any descendant at a branching event).
	
	In the first case, $\Psi_{1, \boldsymbol p}$ is infinite on $\R_-^*$ (Figure~\ref{fig:creation_only}): as the reference BBM cannot reduce its number of particles, it is forbidden to consider competitors of the corresponding RUOT problem which lose mass. In the second case, it is the other way around: $\Psi_{1, \boldsymbol p}$ is infinite on $\R_+^*$ (Figure~\ref{fig:death_only}), and as the reference BBM cannot create new particles, it is forbidden to consider competitors of the corresponding RUOT problem which gain mass. In this second case, taking $p_0 = 1$, we can actually carry explicit computations: we find $\Psi_{1, \boldsymbol p} = l(|r|) = |r| \log |r| - |r| + 1$ on $\R_-$, which coincides with formula (33) in \cite{chen2021most}.
	
	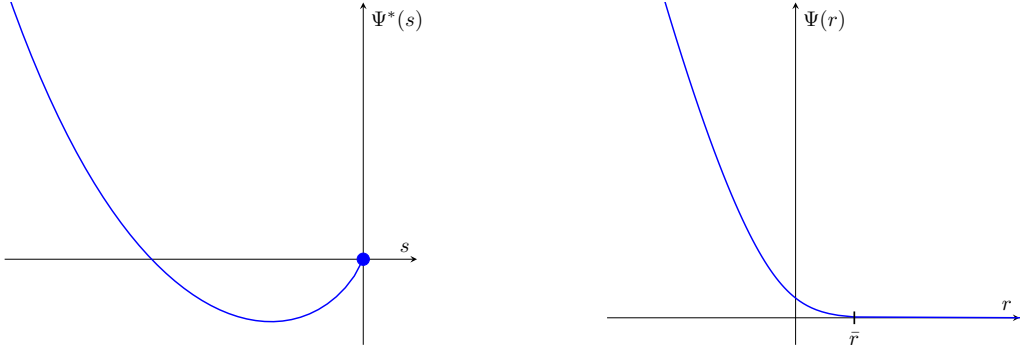
\begin{figure}
		\begin{center}
			\begin{tikzpicture}[scale = .8]
			\begin{axis}[ticks=none,
			axis x line=center,
			axis y line =center,
			xmin=-2, xmax=.3,
			ymin = -.3, ymax = .9,
			restrict x to domain*=-10:10,
			restrict y to domain*=-2:12,
			xlabel={$s$},ylabel={$\Psi^*(s)$}]
			\addplot[color=blue,line width = 0.7pt] table [x=s, y=p]{psi_no_exponential_moment.csv};
			\addplot[only marks,mark=*,mark size=2.9pt,color=blue] coordinates {(0,0)};
			\end{axis}
			\end{tikzpicture}
			\hspace{60pt}
			\begin{tikzpicture}[scale = .8]
			\begin{axis}[ticks=none,
			axis x line=center,
			axis y line =center,
			xmin=-3.2, xmax=3.8,
			ymin = -.3, ymax = 3.5,
			restrict x to domain*=-50:50,
			restrict y to domain*=-2:25,
			xlabel={$r$},ylabel={$\Psi(r)$}]
			\addplot[color=blue,line width = 0.7pt] table [x=r, y=q]{psi_no_exponential_moment.csv};
			\draw[thick] (.997,3pt) -- (.997,-3pt) node[below]{$\bar r$};
			\end{axis}
			\end{tikzpicture}
		\end{center}
		\caption{\label{fig:no_exp_moment} Illustration of $\Psi^*_{1, \boldsymbol p}$ and $\Psi_{1,\boldsymbol p}$ in the case where for all $k \geq 2$, $p_k = 1/(k-1)^{2.2}$, and $p_0 = 1 - \sum_{k\geq 2} p_k$. In this case, $\boldsymbol p$ has no exponential moment, that is, condition~\eqref{eq:exponential_moment_q} from Assumption~\ref{ass:exponential_bounds} is not satisfied.}
	\end{figure}
	
	\section{The case where the branching mechanism has no exponential moment}
	\label{sec:no_exp_moment}
	Still in the setting where $\lambda = \nu = 1$, we are here interested to the case where $\boldsymbol p$ has no exponential moment. Of course, in this case, the support of $\boldsymbol p$ is unbounded. This is exactly the case where condition~\eqref{eq:exponential_moment_q} from Assumption~\ref{ass:exponential_bounds} does not hold. In this case, $\Psi_{1, \boldsymbol p}$ cancels at $\bar r := \sum_k(k-1) p_k$ (which is always assumed to be finite, see the Definition~\ref{def:branching_mechanism} of branching mechanism), but also \emph{stays at zero for larger values of} $r$, see Figure~\ref{fig:no_exp_moment}. If the reference BBM has such a branching mechanism, it is free to have arbitrarily large growth rates in the corresponding RUOT problem.
	
	Note that in this case, $\Psi_{1, \boldsymbol p}$ is not coercive, and it is remarkable that we are still able to get a duality results for the RUOT problem (Theorem~\ref{thm:existence_ruot}) and the equivalence of the values between our two problems (Theorem~\ref{thm:equality_values_dyn}) in this context.
	
	\section{The case where the branching mechanism has exponential moments, but not at any order}
	\label{sec:exp_moment}
	Finally, we treat the case where $\lambda = \nu = 1$, and where $\boldsymbol p$ has exponential moments, but not at any order. Of course, as in the previous section, the support of $\boldsymbol p$ needs to be unbounded. In this case, condition~\eqref{eq:exponential_moment_q} from Assumption~\ref{ass:exponential_bounds} holds.
	
	We call $\theta$ the supremum of those $s \in \R$ for which $\Psi^*_{1, \boldsymbol p}(s) < + \infty$. By assumption, we necessarily have $\theta>0$. There are three scenarios for the behavior of $\Psi^*_{1, \boldsymbol p}(s)$ as $s$ tends to $\theta$ from below: $\Psi^*_{1, \boldsymbol p}(s) \to \infty$ (Figure~\ref{fig:infinite_at_smax}), $\Psi^*_{1, \boldsymbol p}(\theta)< + \infty$ but the derivative $(\Psi^*_{1, \boldsymbol p})'(\theta)= + \infty$ (Figure~\ref{fig:infinite_derivative_at_smax}), or $\Psi^*_{1, \boldsymbol p}(\theta)< + \infty$ and $(\Psi^*_{1, \boldsymbol p})'(\theta)< + \infty$ (Figure~\ref{fig:finite_derivative_at_smax}). But in the three cases, we can check that the slope of $\Psi_{1, \boldsymbol p}$ remains lower than $\theta$. In particular, $\Psi_{1, \boldsymbol p}$ is \emph{not superlinar at} $+\infty$. Notice that at $-\infty$ this type of behavior cannot occur.
	
	This fact is one the main reasons why we need to deal with costs that are not superlinear when introducing the RUOT problem in Section~\ref{sec:presentation_RUOT}.
	
	\begin{figure}
		\begin{center}
			\begin{tikzpicture}[scale = .8]
			\begin{axis}[ticks=none,
			axis x line=center,
			axis y line =center,
			xmin=-2.4, xmax=.9,
			ymin = -.3, ymax = 2.2,
			restrict x to domain*=-10:10,
			restrict y to domain*=-2:12,
			xlabel={$s$},ylabel={$\Psi^*(s)$}]
			\addplot[color=blue,line width = 0.7pt] table [x=s, y=p]{psi_infinite_at_smax.csv};
			\draw[thick, dashed] (0.6932,10) -- (0.6932,0) node[below=1pt]{$\theta$};
			\end{axis}
			\end{tikzpicture}
			\hspace{60pt}
			\begin{tikzpicture}[scale = .8]
			\begin{axis}[ticks=none,
			axis x line=center,
			axis y line =center,
			xmin=-3.2, xmax=11,
			ymin = -.3, ymax = 5.5,
			restrict x to domain*=-50:50,
			restrict y to domain*=-2:25,
			xlabel={$r$},ylabel={$\Psi(r)$}]
			\addplot[color=blue,line width = 0.7pt] table [x=r, y=q]{psi_infinite_at_smax.csv};
			\draw[thick, dashed] (-0.5,-.3466) -- (15,10.397);
			\draw (9.3,5.3) node{$y = \theta r$};
			\end{axis}
			\end{tikzpicture}
		\end{center}
		\caption{\label{fig:infinite_at_smax}Illustration of $\Psi^*_{1, \boldsymbol p}$ and $\Psi_{1,\boldsymbol p}$ in the case where for all $k \geq 2$, $p_k = 1/2^{k-1}(k-1)$, and $p_0 = 1 - \sum_{k\geq 2} p_k$. In this case, $\boldsymbol p$ has some exponential moments, $\theta = \log 2$, and $\Psi^*_{1, \boldsymbol p}(\theta) = + \infty$.}
	\end{figure}
	
	\begin{figure}
		\begin{center}
			\begin{tikzpicture}[scale = .8]
			\begin{axis}[ticks=none,
			axis x line=center,
			axis y line =center,
			xmin=-2.1, xmax=1.1,
			ymin = -.3, ymax = 2,
			restrict x to domain*=-10:10,
			restrict y to domain*=-2:12,
			xlabel={$s$},ylabel={$\Psi^*(s)$}]
			\addplot[color=blue,line width = 0.7pt] table [x=s, y=p]{psi_derivative_infinite_at_smax.csv};
			\draw[thick, dashed] (0.6932,10) -- (0.6932,0) node[below=1pt]{$\theta$};
			\addplot[only marks,mark=*,mark size=2.9pt,color=blue] coordinates {(0.693,0.844)};
			\end{axis}
			\end{tikzpicture}
			\hspace{60pt}
			\begin{tikzpicture}[scale = .8]
			\begin{axis}[ticks=none,
			axis x line=center,
			axis y line =center,
			xmin=-5.2, xmax=7,
			ymin = -.3, ymax = 3.6,
			restrict x to domain*=-50:50,
			restrict y to domain*=-2:25,
			xlabel={$r$},ylabel={$\Psi(r)$}]
			\addplot[color=blue,line width = 0.7pt] table [x=r, y=q]{psi_derivative_infinite_at_smax.csv};
			\draw[thick, dashed] (-0.5,-.3466) -- (15,10.397);
			\draw (3.5,3.4) node{$y = \theta r$};
			\end{axis}
			\end{tikzpicture}
		\end{center}
		\caption{\label{fig:infinite_derivative_at_smax}Illustration of $\Psi^*_{1, \boldsymbol p}$ and $\Psi_{1,\boldsymbol p}$ in the case where for all $k \geq 2$, $p_k = 1/2^{k-1}(k-1)^2$, and $p_0 = 1 - \sum_{k\geq 2} p_k$. In this case, $\boldsymbol p$ has some exponential moments, $\theta = \log 2$, $\Psi^*_{1, \boldsymbol p}(\theta) < + \infty$, but $(\Psi^*)'_{1, \boldsymbol p}(\theta) = + \infty$.}
	\end{figure}
	
	\begin{figure}
		\begin{center}
			\begin{tikzpicture}[scale = .8]
			\begin{axis}[ticks=none,
			axis x line=center,
			axis y line =center,
			xmin=-2.1, xmax=1.1,
			ymin = -.3, ymax = 2.2,
			restrict x to domain*=-10:10,
			restrict y to domain*=-2:12,
			xlabel={$s$},ylabel={$\Psi^*(s)$}]
			\addplot[color=blue,line width = 0.7pt] table [x=s, y=p]{psi_derivative_finite_at_smax.csv};
			\draw[thick, dashed] (0.6932,10) -- (0.6932,0) node[below=1pt]{$\theta$};
			\addplot[only marks,mark=*,mark size=2.9pt,color=blue] coordinates {(0.694,0.4346)};
			\end{axis}
			\end{tikzpicture}
			\hspace{60pt}
			\begin{tikzpicture}[scale = .8]
			\begin{axis}[ticks=none,
			axis x line=center,
			axis y line =center,
			xmin=-5.2, xmax=7,
			ymin = -.3, ymax = 3.6,
			restrict x to domain*=-50:50,
			restrict y to domain*=-2:25,
			xlabel={$r$},ylabel={$\Psi(r)$}]
			\addplot[color=blue,line width = 0.7pt] table [x=r, y=q]{psi_derivative_finite_at_smax.csv};
			\draw[thick, dashed] (-0.5,-.3466) -- (15,10.397);
			\draw (3.5,3.4) node{$y = \theta r$};
			\end{axis}
			\end{tikzpicture}
		\end{center}
		\caption{\label{fig:finite_derivative_at_smax}Illustration of $\Psi^*_{1, \boldsymbol p}$ and $\Psi_{1,\boldsymbol p}$ in the case where for all $k \geq 2$, $p_k = 1/2^{k-1}(k-1)^3$, and $p_0 = 1 - \sum_{k\geq 2} p_k$. In this case, $\boldsymbol p$ has some exponential moments, $\theta = \log 2$, $\Psi^*_{1, \boldsymbol p}(\theta) < + \infty$, and $(\Psi^*)'_{1, \boldsymbol p}(\theta) < + \infty$. Here, for all $r \geq (\Psi^*)'_{1, \boldsymbol p}(\theta)$, we have $(\Psi_{1,\boldsymbol p})'(r) = \theta$. This means that for $r$ large enough, the growth is exactly linear at rate $\theta$.}
	\end{figure}
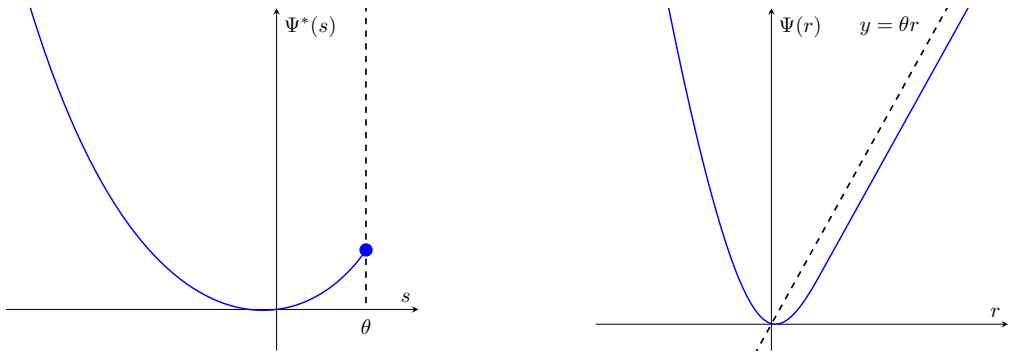

	\chapter{Construction of the processes of Section~\ref{sec:new_processes}}
	\label{app:stochastic_calculus}
	
	The purpose of this appendix is to build the processes introduced in Chapter~\ref{chap:characterization_finite_entropy} to characterize the laws with finite entropy with respect to the branching Brownian motion. Fixing a diffusivity parameter $\nu>0$, a branching mechanism $\boldsymbol{q}$, and a family of initial positions for the particles $\xx = (x_1, \dots, x_p)\in \cup_{k \in \N}(\T^d)^k$ for the whole chapter (except at Section~\ref{sec:large_to_restricted}), we will first define these processes on the large probability space $(\Omega^\xx, \F^\xx, \Prob^\xx)$ introduced to define the BBM in Subsection~\ref{subsec:def_BBM}. In this setting, the heuristic definitions from Section~\ref{sec:new_processes} (in particular the one for $I$, the generalized stochastic integral) can be made rigorous, and the Itô formula from Theorem~\ref{thm:branching_Ito} is easily proved. 
	
	The technical part of the argument will be to show that up to a negligible set, the generalized stochastic integral and the pure jump process that we end-up with are measurable with respect to the $\sigma$-algebra generated by $M^0$, and hence corresponds to processes on the canonical space $\Omega = \cadlag([0,1]; \M_+(\T^d))$, that are well defined $R^\mu$-almost surely, for all $\mu \in \M_\delta(\T^d)$.

	\section{Two different filtrations}
	In this section, we work on the ``large'' probability space $(\Omega^\xx, \F^\xx, \Prob^\xx)$ defined in Subsection~\ref{subsec:def_BBM}, that we endow with the same family of independent random variables $(e_n,L_n,Y_n)_{n \in \mathcal{U}}$ as there. We call $\mu := \delta_{x_1} + \dots +\delta_{x_p}$.
	
	The main difficulty is that we will work with two different filtrations. The first one, that we will call the \emph{large} filtration, is
	\begin{equation*}
	\F^0_t := \sigma \Big\{ \iota(b_n,t),\, \iota(d_n,t),\, X_n(s) \ : \ n \in \mathcal{U}, \, s \in [0,t] \Big\}, \qquad t \in [0,1],
	\end{equation*}
	where $\iota(a,b) = a$ if $a \leq b$ and $+ \infty$ otherwise. In this definition, the conventional choices that we made for those $t$ for which $X_n(t) \notin \mathcal N^0(t)$ in Convention~\ref{conv:trajectory} are important. This filtration contains all the information about the process up to time $t$, including the labeling of the particles. It will be rather easy to define our processes using standard tools on scalar or vector-valued processes (as opposed to measure-valued processes).
	
	The second filtration is:
	\begin{equation}
	\label{eq:def_restricted_filtration}
	\bar \F_t := \sigma \Big( M^0_s \ : \  s \in [0,t] \Big), \qquad t \in [0,1].
	\end{equation}
	We will refer to it as the \emph{restricted} filtration. It only contains the information on the positions of the particles, and not on their labeling. 
	\begin{Rem} 
		\label{rem:large_to_restricted}
		Notice that for all adapted (resp.\ predictable, see Definition~\ref{def:predictable_fields}) field $a$ on the canonical space , composing them to the right with the map
		\begin{equation}
		\label{eq:projection_canonical_space}
		(\omega,t,x) \in \Omega^\xx \times [0,1] \times \T^d \mapsto ((M^0_t(\omega))_{t \in [0,1]}, t, x) \in \XX.
		\end{equation}
		provides a random field $a^0$ on $\Omega^\xx$, that turns out to be adapted (resp.\ predictable) w.r.t.\ the restricted filtration $(\bar \F_t)_{t \in [0,1]}$. In return, for all adapted (resp. predictable) process $Z^0 = (Z^0_t)_{t \in [0,1]}$ on $(\Omega^\xx, (\bar \F_t)_{t \in [0,1]})$, there exists an adapted (resp.\ predictable) process $Z = (Z_t)_{t \in [0,1]}$ on $(\Omega, (\F_t)_{t \in [0,1]})$ such that $Z^0$ is obtained from $Z$ by composing it to the right with the map~\eqref{eq:projection_canonical_space}.
	\end{Rem}
	
	The technical part of this chapter will be to show that the processes that we will define w.r.t.\ the large filtration, will be actually adapted w.r.t.\ the restricted filtration, and hence sent onto the canonical space by the procedure just described. Before getting to that, let us first give some intuition by proving that the large filtration is finer that the restricted one, and by providing some natural measurability properties of the random variables involved in the definition of the BBM.
	\begin{Prop}
		\label{prop:measurability_properties}
		Let $t \in [0,1]$. We have
		\begin{equation*}
		\bar \F_t \subset \F^0_t.
		\end{equation*}
		
		In addition, for all $n \in \mathcal U$, $b_n$ and $d_n$ are $(\F^0_t)_{t \in [0,1]}$-stopping times, and for all $t \in [0,1]$, $\1_{d_n \leq t} L_n$ is $\F^0_t$-measurable.
	\end{Prop} 
	\begin{proof}
		For the first claim, it suffices to prove that  for all $0 \leq s \leq t \leq 1$, $M^0_s$ is measurable w.r.t.\ $\bar \F_t$. This follows from
		\begin{equation*}
		M^0_s = \sum_{n \in \mathcal U} \1_{\iota(b_n,t) \leq s < \iota(d_n,t)} \delta_{X_n(s)}.
		\end{equation*}
		
		For the second claim, we consider $t \in [0,1]$ and $n \in \mathcal U$, and we show that $b_n \wedge t$, $d_n \wedge t$ and $\1_{d_n \leq t} L_n$ are $\F^0_t$-measurable. This comes from:
		\begin{equation*}
		b_n \wedge t = \iota (b_n, t) \wedge t, \quad d_n \wedge t = \iota (d_n, t) \wedge t, \quad \1_{d_n \leq t} L_n = \sum_{i \in \N} \1_{b_{ni} \leq t} = \sum_{i \in \N} \1_{\iota(b_{ni},t) \leq t}.
		\end{equation*}
	\end{proof}
	\section{Generalized stochastic integral at the level of trees}
	
	At the level of the large probability space $\Omega^\xx$, we are able to define our generalized stochastic integral for all random vector field $v^0$ that is predictable in the sense of Definition~\ref{def:predictable_fields} w.r.t.\ the large filtration $(\F^0_t)_{t \in [0,1]}$, provided the following integrability holds:
	\begin{equation}
	\label{eq:integrability_v0}
	\Em_\xx\left[ \int_0^1 \left\cg |v^0(t)|^2, M^0_t \right\cd \D t \right] < + \infty.
	\end{equation}
	So we give ourselves once for all such a random vector field. In a similar way as in Subsection~\ref{subsec:def_I}, we define our generalized stochastic integral (on the large probability space) by the formula:
	\begin{equation}
	\label{eq:def_stochastic_integral_trees}
	\mathcal{I}[v^0]_t := \sum_{n \in \mathcal{U}} \int_0^t \1_{b_n < s \leq d_n} v^0(s, X_n(s))\cdot \D X_n(s), \qquad t \in [0,1].
	\end{equation}
	The following proposition asserts that this definition makes sense, and gives some properties of $\mathcal I$. It can be seen as an analogue of Theorem~\ref{thm:definition_stochastic_integral} at the level of the large probability space $\Omega^\xx$.
	\begin{Prop}
		\label{prop:generalized_stochastic_integral_tree}
		\begin{itemize}
			\item For all $n \in \mathcal U$, the process
			\begin{equation*}
			t \in [0,1] \mapsto \1_{b_n < t \leq d_n} v^0(t, X_n(t))
			\end{equation*}
			is predictable w.r.t.\ to the large filtration $(\F^0_t)_{t \in [0,1]}$.
			\item For all $n \in \mathcal U$, the process $X_n := (X_n(t))_{t \in [0,1]}$ is a continuous $\Prob^\xx$-martingale w.r.t.\ the large filtration $(\F^0_t)_{t \in [0,1]}$. Its quadratic variation satisfies
			\begin{equation}
			\label{eq:quadratic_variation_Xn}
			\mathrm [X_n]_t = \sum_{n' \leq n} \1_{b_{n'} < t \leq d_{n'}}\times \Id \times\nu \D t,
			\end{equation}
			where $\mathrm{Id}$ is the $d$-dimensional identity matrix.
			
			Moreover, if $n, n' \in \mathcal U$, the quadratic covariation of $X_n$ and $X_{n'}$ satisfies:
			\begin{equation}
			\label{eq:formula_bracket_Xn}
			\1_{b_n < t \leq d_n} \1_{b_{n'} < t \leq d_{n'}}\!\times \mathrm d [X_n, X_{n'}]_t = \1_{n = n'} \times  \mathrm{Id} \times \nu \D t, \qquad t \in [0,1].
			\end{equation}
		\end{itemize}
		Consequently, $\mathcal I[v^0]$ from formula~\eqref{eq:def_stochastic_integral_trees} is a well defined continuous martingale w.r.t.\ the large filtration $(\F^0_t)_{t \in [0,1]}$, and its quadratic variation is
		\begin{equation}
		\label{eq:quadratic_variation_I_tree}
		\big[ \mathcal I[v^0] \big]_t = \nu \int_0^t \left\cg |v^0(s)|^2, M^0_s \right\cd \D s, \qquad t \in [0,1].
		\end{equation}
		Finally, for all $0\leq t_0 < t_1 \leq 1$ and for all bounded real-valued random variable $Z^0_{t_0}$ that is $\F^0_{t_0}$-measurable,
		\begin{equation}
		\label{eq:truncation_stochastic_integral_tree}
		\mathcal I[Z^0_{t_0} \1_{(t_0,t_1]} v^0]_t = Z^0_{t_0} \Big( \mathcal I[v^0]_{t_1 \wedge t} - \mathcal I[v^0]_{t_0 \wedge t} \Big), \qquad \forall t \in [0,1].
		\end{equation}
	\end{Prop}
	\begin{proof}
		Let us start with the first point of the statement. Let us take $n \in \mathcal U$. In order to show that $t \mapsto \1_{b_n < t \leq d_n} v^0(t, X_n(t))$ is $(\F^0_t)_{t \in [0,1]}$-predictable, we will show that both $t \mapsto \1_{b_n < t \leq d_n}$ and $t \mapsto v^0(t, X_n(t))$ are $(\F^0_t)_{t \in [0,1]}$-predictable. For the first one, let us observe that it is left continuous, and adapted as a consequence of Proposition~\ref{prop:measurability_properties}, so that the result is a consequence of Definition~\ref{def:predictable_process}. The second one is the composition of the measurable map $v^0: (\XX, \tilde \G) \to (\R^d, \mathcal B(\R^d))$ with the map $(\omega, t) \in (\Omega^\xx \times [0,1], \G) \to (\omega, t, X_n(t))\in (\XX, \tilde \G)$, which is easily seen to be measurable as a consequence of the fact that $X_n$ is adapted w.r.t. $(\F^0_t)_{t \in [0,1]}$ and continuous, and hence predictable.
		
		Let us now prove that for all $n \in \mathcal U$, $X_n$ is a $(\F^0_t)_{t \in [0,1]}$-martingale under $\Prob^\xx$. To do this, we need to be very careful with our choices stated at Convention~\ref{conv:trajectory}. This choices imply the following formula, where we use the notations of Definition~\ref{def:partial_ordering} (in particular, we call $k$ the number such that $n = n(1) \dots n(k)$):
		\begin{equation}
		\label{eq:formula_Xn}
		X_n(t) =  \sum_{l = 1}^k \1_{b_{n_l} < t} Y_{n_l}(t\wedge d_{n_l}-b_{n_l}).
		\end{equation}
		This can for instance be proved by inference over $k$, and by recalling that if $n \notin \mathcal T$, then we choose $b_n = + \infty$. Let us show that all the terms in the r.h.s.\ of formula~\eqref{eq:formula_Xn} is a $(\F^0_t)_{t \in [0,1]}$-martingale under $\Prob^\xx$. Namely, we need to show that for all $n \in \mathcal U$, the process
		\begin{equation*}
		\1_{b_n < t} Y_n(t \wedge d_n - b_n), \qquad t \in [0,1],
		\end{equation*}
		is a $(\F^0_t)_{t \in [0,1]}$-martingale under $\Prob^\xx$. To prove that it is adapted, we just need to observe the following formula, which holds for all $t \in [0,1]$, and which is a consequence of~\eqref{eq:formula_Xn}:
		\begin{equation}
		\label{eq:X_to_Y}
		\1_{b_n < t} Y_n(t \wedge d_n - b_n) = X_n(t) - X_n(t \wedge b_n),
		\end{equation}
		and to recall that $t \wedge b_n$ is $\F^0_t$-measurable. Finally, the martingale property is a consequence of the fact that $Y_n$ is a Brownian motion, that is independent of all the other random variables involved in the definition of the process, and from the fact that both $b_n$ and $d_n$ are $(\F^0_t)_{t \in [0,1]}$-stopping times.	
		
		Formula~\eqref{eq:X_to_Y} easily implies~\eqref{eq:quadratic_variation_Xn} (by inference over $k$, and the fact that the quadratic variation of a Brownian motion $X$ of diffusivity $\nu$ satisfies $\mathrm d[X]_t =\mathrm{Id} \times \nu \D t$). Moreover, it also shows that for all $n \in \mathcal U$ and $t \in [0,1]$,
		\begin{equation*}
		\1_{b_n < t \leq d_n} X_n(t) = \1_{b_n < t \leq d_n}\Big\{ X_n(b_n) + Y_n(t-b_n) \Big\},
		\end{equation*}
		so that~\eqref{eq:formula_bracket_Xn} directly follows from the standard facts that the bracket of two independent semi-martingales cancel.
		
		By the two first points of the statement, for all $n \in \mathcal U$ and all $t \in [0,1]$, the stochastic integral
		\begin{equation*}
		\int_0^t \1_{b_n < s \leq d_n} v^0(s, X_n(s))\cdot \D X_n(s)
		\end{equation*}
		is well d. To deduce that $\mathcal I[v^0]$ defined by formula~\eqref{eq:def_stochastic_integral_trees} is indeed a well-defined martingale, we can use the fact that because of~\eqref{eq:q_finite_mean}, only a finite number of terms do not cancel. Another way that automatically provides~\eqref{eq:quadratic_variation_I_tree} is to use the Hibertian structure of the set of squarred integrable continuous martingales (see \cite[Paragraph~5.1.1]{legall2016brownian}). With this approach, the fact that $\mathcal I[v^0]$ is well defined and formula~\eqref{eq:quadratic_variation_I_tree} are consequences of the following computation, valid for all $t \in [0,1]$:
		\begin{align*}
		\Bigg[ \sum_{n \in \mathcal{U}} \int_0^t \1_{b_n < s \leq d_n} &v^0(t, X_n(s))\cdot \D X_n(s) \Bigg]_t\\
		&= \sum_{n,n' \in \mathcal U} \int_0^t v^0(s, X_n(s)) \cdot v^0(s, X_{n'}(s)) \1_{b_n < s \leq d_n}\1_{b_{n'} < s \leq d_{n'}} \mathrm d [X_n, X_{n'}]_s\\
		&= \nu \int_0^t  \sum_{n\in \mathcal U}  |v^0(s, X_n(s))|^2 \1_{b_n < s \leq d_n}\D s = \nu \int_0^t \left\cg |v^0(s)|^2, M^0_s \right\cd \D s,
		\end{align*}
		where the second equality is obtained using formula~\eqref{eq:formula_bracket_Xn}, and where the last one uses the fact that $M^0$ only undergoes a finite number of discontinuities.
		
		Let us finally prove formula~\eqref{eq:truncation_stochastic_integral_tree}. Let us consider $t_0,t_1, Z_{t_0}^0$ as in the statement and $t \in [0,1]$. Remarking that the vector field $Z^0_{t_0}\1_{(t_0,t_1]}v^0$ is predictable, this is a direct computation using the definition~\eqref{eq:def_stochastic_integral_trees} of $\mathcal I$:
		\begin{align*}
		\mathcal I[Z^0_{t_0}\1_{(t_0,t_1]}v^0]_t &= \sum_{n \in \mathcal{U}} \int_0^t \1_{b_n < s \leq d_n}Z_{t_0}^0 \1_{(t_0,t_1]}(s) v^0(s, X_n(s))\cdot \D X_n(s)\\
		&= Z_{t_0}^0  \sum_{n \in \mathcal{U}} \int_{t \wedge t_0}^{t \wedge t_1} \1_{b_n < s \leq d_n}(s) v^0(s, X_n(s))\cdot \D X_n(s) \\
		&=  Z^0_{t_0} \Big( \mathcal I[v^0]_{t_1 \wedge t} - \mathcal I[v^0]_{t_0 \wedge t} \Big),
		\end{align*}
		as announced.
	\end{proof}

	\section{Simple jump processes at the level of trees}
	The definition of the simple jump processes from Subsection~\ref{subsec:pure_jump} is much easier. At the level of the large probability space $\Omega^\xx$, we will be able to define them for all random scalar fields on $\Omega^\xx$ that are predictable with respect to the large filtration $(\F^0_t)_{t \in [0,1]}$. Given such a field $a^0$ and some $k \in \N\backslash\{1\}$, we define our pure jump process $(\mathcal J_k[a^0]_t)_{t \in [0,1]}$ as follows.
	\begin{Def}
		\label{def:pure_jump_process_tree}
		For all $t \in [0,1]$, we call
		\begin{equation}
		\label{eq:first_formula_J_tree}
		\mathcal J^k[a^0]_t := \sum_{n \in \mathcal U} \1_{d_n \leq t} \1_{L_n = k} a^0\big(d_n,X_n(d_n)\big),
		\end{equation}
		which makes sense because as a consequence of~\eqref{eq:q_finite_mean}, $\Prob^\xx$-\emph{a.s.}, only a finite number of terms is nonzero.
	\end{Def}
	
	Let us check the properties of this pure jump process given at Theorem~\ref{thm:properties_jump_processes}, at the level of the large probability space $\Omega^\xx$. 
	
	\begin{Prop}
		\label{prop:jump_processes_tree}
		Let $(a^0_k)_{k \neq 1}$ be a family of random scalar fields on $\Omega^\xx$, that are all predictable w.r.t.\ the large filtration $(\F^0_t)_{t \in [0,1]}$.
		
		Let us call
		\begin{equation*}
		\mathcal J_t := \sum_{k \neq 1} \mathcal J^k[a^0_k]_t, \qquad t \in [0,1].
		\end{equation*}
		There holds:
		\begin{itemize}
			\item The process $\mathcal J = (\mathcal J_t)_{t \in [0,1]}$ is a well defined simple pure-jump process on the filtered probability space $(\Omega^\xx, \F^\xx, \Prob^\xx,(\F^0_t)_{t \in [0,1]})$.
			\item Its jump measure is
			\begin{equation}
			\label{eq:jump_measure_tree}
			\sum_{k \neq 1} a^0_k(t) \pf\big( q_k M^0_{t-}\big), \qquad t \in [0,1].
			\end{equation}
			\item  Given a function $f \in C_b(\R)$, two times $0 \leq t_0 < t_1 \leq 1$ and a bounded real-valued random variable $Z^0_{t_0}$ that is $\F^0_{t_0}$-measurable, we have for all $t \in [0,1]$,
			\begin{equation}
			\label{eq:truncation_jump_tree}
			Z^0_{t_0} \Big( f(\mathcal J_{t \wedge t_1}) - f(\mathcal J_{t \wedge t_0}) \Big) = \sum_{k \neq 1} \mathcal J^k[b^0_k]_t,
			\end{equation}
			where for all $k \neq 1$ and $t \in [0,1]$, $b^0_k(t) := Z_{t_0}^0\1_{t \in (t_0,t_1]} (f(\mathcal J_{t-} + a^0_k(t)) - f(\mathcal J_{t-}))$.
		\end{itemize} 
	\end{Prop}
	\begin{proof}
		First, $\mathcal J = (\mathcal J_t)_{t \in [0,1]}$ is well defined, almost surely $\cadlag$, piecewise constant, and undergoes a finite number of jumps. Indeed, by definition, for all $t \in [0,1]$,
		\begin{align*}
		\mathcal J_t &= \sum_{k \neq 1} \sum_{n \in \mathcal U} \1_{d_n \leq t} \1_{L_n = k} a_k^0\big(d_n,X_n(d_n)\big) = \sum_{n \in \mathcal U} \1_{d_n \leq t} a_{L_n}^0\big(d_n,X_n(d_n)\big)\\
		&= \sum_{n \in \mathcal U \ : \ d_n \leq 1} \1_{d_n \leq t} a_{L_n}^0\big(d_n,X_n(d_n)\big).
		\end{align*}	
		Hence, the result is a direct consequence of the fact that $\{ n \in \mathcal U \ : \ d_n \leq 1 \}$ is almost surely finite, as a consequence of~\eqref{eq:q_finite_mean}.
		
		Let us show that $\mathcal J$ is adapted w.r.t.\ the large filtration $(\F^0_t)_{t \in [0,1]}$. Let us fix $t \in [0,1]$, and show that $\mathcal J_t$ is $\F^0_t$-measurable. This is a consequence of the formula
		\begin{equation*}
		\mathcal J_t =  \sum_{n \in \mathcal U } \1_{d_n \leq t} a_{L_n}^0\big(d_n,X_n(d_n)\big) = \sum_{n \in \mathcal U } \1_{d_n \leq t} a_{\1_{d_n\leq t} L_n}^0\big(d_n \wedge t, \1_{d_n \leq t} X_n(d_n\wedge t)\big)
		\end{equation*}
		and of the facts that $a^0$ is predictable, that $d_n$ is a $(\F^0_t)_{t \in [0,1]}$-stopping time, and that $\1_{d_n \leq t} L_n$ is $\F^0_t$-measurable.
		
		Before proving~\eqref{eq:jump_measure_tree}, let us prove~\eqref{eq:truncation_jump_tree}. Let us give ourselves $f,t_0,t_1,Z^0_{t_0}$ as in the statement of the theorem, consider $t \in [0,1]$, and define $(b^0_k)_{k \neq 1}$ as above. We have
		\begin{align*}
		\sum_k \mathcal J^k[b^0_k]_t &= \sum_{n \in \mathcal U} \1_{d_n\leq t} Z_{t_0}^0 \1_{d_n \in (t_0,t_1]}\big( f(\mathcal J_{d_n-} + a^0_{L_n}(d_n,X_n(d_n))) - f(\mathcal J_{d_n-}) \big) \\
		&= Z_{t_0}^0 \sum_{n \in \mathcal U} \1_{d_n \in (t_0 \wedge t, t_1 \wedge t]} \big( f(\mathcal J_{d_n}) - f(\mathcal J_{d_n-}) \big).
		\end{align*}
		Since $\mathcal J$ is $\cadlag$, piecewise constant and with a finite number of jumps, and since $\{ d_n \ : \ n\in\mathcal U \}$ are precisely the set of times when it jumps, we are just summing over all the jumps between $t_0 \wedge t$ and $t_1 \wedge t$ the corresponding jump size of $t \mapsto f(\mathcal J_t)$. We conclude that
		\begin{equation*}
		\sum_{k \neq 1} \mathcal J^k[b^0_k]_t = Z_{t_0}^0 \big(f(\mathcal J_{t_1 \wedge t}) - f(\mathcal J_{t_0 \wedge t})\big),
		\end{equation*} 
		as stated above.
		
		We are left with proving that the $(\F^0_t)_{t \in [0,1]}$-predictable measure-valued field defined by formula~\eqref{eq:jump_measure_tree} is the jump measure of $\mathcal J$. To do so, it is enough to prove that whenever the family $(a^0_k)_{k \neq 1}$ is uniformly bounded,
		\begin{equation}
		\label{eq:one_specific_martingale}
		\mathcal J_t - \int_0^t \sum_{k\neq 1} q_k \left\cg a^0_k(s), M^0_{s-}\right\cd \D s, \qquad t \in [0,1]
		\end{equation}
		is a martingale. Indeed, condition~\eqref{eq:integrability_jump_measure} is straightfoward to check using~\eqref{eq:q_finite_mean}, and it suffices to use formula~\eqref{eq:truncation_jump_tree} to treat the general case of all martingales given in the Definition~\ref{def:simple_pure-jump_processes} of simple pure-jump processes only using martingales of the form~\eqref{eq:one_specific_martingale}. Consequently, we can restrict even more the cases that we need to study. Indeed, now everything is linear with respect to $(a_k)_{k \neq 1}$. Therefore, we can decouple the different values of $k$, and only prove that for all $k \neq 1$, which is kept fixed from now on, the process
		\begin{equation*}
		\mathcal J^k[a^0_k]_t - \int_0^t q_k \left\cg a^0_k(s), M^0_{s-}\right\cd \D s, \qquad t \in [0,1]
		\end{equation*}
		is a martingale. With the terminology of~\cite{jacod2013limit} this martingale property exactly means that we need to prove that the predictable random measure $q_k \D s \otimes M_{s-}$ is the compensator of the optional and integrable random measure $\sum_{n \in \mathcal U} \1_{d_n \leq 1} \1_{L_n = k}\delta_{X_n(d_n)}$ (see~\cite[Definition~II.1.6]{jacod2013limit}). Because of~\cite[Theorem~11.1.8]{jacod2013limit}, this is the same as proving that for all predictable bounded scalar field $a^0_k$,
		\begin{equation*}
		\Em_\xx \Big[\mathcal J^k[a_0^k]_1 \Big] = \Em_\xx \left[  \int_0^1 q_k \left\cg a^0_k(t), M^0_{t-}\right\cd \D t \right].
		\end{equation*}
		By approximation, using the dominated convergence theorem and adapting the proof of~\cite[Proposition~5.3]{legall2016brownian}, we can even restrict ouselves to considering $a^0_k$ of the form $(\omega, t, x) \in \XX \mapsto \1_A(\omega) \1_{(t_0, t_1]}(t)\1_{U}(x)$, where $0 \leq t_0 \leq t_1 \leq 1$, $A \in \F^0_{t_0}$ and $U \subset \T^d$ is an open ball. Therefore, we are left with proving that for all such $t_0,t_1,A,U$, 
		\begin{equation*}
		\Em_\xx\left[\1_A \sum_{n \in \mathcal U} \1_{t_0< d_n \leq t_1} \1_{L_n = k} \1_{U}(X_n(d_n))\right] = \Em_\xx\left[ \1_A\int_{t_0}^{t_1}q_k M^0_{s-}(U) \D s \right].
		\end{equation*}
		A way to prove it is to derivate both sides of this formula w.r.t. $t_1$. On the r.h.s., we get
		\begin{equation*}
		\frac{\D}{\D t_1 +} \Em_\xx\left[ \1_A\int_{t_0}^{t_1}q_k M^0_{s-}(U) \D s \right] = \Em_\xx \Big[ \1_A q_k M^0_{t_1-}(U) \Big].
		\end{equation*}
		On the l.h.s., we have for $h>0$ small enough, using the notations of Definition~\ref{def:N_S}, and calling for all $t \in [0,1)$, $\bar n(t) := \argmin \{ d_n\, : \,  n \in \mathcal U \mbox{ and } d_n > t \}$ and $\bar L(t) := L_{\bar n(t)}$ in the last line:
		\begin{align*}
		\Em_\xx\Bigg[\1_A \sum_{n \in \mathcal U} &\1_{t_0< d_n \leq t_1+h} \1_{L_n = k} \1_{U}(X_n(d_n))\Bigg] - \Em_\xx\left[\1_A \sum_{n \in \mathcal U} \1_{t_0< d_n \leq t_1} \1_{L_n = k} \1_{U}(X_n(d_n))\right]  \\
		&= \Em_\xx\left[\1_A \sum_{n \in \mathcal U} \1_{t_1< d_n \leq t_1+h} \1_{L_n = k} \1_{U}(X_n(d_n))\right]\\
		&=\Em_\xx\bigg[ \1_A \1_{S(M^0,t_1) \leq t_1 + h} \1_{\bar L(t_1) = k}\1_U(X_{\bar n(t_1)}(t_1))\bigg]+ \underset{h \to 0}{o}(h).
		\end{align*}
		For the last equality, one needs to observe that the probabilities
		\begin{gather*}
		\Prob^\xx\Big( \mbox{There is more than one branching event in the set of times }(t_1, t_1 + h] \Big),\\
		\Prob^\xx \Big(X_n(t_1) \in U \mbox{ but } \exists s \in (t_1, t_1 + h] \mbox{ s.t.} \ X_n(s) \notin U\Big),\\
		\Prob^\xx \Big(X_n(t_1) \notin U \mbox{ but } \exists s \in (t_1, t_1 + h] \mbox{ s.t.} \ X_n(s) \in U\Big)
		\end{gather*} 
		are of order $o(h)$ when $h$ is small. 
		
		To conclude, we need to observe, letting the details to the reader, that as a consequence of the definition of the BBM and by standard considerations on exponential times, conditionally on $\mathcal F^0_{t_1}$, the random variables $S(M^0,t_1)$, $\bar n(t_1)$ and $\bar L(t_1)$ are independent, of respective laws the exponential law of parameter $\lambda^{\boldsymbol q} M_{t_1}(\T^d)$, the uniform law on $\mathcal N^0(t_1)$ and the law $\boldsymbol p^{\boldsymbol q}$. Hence, coming back to our computation, we find
		\begin{align*}
		\Em_\xx\bigg[ \1_A \1_{S(M^0,t_1) \leq t_1 + h} \1_{\bar L(t_1) = k}\1_U(X_{\bar n(t_1)}(t_1))\bigg] &=\Em_\xx\bigg[ \1_A \Em_\xx\bigg[\1_{S(M^0,t_1) \leq t_1 + h} \1_{\bar L(t_1) = k}\1_U(X_{\bar n(t_1)}(t_1))\bigg| \mathcal F^0_{t_1} \bigg]\bigg] \\
		&= \Em_\xx\left[ \1_A \Big( 1 - \exp(- \lambda^{\boldsymbol q} M^0(t_1) h)\Big) p^{\boldsymbol q}_k \times \frac{M^0_{t_1}(U)}{M^0_{t_1}(\T^d)} \right]\\
		&= \Em_\xx \Big[ \1_A q_k M^0_{t_1}(U) \Big]h + \underset{h \to 0}{o}(h).
		\end{align*}
		where on the second line, the quantity inside the expectation cancels on the event $\{M^0_{t_1}(\T^d) = 0 \}$. The result follows from the fact that for all $t \in [0,1]$, $\Prob^\xx$-\emph{a.s.}, $M^0_{t-} = M^0_t$.
	\end{proof}
	
	\section{Generalized Itô formula at the level of trees}
	In this section, we prove the following extended Itô formula, at the level of the large probability space~$\Omega^\xx$. Recall that $\mu\in \M_\delta(\T^d)$ is the (deterministic) initial condition of the BBM introduced in Subsection~\ref{subsec:def_BBM}.
	\begin{Thm}
		Let $\varphi:[0,1]\times\T^d \to \R$ be a smooth function. We have $\Prob^\xx$-\emph{a.s.}\ for all $t \in [0,1]$:
		\begin{equation}
		\label{eq:branching_Ito_tree}
		\cg \varphi(t), M^0_t \cd = \cg \varphi(0), \mu\cd + \mathcal I[\nabla \varphi]_t + \sum_{k=0}^{+\infty} \mathcal J^k[(k-1)\varphi]_t + \int_0^t \Big\cg \partial_t \varphi(s) + \frac{\nu}{2} \Delta \varphi(s), M^0_s \Big\cd \D s.
		\end{equation}
	\end{Thm}
	\begin{proof}
		Let us introduce the sequence $(T_i)_{i \in \N}$ of $(\F_t^0)_{t \in [0,1]}$ stopping times defined by
		\begin{equation*}
		T_0 := 0, \qquad T_{i + 1} :=S(M^0,T_i), \quad i \in \N.
		\end{equation*}
		They correspond to times of branching events, ordered chronologically.
		
		We start by applying the classical Itô formula~\cite[Theorem~5.10]{legall2016brownian} to the continuous semi-martingale $(\varphi(t,X_n(t)))_{t \in [0,1]}$ for some given smooth function $\varphi$ of $t$ and $x$, some $n \in \mathcal U$, and between the two stopping times $T_{i+1} \wedge t$ and $T_i \wedge t$, $i \in \N$ and $t \in [0,1]$. We get thanks to Theorem~\ref{prop:generalized_stochastic_integral_tree}
		\begin{align*}
		\varphi\big(T_{i+1} &\wedge t,X_n(T_{i+1} \wedge t)\big) - \varphi\big(T_i \wedge t,X_n(T_i \wedge t)\big)\\
		&=  \int_{T_i \wedge t}^{T_{i+1} \wedge t} \hspace{-5pt} \nabla \varphi(s, X_n(s))\D X_n(s) +  \int_{T_i \wedge t}^{T_{i+1} \wedge t} \hspace{-5pt}  \partial_t \varphi(s,X_n(s)) \D s + \frac{1}{2}  \int_{T_i \wedge t}^{T_{i+1} \wedge t} \hspace{-5pt} \mathrm{D}^2 \varphi(s, X_n(s)) \cdot \mathrm{d} [X_n]_s\\
		&= \int_{T_i \wedge t}^{T_{i+1} \wedge t} \hspace{-5pt} \nabla \varphi(s, X_n(s))\D X_n(s) + \int_{T_i \wedge t}^{T_{i+1} \wedge t} \hspace{-5pt}\Big\{ \partial_t \varphi(s,X_n(s)) + \frac{\nu}{2} \sum_{n' \leq n}\1_{ b_{n'} < s \leq d_{n'} } \Delta \varphi(s, X_n(s)) \Big\} \D s.
		\end{align*}
		
		By definition, there is no branching event between $T_i \wedge t$ and $T_{i+1} \wedge t$. It means that for all $s \in [T_i \wedge t, T_{i+1}\wedge t)$, we have $\1_{b_n < s \leq d_n} = \1_{b_n < T_i \wedge t \leq d_n}$. Therefore, multiplying the previous formula by this indicator function and summing it over $n \in \mathcal U$, we get
		\begin{align}
		\notag\cg \varphi(T_{i+1} &\wedge t), M^0_{T_{i+1} \wedge t -} \cd - \cg \varphi(T_i \wedge t) , M^0_{T_i \wedge t} \cd \\
		\notag&= \sum_{n \in \mathcal U} \int_{T_i \wedge t}^{T_{i+1} \wedge t} \1_{b_n < s \leq d_n} \nabla \varphi(s, X_n(s))\D X_n(s) + \int_{T_i \wedge t}^{T_{i+1} \wedge t}\left\cg \partial_t \varphi(s) + \frac{\nu}{2} \Delta \varphi(s) , M^0_s \right\cd \D s\\
		\label{eq:Ito_continuous}&= \mathcal I[\nabla \varphi]_{T_{i+1} \wedge t} - \mathcal I[\nabla \varphi]_{T_{i} \wedge t} + \int_{T_i \wedge t}^{T_{i+1} \wedge t} \hspace{-5pt}\Big\cg \partial_t \varphi(s) + \frac{\nu}{2} \Delta \varphi(s), M_s^0\Big\cd \D s.
		\end{align}
		
		Next, remark that for all $i \in \N$, the set $\{ n \in \mathcal U \ : \ d_n = T_{i+1}\}$ is $\Prob^\xx$-\emph{a.s.}\ a singleton. We call $\bar n_i$ its only element. By definition of the BBM, we have $\Prob^\xx$-\emph{a.s.}\ for all $i \in \N$, 
		\begin{equation*}
		M^0_{T_{i+1} \wedge t} - M^0_{T_{i+1} \wedge t -} = \1_{T_{i+1} \leq t} (L_{\bar n_i} - 1) \delta_{X_{\bar n_i}(T_{i+1})},
		\end{equation*} 
		so that the following formula holds $\Prob^\xx$-\emph{a.s.}\ for all $t \in [0,1]$:
		\begin{equation}
		\label{eq:Ito_jump}
		\cg \varphi(T_{i+1}\wedge t), M^0_{T_{i+1} \wedge t} \cd - \cg \varphi(T_{i+1} \wedge t) , M^0_{T_{i+1} \wedge t-} \cd = \1_{T_{i+1} \leq t} (L_{\bar n_i}-1)\varphi(T_{i+1}, X_{\bar n_i}(T_{i+1})).
		\end{equation}
		
		If we sum the formulas~\eqref{eq:Ito_continuous} and~\eqref{eq:Ito_jump} over $i \in \N$, as we know that $T_i \wedge t = t$ for $i$ random but sufficiently large, and that $\Prob^\xx$-\emph{a.s.}, $M^0_0 = \mu$, we get
		\begin{align*}
		\cg \varphi( t), M^0_{t} \cd - \cg \varphi(0) , \mu \cd &= \mathcal I[\nabla \varphi]_t + \int_{0}^{t} \hspace{-5pt}\Big\cg \partial_t \varphi(s) + \frac{\nu}{2} \Delta \varphi(s), M_s^0\Big\cd \D s + \sum_{i \in \N}  \1_{T_{i+1} \leq t}  (L_{\bar n_i}-1)\varphi(T_{i+1}, X_{\bar n_i}(T_{i+1}))\\
		&= \mathcal I[\nabla \varphi]_t + \int_{0}^{t} \hspace{-5pt}\Big\cg \partial_t \varphi(s) + \frac{\nu}{2} \Delta \varphi(s), M_s^0\Big\cd \D s +  \sum_{n\in \mathcal U} \1_{d_n \leq t}(L_n-1)\varphi(d_n, X_n(d_n)),
		\end{align*}
		where the second line follows from $\{ (n,d_n) \ : \ n \in \mathcal U \mbox{ and } d_n < + \infty \} = \{ (\bar n_i, T_{i+1}) \ : \ i \in \N \}$. But the last term is nothing but $\sum_{k \neq 1} \mathcal J^k[(k-1) \varphi]_t$, and the result follows.
	\end{proof}
	
	\section{Projection on the canonical space}
	
	This subsection is rather technical. The goal is to prove that provided $v^0$ and $a^0$ appearing in the definitions of $\mathcal I$ and $\mathcal J_k$, $k \neq 1$ are predictable w.r.t.\ the restricted filtration $(\bar \F_t)_{t \in [0,1]}$ defined in~\eqref{eq:def_restricted_filtration}, then these processes are adapted with respect to the same restricted filtration $(\bar \F_t)_{t \in [0,1]}$.
	
	Heuristically, it means that these processes do not depend on a specific labeling of the particles, but only on the curve of empirical measures. This result includes some subtle symmetry properties of the BBM: for instance, the value of the generalized stochastic integral of an $(\bar \F_t)_{t \in [0,1]}$ predictablely measurable random vector field is not affected if particles are relabeled when they cross. This kind of behaviors is not easy to handle, as for example, in dimension one or two, whenever two particles cross, they cross an infinite number of times. Our proof uses some approximated processes that are obviously adapted w.r.t.\ $(\bar \F_t)_{t \in [0,1]}$.

	We start with the hardest result, that is, the case of the generalized stochastic integral.
	\begin{Prop}
		\label{prop:I_adapted_restricted}
		Let $v^0$ be a predictable vector field w.r.t.\ the restricted filtration $(\bar \F_t)_{t \in [0,1]}$ defined in~\eqref{eq:def_restricted_filtration}, that satisfies~\eqref{eq:integrability_v0}. Let $(\mathcal{I}[v^0]_t)_{t \in [0,1]}$ be the process defined by formula~\eqref{eq:def_stochastic_integral_trees}. This process is adapted with respect to the filtration $(\bar \F_t)_{t \in [0,1]}$.
	\end{Prop}
	\begin{proof}
		We need to check that for a given $v^0$ as in the statement of the proposition, and a given $t \in [0,1]$, $\mathcal{I}[v^0]_t$ is $\bar \F_t$-measurable. By approximation (adapt once again the proof of~\cite[Proposition~5.3]{legall2016brownian}), we can restrict ourselves to predictable vector fields $v^0$ of the form
		\begin{equation}
		\label{eq:def_v_proof_I_measurable}
		v^0(t,x) = \1_{A} \1_{s < t}\1_{x \in U} w,
		\end{equation}
		where $A \in \bar \F_s$, $s \in [0,1)$ $w \in \R^d$ and $U \subset \T^d$ is an open ball. So from now on, we fix $A$, $s$, $w$ and $U$ as such, and we define $v^0$ through~\eqref{eq:def_v_proof_I_measurable}.
		
		If $t \leq s$, $\I[v^0]_t = 0$ everywhere, so that $\sigma(\I[v^0]_t) = \{ \emptyset, \Omega^0 \} \subset \bar \F_t$.

		If $t>s$, we will approximate $\I[v^0]_t$ by a sequence of $\bar \F_t$-measurable random variables. These random variables will have two parameters: $N \in \N^*$ that will parameterize a discretization in time, and $P \in \N^*$ that will have to do with the modulus of continuity of the trajectories of the particles, up to time $t$. So we give ourselves $N,P \in \N^*$.

		For $i = 0, \dots, 2^N$, we call $t_i^N := s + (t-s) i / 2^N$. For $i= 0, \dots, 2^N-1$, we also call $\xi_i^N(x)$, $x \in \T^d$ the random vector field:
		\begin{equation*}
		\xi_i^N(x) := \left\{ 
		\begin{aligned}
		&\underset{\big\{z \in \R^d \ : \ x + z \in \{ X_n(t_{i+1}^N), \ n \in \mathcal{N}(t_{i+1}^N)\}\big\}}{\argmin}  |z|,  && \mbox{if }\mathcal{N}(t_{i+1}^N) \neq \emptyset \mbox{ and the minimizer is unique}, \\
		&0, && \mbox{else},
		\end{aligned}
		\right.
		\end{equation*}
		where whenever $x \in \T^d$ and $z \in \R^d$, we see $x + z$ as an element of $\T^d$. This vector field is chosen in such a way that with high probability, at least when $N$ is large, for all $n \in \mathcal N(t_{i}^N)$, $X_n(t_{i}^N) + \xi_i^N = X_n(t_{i+1}^N)$. Seen as a function on $\Omega^0 \times \T^d$, we can check that it is $\bar \F_{t_{i+1}^N} \otimes \mathcal{B}(\T^d)\subset \bar \F_t\otimes \mathcal{B}(\T^d)$-measurable. 
		
		We also consider the following random variable, indexed by $\tau \in [s,t]$:
		\begin{equation*}
		D(\tau) := \inf \Big\{ \mathrm{dist}\Big(X_n(\tau), X_{n'}(\tau)\Big) \ : \ n,n' \in \mathcal N(\tau)\Big\},
		\end{equation*}
		where $\mathrm{dist}$ is the geodesic distance on the torus. We can also check that $D(t)$ is $\bar \F_\tau \subset \bar \F_t$-measurable  (it is clearly a measurable function of $M^0_{t})$.
		
		Now, we define the approximated increments for $i = 0, \dots 2^N - 1$ as:
		\begin{equation*}
		\Delta I^{N,P}_i := \left\{  
		\begin{aligned}
		&\1_A w \cdot \cg \1_{\underline U{}_P} \xi_i^N  , M^0_{t_{i}^N} \cd, && \mbox{if }M^0_{\tau}(\T^d) = M^0_{t_i^N}(\T^d), \, \forall \tau \in [t_i^N, t_{i+1}^N] \mbox{ and } D(t_i^N) > \frac{1}{P}, \\
		&0, &&\mbox{else},
		\end{aligned}
		\right.
		\end{equation*}
		where $\underline U{}_P$ is the smaller open ball $\{ x \in \T^d \ : \ \mathrm{dist}(x, \T^d \backslash U) > 1/P \}$. We also define the bigger closed ball $\overline U{}_P := \{ x \in \T^d \ : \ \mathrm{dist}(x, U) \leq 1/P \}$. Finally, we define the approximated generalized stochastic integral as
		\begin{equation*}
		I^{N,P} := \sum_{i=0}^{2^N - 1} \Delta I^{N,P}_i.
		\end{equation*}
		From what we said up to now, $I^{N,P}$ is $\bar \F_t$-measurable. We will show that for all $\eps>0$,
		\begin{equation*}
		\limsup_{P \to + \infty} \limsup_{N \to + \infty}\Prob^\xx\Big( \left|\mathcal I[v^0]_t - I^{N,P}\right| \geq \eps \Big) = 0.
		\end{equation*}
		By a diagonal argument and an extraction, it is sufficient to conclude.
		
		We defined $I^{N,P}$ in such a way that when $N$ and $P$ are large, with large probability, $\Delta \mathcal I_i^{N} = \Delta I_i^{N,P}$ for most of the indices $i$, where $\Delta \mathcal I_i^{N}$ stands for the increment of the generalized stochastic integral, that is, $\Delta \mathcal I_i^{N} := \I[v^0]_{t_{i+1}^N} - \I[v^0]_{t_{i}^N}$. Let us formalize this idea. Let $N,P \in \N^*$. First, we define two (random) sets of exceptional indices:
		\begin{gather*}
		S_1^N := \Big\{ i \in \{ 0, \dots, 2^N-1\} \ : \ \exists \tau \in [t_i^N, t_{i+1}^N] \mbox{ s.t.\ }M_\tau^0(\T^d) \neq M_{t_i^N}^0(\T^d) \Big\},\\
		S_2^{N,P} := \Big\{ i \in \{ 0,\dots,  2^N-1 \}\ : \ D(t_i^N) \leq 1/P \mbox{ or } M^0_{t_i^N} (\overline U{}_P \backslash \underline U{}_P)>0 \Big\}.
		\end{gather*}
		
		The first one is the set of indices $i$ for which there is a branching event in the time interval $(t_i^N, t_{i+1}^N]$, and the second one is the set of indices $i$ for which at time $t_i^N$, at least two particles are closer than $1/P$ or at least one particle is in the small set $\overline U{}_P \backslash \underline U{}_P$.
		
		Then, we define an event of large probability such that on this event, $\Delta \mathcal I_i^{N} = \Delta I_i^{N,P}$, unless $i$ is inside one of the exceptional sets $S_1^N$ or $S^{N,P}_2$. To do this, we define $g: \R_+ \to \R_+$ a (random) uniform modulus of continuity of all the curves of particles alive before time $t$:
		\begin{equation*}
		g(a) = \sup \Big\{ \mathrm{dist}(X_n(t_2) , X_n(t_1)) \ : \ t_1,t_2 \in [0,1],\ |t_2 - t_1| \leq a, \ n \in \mathcal T \mbox{ s.t. } b_n \leq 1. \Big\}, \quad a \in (0,1].
		\end{equation*}
		Almost surely, $g(a)$ tends to zero as $a \to 0$. We call 
		\begin{equation*}
		A^{N,P} := \left\{ g\Big((2^{-N}(t-s)\Big) \leq \frac{1}{2P} \right\}.
		\end{equation*} 
		Observe that for all $P$, $\Prob^\xx (A^{N,P}) \to 1$ as $N \to + \infty$.
		
		Let us show that on $ A^{N,P}$, if $i \notin S_1^N$ and $i\notin S_2^{N,P}$, then $\Delta \mathcal I_i^{N} = \Delta I_i^{N,P}$. Define \begin{equation*}
		Z_{n,i}^N:= \int_{t_{i}^N}^{t_{i+1}^N}\D X_n(\tau),
		\end{equation*}
		which is seen as an element of $\R^d$, as opposed to $\T^d$. Notice that $X_n(t_{i+1}^N) = X_n(t_i^N) + Z_{n,i}^N$. Because we are on $A^{N,P}$, we have for all $n \in \mathcal N(t_i^N)$,
		\begin{equation*}
		|Z_{n,i}^N| \leq g\Big(2^{-N}(t-s)\Big) \leq \frac{1}{2P}.
		\end{equation*}
		On the other hand, whenever $i \notin S_2^{N,P}$, for all $n' \in \mathcal N(t_i^N)$ with $n' \neq n$,
		\begin{equation*}
		\mathrm{dist}(X_n(t_i^N) , X_{n'}(t_{i+1}^N)) \geq \mathrm{dist}(X_n(t_i^N) , X_{n'}(t_{i}^N)) - \mathrm{dist}(X_{n'}(t_i^N) , X_{n'}(t_{i+1}^N)) > \frac{1}{P} - \frac{1}{2P} = \frac{1}{2P}.
		\end{equation*}
		Hence, $\xi_i^N(X_n(t_i^N)) = Z_{n,i}^N$.
		
		Furthermore, whenever $i \notin S_1^N$, no particle dies in the interval of time $(t_i^N, t_{i+1}^N]$, so that
		\begin{equation*}
		\Delta \mathcal I_i^N = \mathcal{I}[v^0]_{t_{i+1}^N} - \mathcal{I}[v^0]_{t_{i}^N} = \1_A w \cdot \sum_{n \in \mathcal{N}(t_i^N)} \int_{t_i^N}^{t_{i+1}^N}\1_U(X_n(t)) \D X_n(t).
		\end{equation*}
		But as we are working on $A^{N,P}$ and with $i \in S_2^{N,P}$, for all $n \in \mathcal N(t_i^N)$, for all $t \in [t_i^N, t_{i+1}^N]$, we clearly have $X_n(t) \in U$ if and only if $X_n(t_i^N)\in \underline U{}_P$, that is, $\1_U(X_n(t)) = \1_{\underline U{}_P}(X_n(t_i^n))$. As a consequence,
		\begin{equation*}
		\Delta \mathcal I_i^N = \1_A w \cdot \sum_{n \in \mathcal{N}(t_i^N)} \1_{\underline U{}_P}(X_n(t_i^n)) Z_{n,i}^N = \1_A w \cdot \sum_{n \in \mathcal{N}(t_i^N)} \1_{\underline U{}_P}(X_n(t_i^n))\xi_i^N(X_n(t_i^N)) = \Delta I^{N,P}_i.
		\end{equation*}
		From these identities, we deduce that on $A^{N,P}$
		\begin{equation*}
		\mathcal I[v^0]_t - I^{N,P} = \sum_{i=0}^{2^N - 1} \Delta \mathcal I^N_i - \Delta I^{N,P}_i = \sum_{i \in S_1^N \cup S_2^{N,P}} \Delta \mathcal I^N_i - \Delta I^{N,P}_i.
		\end{equation*}
		Moreover, $\Delta I^{N,P}_i$ cancels for $i \in S_1^N \cup S^{N,P}_2$ (we defined it as such). So we end up with
		\begin{equation*}
		\mathcal I[v^0]_t - I^{N,P} = \sum_{i \in S_2^{N,P}} \Delta \mathcal I^N_i + \sum_{i \in S_1^N \backslash S_2^{N,P}} \Delta \mathcal I^N_i, \qquad \Prob^\xx\mbox{-\emph{a.s.}\ on }A^{N,P}.
		\end{equation*}
		By a standard trick, we bound our target probability in the following way:
		\begin{equation*}
		\Prob^\xx\Big( \left|\mathcal I[v^0]_t - I^{N,P}\right| \geq \eps \Big) \leq \Prob^\xx\Big(\Omega^0 \backslash A^{N,P}\Big) +  \Prob^\xx\bigg( \bigg|\sum_{i \in S_2^{N,P}} \Delta \mathcal I_i^{N} \bigg| \geq \frac{\eps}{2} \bigg) + \Prob^\xx\bigg( \bigg|\sum_{i \in S_1^N \backslash S_2^{N,P}} \Delta \mathcal I^N_i \bigg| \geq \frac{\eps}{2} \bigg).
		\end{equation*}
		For all $P \in \N^*$, the first term tends to $0$ as $N \to + \infty$. Let us show that the two other terms are small when $N \to + \infty$ and then $P \to + \infty$.
		
		The easy one is the second one, because $P$ plays no role. Indeed, we have
		\begin{equation}
		\label{eq:bound_jump_instants}
		\bigg|\sum_{i \in S_1^N \backslash S_2^{N,P}} \Delta \mathcal I^N_i \bigg| \leq \sum_{i \in S_1^N \backslash S_2^{N,P}} \left|\Delta \mathcal I^N_i\right| \leq \# S_1^N \times \max_{i \in S_1^N}\left|\Delta \mathcal I^N_i\right| .
		\end{equation} 
		But the cardinal of $S_1^N$ is smaller than the total number of jumps,  and hence is $\Prob^\xx$-\emph{a.s.}\ bounded uniformly in $N$, as a consequence of~\eqref{eq:q_finite_mean}. Furthermore, almost surely, $\max_{i \in S_1^N} |\Delta \mathcal I_i^{N}| \leq \max_{i = 0, \dots, 2^N-1} |\Delta \mathcal I_i^{N}|$ tends to $0$ as $N \to + \infty$, by almost sure continuity of $\mathcal{I}[v^0]$. We conclude that for all $P$, the quantity in the l.h.s.\ of~\eqref{eq:bound_jump_instants} converges \emph{a.s.}\ towards $0$ as $N \to + \infty$ so that the probability that it is bigger than $\eps/2$ tends to zero as well.
		
		We are left with proving
		\begin{equation*}
		\limsup_{P \to + \infty} \limsup_{N \to +\infty}\Prob^\xx\bigg( \bigg|\sum_{i \in S_2^{N,P}} \Delta \mathcal I_i^{N} \bigg| \geq \frac{\eps}{2} \bigg) = 0.
		\end{equation*}
		To do this, with start with Markov's inequality:
		\begin{equation*}
		\Prob^\xx\bigg( \bigg|\sum_{i \in S_2^{N,P}} \Delta \mathcal I_i^{N} \bigg| \geq \frac{\eps}{2} \bigg) \leq \frac{4}{\eps^2}\Em_\xx\left[ \left| \sum_{i\in S_2^{N,P}} \Delta \mathcal I_i^{N} \right|^2 \right], 
		\end{equation*}
		so that we want to prove
		\begin{equation*}
		\limsup_{P \to + \infty} \limsup_{N \to + \infty} \Em_\xx\left[ \left| \sum_{i\in S_2^{N,P}} \Delta \mathcal I_i^{N} \right|^2 \right] =0.
		\end{equation*}
		But this quantity can be computed using the properties of the generalized stochastic integral. Indeed, it is easy to check that we have exactly
		\begin{equation*}
		\sum_{i\in S_2^{N,P}} \Delta \mathcal I_i^{N} = \mathcal I [v^{N,P}]_t,
		\end{equation*}
		where $v^{N,P}$ cancels outside of $(s,t]$, and where for all $\tau \in (t_i^N, t_{i+1}^N]$,
		\begin{equation*}
		v^{N,P}(\tau,x) := \1_A \1_{x \in U} \1_{i \in S_2^{N,P}} w.
		\end{equation*}
		(It is easy to check that $\1_{i \in S_2^{N,P}}$ is $\bar \F_{t_i^N}$-measurable, so that $v^{N,P}$ is a predictable bounded vector field, and hence $\mathcal{I}[v^{N,P}]_t$ is well defined.) As a consequence,
		\begin{align*}
		\Em_\xx\left[ \left| \sum_{i\in S_2^{N,P}} \Delta \mathcal I_i^{N} \right|^2 \right] &= \Em_\xx\left[ \left| \mathcal{I}[v^{N,P}]_t \right|^2 \right]=\Em_\xx\left[ \int_s^t \left\cg\left| v^{N,P}(\tau) \right|^2, M^0_\tau \right\cd \D \tau \right]\\ &= |w|^2 \Em_\xx\left[\1_A\sum_{i=0}^{2^N - 1} \int_{t_i^N}^{t_{i+1}^N} \1_{i \in S_2^{N,P}}   M^0_\tau(U)  \D \tau \right] \\
		&\leq |w|^2 \Em_\xx\left[\sup_{\tau \in [0,1]} M^0_\tau(\T^d) \times \frac{t-s}{2^N} \sum_{i=0}^{2^N - 1}  \1_{i \in S_2^{N,P}}   \right]\\
		&\leq |w|^2\times (t-s) \times \sqrt{\Em_\xx \left[ \sup_{\tau \in [0,1]} M^0_\tau(\T^d)^2 \right]}\sqrt{\frac{1}{2^N}\sum_{i=0}^{2^N-1}\Em_\xx\left[  \1_{i \in S^{N,P}_2} \right]}  .
		\end{align*}
		The first square root is easily seen to be finite as a consequence of~\eqref{eq:q_finite_mean}.
		
		Therefore, we need to prove
		\begin{equation*}
		\limsup_{P \to + \infty}\limsup_{N \to + \infty} \frac{1}{2^N}\sum_{i=0}^{2^N-1}\Em_\xx\left[  \1_{i \in S^{N,P}_2} \right] = 0.
		\end{equation*}
		
		Observe the following manipulations
		\begin{align*}
		\frac{1}{2^N}\sum_{i=0}^{2^N-1}\Em_\xx\left[  \1_{i \in S^{N,P}_2} \right] &= \frac{1}{2^N}\sum_{i=0}^{2^N-1}\Prob^\xx\Big( D(t_i^N) \leq 1/P \mbox{ and } M_{t_i^N}^0(\overline U{}_P \backslash \underline U{}_P)>0 \Big)\\
		&\leq \frac{1}{2^N}\sum_{i=0}^{2^N-1}\Prob^\xx\Big( D(t_i^N) \leq 1/P \Big) + \frac{1}{2^N}\sum_{i=0}^{2^N-1}\Prob^\xx\Big(M_{t_i^N}^0(\overline U{}_P \backslash \underline U{}_P)>0  \Big).
		\end{align*}
		It is possible to check that both $\tau \mapsto \Prob^\xx(D(\tau) \leq 1/P)$ and $\tau \mapsto \Prob^\xx(M_{\tau}^0(\overline U{}_P \backslash \underline U{}_P)>0)$ are continuous. Thus, the Riemann approximations converge, and we get for all $P \in \N^*$
		\begin{equation*}
		\limsup_{N \to + \infty} \frac{1}{2^N}\sum_{i=0}^{2^N-1}\Em_\xx\left[  \1_{i \in S^{N,P}_2} \right] \leq \frac{1}{t-s}\int_s^t\left\{ \Prob^\xx\Big( D(\tau)\leq 1/P \Big) + \Prob^\xx\Big( M_{\tau}^0(\overline U{}_P \backslash \underline U{}_P)>0  \Big)\right\}\D \tau.
		\end{equation*}
		As for all $\tau$, almost surely, all particles are at different locations and are not located on $\partial U$, the dominated convergence theorem implies that the r.h.s.\ converges to $0$ as $P \to + \infty$, so that we can conclude.
	\end{proof}
	
	Let us now treat the simpler case of our pure-jump processes.
	\begin{Prop}
		Let $a^0$ be a random scalar field on $\Omega^\xx$ that is predictable with respect to the restricted filtration $(\bar \F_t)_{t \in [0n1]}$ defined in~\eqref{eq:def_restricted_filtration}. Let $k \in \N\backslash\{1\}$. The process $(\mathcal{J}^k[a^0]_t)_{t \in [0,1]}$ defined in Definition~\ref{def:pure_jump_process_tree} satisfies:
		\begin{equation}
		\label{eq:second_formula_J}
		\Prob^\xx\mbox{\emph{a.s.}},\quad \forall t \in [0,1], \qquad \mathcal J^k[a^0]_t :=  \sum_{s \leq t} \1_{\Delta M^0_s(\T^d) = k-1} \left\cg a^0(s), \frac{\Delta M_s}{k-1} \right\cd.
		\end{equation}
		In particular, the process in the r.h.s.\ is a version of the one in the l.h.s.\ that is adapted with respect to the restricted filtration $(\bar \F_t)_{t \in [0,1]}$.
	\end{Prop}
	\begin{proof}
		The fact that the first claim implies the second one is obvious, so let us focus on proving formula~\eqref{eq:second_formula_J}. We fix $a^0$ and $k$ as in the statement of the proposition.
		
		As a direct consequence of the definitions of Subsection~\ref{subsec:def_BBM}, up to a $\Prob^\xx$-negligible set, for all $s \in [0,1]$, we have $\1_{\Delta M^0_s(\T^d) = k-1} = \sum_{n \in \mathcal U}\1_{d_n = s} \1_{L_n = k}$, where at most one term of the sum in the r.h.s. does not cancel. Moreover, outside the above mentioned $\Prob^\xx$-negligible set, if $s$ is such that $\Delta M^0_s(\T^d) = k-1$, and if $n \in \mathcal U$ is such that $d_n = s$ and $L_n = k$, we have $\Delta M^0_s = (k-1) \delta_{X_n(d_n)}$. Consequently, \eqref{eq:first_formula_J_tree} implies~\eqref{eq:second_formula_J}.
	\end{proof}
	\section{Conclusion: proofs of Theorem~\ref{thm:definition_stochastic_integral}, \ref{thm:properties_jump_processes} and~\ref{thm:branching_Ito}}
	\label{sec:large_to_restricted}
	In order to close this chapter, we give a few lines to connect the results obtained in the previous sections to the ones presented in Section~\ref{sec:new_processes}. Here, the diffusivity parameter $\nu >0$ and the branching mechanism $\boldsymbol q$ are still fixed, but we also chose 
	a law $R_0 \in \P(\M_{\delta}(\T^d))$, and we call $R\sim \BBM(\nu, \boldsymbol q, R_0)$. We start by writing a proof of Theorem~\ref{thm:definition_stochastic_integral}.
	\begin{proof}[Proof of Theorem~\ref{thm:definition_stochastic_integral}]
		Let $v \in L^2(\XX, \tilde \G, \Pi_R)$ as in the statement of the theorem. We have with the notations of Subsection~\ref{subsec:def_BBM}:
		\begin{equation*}
		\E_{R_0}\left[ \Em_{M_0} \left[ \int_0^1 \cg |v(t)|^2, M_t \cd \D t \right] \right] < + \infty.
		\end{equation*}
		In particular, for $R_0$-almost all $\mu$, we have
		\begin{equation*}
		\Em_{\mu} \left[ \int_0^1 \cg |v(t)|^2, M_t \cd \D t \right] < + \infty.
		\end{equation*}
		Let us fix such a $\mu\in \M_{\delta}(\T^d)$, and $\xx = (x_1, \dots, x_p)$ such that $\mu = \delta_{x_1} + \dots + \delta_{x_p}$. Call $v^0$ the vector field on $\Omega^\xx$ defined for $\Prob^\xx$-almost all $(\omega, t, x) \in \Omega^\xx \times [0,1] \times \T^d$ by $v^0(\omega, t, x) = v(M^0(\omega), t, x)$. By Remark~\ref{rem:large_to_restricted}, $v^0$ is $(\bar \F_t)_{t \in [0,1]}$-predictable on $\Omega^\xx$, and we have
		\begin{equation*}
		\Em_{\xx} \left[ \int_0^1 \cg |v^0(t)|^2, M^0_t \cd \D t \right] = \Em_{\mu} \left[ \int_0^1 \cg |v(t)|^2, M_t \cd \D t \right] < + \infty.
		\end{equation*}
		Therefore, the process $\mathcal I[v^0]$ is well defined $\Prob^\xx$-\emph{a.s.}\ by Theorem~\ref{prop:generalized_stochastic_integral_tree}, and it is adapted w.r.t.\ $(\bar \F_t)_{t \in [0,1]}$ thanks to Proposition~\ref{prop:I_adapted_restricted}. Hence, by Remark~\ref{rem:large_to_restricted}, we deduce that there exists an adapted process on $\Omega$, defined $R^\mu$-\emph{a.s.}, that we call $I[v]$, and that satisfies $\Prob^\xx$-\emph{a.s.}, for all $t \in [0,1]$:
		\begin{equation}
		\label{eq:I_large_to_restricted}
		I[v]_t(M_0) = \mathcal I[v^0]_t .
		\end{equation} 
		Note that this definition only depends on the choice of $\mu$ and not on the choice of $\xx$. This is precisely because $\mathcal I[v^0]$ is adapted w.r.t.\ to the filtration generated by $M^0$, whose law does not depend on $\xx$.
		
		In that way, $I[v]$ is well defined $\R^\mu$-\emph{a.s.}, for $R^0$-almost all $\mu$, and so $R$-\emph{a.s.} The remaining part of the result is a direct consequence of Proposition~\ref{prop:generalized_stochastic_integral_tree} and formula~\eqref{eq:I_large_to_restricted}.
	\end{proof}
	
	Let us perform the same argument for the pure-jump processes, and prove Theorem~\ref{thm:properties_jump_processes}.
	\begin{proof}[Proof of Theorem~\ref{thm:properties_jump_processes}]
		By Definition~\ref{def:jump_processes} and Proposition~\ref{prop:jump_processes_tree}, for $R_0$-almost all $\mu$ and all $\xx = (x_1, \dots, x_p)$ such that $\mu = \delta_{x_1} + \dots + \delta_{x_p}$, we have $\Prob^\xx$-\emph{a.s.}, for all $t \in [0,1]$ and all $k \neq 1$:
		\begin{equation}
		\label{eq:J_large_to_restricted}
		J^k[a_k]_t = \mathcal J^k[a^0_k]_t,
		\end{equation}
		where $a^0_k$ is the $(\bar \F_t)_{t \in [0,1]}$-predictable field on $\Omega^\xx$ defined for all $(\omega, t, x) \in \Omega^\xx \times [0,1] \times \T^d$ by $a^0_k(\omega, t, x) := a(M^0(\omega), t, x)$.
		
		Therefore, \eqref{eq:jump_measure_J} and~\eqref{eq:truncation_jump_process} are direct consequences of~\eqref{eq:J_large_to_restricted}, \eqref{eq:jump_measure_tree} and~\eqref{eq:truncation_jump_tree}.
	\end{proof}
	
	Finally, we prove Theorem~\ref{thm:branching_Ito}.
	\begin{proof}[Proof of Theorem~\ref{thm:branching_Ito}]
		Let us consider $\mu \in \M_{\delta}(\T^d)$ and $\xx = (x_1, \dots, x_p)$ such that $\mu = \delta_{x_1} + \dots + \delta_{x_p}$. Up to a $\Prob^\xx$-negligible set, formulas~\eqref{eq:branching_Ito_tree}, \eqref{eq:I_large_to_restricted} and \eqref{eq:J_large_to_restricted} hold for all $t \in [0,1]$. Therefore, for all $\mu \in \M_{\delta}(\T^d)$, formula~\eqref{eq:branching_Ito} holds $R^\mu$-\emph{a.s.}, for all $t \in [0,1]$. Consequently, it also holds $R$-\emph{a.s.}, for all $t \in [0,1]$.
	\end{proof}

	\chapter{Representation of linear forms in Orlicz spaces}
	\label{app:riesz}
	
	In this section, we state and prove a Riesz type theorem for representing the linear forms on the set of functions having an exponential moment. By no mean we pretend that this result is new, and the book~\cite{rao1991theory} provides an extensive presentation of the general theory of Orlicz spaces, from which our result can be deduced. Let us also mention~\cite{leonard2001convex}, where the author explains how to adapt classical results of the field in the exotic case where the set of functions that we consider is not invariant through a change of sign. Still, we wanted to present a short and elementary proof of exactly what we need. Our proof uses the $L^1$ convergence of uniformly integrable martingales.
	
	The context is the following. Let us consider a measured space $(\XX, \G, \pi)$ and assume:
	\begin{itemize}
		\item The measure $\pi$ is nonnegative and finite.
		\item The $\sigma$-algebra $\G$ is generated by a countable family of sets.
	\end{itemize}
	
	The theorem is the following.
	\begin{Thm}
		\label{thm:riesz}
		Let us consider $(\XX, \G, \pi)$ as above. Let $\Lambda$ be a real linear functional defined on the set $\mathcal{L}^\infty(\XX, \G)$ of bounded measurable functions\footnote{Let us stress that \emph{a priori}, we allow $\Lambda(f)$ to be nonzero even if $f$ cancels $\pi$-almost everywhere. Actually, when~\eqref{eq:ineq_orlicz} holds, this cannot happen.}. Assume that there exists $C>0$ such that for all $f \in \mathcal{L}^\infty(\XX, \G)$,
		\begin{equation}
		\label{eq:ineq_orlicz}
		\Lambda(f) \leq C + \int \big\{ \exp(f) - 1 \big\} \D \pi.
		\end{equation}
		Then, there exists a unique $g \in L^1(\XX, \G, \pi)$ such that for all $f \in \L^\infty(\XX, \G)$,
		\begin{equation*}
		\Lambda(f) = \int fg \D \pi.
		\end{equation*}
		In particular, $\Lambda(f)$ does not depend on the value of $f$ on $\pi$-negligible sets.
		
		Finally, $g$ is nonnegative, and satisfies the following bound:
		\begin{equation}
		\label{eq:estim_g}
		\int \big\{ g \log g + 1 - g \big\}\D \pi \leq C.
		\end{equation}
	\end{Thm}
	\begin{Rem}
		The proof that we present can be easily adapted replacing $(\exp(x) - 1, y \log y + 1 -y)$ by other pairs $(\Psi(x), \Psi^*(y))$ of convex conjugate functions, where $\Psi^*$ is superlinear and $\Psi(0)=0$, under minimal additional assumptions. It can hence be used in greater generality to provide Riesz type theorems in Orlicz spaces. We decided to stick to our case because we wanted to treat its specificities (for instance, $\Psi$ is superlinear when $x \to + \infty$, but remains negative for $x <0$, and $\Psi^*$ is singular at zero) without introducing artificial assumptions.
	\end{Rem}
	
	In the case where $\pi$ is a probability measure, it is sometimes useful to consider the following corollary instead of Theorem~\ref{thm:riesz}. In this case, we call $\E_\pi$ the integral with respect to $\pi$. We do not formulate this result in the same way as Theorem~\ref{thm:riesz} because the context in which we apply it is a bit different, but the spirit is the same.
	\begin{Cor}
		\label{cor:riesz}
		Let us consider $(\XX, \G, \pi)$ as above, and let us assume that $\pi$ is a probability measure, that is, $\pi(\XX)=1$. Let $\Lambda$ be a real linear functional on the vector space of all variables whose absolute value have exponential moments at any order, that is, on
		\begin{equation*}
		\mathcal{V} := \Big\{ f:\XX \to \R, \quad \G\mbox{-measurable, s.t. } \forall \kappa>0,\ \E_\pi\big[ \exp\big( \kappa |f| \big) \big] < + \infty \Big\}.
		\end{equation*}
		Assume that there exists $C>0$ such that for all $f \in \mathcal{V}$, there holds
		\begin{equation}
		\label{eq:ineq_cor_riesz}
		\Lambda(f) \leq \inf_{\kappa > 0} \frac{1}{\kappa} \Big( C + \log \E_\pi\big[ \exp\big( \kappa f \big) \big] \Big).
		\end{equation}
		Then there exists $\pi' \in \P(\XX)$ such that $H(\pi' | \pi) \leq C$ and for all $f \in \mathcal{V}$, $\Lambda(f) = \E_{\pi'}[f]$.
	\end{Cor}
	
	We first prove Theorem~\ref{thm:riesz}, and then Corollary~\ref{cor:riesz}.
	\begin{proof}[Proof of Theorem~\ref{thm:riesz}]
		Before entering into the core of the proof, let us just mention that the uniqueness of $g$ is obvious, so we will not talk about it.
		
		\begin{steph}{Reduction to probability spaces}
			Let us show that it suffices to prove the existence of $g$ in the case where $\pi(\XX) = 1$. Let us assume that the result is true in that case, and let us consider $(\XX, \G, \pi)$ and $\Lambda$ as in the statement of the theorem. By our assumption, the result is true in $(\XX, \G, \tilde \pi)$, where $\tilde \pi = \pi / \pi(\XX)$. Let us call $\tilde \Lambda (f) := \Lambda (f / \pi(\XX))$, $f \in \L^\infty(\XX, \G)$. Dividing~\eqref{eq:ineq_orlicz} by $\pi(\XX)$, we get for all $f \in \L^\infty(\XX, \G)$:
			\begin{equation*}
			\tilde \Lambda(f) \leq \frac{C}{\pi(\XX)} + \int \big\{ \exp(f) - 1 \big\}\D \tilde \pi.
			\end{equation*}
			So using the result in $(\XX, \G, \tilde \pi)$ we find a nonnegative $g \in L^1(\XX, \G, \tilde \pi) = L^1(\XX, \G, \pi)$ such that for all $f \in \L^\infty(\XX, \G)$,
			\begin{equation*}
			\Lambda(f) = \pi(\XX) \tilde \Lambda(f) = \pi(\XX) \int fg \D \tilde \pi = \int fg \D \pi.
			\end{equation*}
			Finally, we have
			\begin{equation*}
			\int \big\{ g \log g + 1 - g \big\}\D \pi = \pi(\XX) \int \big\{ g \log g + 1 - g \big\}\D \tilde \pi \leq \pi(\XX) \frac{C}{\pi(\XX)} = C,
			\end{equation*}
			so that the result also holds in $(\XX, \G, \pi)$. From now on, we assume that $\pi(\XX) = 1$ and we denote by $\E_\pi$ the integral w.r.t.\ $\pi$.
		\end{steph}
		
		\begin{steph}{Conditioning by finite $\sigma$-algebras}
			
			Let $(A_k)_k \in \N^*$ be a countable family of sets generating $\G$, and call $\G_n := \sigma(A_1, \dots, A_n)$, $n \in \N^*$, the $\sigma$-algebra generated by $A_1, \dots, A_n$. For all $n \in \N^*$, $\G_n$ is a finite $\sigma$-algebra. It is hence also generated by a finite partition $B_1^n, \dots, B_{p_n}^n \in \G$, where $p_n \in \N^*$ is the cardinality of this partition.
			
			For $i = 1, \dots, p_n$, let us call $y_i := \Lambda(\1_{B_i^n})$.
			Applying~\eqref{eq:ineq_orlicz} to $f = x \1_{B_i^n}$ for $x \in \R$, leads to
			\begin{equation*}
			x y_i = \Lambda(x \1_{B^n_i}) \leq C + \big\{ \exp(x) - 1 \big\} \pi(B^i_n).
			\end{equation*}
			This inequality has two interesting consequences. First, $y_i \geq 0$, else the inequality would be violated for $x \to - \infty$. Second, if $\pi(B_i^n) = 0$, then $y_i=0$, else the inequality would be violated for $x \to + \infty$. Using these remarks and the convention $0/0 = 0$, we define the following nonnegative $\G_n$-measurable function
			\begin{equation*}
			g_n := \frac{y_1}{\pi(B_1^n)} \1_{B^n_1} + \dots + \frac{y_{p_n}}{\pi(B^n_{p_n})} \1_{B^n_{p_n}}.
			\end{equation*}
			With these definitions we have for all $f \in \L^\infty(\XX, \G)$,
			\begin{equation}
			\label{eq:approx_lambda}
			\Lambda\Big(\E_\pi\big[f\big|\G_n\big]\Big) = \E_\pi\big[f g_n \big].
			\end{equation}
		\end{steph}
		
		\begin{steph}{$(g_n)_{n \in \N}$ is a nonnegative uniformly integrable martingale}
			Let us show that under $\pi$, $(g_n)_{n \in \N}$ is a martingale w.r.t.\ the filtration $(\G_n)_{n \in \N}$. For all $n \in \N^*$, $g_n$ is $\G_n$-measurable, so it is enough to show that for any $f$ that is $\G_{n}$-measurable and bounded, we have $\E_\pi[fg_{n+1}] = \E_\pi[fg_n]$. For such an $f$, we have $\E_\pi[f|\G_{n+1}] = \E_\pi[f|\G_n] = f$. So using \eqref{eq:approx_lambda}, we get
			\begin{equation*}
			\E_\pi[fg_{n+1}] = \Lambda\Big(\E_\pi\big[f\big|\G_{n+1}\big]\Big) = \Lambda\Big(\E_\pi\big[f\big|\G_n\big]\Big) = \E_\pi[fg_n].
			\end{equation*}
			
			We already showed that $(g_n)_{n \in \N^*}$ is nonnegative, let us show that it is uniformly integrable. For this, we just have to use~\eqref{eq:ineq_orlicz} with the $\G_n$ measurable function $f = \log g_n \1_{g_n>0}$, which is bounded. This leads to
			\begin{equation*}
			\E_\pi\big[g_n \log g_n\big]  = \Lambda\Big( \E_\pi\big[ f \big| \G_n \big] \Big) = \Lambda(f) \leq C + \E_\pi\big[ \exp( f ) - 1 \big] = C + \E_\pi\big[g_n -1\big],
			\end{equation*} 
			that we rewrite
			\begin{equation}
			\label{eq:unif_bound_g_n}
			\E_\pi\big[g_n \log g_n + 1 - g_n\big] \leq C,
			\end{equation}
			and which provides uniform integrability.
		\end{steph}
		
		\begin{steph}{Convergence}
			By standard martingale convergence theorems, there is a nonnegative $\G$-measurable function $g$ such that $g_n \to g$ as $n \to + \infty$, $\pi$-\emph{a.e.} and in $L^1(\XX,\G, \pi)$. Inequality~\eqref{eq:estim_g} follows from~\eqref{eq:unif_bound_g_n} by Fatou's lemma. 
			
			We need to show that for all $f \in \L^\infty(\XX, \G)$, we have $\Lambda(f) = \E_\pi[fg]$. We will do it by showing that both side of this equality is the limit of $\Lambda(\E_\pi[f|\G_n])$ as $n \to + \infty$. The r.h.s.\ is a consequence of the $L^1$ convergence:
			\begin{equation*}
			\Lambda\Big(\E_\pi\big[f\big|\G_n\big]\Big) = \E_\pi\big[ f g_n \big] \underset{n \to +_\infty}{\longrightarrow} \E_\pi[fg].
			\end{equation*}
			The l.h.s.\ is a bit more tricky as it amounts to some continuity of $\Lambda$. We will use~\eqref{eq:ineq_orlicz}. For all $h \in \L^\infty(\XX, \G)$ and $\eps>0$, inequality~\eqref{eq:ineq_orlicz} implies
			\begin{equation*}
			\Lambda(h) = \frac{\eps}{C} \Lambda \left( \frac{C}{\eps}h \right) \leq \eps + \frac{\eps}{C}\E_\pi\left[\exp\left( \frac{C}{\eps}h \right) - 1 \right].
			\end{equation*}
			Writing the same inequality for $-h$ instead of $h$ and bounding both $h$ and $-h$ by $|h|$ in the r.h.s.\ leads to
			\begin{equation*}
			\big|\Lambda(h)\big| \leq \eps + \frac{\eps}{C}\E_\pi\left[\exp\left( \frac{C}{\eps}|h| \right) - 1 \right].
			\end{equation*}
			As for all $x \in \R$, $\exp x - 1 \leq x \exp x$, calling $K:= \sup |h|$, we have
			\begin{equation*}
			\big|\Lambda(h)\big| \leq \eps + \E_\pi\left[ \exp\left( \frac{C}{\eps}|h|\right) |h| \right] \leq \eps + \exp\left( \frac{CK}{\eps} \right)\E_\pi\big[ |h| \big].
			\end{equation*}
			This inequality shows that $\Lambda$ is continuous in the $L^1$ distance when restricted to the bounded sets of $\L^\infty(\XX, \G)$. Indeed for all $h$ in the centered ball of $\L^\infty$ of radius $K$, $|\Lambda(h)| \leq F_K(\|h\|_1)$, where 
			\begin{equation*}
			F_K: N \in \R_+ \mapsto \inf_{\eps>0} \left\{ \eps + \exp\left( \frac{CK}{\eps} \right)N \right\}
			\end{equation*}
			is a nonnegative upper semi-continuous concave function canceling at zero. We conclude by applying this continuity property to the convergence $\E_\pi[f|\G_n] \to f$ in $L^1(\XX, \G, \pi)$, for all $f \in \L^\infty(\XX, \G)$.
		\end{steph}
	\end{proof}
	
	\begin{proof}[Proof of Corollary~\ref{cor:riesz}]
		As $\pi$ is finite, we have $\L^\infty(\XX, \G) \subset \mathcal{V}$. Moreover, for all $f \in \L^\infty(\XX, \G)$, taking $\kappa = 1$ in the r.h.s. of~\eqref{eq:ineq_cor_riesz}, exploiting the inequality $\log z \leq z - 1$, $z \in \R_+^*$ and using $\pi(\XX)=1$, we get
		\begin{equation*}
		\Lambda(f) \leq C + \E_\pi\big[ \exp\big( f \big) - 1 \big].
		\end{equation*}
		So Theorem~\ref{thm:riesz} applies, and we get a nonnegative $g \in L^1(\XX, \G, \pi)$ satisfying the bound
		\begin{equation*}
		\E_\pi\big[ g \log g + 1 - g \big] \leq C,
		\end{equation*} 
		and such that for all $f\in \L^\infty(\XX, \G)$, $\Lambda(f) = \E_\pi[fg]$. Applying~\eqref{eq:ineq_cor_riesz} with $\kappa := 1$ and constant functions $f \equiv x \in \R$, we find $\Lambda(f) = x \E_\pi[g] \leq C + x$. This is possible for all $x \in \R$ only provided $\E_\pi[g] = 1$. Therefore, we set $\pi' := g \cdot \pi$, and we conclude that $\pi'$ is such that $H(\pi' | \pi) \leq C$, and satisfies $\Lambda(f) = \E_{\pi'}[f]$, for all $f \in \L^\infty(\XX, \G)$. 
		
		So to close the proof, we just need to show that $\Lambda(f) = \E_{\pi'}[f]$ extends to all $f \in \mathcal{V}$. Using the decomposition $f = f_+ - f_-$ and the linearity of $\Lambda$ and $\E_{\pi'}$, we can assume that $f$ is nonnegative. Let us call $f_n := \min(f,n)$, $n \in \N$. For all $n$, $f_n \in \L^\infty(\XX, \G)$, so $\Lambda(f_n) = \E_{\pi'}[f_n]$, and it remains to show that as $n \to \infty$, $\Lambda(f_n) \to \Lambda(f)$, and $\E_{\pi'}[f_n] \to \E_{\pi'}[f]$. The second limit follows from the monotone convergence theorem. The
		first one follows from the following observations. First, if $h \in
		\mathcal{V}$ is nonnegative, then so is $\Lambda(h)$. Indeed,
		applying~\eqref{eq:ineq_cor_riesz} to $-h$ leads to
		\begin{equation*}
		-\Lambda(h) \leq \lim_{\kappa \to +\infty} \frac{1}{\kappa}\Big(
		C + \log \E_\pi\big[\exp\big( - \kappa h \big)\big] \Big) =0,
		\end{equation*}
		where the limit is obtained using the dominated convergence theorem.
		Second, for $\eps>0$, let us take $\kappa :=C / \eps$. For $n
		\in \N$, applying inequality~\ref{eq:ineq_cor_riesz} to the nonnegative
		function $f - f_n$, we get
		\begin{equation*}
		\limsup_{n \to + \infty}|\Lambda(f) - \Lambda(f_n)| = \limsup_{n \to + \infty}\Lambda(f-f_n) \leq \eps
		+ \frac{\eps}{C} \limsup_{n \to + \infty} \log \E_\pi\left[ \exp\left(
		\frac{C}{\eps}\big\{f-f_n\big\}\right)\right] \leq \eps,
		\end{equation*}
		and the result follows.
	\end{proof}

	\newpage
	
	\addcontentsline{toc}{chapter}{Bibliography}
	
	\bibliographystyle{alpha}
	
	\bibliography{bibliography}
	
\end{document}